\documentclass[a4paper,12pt]{report}

\usepackage{verbatim}
\usepackage{graphicx}
\usepackage{float}

\usepackage{amsmath}
\usepackage{amsthm}
\usepackage{amssymb}
\usepackage{amsbsy}
\usepackage{braket}

\usepackage{fouridx}
\usepackage{fullpage}

\usepackage{a4wide}
\usepackage{a4}

\usepackage[bottom]{footmisc}

\usepackage{longtable}

\usepackage{chngcntr}
\counterwithout{footnote}{chapter}

\usepackage{enumitem}
\usepackage{soul}

\newcommand\relphantom[1]{\mathrel{\phantom{#1}}}

\numberwithin{equation}{chapter}
\numberwithin{figure}{chapter}
\newtheorem{Theorem}{Theorem}
\numberwithin{Theorem}{chapter}

\newtheorem{Lemma}[Theorem]{Lemma}
\newtheorem{Definition}[Theorem]{Definition}
\newtheorem{Remark}[Theorem]{Remark}

\newtheorem{Corollary}[Theorem]{Corollary}

\usepackage[toc,page]{appendix}

\allowdisplaybreaks

\begin{document}

\begin{titlepage}
\centering
\LARGE 
\textcaps{University of Reading} \\ 
\textcaps{Department of Mathematics and Statistics}\\
\vspace{6cm}
Stochastic Resonance for a Model with Two Pathways\\
\vspace{5cm}
\Large Tommy Liu 
\vfill
\normalsize
Thesis submitted for the Degree of Doctor of Philosophy\\
September 2016
\end{titlepage}

\newpage\null\thispagestyle{empty}\addtocounter{page}{-1}\newpage


\chapter*{Abstract}
\addcontentsline{toc}{chapter}{Abstract}

In this thesis we consider stochastic resonance for a diffusion with drift given by a potential, which has two metastable states and two pathways between them. Depending on the direction of the forcing the height of the two barriers, one for each path, will either oscillate alternating or in synchronisation. 

We consider a simplified model given by  discrete and continuous 
 time Markov Chains with two states. 
This was done for  alternating and synchronised wells. 
The invariant measures are derived for both cases and shown to be constant for the synchronised case. 
A PDF for the escape time from an oscillatory potential is reviewed. 

Methods of detecting stochastic resonance are presented, which are linear response, signal-to-noise ratio, energy, out-of-phase measures, relative entropy and entropy.  
A new statistical test called the  conditional Kolmogorov-Smirnov test is developed, which can be used to analyse stochastic resonance. 

An explicit two dimensional potential is introduced, the critical point structure derived and the dynamics, the invariant state  and escape time studied numerically.

The six measures are unable to detect the stochastic resonance in the case of synchronised saddles. The distribution of escape times however not only shows a clear sign of stochastic resonance, but changing the direction of the forcing from alternating to synchronised  saddles an additional resonance at double the forcing frequency starts to appear. The conditional KS test  reliably detects the stochastic resonance even for forcing quick enough and for data so sparse that the stochastic resonance is not obvious directly from the histogram of escape times.

\chapter*{Declaration}
\addcontentsline{toc}{chapter}{Declaration}
I confirm that this is my own work and the use of all material from other sources
has been properly and fully acknowledged.
\\
\\
\noindent Tommy Liu

\chapter*{Acknowledgement}
\addcontentsline{toc}{chapter}{Acknowledgement}
I would like to thank Tobias Kuna for supervising this thesis;
Valerio Lucarini for co-supervising;
Tristan Pryer, 
Horatio Boedihardjo,
Martin Kolb
and the late Professor Alexei Likhtman 
for being on the Monitoring Committee;
Jochen Broecker and Ostap Hryniv for being the examiners on my viva; 
Peter Imkeller for helpful discussions; 
Pawel Stasiak for introducing me to the Meteorology Computer Clusters; 
Peta-Ann King and Sue Davis for their pastoral care; 
the EPSRC for funding
and finally to my family and friends for their support over the years. 
\\
\\
\noindent Tommy Liu\\
\noindent September 2016\\
\noindent University of Reading

\tableofcontents

\listoffigures

\chapter*{Introduction}
\addcontentsline{toc}{chapter}{Introduction}

\section*{Outline of Problem}
\addcontentsline{toc}{section}{Outline of Problem}

Consider the following problem. 
Let $X_t^\epsilon$ be the random variable describing the trajectory of a diffusion process in $\mathbb{R}^r$ where $t$ is the time and $\epsilon^2$ is the variance level. 
More precisely we consider processes described by the following type of stochastic differential equation 
\begin{align*}
dX^\epsilon_t=b\left(X^\epsilon_t,t\right)dt+\epsilon\,dW_t
\end{align*}
where $b:\mathbb{R}^r\times \mathbb{R}\longrightarrow \mathbb{R}^r$ and $W_t$ is a Wiener process in $\mathbb{R}^r$. 
We suppose that the drift term $b$ has the form
\begin{align*}
b(x,t)=-\nabla V_0 (x) + F\cos \Omega t 
\end{align*}
where $F,x\in\mathbb{R}^r$ and $V_0:\mathbb{R}^r\longrightarrow \mathbb{R}$ is called the unperturbed potential. 
We consider unperturbed potentials with two or more minimas (wells). 
Most importantly, we consider potentials where there are multiple pathways between the wells. 
To our knowledge systems with two pathways have not been studied in the context of stochastic resonance.

Consider the case $\Omega=0$ and where the noise $\epsilon$ is very small. 
The particle will stay very close to one of the wells of the potential and will occasionally escape to the other well. 
The time of the actual transition from one well to the other is very short compare to the time it stays in any particular well. 

Now consider the case where $\Omega>0$. For particular choices of $\Omega>0$ and $\epsilon>0$, these transitions between the two wells will become synchronised with the driving frequency $\Omega$. 
This is called stochastic resonance. 
Thus the term \emph{noise induced synchronisation} was used for systems where the amplitude of the forcing $F$ was not large \cite{tk_14, tk_15} (see also the discussions in \cite{tk_16}). 
New insights into the exact manner of these synchronised transitions will be studied in this thesis,
which may be more appropriate  in light of the results obtained in this thesis.

For small noise $\epsilon$, one would expect that stochastic resonance depends only on the essential properties of the system, such as the height difference between the wells and the pathways for escape. 
We investigate what effects these multiple pathways have on the appearance of stochastic resonance. 
Varying $F$, $\Omega$ and $\epsilon$ should thus reveal the qualitative structure of the unperturbed potential $V_0$. 
In this thesis we test this paradigm by studying a two dimensional example with two wells and two independent pathways between them, see Chapter 
\ref{sum_chap_mexican_hat_toy_model}.

\section*{Historical Background}
\addcontentsline{toc}{section}{Historical Background}

Stochastic resonance has attracted  interest among mathematicians and physicist. 
An overview of the studies that have occurred in both physics and mathematics are given here. 

\subsection*{Physical Background}
\addcontentsline{toc}{subsection}{Physical Background}

Stochastic resonance was first observed in 1981 \cite{benzi81,tk_17c,tk_17b}. 
The first example \cite{benzi81} considered transitions between two metastable states to model the cyclic occurrences of ice ages. 
Since then many examples of stochastic resonance were found in optics 
\cite{opt_1_PhysRevLett.60.2626,
opt_2_Guidoni1995,
opt_3_PhysRevA.49.2199,
opt_4_PhysRevLett.68.3375},
electronics 
\cite{elec_1_FAUVE19835,
elect_2_PhysRevE.49.R1792,
elect_3_Mantegna1995,
elect_4_PhysRevLett.76.563,
elect_5_PhysRevLett.74.3161,
elect_6_:/content/aip/journal/jap/76/10/10.1063/1.358258,
elect_7_PhysRevA.39.4323,
elect_8_PhysRevE.49.4878,
elect_9_PhysRevLett.67.1799},
neuronal systems 
\cite{neuron_1_PhysRevLett.67.656}, 
quantum systems 
\cite{quant_1_:/content/aip/journal/jap/77/6/10.1063/1.358720,
quant_2_:/content/aip/journal/apl/66/1/10.1063/1.114161} 
and paddlefish 
\cite{fish_1_PhysRevLett.84.4773,
fish_2_FREUND200271}.
Stochastic resonance could be thought of as quasi-deterministically periodic transition between two metastable states. 
For example, the climate of the Earth could be modelled by two states. There is a state corresponding to an Ice Age and another corresponding to the opposite of an Ice Age, a so-called ``Hot Age''. As the Earth's climate cyclically changes many times between Cold Ages and Hot Ages, its behaviour could be modelled by stochastic resonance.

A range of techniques for example
linear response 
\cite{lin_1_TUS:TUS1787,
lin_2_doi:10.1137/0143037}, 
signal-to-noise ratio
\cite{sign_1_PhysRevA.39.4668,
sign_2_escape_2_PhysRevA.41.4255}
and distribution of escape times
\cite{escape_1_PhysRevA.42.3161,
sign_2_escape_2_PhysRevA.41.4255,
escape_3_PhysRevE.49.4821}
were used to define, analyse and study stochastic resonance. 
These techniques along with other examples of stochastic resonance are reviewed in the long overview paper  by Gammaitoni, H\"anggi, Jung and Marchesoni \cite{RevModPhys.70.223}. 
We will evaluate the usefulness of some of these techniques for our problem, see Chapter \ref{sum_chap_results}.

\subsection*{Mathematical Background}
\addcontentsline{toc}{subsection}{Mathematical Background}

There are various mathematical studies of stochastic resonance. 
These often involve different orders of approximations for small noise levels. 
The first and second order of approximations are discussed below.
Adiabatic large deviation is also presented. 

In the first leading order of approximation, a key element of study is to control the escape times from the wells as given by the so called large deviation theory, see the monograph of Freidlin and Wentzell \cite{freidlin98}.
The distribution of the exit time was derived by Day in \cite{tk_19} and by 
Galves,
Kifer, Olivieri and Vares
\cite{freid_2_doi:10.1137/1119057,
oliveri_expoen,
freid_3}. 
To go beyond leading order has been much more difficult for the transition problem between two wells as WKB theory could up to now not be rigorously applied.

The next order of approximation was rigorously derived by  Bovier, Eckhoff, Gayrard, Klein \cite{Bovier02metastabilityin} and Berglund and Gentz  \cite{kram_2_2008arXiv0807.1681B} using techniques from potential theory.  
Berglund and Gentz in a series of papers studied the situation of low, non-quadratic barriers and drifts not given by autonomous potentials 
\cite{kram_2_2008arXiv0807.1681B, tk_16}. 
A review of different techniques used to derive Kramers' formula can be found in the review paper
\cite{krammer_review_2011arXiv1106.5799B}. 

In \cite{ld_1_comp_1_Freidlin2000333} Friedlin considered stochastic resonance in the adiabatic regime. 
This means the diffusion can effectively be described by a Markov process which describes the jumps between wells. 
This problem was revisited by Hermann, Imkeller and Pavlyukevich, see Chapter 4 in \cite{tran2014} and references therein, to derive results uniformly for varying time scale to identify the optimal resonance point asymptotically for small noise even outside the adiabatic regime leading to different logarithmic corrections including the famous cycling effect discovered by Day \cite{tk_20}, see also \cite{tk_18} for the connection with stochastic resonance.
Escape time outside of adiabatic regime is studied in \cite{tk_00_doi:10.1137/120887965}.

As mentioned above in leading order the transitions of the diffusion process $X^\epsilon_t$ between the wells can be approximated by a two state Markov Chain $Y^\epsilon_t=\pm1$ which have been studied \cite{pav_thesis,pav_imkell,pav03,tran2014}. 
Further comparative studies of the stochastic resonance for the diffusion case $X^\epsilon_t$ versus the Markov Chain $Y^\epsilon_t$ case were done by 
Hermann, Imkeller, Pavlyukevich and Peithmann 
in 
\cite{ld_3_comp_4_imkeller2004stochastic,
comp_2_Herrmann2005,
ld_2_comp_3_herrmann2005large,
herrmann2005}.
A collection of papers on comparative studies between stochastic resonance in  diffusion and Markov Chains can be found in the monograph \cite{tran2014}. 
One of the main conclusions in 
\cite{ld_2_comp_3_herrmann2005large,
herrmann2005,
comp_2_Herrmann2005,
tran2014}
is rigorously showing that using linear response and signal-to-noise ratio to analyse  stochastic resonance in the diffusion case $X^\epsilon_t$ gives a different result to analysing the Markov Chain case $Y^\epsilon_t=\pm1$ with the same techniques even asymptotically in the small noise limit.
Other common methods used to study stochastic resonance  include invariant measures and Fourier transforms. 
We consider six measures of stochastic resonance frequently used and considered by Pavlyukevich in his thesis \cite{pav_thesis,tran2014} which are  linear response, signal-to-noise ratio, energy, out-of-phase measure, relative entropy and entropy.

In this thesis we will study stochastic resonance on a two dimensional toy model, in  both the diffusion and Markov Chain cases, and where there are two independent pathways between the wells going through two different saddles. 
The escape times and the six measures of stochastic resonance introduced above are studied.

\section*{Summary of Research}
\addcontentsline{toc}{section}{Summary of Research}

In Chapter \ref{sum_chap_stochastic_resonance}
we review the first model in which  stochastic resonance was observed, that is, we are considering the unperturbed potential 
\begin{align*}
V_0(x)=\frac{x^4}{4}-a\frac{x^2}{2}
\end{align*}
and the corresponding SDE 
\begin{align*}
dX^\varepsilon_t=\left[-\nabla V_0+F\cos(\Omega t)\right]dt+\epsilon\,dW_t .
\end{align*}
In one dimension the escape time can be explicitly computed as the solution to an ODE and using Laplace method asymptotic formulas can be derived. 
In Chapter \ref{sum_chap_escape_static} a review of large deviation theory and results concerning escape times are given. 
In Chapter \ref{chap_sect_kram} further results, based on potential theory, are given and the analogue of Kramers' formula for our case is presented. 
In Chapter \ref{sum_chap_escape_oscill} discrete and continuous time Markov Chains are considered. 
The associated invariant measures and the relaxation time to this invariant measure is derived for alternating and synchronised wells. 
The probability density  function of escape times is derived as well. 
In Chapter \ref{sum_chap_analysis_theory} the six measures used to analyse stochastic resonance mentioned above are introduced. 
Furthermore, methods used  to study escape times are given and in particular a new version of the Kolmogorov-Smirnov test suitable for this problem is  discussed. 
In Chapter \ref{sum_chap_mexican_hat_toy_model} the main model under consideration in this thesis is studied, which has two wells and two saddles. The two wells are connected through to independent pathways each through one of the saddles. 
Due to its form, we nicknamed it the Mexican Hat Toy Model
\begin{align*}
V_0(x,y)=\frac{1}{4}r^4-\frac{1}{2}r^2-ax^2+by^2
\quad \text{where} \quad r=\sqrt{x^2+y^2}.
\end{align*}
We rigorously derive the qualitative structure of the potential with and without external forcing.
In Chapter \ref{sum_chap_numerical_methods} the numerical methods used to simulate the associated SDE
\begin{align*}
dx&=\left[-\frac{\partial V_0}{\partial x}+F_x\cos \Omega t \ \right]dt+\epsilon \ dw_x \\
dy&=\left[-\frac{\partial V_0}{\partial y}+F_y\cos \Omega t \ \right]dt+\epsilon \ dw_y
\end{align*}
are discussed and non rigorous estimates of all relevant error sources are given  necessary to be confident about the precision of the simulation needed.
The $dw_x$ and $dw_y$ are $x$ and $y$ components of the two dimensional Wiener processes. 
The numerical algorithm used is the Euler method \cite{num_2002IJMPC..13.1177M} which is sufficiently accurate for our purposes.
In Chapter \ref{sum_chap_results} the results from simulating the SDE are presented and interpreted. The six measures are studied and the quality of the approximation by the aforementioned Markov chains is tested using the Kolomogorov-Smirnov test developed. The results were repeated in a sparse data context. 

In Chapter \ref{sum_chap_results} the main findings and conclusions of this thesis are presented. It is shown that the six measures are unable to detect stochastic resonance in the case of synchronised saddles. The six measures show no sharp signature as the saddles change from alternating to synchronised saddles. This is due to the fact that the invariant measures are constant for synchronised saddles.  By contrast, not only did the distribution of escape times show a signature for stochastic resonance with synchronised saddles;  the distribution of escape times did show a clear signature as the saddles change from alternating to synchronised, by exhibiting signatures which we call the Single, Intermediate and Double Frequency. The newly developed conditional Kolomogorov-Smirnov test was shown to be a good method to analyse the statistics of the escape times. 

This thesis then finishes with a conclusion of the results obtained. 
In Appendix \ref{appendix_potential} the conventions used are collected. 
In Appendix \ref{append_further_methods} the methods used to calculate the Fourier transforms and the escape times are explained.

\chapter{Stochastic Resonance}
\label{sum_chap_stochastic_resonance}
The earliest known and simplest example of stochastic resonance is reviewed. 
This was done in 1981 \cite{benzi81}. 
Properties about its escape times are derived. 
Estimates for the resonance noise level $\epsilon_{res}$ are given.
The techniques involved include a review of Laplace method.
This study only works for $\epsilon$ in the small noise approximation.  
Only one dimensional systems will be studied in this Chapter. 
Deducing properties about the underlying potential is trivial. 

\section{Laplace Method}

The main technique used to study exit times in one dimension is the so called Laplace Method. 
For completeness and to get a better understanding of the mechanism we are going to study, a proof will be provided later on. 

\begin{Theorem}\label{chap_1_laplace_method}
(Laplace Method) Let $f:[a,b] \rightarrow \mathbb{R}$ be twice differentiable on $[a,b]$. Let $x_0\in(a,b)$ be unique such that $f(x_0)=\max_{x\in(a,b)}f(x)$. Assuming $f''(x)$ is continuous on $[a,b]$ with $f'(x_0)=0$ and $f''(x_0)<0$ then 
\begin{equation*}
\lim_{n\rightarrow \infty} \left(  \frac{  \int^b_a e^{   nf(x) }    dx}{e^{nf(x_0)} \sqrt{      \frac{2\pi}{n(-f''(x_0))}        }}\right)=1. 
\end{equation*}
\end{Theorem}

\noindent A Corollary follows from Laplace Method as a special  case of Theorem \ref{chap_1_laplace_method}. 

\begin{Corollary}\label{chap_1:ts1}
Let $f:[a,b] \rightarrow \mathbb{R}$ be twice differentiable on $[a,b]$. Let $x_0=a$ or $x_0=b$ be unique such that $f(x_0)=\max_{x\in[a,b]}f(x)$. Assuming $f''(x)$ is continuous on $[a,b]$ with $f'(x_0)=0$ and $f''(x_0)<0$ then 
\begin{equation*}
\lim_{n\rightarrow \infty} \left(  \frac{  \int^b_a e^{   nf(x) }    dx}{\frac{1}{2}e^{nf(x_0)} \sqrt{      \frac{2\pi}{n(-f''(x_0))}        }}\right)=1. 
\end{equation*}
\end{Corollary}

\noindent We recall Taylor's Remainder Theorem which is needed in the proofs. 

\begin{Theorem}
Suppose that $f:\mathbb{R} \rightarrow \mathbb{R}$ is  $(n+1)$ times differentiable on $\mathbb{R}$. Let $x,a \in \mathbb{R}$, with $x>a$ then $f$ can be expressed as 
\begin{equation*}
f(x)=f(a) + \frac{f'(a)}{1!}(x-a)+\frac{f''(a)}{2!}(x-a)^2 + \dots + \frac{f^n(a)}{n!}(x-a)^n + R_{n+1}(x)
\end{equation*}
where $R_{n+1}$ the remainder can be expressed as
\begin{align*}
\text{Integral Form} \ \ R_{n+1}(x)&=\frac{1}{n!} \int^x_a(x-t)^nf^{n+1}(t)dt \\[0.5em]
\text{Lagrange Form} \ \ R_{n+1}(x)&=\frac{f^{n+1}    (\xi)}{(n+1)!}(x-a)^{n+1} \ , \ \xi\in[a,x]. 
\end{align*}
\end{Theorem}

\noindent The following simple Lemma is also needed in the proof of Laplace Method. 

\begin{Lemma}\label{chap_1_lemma_inf}
Let $f:[a,b]\longrightarrow \mathbb{R}$ be continuous.
Let $x_0$ be a unique maximum such that $f(x_0)=\max_{x\in[a,b]}f(x)$,
then for any fixed $\delta>0$, there exists an $\eta>0$ such that for any $s\notin (x_0-\delta, x_0+\delta)$ we have 
\begin{align*}
\eta\leq f(x_0)-f(s). 
\end{align*}
\end{Lemma}

\noindent Now we review proofs of the methods needed. 

\begin{proof}[Proof of Lemma \ref{chap_1_lemma_inf}]
We know that $x_0$ is the unique maximum, which means 
\begin{align*}
0<f(x_0)-f(s)
\end{align*}
for any $s\notin(x_0-\delta, x_0+\delta)$. 
This means the infinum of the set is bounded by zero
\begin{align*}
\inf \left\{ f(x_0)-f(s) : s\notin (x_0-\delta, x_0+\delta) \right\}
\geq0. 
\end{align*}
Suppose that the infinum of the set is zero
\begin{align*}
\inf \left\{ f(x_0)-f(s) : s\notin (x_0-\delta, x_0+\delta) \right\}
=0
\end{align*}
and yet all elements of the set are strictly greater than zero. 
This means  some members would be arbitrarily close to zero, 
\begin{align*}
0<f(x_0)-f(s)<\epsilon
\end{align*}
where $\epsilon$ is arbitrarily small. 
There exists a sequence 
\begin{align*}
(s_n)_{n\geq1}\subset [a,b]\backslash(x_0-\delta,x_0+\delta)=[a,x_0-\delta]\cup[x_0+\delta,b]
\end{align*}
such that 
\begin{align*}
0<f(x_0)-f(s_n)<\frac{1}{n}
\Longrightarrow
f(x_0)<f(s_n)+\frac{1}{n}. 
\end{align*}
But this sequence is in a compact set, which must have a subsequence which converges to a member $s'\in[a,x_0-\delta]\cup[x_0+\delta,b]$, that is 
\begin{align*}
f(x_0)\leq f(s')
\end{align*}
which contradicts the fact $x_0$ is the unique maximum. 
This implies that 
\begin{align*}
\inf \left\{ f(x_0)-f(s) : s\notin (x_0-\delta, x_0+\delta) \right\}
>0
\end{align*}
so the $\eta>0$ as in the assertion of the Lemma must exist. 
\end{proof}

\begin{proof}[Proof of Theorem \ref{chap_1_laplace_method}]
A differentiable function is also a continuous function. 
Since $f(x_0)=\max_{x\in[a,b]}f(x)$ we can say  $f'(x_0)=0$. 
Using the Taylor's Remainder Theorem we can rewrite $f(x)$ for $x\in [x_0, x_0+\delta]$ for some $\delta>0$ and $\xi\in[x_0,x]$ as 
\begin{align*}
f(x)&=f(x_0)+\frac{f''(\xi)}{2}(x-x_0)^2. 
\end{align*}
We can obtain an upper and lower bound for $f''(\xi)$ by exploiting its continuity on $[a,b]$. 
Since $x\in [x_0, x_0+\delta]$ we must also have $\xi\in [x_0, x_0+\delta]$. 
So 
\begin{equation*}
|\xi-x_0| \leq \delta .
\end{equation*}
For any $\epsilon>0$  and for a sufficiently small $\delta$, we can have 
\begin{equation*}
|f''(\xi)-f''(x_0)|<\epsilon. 
\end{equation*}
This means we can say 
\begin{align*}
-\epsilon < f''(\xi) &- f''(x_0) < \epsilon \\
f''(x_0)-\epsilon < &f''(\xi) < f''(x_0) +\epsilon 
\end{align*}
which gives 
\begin{align}
f(x) &\leq f(x_0) + \frac{1}{2} (f''(x_0)+\epsilon)(x-x_0)^2 \label{chap_1_laplace_method_eqn2}\\ 
f(x) &\geq f(x_0) + \frac{1}{2} (f''(x_0)-\epsilon)(x-x_0)^2.\label{chap_1_laplace_method_eqn3}
\end{align}
We start with the lower bound for $f(x)$ as in Equation \ref{chap_1_laplace_method_eqn2}
\begin{align*}
\int^b_a e^{nf(x)} dx &\geq \int^{x_0+\delta}_{x_0-\delta} e^{nf(x)}dx \\ 
&\geq e^{nf(x_0)}\int^{x_0+\delta}_{x_0-\delta}e^{\frac{n}{2}(f''(x_0)-\epsilon)(x-x_0)^2}dx \\
&=e^{nf(x_0)} \frac{1}{\sqrt{n(-f''(x_0)+\epsilon )} }\int^{+\delta\sqrt{ n(-f''(x_0)+\epsilon  )  }}_{-\delta\sqrt{  n(-f''(x_0)+\epsilon )}}e^{-\frac{1}{2}y^2} dy 
\end{align*}
where we have made a transformation using
\begin{align*}
y=\sqrt{n(-f''(x_0)+\epsilon   )}(x-x_0). 
\end{align*}
Dividing both sides by $e^{nf(x_0)} \sqrt{      \frac{2\pi}{n(-f''(x_0))}        }$ gives 
\begin{align}
 \left(  \frac{  \int^b_a e^{   nf(x) }    dx}{e^{nf(x_0)} \sqrt{      \frac{2\pi}{n(-f''(x_0))}        }}\right) &\geq \frac{1}{\sqrt{2\pi}} \sqrt{   \frac{   -f''(x_0)       }{     -f''(x_0)+\epsilon       }                    }         \int^{+\delta\sqrt{ n(-f''(x_0)+\epsilon  )  }}_{-\delta\sqrt{  n(-f''(x_0)+\epsilon )}}e^{-\frac{1}{2}y^2} dy.  
 \label{chap_1_laplace_method_eqn1}
\end{align}
Using Lemma \ref{chap_1_lemma_inf} we can say that for any fixed $\delta$, there exists an $\eta>0$ such that for any $s\notin(x_0-\delta,x_0+\delta)$ we have 
\begin{align*}
\eta\leq f(x_0)-f(s). 
\end{align*}
So we can proceed with 
\begin{align*}
\int^b_a e^{nf(x)} dx &= \int^{x_0-\delta}_{a} e^{nf(x)} dx + \int^{x_0+\delta}_{x_0-\delta}e^{nf(x)}dx + \int^b_{x_0+\delta}e^{nf(x)}dx\\
&\leq \int^{x_0-\delta}_{a} e^{n(f(x_0)-\eta)} dx + \int^{x_0+\delta}_{x_0-\delta}e^{nf(x)}dx + \int^b_{x_0+\delta}e^{n(f(x_0)-\eta)}dx\\
&=(x_0-\delta-a)e^{n(f(x_0)-\eta)}+(b-x_0-\delta)e^{n(f(x_0)-\eta)}+\int^{x_0+\delta}_{x_0-\delta}e^{nf(x)}dx\\
&=(b-a-2\delta)e^{n(f(x_0)-\eta)}+\int^{x_0+\delta}_{x_0-\delta}e^{nf(x)}dx. 
\end{align*}
Now we use the upper bound for $f(x)$ from Equation \ref{chap_1_laplace_method_eqn3}. So 
\begin{align*}
\int^b_a e^{nf(x)}dx &\leq (b-a)e^{n(f(x_0)-\eta)}+ e^{nf(x_0)}\int^{x_0+\delta}_{x_0-\delta}e^{\frac{n}{2}(f''(x_0)+\epsilon)(x-x_0)^2}dx\\
&\leq(b-a)e^{n(f(x_0)-\eta)}+ e^{nf(x_0)}\int^{+\infty}_{-\infty}e^{\frac{n}{2}(f''(x_0)+\epsilon)(x-x_0)^2}dx\\
&=(b-a)e^{n(f(x_0)-\eta)}+ e^{nf(x_0)} \sqrt{      \frac{2\pi}{n(-f''(x_0)-\epsilon)}       } 
\end{align*}
where $\epsilon$ is chosen small enough so that $(f''(x_0)+\epsilon)<0$ is still negative. Now divide both sides by $ e^{nf(x_0)} \sqrt{      \frac{2\pi}{n(-f''(x_0))}       } $ which gives 
\begin{align}
 \left(  \frac{  \int^b_a e^{   nf(x) }    dx}{e^{nf(x_0)} \sqrt{      \frac{2\pi}{n(-f''(x_0))}        }}\right)&\leq   \left(          (b-a)e^{-n\eta}\sqrt{ \frac{n(-f''(x_0))}{2\pi}    }+ \sqrt{\frac{  -f''(x_0)           }{        -f''(x_0)-\epsilon            }}     \   \    \right). 
 \label{chap_1_laplace_method_eqn4}
\end{align}
Now using the other bound for $ \left(  \frac{  \int^b_a e^{   nf(x) }    dx}{e^{nf(x_0)} \sqrt{      \frac{2\pi}{n(-f''(x_0))}        }}\right)$ from Equation \ref{chap_1_laplace_method_eqn1} together with Equation \ref{chap_1_laplace_method_eqn4} gives 
\begin{align*}
\frac{1}{\sqrt{2\pi}} \sqrt{   \frac{   -f''(x_0)       }{     -f''(x_0)+\epsilon       }                    }         \int^{+\delta\sqrt{ n(-f''(x_0)+\epsilon  )  }}_{-\delta\sqrt{  n(-f''(x_0)+\epsilon )}}e^{-\frac{1}{2}y^2} dy &\leq  \left(  \frac{  \int^b_a e^{   nf(x) }    dx}{e^{nf(x_0)} \sqrt{      \frac{2\pi}{n(-f''(x_0))}        }}\right)   \\
&\leq   \left(          (b-a)e^{-n\eta}\sqrt{ \frac{n(-f''(x_0))}{2\pi}    }+ \sqrt{\frac{  -f''(x_0)           }{        -f''(x_0)-\epsilon            }}     \   \    \right). 
\end{align*}
Now we can take the limit as $n \rightarrow +\infty$ which gives  
\begin{equation*}
\sqrt{\frac{  -f''(x_0)           }{        -f''(x_0)+\epsilon            }}\leq \lim_{n\rightarrow \infty} \left(  \frac{  \int^b_a e^{   nf(x) }    dx}{e^{nf(x_0)} \sqrt{      \frac{2\pi}{n(-f''(x_0))}        }}\right) \leq \sqrt{\frac{  -f''(x_0)           }{        -f''(x_0)-\epsilon            }}
\end{equation*}
after noting that $\lim_{n\rightarrow0}e^{-n\eta}\sqrt{n}=0$  and $\lim_{n\rightarrow \infty}\int^{+\delta\sqrt{ n(-f''(x_0)+\epsilon  )  }}_{-\delta\sqrt{  n(-f''(x_0)+\epsilon )}}e^{-\frac{1}{2}y^2} dy=\sqrt{2\pi}$. 
Since $\epsilon$ can be chosen to be arbitrarily small using the Sandwich Theorem gives
\begin{equation*}
 \lim_{n\rightarrow \infty} \left(  \frac{  \int^b_a e^{   nf(x) }    dx}{e^{nf(x_0)} \sqrt{      \frac{2\pi}{n(-f''(x_0))}        }}\right)=1. 
\end{equation*}
\end{proof}

\begin{proof}[Proof of Corollary \ref{chap_1:ts1}] 
If $x_0=a$ then the proof is the same but with a few adjustments. 
In other words, the interval $[x_0-\delta,a]$ does not need to be considered as it is outside the region of integration. 
\begin{align*}
\int^{x_0+\delta}_{x_0-\delta}&\rightarrow \int^{x_0+\delta}_{x_0}\\
\int^{x_0-\delta}_{a}&\rightarrow0\\
\int^{b}_{x_0+\delta}&\rightarrow \int^{b}_{x_0+\delta}
\end{align*}
and the resulting computation would give the extra factor of $\frac{1}{2}$ after using 
\begin{equation*}
\int^{x_0+\delta}_{x_0} e^{\frac{n}{2}(f(x_0)\pm \epsilon)} dy= \frac{1}{2}\int^{x_0+\delta}_{x_0-\delta} e^{\frac{n}{2}(f(x_0)\pm \epsilon)} dy.
\end{equation*}
A similar argument holds for $x_0=b$.
Note that the computation shows that only a small neighbourhood of $x_0$ is relevant asymptotically and that the average term is exponentially small in $n$. 
\end{proof}

\section{One Dimensional Potential}
The potential we are interested in is
\begin{align*}
V_0=\frac{x^4}{4}-a\frac{x^2}{2}
\end{align*}
where $a>0$. 
When this is given a driving frequency it is\footnote{See Appendix \ref{appendix_potential} for a full explanation of the notation used for the potentials.}
\begin{align*}
V_t&=V_0-Fx\cos\Omega t\\
&=\frac{x^4}{4}-a\frac{x^2}{2}-Fx\cos\Omega t
\end{align*}
which when the forcing is zero, the potential has two wells at $x_0=\pm\sqrt{a}$. 
The SDE we want to study is 
\begin{align*}
\frac{dx}{dt}
&=-\nabla V_t +\epsilon \frac{dw}{dt}\\
dx&=\left[x(a-x^2)+F\cos\Omega t\right]dt+\epsilon \ dw
\end{align*}
where $w$ is a one dimensional Wiener process. 
Consider a realisation of the trajectory $x(t)$ starting at $x(0)=y$. 
Its escape time from the left well $\tau_1$ and right well $\tau_2$ are defined as 
\begin{align}
\tau_1(y)&=\inf \{ t:x(t)=0 \quad \text{and} \quad x(0)=y\}\quad \text{where}\quad y\in(-\infty,0) \label{chap_1_benzi_escape_time_1}\\
\tau_2(y)&=\inf \{ t:x(t)=0 \quad \text{and} \quad x(0)=y\}\quad \text{where}\quad y\in(0,+\infty). \label{chap_1_benzi_escape_time_2}
\end{align}
Note that the trajectory $x(t)$  is related to the escape times $\tau_1$ and $\tau_2$. 
Define a new quantity by 
\begin{equation*}
f^i_n(y)=\left\langle \tau_i(y)^n\right\rangle
\end{equation*}
with $i=1,2$ for the two wells and $n=1,2,\ldots$.
Note $\langle\cdot\rangle$ denotes the mean average over all realisations. 
Also note that the $n$th moment is being used here. 
In \cite{benzi81} a method by Gihman and Skorohod   \cite{gihman72} was used to derive the following equation
\begin{equation}
\frac{1}{2} \epsilon^2\frac{d^2}{dy^2} f^i_n(y)-V' \frac{d}{dy}f^i_n(y)=-nf^i_{n-1}(y) \label{chap_1:es3}
\end{equation}
where $V'$ is a shorthand for $V'=\nabla V_0$ if we are escaping from the potential described by $V_0$.
Note that the potential is frozen in the case of $V_0$. 
Similarly $V'$ is a shorthand for $V'=\nabla V_t$ if we are escaping from the potential described by $V_t$. 
The following boundary conditions are
\begin{align*}
f^i_n(0)=0, \quad \frac{d}{dy}f^1_n(-\infty)=0 \quad \text{and} \quad \frac{d}{dy}f^2_n(+\infty)=0. 
\end{align*}
Having $f^i_n(0)=0$ is appropriate since being at $y=0$ means it is in neither well and so has already escaped at $t=0$ anyway. 
We can see how    $\frac{d}{dy}f^1_n(-\infty)=0$ and $\frac{d}{dy}f^2_n(+\infty)=0$     make sense by considering $f^2_1$ as an example. 
If the particle starts at $x(0)=y$, where $y$ is a very large positive number $y\gg N \sqrt{a}$ (where $N\gg1$) then with the equilibrium point being an attractor, it would more or less deterministically slide towards $x=+\sqrt{a}$.  
We call the time it takes for it to travel to $x=+\sqrt{a}$, $\tau'$. 
If the particle starts somewhere further beyond $y$ say $x=y+\delta$, where $\delta>0$ , it would also slide down to $x=+\sqrt{a}$ almost deterministically. 
We call this new time to get to $x=+\sqrt{a}$, $\tau''$. 
Intuitively, we would expect $\tau'\approx \tau''$ so $\frac {d}{dy}f^2_1(-\infty)=0$. 

The aim now is to solve Eqn \ref{chap_1:es3} for different cases. 
These are for $F=0$ and $F\neq0$, in the small noise approximation. 

\subsection{One Dimensional Potential - Case $F=0$}
The potential is stationary and does not depend on time. 
It  is symmetric at $x=0$ so we must have
\begin{align*}
f^1_1(-y)=f^2_1(y). 
\end{align*}
For simplicity we denote the following 
\begin{align*}
f=f^1_1 
\quad \text{and} \quad 
I=\frac{df}{dy}
\end{align*}
which rewrites the differential equation as 
\begin{align*}
\frac{1}{2} \epsilon^2\frac{dI}{dy} -V_0' I=-1
\end{align*}
where $V'_0=\nabla V_0$. 
Using an integrating factor gives 
\begin{align*}
\frac{d}{dy}
\left(
I\times
\exp
\left\{
\int^y_0-\frac{2}{\epsilon^2}
V'_0(s)\,ds
\right\}
\right)
&=
-\frac{2}{\epsilon^2}
\exp
\left\{
\int^y_0
-\frac{2}{\epsilon^2}
V'_0(s)\,ds
\right\}
\end{align*}
which gives 
\begin{align*}
\frac{d}{dy}
\left(
I\times
\exp
\left\{
-\frac{2}{\epsilon^2}
V_0(y)
\right\}
\right)
=
-\frac{2}{\epsilon^2}
\exp
\left\{
-\frac{2}{\epsilon^2}
V_0(y)
\right\}. 
\end{align*}
We know that $I(-\infty)=0$ so integrating we have 
\begin{align*}
I(y)\times
\exp
\left\{
-\frac{2}{\epsilon^2}
V_0(y)
\right\}
-
I(-\infty)\times
\exp
\left\{
-\frac{2}{\epsilon^2}
V_0(-\infty)
\right\}
=
-\frac{2}{\epsilon^2}
\int^y_{-\infty}
\exp
\left\{
-\frac{2}{\epsilon^2}
V_0(s)
\right\}
\,ds
\end{align*}
and proceeding we have 
\begin{align*}
I(y)\times
\exp
\left\{
-\frac{2}{\epsilon^2}
V_0(y)
\right\}
&=
-\frac{2}{\epsilon^2}
\int^y_{-\infty}
\exp
\left\{
-\frac{2}{\epsilon^2}
V_0(s)
\right\}
\,ds\\[0.5em]
I(y)
&=
-\frac{2}{\epsilon^2}
\exp
\left\{
+\frac{2}{\epsilon^2}
V_0(y)
\right\}
\int^y_{-\infty}
\exp
\left\{
-\frac{2}{\epsilon^2}
V_0(s)
\right\}
\,ds. 
\end{align*}
We know that $f(0)=0$ so integrating we have 
\begin{align*}
f(y)-f(0)
&=
-\frac{2}{\epsilon^2}
\int^y_0
\exp
\left\{
+\frac{2}{\epsilon^2}
V_0(u)
\right\}
\int^u_{-\infty}
\exp
\left\{
-\frac{2}{\epsilon^2}
V_0(s)
\right\}
\,ds
\,du
\end{align*}
which we rewrite as 
\begin{align}
f(y)&=
-\frac{2}{\epsilon^2}
\int^y_0
\exp
\left\{
+\frac{2}{\epsilon^2}
V_0(u)
\right\}
g(u)
\,du \label{chap_1_eval_eqn} \\[0.5em]
\text{where}\quad
g(u)&=
\int^u_{-\infty}
\exp
\left\{
-\frac{2}{\epsilon^2}
V_0(s)
\right\}
\,ds.\nonumber
\end{align}
Up to now the methods we have used for solving $f(\cdot)$ are exact and the boundary conditions on $f(\cdot)$ have also been kept. 
There are no approximations to our approach so far. 
Recall that $y<0$. 
We seek an approximate solution for $f(y)$ in the region $y\in[-\sqrt{a},0]$. 
Now we use Laplace Method in the small noise limit (small $\epsilon$) to evaluate the integrals in Equation \ref{chap_1_eval_eqn}. 
Note that 
\begin{align*}
\max_{u\in[-\infty,0]}
\left(
-\frac{2}{\epsilon^2}
V_0(u)
\right)
=
-\frac{2}{\epsilon^2}V_0\left(-\sqrt{a}\right). 
\end{align*}
Using the Laplace Method for small $\epsilon$ gives the approximation
\begin{align*}
g(u)\approx
\left\{
\begin{array}{lll}
\sqrt{\frac{\pi\epsilon^2}{2a}}
\exp\left\{\frac{a^2}{2\epsilon^2}\right\}
&
\text{if}
&
u\in(-\sqrt{a},0]\\[0.5em]
\frac{1}{2}
\sqrt{\frac{\pi\epsilon^2}{2a}}
\exp\left\{\frac{a^2}{2\epsilon^2}\right\}
&
\text{if}
&
u=-\sqrt{a}
\end{array}
\right.
\end{align*}
This means $g(u)\approx
\sqrt{\frac{\pi\epsilon^2}{2a}}
\exp\left\{\frac{a^2}{2\epsilon^2}\right\}$
almost everywhere with respect to the Lebesgue measure on $[-\sqrt{a},0]$. 
This approximates $f(\cdot)$ to 
\begin{align*}
f(y)\approx
+\frac{2}{\epsilon^2}
\sqrt{\frac{\pi\epsilon^2}{2a}}
\exp\left\{\frac{a^2}{2\epsilon^2}\right\}
\int_y^0
\exp\left\{+\frac{2}{\epsilon^2}V_0(s)\right\}
\,ds
\end{align*}
where we have switched the limits of the integral. 
We use Laplace Method again after noting that 
\begin{align*}
\max_{s\in[-\sqrt{a},0]}
\left(
+\frac{2}{\epsilon^2}V_0(s)
\right)
=+\frac{2}{\epsilon^2}V_0(0)
\end{align*}
here the maximum is on the edge on the boundary meaning we would need an extra factor of $\frac{1}{2}$. 
So
\begin{align*}
f(y)&\approx
+\frac{2}{\epsilon^2}
\sqrt{\frac{\pi\epsilon^2}{2a}}
\exp\left\{\frac{a^2}{2\epsilon^2}\right\}
\frac{1}{2}
\sqrt{\frac{\pi\epsilon^2}{a}}\\
&=\frac{1}{\sqrt{2}}
\frac{\pi}{a}
\exp\left\{\frac{a^2}{2\epsilon^2}\right\}. 
\end{align*}

\subsection{One Dimensional Potential - Case $F\neq0$}
The potential is now oscillating. 
We aim to do a similar calculation to the static potential case. 
We make an approximation by assuming that the amplitude of the oscillations $F$ is small enough such that there will always be two distinct wells.
The positions of the critical points and wells will be very close to the static potential case. 
We can just focus on one well $x_0(t)$ on the left of the hill which is dependent on the time $t$. 
The method is similar to what we have used in the  $F=0$ case.
We have 
\begin{align}
f(y)
&=
-\frac{2}{\epsilon^2}
\int^y_0
\exp
\left\{
+\frac{2}{\epsilon^2}
V_t(u)
\right\}
\int^u_{-\infty}
\exp
\left\{
-\frac{2}{\epsilon^2}
V_t(s)
\right\}
\,ds
\,du\label{chap_1_eva_eqn_2}\\[0.5em]
&=
+\frac{2}{\epsilon^2}
\int^0_y
\exp
\left\{
+\frac{2}{\epsilon^2}
V_t(u)
\right\}
g_t(u)
\,du\nonumber\\[0.5em]
\text{where}\quad 
g_t(u)&=\int_{-\infty}^u
\exp
\left\{
-\frac{2}{\epsilon^2}
V_t(s)
\right\}
\,ds\nonumber
\end{align}
since $y<0$ the limits of the integral may be switched. 
Notice we have approximated the situation by assuming that the hill moves very little away from $x=0$ which is what makes Equation \ref{chap_1_eva_eqn_2} valid. 
We seek a solution for $f(y)$ in the region $y\in[x_0(t),0]$. 
For small $\epsilon$, we can use Laplace's Method to approximate the integrals in Equation \ref{chap_1_eva_eqn_2}.
Note that
\begin{align*}
\max_{u\in[-\infty,0]}\left(-\frac{2}{\epsilon^2}V_t(u)\right)
=-\frac{2}{\epsilon^2}
V_t
\left(
x_0(t)
\right). 
\end{align*}
Now using the Laplace's Method gives 
\begin{align*}
g_t(u)\approx
\left\{
\begin{array}{lll}
\exp\left\{-\frac{2}{\epsilon^2}V_t\left(x_0(t)\right)\right\}
\sqrt{\frac{\pi\epsilon^2}{V_t''\left(x_0(t)\right)}}
&
\text{if}
&
u\in(x_0(t),0]\\[0.5em]
\frac{1}{2}\exp\left\{-\frac{2}{\epsilon^2}V_t\left(x_0(t)\right)\right\}
\sqrt{\frac{\pi\epsilon^2}{V_t''\left(x_0(t)\right)}}
&
\text{if}
&
u=x_0(t)
\end{array}
\right.
\end{align*}
where $V''_t=\nabla^2V_t$. 
In other words $g_t\approx
\exp\left\{-\frac{2}{\epsilon^2}V_t\left(x_0(t)\right)\right\}$ almost everywhere on $[x_0(t),0]$ with respect to the Lebesgue measure.  
So
\begin{align*}
f(y)\approx
+\frac{2}{\epsilon^2}
\exp\left\{-\frac{2}{\epsilon^2}V_t\left(x_0(t)\right)\right\}
\sqrt{\frac{\pi\epsilon^2}{V_t''\left(x_0(t)\right)}}
\int^0_y
\exp\left\{+\frac{2}{\epsilon^2}V_t(u)\right\}
\,du. 
\end{align*}
Now
\begin{align*}
\max_{u\in[x_0(t),0]}
\left(
+\frac{2}{\epsilon^2}
V_t(u)
\right)
=+\frac{2}{\epsilon^2}V_t(0)
\end{align*}
where the maximum is on the boundary of $[x_0(t),0]$ meaning we would need an extra factor of $\frac{1}{2}$.
So
\begin{align*}
f(y)\approx
+\frac{2}{\epsilon^2}
\exp\left\{-\frac{2}{\epsilon^2}V_t\left(x_0(t)\right)\right\}
\sqrt{\frac{\pi\epsilon^2}{V_t''\left(x_0(t)\right)}}
\frac{1}{2}
\sqrt{\frac{\pi\epsilon^2}{-V''_t(0)}}
\end{align*}
which gives 
\begin{align*}
f(y)\approx
\frac{\pi}{\sqrt{aV''_t\left(x_0(t)\right)}}
\exp\left\{-\frac{2}{\epsilon^2}V_t\left(x_0(t)\right)\right\}. 
\end{align*}
Finding $f(y)$ is very hard so we consider when the oscillations are very slow, that is for very small $\Omega$. 
At the two extremes we have 
\begin{align}
dx &= [x(a-x^2)+F]dt\,+\,\epsilon\, dw \ \ \text{when $t=0$} \label{chap_1:e111}\\
dx &= [x(a-x^2)-F]dt\,+\,\epsilon\, dw\ \ \text{when $t=\frac{\pi}{\Omega}$}. \label{chap_1:e112}
\end{align}
We also assume that the oscillations are so small, the now time dependent equilibrium point does not differ much from the time independent case $x_0=-\sqrt{a}$. 
We solve $f(y)$ for the case of Equation \ref{chap_1:e111}. 
The case for Equation \ref{chap_1:e112} is similar. 
Let $x_0(t)$ be approximated and denoted with a new notation by 
\begin{align*}
x_0(t)=z_0+\delta =s
\end{align*}
where $z_0=-\sqrt{a}$.  
We seek an expression for $\delta$ by 
\begin{align*}
[s(a-s^2)+F]dt&=0\quad \text{from Equation \ref{chap_1:e111}}\\[0.5em]
s(a-s^2)&=-F\\[0.5em]
\delta&\approx\frac{-F}{a-3z_0^2}= +\frac{F}{2a}
\end{align*}
after  ignoring terms of higher order than $\delta^2$.
Progressing further gives 
\begin{align}
-\frac{2}{\epsilon^2}V_{t=0}(s)
&=
-\frac{2}{\epsilon^2}
\left\{
\frac{s^4}{4}-a\frac{s^2}{2}-Fs
\right\}\nonumber
\\[0.5em]
&\approx-\frac{2}{\epsilon^2}\left[V_0(z_0)+  \delta(z_0^3-az_0-F) -Fz_0\right]\label{chap_1_eva_eqn_3}
\end{align}
again after ignoring terms of higher order than $\delta^2$. 
Equation \ref{chap_1_eva_eqn_3} is now approximated by 
\begin{align*}
-\frac{2}{\epsilon^2}V_{t=0}(s)&\approx-\frac{2}{\epsilon^2}\left[V_0(z_0)+  \delta(z_0^3-az_0-F) -Fz_0\right]
\\[0.5em]
&=\frac{a^2}{2\epsilon^2}\left\{  1+\frac{4F^2}{2a^3}  - \frac{4F}{a^{\frac{3}{2}}}  \right\}\\[0.5em]
&\approx\frac{a^2}{2\epsilon^2}\left\{  1  - \frac{4F}{a^{\frac{3}{2}}}  \right\}
\end{align*}
after assuming $F^2$ is small. 
We make another approximation by 
\begin{align*}
\sqrt{aV_{t=0}''\left(x_0(t)\right)}&\approx\sqrt{aV_0''(z_0)}\\
&=\sqrt{a[3z_0^2-a]} \\
&=a\sqrt{2}
\end{align*}
after noting that $z_0=-\sqrt{a}$. 
So $f(y)$ for Equation \ref{chap_1:e111} and \ref{chap_1:e112} are 
\begin{align*}
f(x_0(t))&= \frac{\pi}{a\sqrt{2}}\exp \left\{     \frac{a^2}{2\epsilon^2}\left(  1  - \frac{4F}{a^{\frac{3}{2}}}  \right)     \right\} \ \ \text{when $t=0$} \\[0.5em]
f(x_0(t))&= \frac{\pi}{a\sqrt{2}}\exp \left\{     \frac{a^2}{2\epsilon^2}\left(  1  + \frac{4F}{a^{\frac{3}{2}}}  \right)     \right\} \ \ \text{when $t=\frac{\pi}{\Omega}$}.
\end{align*}
We can see how the solution make physical sense because when $t=\frac{\pi}{\Omega}$ the left well is lower, and so the probability to escape is lower and the time to escape would also be longer. 
Now that all of our calculations are done for both the time independent and time dependent case we can compare them.

\subsection{Conclusion and Resonance Condition $\epsilon_{res}$}
We have effectively reviewed $f^1_1$ in the limit of small noise,
which is $\langle \tau\rangle$ the averaged escape time for small $\epsilon$. 
Comparing them more clearly here gives 
\begin{align*}
F=0 \quad 
\langle \tau \rangle
&=
\frac{1}{\sqrt{2}}
\frac{\pi}{a}
\exp\left\{\frac{a^2}{2\epsilon^2}\right\}\\[0.5em]
F\neq0 \quad 
\langle \tau \rangle
&=\frac{\pi}{a\sqrt{2}}\exp \left\{     \frac{a^2}{2\epsilon^2}\left(  1  - \frac{4F}{a^{\frac{3}{2}}}  \right)     \right\}
\quad \text{when} \quad t=0\\[0.5em]
\langle \tau \rangle
&=\frac{\pi}{a\sqrt{2}}\exp \left\{     \frac{a^2}{2\epsilon^2}\left(  1  + \frac{4F}{a^{\frac{3}{2}}}  \right)     \right\}
\quad \text{when} \quad t=\frac{\pi}{\Omega}.
\end{align*}
For the $F\neq0$ case if we impose
\begin{align}
\langle \tau \rangle&=\frac{\pi}{\Omega}\quad \text{for}\quad t=0\label{chap_1_noise_1}\\[0.5em]
\langle \tau \rangle&=\frac{\pi}{\Omega}\quad \text{for}\quad t=\frac{\pi}{\Omega}\label{chap_1_noise_2}
\end{align}
and solve for the noise in both cases (that is solving Equation \ref{chap_1_noise_1} and \ref{chap_1_noise_2} for $\epsilon$) we have 
\begin{align*}
\epsilon_1&= a\left(\frac{1-4F/a^{3/2}}{2\ln(2\sqrt{2}a / \Omega )}\right)^{1/2}\quad \text{for}\quad t=0\\[0.5em]
\epsilon_2&= a\left(\frac{1+4F/a^{3/2}}{2\ln(2\sqrt{2}a / \Omega )}\right)^{1/2}\quad \text{for}\quad t=\frac{\pi}{\Omega}
\end{align*}
then the resonance condition $\epsilon_{res}$ should be inside the interval $\epsilon_{res}\in[\epsilon_1,\epsilon_2]$. 
For the example below it just so happen that $[\epsilon_1,\epsilon_2]=[0.18,0.31]$ and $\epsilon_{res}\approx0.26$, which gives the trajectory

\begin{figure}[H]
\begin{center}
\centerline{\includegraphics[scale=0.29]{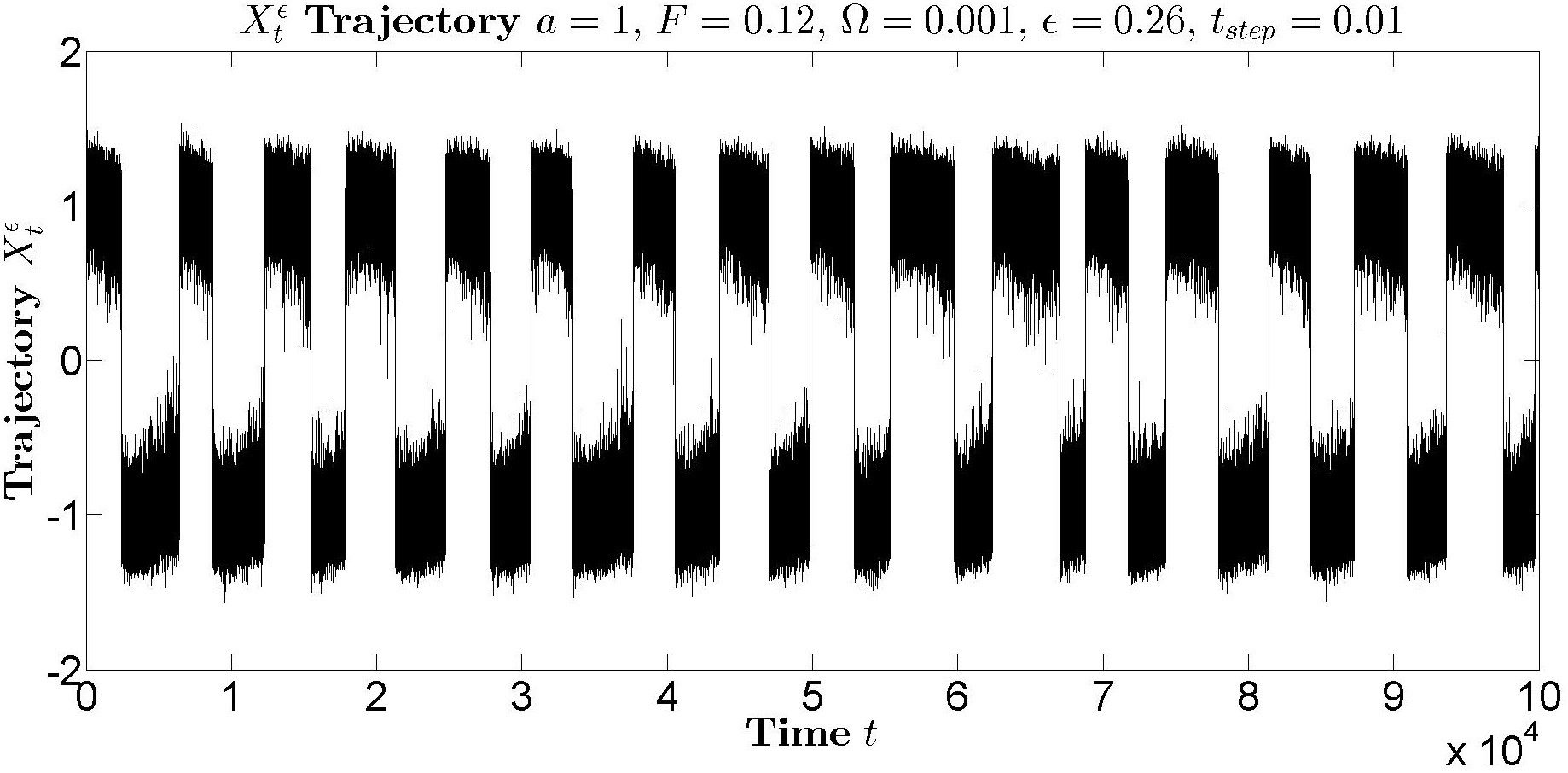}}
\caption{This trajectory is exhibiting quasi-determinism. It is near stochastic resonance.}
\label{chap_1_fig_1}
\end{center}
\end{figure}

\noindent We can increase and decrease the noise away from $\epsilon_{res}\approx0.26$, and transitions will occur more frequently or less frequently as we move away from resonance. 

\begin{figure}[H]
\centerline{\includegraphics[scale=0.29]{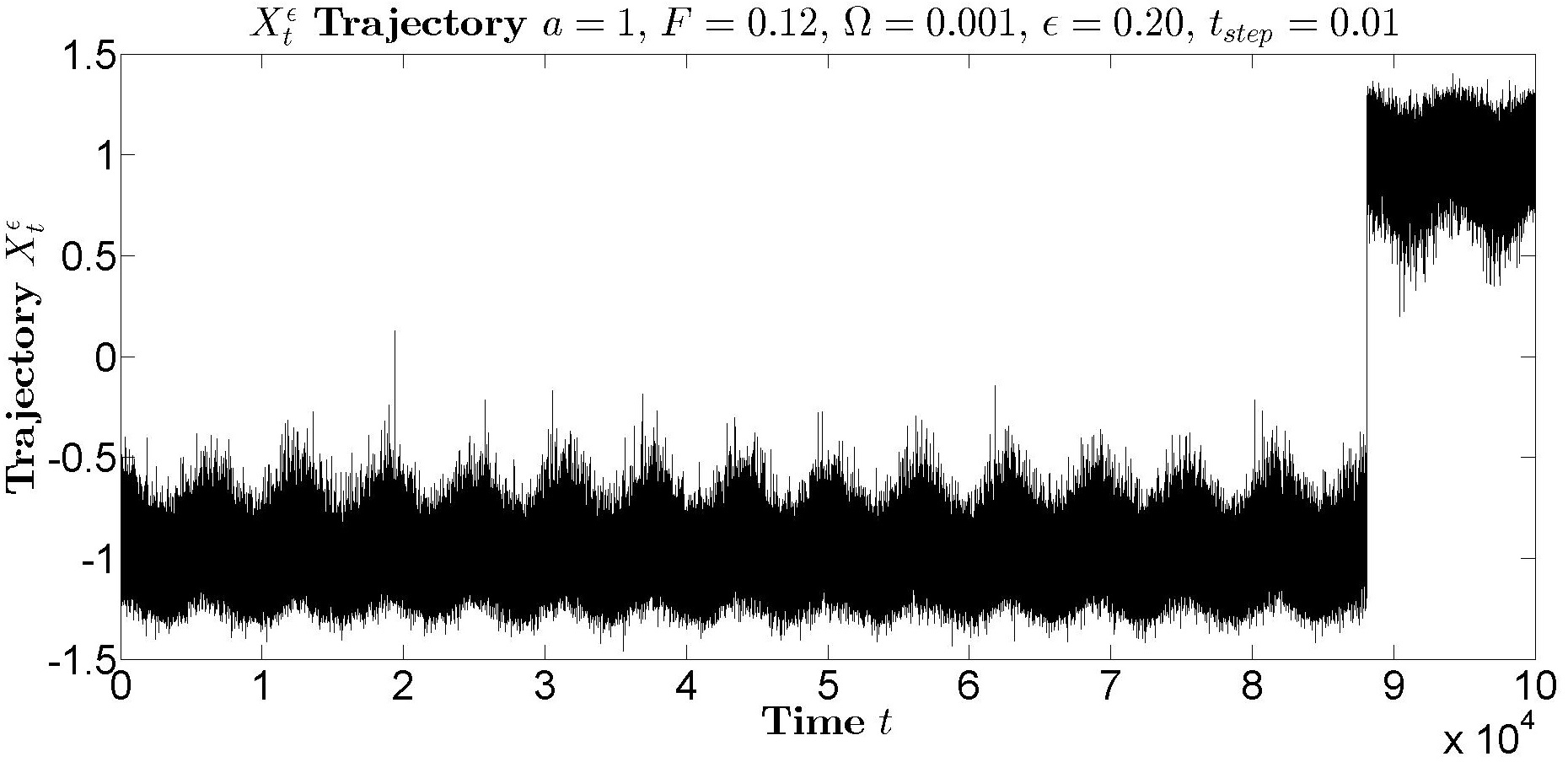}}
\caption{Transitions occur irregularly and are rare.}\label{chap_0_fig_4}
\end{figure}

\begin{figure}[H]
\centerline{\includegraphics[scale=0.29]{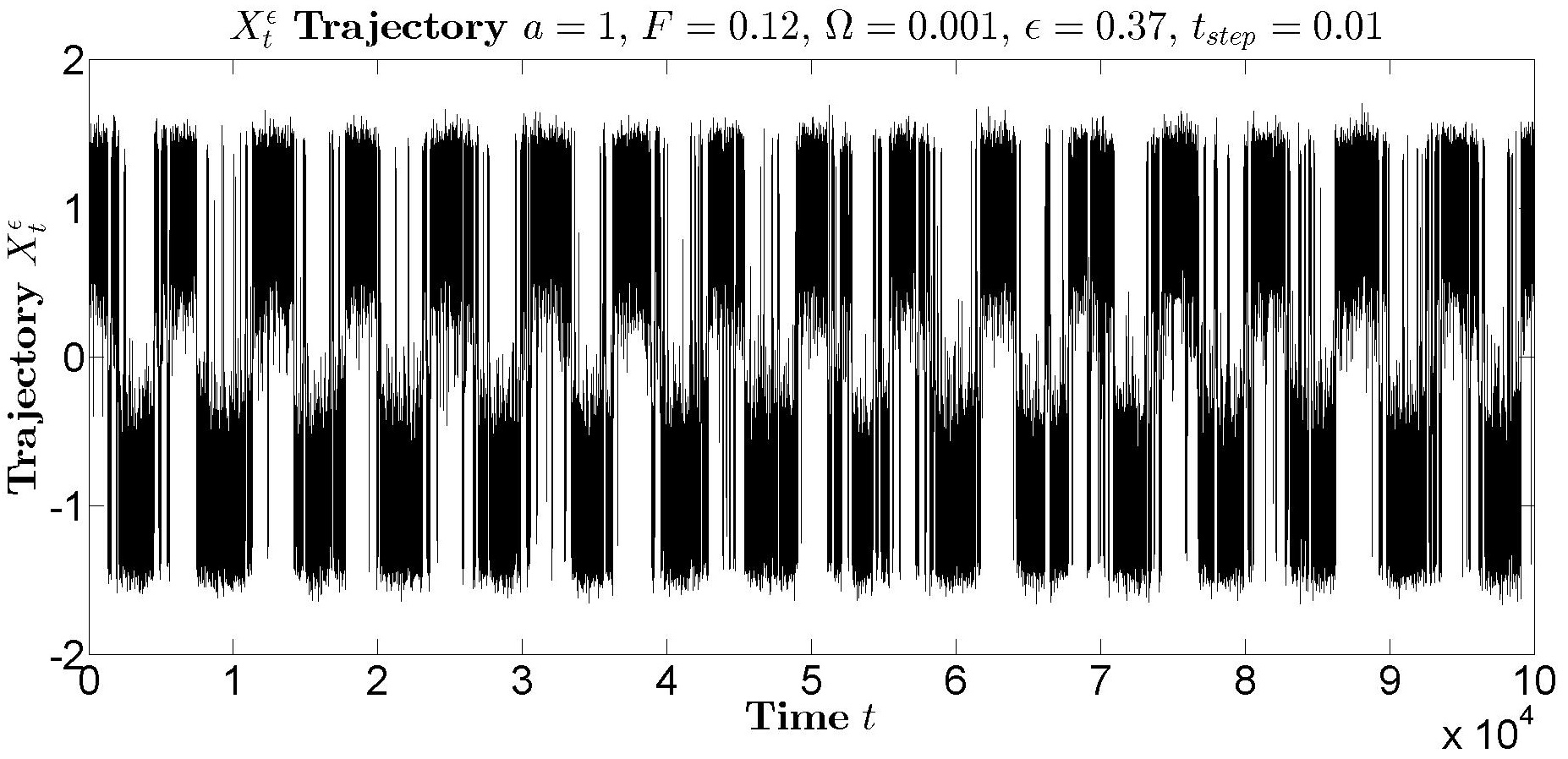}}
\caption{Transitions occur very often.}\label{chap_0_fig_5}
\end{figure}

\section{Remarks on One Dimensional Potential}
Notice that this is a very crude way to study the system.
The oscillating potential is being approximated by a frozen static potential. 
This is precisely the adiabatic approximation.  
Not only does it assume small noise, small forcing and adiabatic time development, the exact positions of the critical points (the two wells and the hill) were not calculated. 
All the calculations assumed that the hill was near $x=0$. 
This means the forcing is assumed to be small enough such that the hill does not move far away from $x=0$. 
Benzi et al's definition of the escape time is so crude it will be used only as a rough guide. 

Also note it is hard to tell if  Figures \ref{chap_1_fig_1},
\ref{chap_0_fig_4} and \ref{chap_0_fig_5}
show any regularity or not. 
As we shall see, regularity shows itself in the distribution of the escape times.

\chapter{Theoretical Escape Time from a Well of a Static Potential}
\label{sum_chap_escape_static}
We consider the theoretical escape time of a particle from a well of a static potential. 
This is done in two parts. 
The first part is Freidlin-Wentzell theory or large deviation
and the second part is Kramers' formula which is derived using potential theory. 
For small noise levels large deviation is considered and for higher noise levels potential theory is used.

\section{Freidlin-Wentzell Theory and Large Deviation}
A review of the major results of  the Freidlin-Wentzell theory is presented which is found in \cite{freidlin98}. 
This is done by considering stochastic systems converging to the deterministic limit for small noise, action functional for Wiener processes, action functional for general processes and the main theorems concerning the escape time.

\subsection{Stochastic Processes}
Let $\left(\Omega, \mathcal{F}, P\right)$ be a probability space. 
Let $\left(\mathbb{R}^r,\mathcal{B}\right)$ be a measure space on $\mathbb{R}^r$.  
Let $T$ be an indexing set. 
For $\omega\in\Omega$ and $t\in T$ define a mapping $\Omega \times T \longrightarrow \mathbb{R}^r$ by 
\begin{align*}
X_t(\omega):\Omega \times T \longrightarrow \mathbb{R}^r 
\end{align*}
where $X_t(\omega)\in \mathbb{R}^r$ is called a stochastic process on $\mathbb{R}^r$. 
The probability measure defined on $A\in\mathcal{F}$ is denoted by $P(A)$. 
But if this probability measure can depend on a value $x\in \mathbb{R}^r$, 
we will often put this dependence explicitly into the notation
\begin{align*}
P(A,x)=P_x(A). 
\end{align*}
We will only consider Markov processes in this thesis. 
There are further technical properties a Markov process has to fulfil.\footnote{For more details see page 20 in \cite{freidlin98}.} 
Intuitively this can be understood in the following way; 
the $\Omega$ can be thought of as the set of all trajectories of a stochastic process; 
the $T$ can be thought of as the set of time, for example $T=[0,\infty)$; 
the $\mathbb{R}^r$ can be thought as the space in which the trajectory is in, 
then $X_t(\omega)$ is a trajectory in $\mathbb{R}^r$ with continuous time $t\geq0$.

\subsection{Deterministic Limit}
Consider the following system. 
We have the $r$-dimensional real space $\mathbb{R}^r$. 
Let $x_t \in \mathbb{R}^r$ be a time dependent variable in $\mathbb{R}^r$.
Let $b(x_t)$ be a function $b:\mathbb{R}^r\rightarrow\mathbb{R}^r$ on $x_t$. 
We then let 
\begin{equation}
dx_t=b(x_t)dt \label{chap_2_tobias_deter}
\end{equation}
which can be seen as a system of $r$ differential equations for each of the elements of $x_t$. 
When we consider random systems we denote the (random) variable by $X^\epsilon_t$ with values in $\mathbb{R}^r$. 
The stochastic processes we consider in this thesis are diffusion processes. 
More precisely we consider random dynamical systems which are solutions of the following system of stochastic differential equations 
\begin{equation}
dX_t^\epsilon=b(X_t^\epsilon)dt+\epsilon\sigma(X_t^\epsilon)dw_t
\label{chap_2_tobias_rand}
\end{equation}
where $\epsilon$ is the noise level, $w_t$ is a $l$-dimensional Wiener process and $\sigma(X_t^\epsilon)$ is a function on $X_t^\epsilon$ returning a $l \times r$ matrix.

The first circle of results in the book of Freidlin and Wentzell are about how the solutions of the random system Equation \ref{chap_2_tobias_rand} approximate the solutions of the deterministic system Equation \ref{chap_2_tobias_deter}. 
For example we know that 
\begin{equation*}
\lim_{\epsilon\rightarrow0}X_t^\epsilon=x_t. 
\end{equation*}
But the exact manner of this limit and the conditions under which $X_t^\epsilon\rightarrow x_t$ is reached  is documented in Freidlin-Wentzell.\footnote{See pages 44-59 in \cite{freidlin98}.} 
For example we have\footnote{Adapted from Theorem 1.2 page 45 of \cite{freidlin98}.} 

\begin{Theorem} \label{thm:lip:growth}
Suppose  that the coefficients of Equation \ref{chap_2_tobias_rand} satisfy a Lipschitz condition and a growth condition given by 
\begin{align}
\sum_i\left[
b_i(x)-b_i(y)
\right]^2
+
\sum_{i,j}
\left[
\sigma_{ij}(x)-\sigma_{ij}(y)
\right]^2
&\leq K^2 |x-y|^2 \label{eqn:lip}\\
\sum_i\left[b_i(x)\right]^2
+
\sum_{i,j}
\left[\sigma_{ij}(x)\right]^2
&\leq K^2 (1+|x|^2)\label{eqn:growth}
\end{align}
then for all $t>0$ and $\delta>0$ we have 
\begin{align*}
E\left|X^\epsilon_t-x_t\right|^2
\leq \epsilon^2 a(t)
\quad \text{and} \quad 
\lim_{\epsilon\rightarrow0}P\left\{\max_{0\leq s \leq t}\left|X_s^\epsilon-x_s\right|>\delta\right\}=0
\end{align*}
where $a(t)$ is a monotone increasing function, which is expressed in terms of $|x|$ and $K$. 
\end{Theorem}

\noindent Theorem \ref{thm:lip:growth} can be explained in another way.  Intuitively as $\epsilon\rightarrow0$ we would expect to be back in the deterministic system $x_t$. 
Note that Equation \ref{eqn:lip} is the Lipschitz condition and \ref{eqn:growth} is the growth condition. 

There is also a stochastic analogue of Taylor's Remainder's Theorem where it can be shown that $X_t^\epsilon$ admits the following decomposition\footnote{See Theorem 2.1 page 52 of \cite{freidlin98}.}
\begin{equation*}
X_t^\epsilon=X_t^{(0)}+\epsilon X_t^{(1)}+\cdots+\epsilon^k X_t^{(k)}+R_{k+1}^\epsilon(t)
\end{equation*}
where the remainder is bounded by new functions
\begin{equation*}
\sup_{0\leq t \leq T}\left|R_{k+1}^\epsilon(t)\right|<C(\omega)\epsilon^{k+1}\quad \text{and} \quad P\left\{C(\omega)<\infty\right\}=1
\end{equation*}
and the $X^{(k)}_t$ are solutions of stochastic differential equations.

\subsection{Action Functional for Wiener processes} 
\label{chap_action_functional}
The second part of Freidlin-Wentzell has a new setting.\footnote{See pages 70-79 in \cite{freidlin98}.}  
Let $b(X_t^\epsilon)=0$, $\sigma(X_t^\epsilon)=1$ and $w_t$ be a $r$-dimensional Wiener process, that is to say 
\begin{equation*}
dX_t^\epsilon=\epsilon\, dw_t
\end{equation*}
where we have reduced the system to a $r$-dimensional Wiener process. 
Let $C_{T_1T_2}=C_{T_1T_2}(\mathbb{R}^r)$ denote the set of all continuous paths in $\mathbb{R}^r$ starting at time $T_1$ and ending at $T_2$. 
On this set we define a metric by 
\begin{equation*}
\rho_{T_1T_2}(\psi,\varphi)=\sup_{T_1\leq t \leq T_2}|\psi_t-\varphi_t|. 
\end{equation*}
We define a new functional by 
\begin{equation*}
S(\varphi)=S_{T_1T_2}(\varphi)=\frac{1}{2}\int^{T_2}_{T_1}|\dot{\varphi_s}|^2ds
\end{equation*}
for absolutely continuous (and differentiable) $\varphi_t$. 
If $\varphi_t$ is not absolutely continuous or if the integral is divergent, we set $S(\varphi)=+\infty$. 
We define the action functional by 
\begin{equation*}
I^\epsilon_{T_1T_2}(\varphi)=\epsilon^{-2}S_{T_1T_2}(\varphi)
\end{equation*}
and $S_{T_1T_2}(\varphi)$ will be called the normalized action functional. 
The paths should be interpreted as points,
that is elements of the functional space of paths, 
that is each point is itself a path. 
The distance between these points, and hence paths, is given by the metric just defined.  
We define a new set
\begin{equation*}
\Phi(s)=\left\{\varphi\in C_{0T}  \ \text{such that} \ \varphi_0=0\ \text{and}\ S_{0T}(\varphi)\leq s\right\}
\end{equation*}
which can be shown to be compact in the uniform topology.\footnote{See Lemma 2.1 page 77 of \cite{freidlin98}.}
 Now the SDE 
\begin{equation*}
dX_t^\epsilon=\epsilon\, dw_t
\end{equation*}
cannot be solved pathwise as in the deterministic case. 
The solution is a randomly chosen path 
out of an infinitude of possible paths. 
This is described by the probability of the path having certain properties. 
Also $X_t^\epsilon$ is self-similar and non-differentiable, but may be approximated by differentiable functions $\varphi_t$. 
The next major theorems in Freidlin-Wentzell show that for any $\delta>0$ and $\gamma>0$ we have\footnote{See Theorem 2.1 page 74 in \cite{freidlin98}.}
\begin{equation*}
P\left\{\rho_{0T}(X_t^\epsilon,\varphi_t)<\delta\right\}\geq \exp\left\{-\epsilon^{-2}\left[S_{0T}(\varphi)+\gamma\right]\right\}
\end{equation*}
and for any $\delta>0$, $\gamma>0$, $s_0>0$ with $s<s_0$ we have\footnote{See Theorem 2.2 page 74 in \cite{freidlin98}.} 
\begin{equation*}
P\left\{\rho_{0T}(X_t^\epsilon,\Phi(s))\geq \delta \right\}\leq \exp \left\{-\epsilon^{-2}(s-\gamma)\right\}. 
\end{equation*}
These two statements may be interpreted as a Laplace  type theorem in function spaces. 
A physical interpretation is that this gives an asymptotic description (in small $\epsilon$) for the probability that the path $X^\epsilon_t$ is near to $\varphi_t$, that is 
\begin{equation*}
P\left\{\rho(X_t^\epsilon,\varphi)<\delta \right\}\approx\exp\left\{-\epsilon^{-2}S(\varphi)\right\}. 
\end{equation*}
In the next section we develop the action functional for more general processes.

\subsection{Action Functional for General processes}
So far the action theory was developed for just one particular example of a stochastic process, that is the Wiener process. 
Now we develop an action theory for a general stochastic process described by Equation \ref{chap_2_tobias_rand}.\footnote{See pages 79-92 of \cite{freidlin98}.}  

Before we do that we state a list of properties a functional should have so that we can consider a suitable action functional. 
We state these properties in a more general context. 
Let $(X,\rho)$ be a metric space with metric $\rho$. 
On the $\sigma$-algebra of its Borel subsets let $\mu^h$ be a family of probability measures depending on a parameter $h>0$. Let $\lambda(h)$ be a positive function going to $+\infty$ as $h\downarrow0$. 
Let $S(x)$ be a function such that $S:X\rightarrow[0,\infty]$. 
We say that $\lambda(h)S(x)$ is an action function if the following holds. 
\\ \\
\indent (0) the set $\Phi(s)=\{x:S(x)\leq s\}$ is compact for every $s\geq 0$. \\
\indent (I) for any $\delta>0$, any $\gamma>0$ and any $x \in X$ there exists an $h_0>0$ such that
\begin{equation*}
\mu^h\{y:\rho(x,y)<\delta\}\geq\exp\{-\lambda(h)[S(x)+\gamma]\}
\end{equation*}
\indent $\relphantom{\text{(I)}}$ for all $h\leq h_0$\\
\indent (II) for any $\delta>0$, any $\gamma>0$ and any $s>0$ there exists an $h_0>0$ such that 
\begin{equation*}
\mu^h\{y:\rho(y,\Phi(s))\geq \delta\}\leq \exp\{-\lambda(h)(s-\gamma)\}
\end{equation*}
\indent $\relphantom{\text{(II)}}$ for all $h\leq h_0$.\\ \\
$S(x)$ and $\lambda(h)$ will be called the normalized action functional and normalizing coefficient. 

The results given in Chapter \ref{chap_action_functional} show that the functional considered there has all the above properties, 
where   $X=C_{T_1T_2}(\mathbb{R}^r)$, $S(\varphi)=S_{T_1T_2}(\varphi)$, $\lambda(h)=\epsilon^{-2}$ and $\epsilon=h$. 
Thus we can see how 
$X=C_{T_1T_2}(\mathbb{R}^r)$ with $I^\epsilon_{T_1T_2}$ is an action functional since it satisfied all three properties. 
Doubtless that there will be many other systems which satisfy all these three properties as well. This higher level of abstraction would allow us to prove powerful theorems.

\subsection{Main Theorems}
The action functional was given for a diffusion with drift term zero i.e. $b(X_t^\epsilon)=0$.
 We now put this term back in to consider the equation 
\begin{equation*}
dX_t^\epsilon=b(X_t^\epsilon)dt+\epsilon\,dw_t
\end{equation*}
where $w_t$ is a $r$-dimensional Wiener process. 
It can be shown that letting\footnote{See Theorem 1.1 page 104 of \cite{freidlin98}.}
\begin{equation*}
S_{0T}(\varphi)=\frac{1}{2}\int^T_0|\dot{\varphi_s}-b(\varphi_s)|^2ds
\end{equation*}
with $\lambda(h)=\epsilon^{-2}$ and $h=\epsilon$ satisfy the three properties of the action functional. 
The action functional then allows us to compute asymptotically different probabilities. 
For example, 
let $D\subset \mathbb{R}^r$ be a region of space in $\mathbb{R}^r$ and let 
\begin{align*}
H_D(t,x)&=\{\varphi\in C_{0T}(\mathbb{R}^r): \varphi_0=x, \ \varphi_t\in D \cup \partial D\}\\
\overline{H}_D(t,x)&=\{\varphi\in C_{0T}(\mathbb{R}^r): \varphi_0=x, \ x_s\notin D \ \text{for some } \ s \in [0,t]\}
\end{align*}
then it can be shown that\footnote{See Theorem 1.2 page 105 of \cite{freidlin98}.}  
\begin{align}
\lim_{\epsilon\rightarrow 0}\epsilon^2 \ln P_x\{X_t^\epsilon\in D\}&=-\min_{\varphi \in H_D(t,x)}S_{0T}(\varphi)\label{aug:eq1}\\
\lim_{\epsilon\rightarrow 0}\epsilon^2 \ln P_x\{\tau^\epsilon\leq t\}&=-\min_{\varphi \in \overline{H}_D(t,x)}S_{0T}(\varphi)\label{aug:eq2}
\end{align}
where $\tau^\epsilon=\min\{t:X_t^\epsilon \notin D\}$ is the escape time from $D$. 
This theorem gives us the leading term for the probabilities leaving this region of space $D$. 
The $H_D$ is the set of all paths that stay in $D$ and its boundary. 
The $\overline{H}_D$ is the set of all paths that leave $D$ at some time. 
Equation \ref{aug:eq1} is thus the probability of remaining in $D$ and Equation \ref{aug:eq2} is the probability of the escape time being less than $t$. 

So far  the above results hold for a general region $D$.
Now we want to consider the case where $D$ is the vicinity of a well, that is the region near and around a metastable state. 
This means $D$ is attracted to a point inside of $D$. 
Without loss of generality we can choose this point, which is the position of the well, to be zero $x_{well}=0$.
But the minimiser of $S_{0T}(\varphi)$ is very difficult to compute explicitly using the usual differential equations.
Let  
\begin{align*}
f(t,x,y)=
\min_{     \substack{\varphi_0=x  \\ \varphi_t=y}      }
S_{0t}(\varphi). 
\end{align*}
The Hamilton-Jacobi equations are given by 
\begin{align}
\frac{\partial f(t,x,y)}{\partial t}
+\frac{1}{2}
\left|
\nabla_yf(t,x,y)
\right|^2
+
\left(
b(y),\nabla_yf(t,x,y)
\right)
=0 \label{chap_1_euler_lang}
\end{align}
where $\nabla_y$ is the gradient operator in the variable $y$ and note that
\begin{align*}
\min_{\varphi\in H_D(t,x)}
S_{0t}(\varphi)
&=
\min_{y\in D\cup \partial D}
f(t,x,y)\\[0.5em]
\min_{\varphi\in \overline{H}_D(t,x)}
S_{0t}(\varphi)
&=
\min_{  \substack{0\leq s \leq t  \\ y\notin D}     }
f(s,x,y)
\end{align*}
and the solution to Equation \ref{chap_1_euler_lang} would be closely related to Equation \ref{aug:eq1} and \ref{aug:eq2}.\footnote{See pages 105-108 of \cite{freidlin98}.}
Let us introduce the so-called quasipotential
\begin{equation*}
\tilde{V}(x,y)=\inf\{S_{T_1T_2}(\varphi): \varphi\in C_{T_1T_2}(\mathbb{R}^r), \ 
\varphi_{T_1}=x, \ \varphi_{T_2}=y \}
\end{equation*}
which is the least action over all paths which starts at   $x$ and ends 
at $y$. 
Suppose that the drift term can be written as the gradient of a potential $V$
\begin{equation*}
b(x)=-\nabla V(x)
\end{equation*}
then it can be shown that\footnote{See Theorem 3.1 page 118 of \cite{freidlin98}.}  
\begin{align}
\tilde{V}(0,x)=2V(x). 
\label{eqn:main:height}
\end{align}
Note that this only holds for points $x\in\mathbb{R}^r$ such that $V(x)\leq\min_{y\in \partial D}V(y)$, that is for points lower than the exit point. 
Now suppose that there exists a unique point $y_0\in\partial D$ for which $\tilde{V}(0,y_0)=\min_{y\in\partial D}\tilde{V}(0,y)$ then\footnote{See Theorem 2.1 page 108 of \cite{freidlin98}.}
\begin{align}
\lim_{\epsilon\rightarrow 0} P_x\{\rho(X_s^\epsilon,y_0)<\delta\}=1
\quad \text{where} \quad 
s=\inf\left\{t:X^\epsilon_t\in\partial D\right\}
\label{eqn:main:escape:time}
\end{align}
for every $\delta>0$ and any $x\in D$. 
This means $X_t^\epsilon$ will exit near points of least height in the small noise limit. 

\subsection{Remarks on Freidlin-Wentzell and Large Deviation Theory}
The main Theorems of Freidlin-Wentzell were developed on a precise and rigorous mathematical setting. It would be appropriate to interpret what they mean in a more physical setting. 
Consider Equation \ref{eqn:main:height} and \ref{eqn:main:escape:time}. 
Equation \ref{eqn:main:height} gives an easier way to calculate the quasipotential, because the quasipotential is related in a very simple way to the height of the potential. Equation \ref{eqn:main:escape:time} means that the particle will escape whilst travelling through a path which gives the least height, or interpreted in another way, a path of least action. 
Thus, one of the main conclusion of the Fredlin-Wentzell theory is that the particle will tend to escape following close to a path which gives the least distance to climb out of a well.

\section{Kramers' Formula and Potential Theory}
\label{chap_sect_kram}
Let $V:\mathbb{R}^r\longrightarrow\mathbb{R}$.
Let 
$x\in\mathbb{R}^r$ be a well and  
$z_i\in\mathbb{R}^r$ be  saddles labelled by $i=1\ldots n$. 
The saddles would be gateways providing a passage for escape from the well. 
Define 
\begin{align*}
\Delta V_i=V(z_i)-V(x)
\end{align*}
which is the height difference between the well and the $i$th saddle. 
For small noise $\epsilon$, an approximate expression can be estimated for the escape time of the particle going through  the $i$th saddle. 
In the smallest order of the noise $\epsilon$ the mean exit time, as described in the previous section, is given by\footnote{See Theorem 4.1 and 4.2 on pages 124-127 of \cite{freidlin98}.}
\begin{align*}
\tau_i=e^{+2\Delta V_i /\epsilon^2}. 
\end{align*}
Inverting this gives the escape rate
\begin{align*}
R_i=e^{-2\Delta V_i /\epsilon^2}
\end{align*}
and the total escape rate would be to sum over all the saddles
\begin{align*}
R=\sum_{i=1}^n R_i
=\sum_{i=1}^n e^{-2\Delta V_i /\epsilon^2}. 
\end{align*}
The order correction is done by adding a coefficient called Kramers' coefficient and the resulting corrected rate is called Kramers' rate
\begin{align*}
R_i=k_ie^{-2\Delta V_i /\epsilon^2}
\end{align*}
where 
\begin{align*}
k_i=\frac{\sqrt{\left|\nabla^2V(x)\right|}}{2\pi}
\frac{\left|\lambda(z_i)\right|}{\sqrt{\left\Vert \nabla^2V(z_i) \right\Vert}}
\end{align*}
where $\left|\nabla^2(x)\right|$ denotes the determinant of the  Hessian of the potential at the well $x$, 
$\left\Vert \nabla^2V(z_i) \right\Vert$ denotes the modulus of the determinant of the potential at the saddle $z_i$
and 
$\left|\lambda(z_i)\right|$ denotes the minimum eigenvalue of the Hessian of the potential at the saddle $z_i$. 
This gives the escape rate in the next order of approximation to be 
\begin{align*}
R&=\sum_{i=1}^nR_i
=\sum_{i=1}^n k_ie^{-2\Delta V_i /\epsilon^2}
\end{align*}
which is rewritten as 
\begin{align*}
R=
\frac{\sqrt{\left|\nabla^2V(x)\right|}}{2\pi}
\sum_{i=1}^n
\frac{\left|\lambda(z_i)\right|}{\sqrt{\left\Vert \nabla^2V(z_i) \right\Vert}}
\exp\left\{\frac{-2\left(V(z_i)-V(x)\right)}{\epsilon^2}\right\}. 
\end{align*}
The last order of approximation for higher noise $\epsilon$ is done by bounding the error on Kramers' coefficient. 
This is 
\begin{align*}
k_i=\frac{\sqrt{\left|\nabla^2V(x)\right|}}{2\pi}
\frac{\left|\lambda(z_i)\right|}{\sqrt{\left\Vert \nabla^2V(z_i) \right\Vert}}
\left(
\frac{1}{1+\mathcal{O}\left(\frac{\epsilon^2}{2}\ln\frac{\epsilon^2}{2}\right)}
\right)
\end{align*}
which is rewritten as 
\begin{align}
\label{chap_2_krammers}
\frac{1}{k_i}
=
\frac{2\pi}{\sqrt{\left|\nabla^2V(x)\right|}}
\frac{\sqrt{\left\Vert \nabla^2V(z_i) \right\Vert}}{\left|\lambda(z_i)\right|}
\left[
1+\mathcal{O}\left(\frac{\epsilon^2}{2}\ln\frac{\epsilon^2}{2}\right)
\right]. 
\end{align}
We conclude with a few words on the derivation of Kramers' formula and the bound on its error.
Equation \ref{chap_2_krammers} was derived rigorously using techniques from potential theory instead of large deviations. 
Part of the technique involves the escape time being expressed in terms of a partial differential equation, similar to Equation \ref{chap_1:es3} for example. 
This derivation was done in  \cite{Bovier02metastabilityin}  which is beyond the scope of this thesis. 
This Chapter reviewed the escape rates of a particle from a static well, which will be relevant when we consider escape rates from an oscillating well.

\chapter{Theoretical Escape Time from a Well of an Oscillatory Potential}
\label{chapter_oscil_times}
\label{sum_chap_escape_oscill}

Stochastic resonance usually involves studying transitions between two stable states. 
Studying a stochastic differential equation in multidimensional real space can be complicated. 
It would be useful to simplify stochastic resonance  down to a  Markov Chain with transitions between two states $+1$ and $-1$, 
then we try to model stochastic resonance with a two state Markov Chain.  
This is done for both discrete and continuous time Markov Chains with two states, for both alternating and synchronised saddles.

\section{Markov Chain Reduction}
Let $V:\mathbb{R}^r\longrightarrow\mathbb{R}$ be a potential with two wells. 
This potential is subjected to a periodic forcing $F$ with frequency $\Omega$ and perturbed by noise $\epsilon$, which is described by the SDE\footnote{See Appendix \ref{appendix_potential} for how the forcing is denoted.}
\begin{align}
\dot{X^\epsilon_t}=-\nabla V+F\cos(\Omega t)+\epsilon \dot{W_t}
\label{chap_4:markov}
\end{align}
where $W_t$ is a Wiener process in $\mathbb{R}^r$ and  $F\in\mathbb{R}^r$. 
We call $X^\epsilon_t$  the diffusion case. 
The $X^\epsilon_t$ can be reduced to a Markov Chain $Y^\epsilon_t$ on $\{-1,+1\}$
\begin{align*}
X^\epsilon_t
\longrightarrow
Y^\epsilon_t
\end{align*}
in the following way. 
In what follows we will assume that the diffusion $X^\epsilon_t$ is continuous in time and space. 
Let $w_{l}(t)$ denote the position of the left well at time $t$ and $w_{r}(t)$ the position of the right well at time $t$. 
Note that $w_l(t)$ and $w_r(t)$ are also continuous in time. 
Let $R\in\mathbb{R}$ be constant. 
The reduction from the  $X^\epsilon_t$ to the Markov Chain $Y^\epsilon_t$ is
\begin{align*}
Y^\epsilon_t&=
\left\{
\begin{array}{lll}
-1 & \text{if} &
\left|X^\epsilon_t-w_l(t)\right|\leq R\\[0.5em]
+1 & \text{if} &
\left|X^\epsilon_t-w_r(t)\right|\leq R\\[0.5em]
Z & \text{if} & \left|X^\epsilon_t-w_l(t)\right|> R \quad \text{and} \quad \left|X^\epsilon_t-w_r(t)\right|> R
\end{array}
\right.
\end{align*}
where $Z$ is given by 
\begin{align*}
Z&=
\left\{
\begin{array}{lll}
-1 & \text{if} &
s_2<s_1\\[0.5em]
+1 & \text{if} &
s_1<s_2
\end{array}
\right.
\end{align*}
where $s_1$ and $s_2$ are given by 
\begin{align*}
s_1&=\max_{u<t}\left\{u:\left|X^\epsilon_u-w_l(u)\right|\leq R\right\}\\
s_2&=\max_{u<t}\left\{u:\left|X^\epsilon_u-w_r(u)\right|\leq R\right\}.
\end{align*}
When $Y^\epsilon_t=-1$ we say the particle is  in the left well 
and when $Y^\epsilon_t=+1$ we say the particle is  in the right well. 
Only when it enters the other well would $Y^\epsilon_t$ change sign. 
When the condition $\left|X^\epsilon_u-w_l(u)\right|\leq R$ is satisfied we say the particle is covered by the left well. 
When the condition $\left|X^\epsilon_u-w_r(u)\right|\leq R$ is satisfied we say the particle is covered by the right well.
Note that $R$ is chosen small enough such that it is impossible for the particle to be covered by both wells at any time, that is 
\begin{align*}
\left\{
x\in\mathbb{R}^r\,:\,
\left|X^\epsilon_t-w_l(t)\right|<R 
\quad \text{and} \quad 
\left|X^\epsilon_t-w_r(t)\right|<R 
\right\}=\emptyset
\end{align*}
for all times $t\geq0$. 
This means that $s_1$ is the most recent time the particle is covered by the left well and $s_2$ is the most recent time the particle is covered by the right well. 
Notice that if initially at $t=0$, the particle is covered by neither well then $Y^\epsilon_t$ cannot be derived nor defined by the above definitions. 
In this case either $Y^\epsilon_0=-1$ or $Y^\epsilon_0=+1$ is chosen depending on what initial conditions are required.  
In other words if $Y^\epsilon_0=-1$  is chosen as the initial condition then the particle is covered by the left well for $t<0$.  
If $Y^\epsilon_0=+1$ is chosen as the initial condition then the particle is covered by the right well for $t<0$.

The escape time from the left to right well $\tau_{-1+1}$ and from the right to left well $\tau_{+1-1}$ are defined in the following way\footnote{See Appendix \ref{appendix_escape_times} for details of the actual use of $R$ in the measurement of the escape times. \label{chapter_oscil_times_R}}
\begin{align*}
\tau_{-1+1}&=\mu\left(\left\{t:Y^\epsilon_t=-1\right\}\right)\quad \text{where} \quad \left\{t:Y^\epsilon_t=-1\right\} \quad \text{is an interval}\\
\tau_{+1-1}&=\mu\left(\left\{t:Y^\epsilon_t=+1\right\}\right)\quad \text{where} \quad \left\{t:Y^\epsilon_t=+1\right\} \quad \text{is an interval}
\end{align*}
where $\mu$ denotes the Lebesgue measure. 
In other words the time spent being in the state $Y^\epsilon_t=-1$ is $\tau_{-1+1}$ and the time spent being in the state $Y^\epsilon_t=+1$ is $\tau_{+1-1}$. 
These intervals will always be closed intervals. 
The process $Y^\epsilon_t$ has two states, hence each sample is a piecewise constant function.
The length of each piece is the escape time $\tau_{-1+1}$ or $\tau_{+1-1}$. 
Note that $\tau_{-1+1}$ and $\tau_{+1-1}$ are random times and random variables.

For the diffusion the escape time can be explained in the following way. 
Each well is surrounded by a circle with a constant radius $R$ which moves with the well.  
A particle is said to have entered the left well if it enters the region covered by the radius $R$ over the left well. 
The particle is then said to have entered the right well  when it enters the region covered by $R$ in the right well. 
The time difference between entering the left well and entering the right well is defined to be the escape time from left to right $\tau_{-1+1}$. 
A similar argument is said for $\tau_{+1-1}$. 
This also means the escape times in the Markov Chain is the same as the diffusion trajectory $X^\epsilon_t$ by definition. 
Notice that the diffusion $X^\epsilon_t$ has to be defined first before the Markov Chain $Y^\epsilon_t$ which is a derived quantity.  

Notice that all of our reasoning in deriving $Y^\epsilon_t$ only assumes that $X^\epsilon_t$ is a continuous time process in $\mathbb{R}^r$. We did not check whether $Y^\epsilon_t$ satisfy the strict definitions of a continuous time Markov Chain. If $Y^\epsilon_t$ is a Markov Chain it should also satisfy the Markov property, that is 
\begin{align*}
P\left(Y^\epsilon_t=i|Y^\epsilon_{t_1}=-1,Y^\epsilon_{t_2}=+1\right)&=P\left(Y^\epsilon_t=i|Y^\epsilon_{t_1}=-1\right)\\[0.5em]
&=P\left(Y^\epsilon_t=i|Y^\epsilon_{t_2}=+1\right)
\end{align*}
for any $0\leq t_1<t_2<t$. 
Again we stress that the only assumption we made when deriving $Y^\epsilon_t$ from $X^\epsilon_t$ is that $X^\epsilon_t$  is a continuous time process in $\mathbb{R}^r$, which is not sufficient for $Y^\epsilon_t$ to be a Markov Chain nor for $Y^\epsilon_t$ to satisfy the Markov property. But throughout the rest of this thesis the diffusion $X^\epsilon_t$ will be a Markov process, which means $Y^\epsilon_t$ should be a good approximation to a Markov Chain.\footnote{Whether $Y^\epsilon_t$ is a Markov Chain for a Markov process $X^\epsilon_t$, or for $X^\epsilon_t$ described by an SDE requires proof. This is an open question.}

A discrete time and continuous time Markov Chain model for Equation \ref{chap_4:markov} are studied in the following sections.

\section{Discrete Time Markov Chain}
Let the time be discrete. This to say time $t$ belongs to 
\begin{align*}
t\in\left\{0,1,2,\ldots\right\}.
\end{align*}
The Markov Chain is a time dependent stochastic process which can take values $+1$ or $-1$
\begin{align*}
Y_t=\pm1.
\end{align*}
At time $t$ the probability of $Y_t$ jumping from  $-1$ to $+1$ is denoted by $p_{-1+1}(t)$;
the probability of  jumping from  $+1$ to $-1$ is denoted by $p_{+1-1}(t)$;
the probability of staying in $-1$ is denoted by $p_{-1-1}(t)$;
and the probability of staying in $+1$ is denoted by $p_{+1+1}(t)$.
Notice that they have the following properties for all time $t$
\begin{align*}
p_{-1-1}(t)+p_{-1+1}(t)&=1\\
p_{+1+1}(t)+p_{+1-1}(t)&=1.
\end{align*}
A transition matrix can be defined as
\begin{align*}
P_t:=
\left(
\begin{array}{cc}
p_{-1-1}(t) & p_{-1+1}(t)\\[0.5em]
p_{+1-1}(t) & p_{+1+1}(t)
\end{array}
\right). 
\end{align*}
At every point in time it is possible to define a state probability, that is the probability of the trajectory being  $-1$ or $+1$,
\begin{align*}
P\left(Y_t=-1\right)&=\nu_-(t)\\
P\left(Y_t=+1\right)&=\nu_+(t).
\end{align*}
Notice that the state probability satisfy the following  condition for all time $t$
\begin{align*}
\nu_-(t)+\nu_+(t)=1.
\end{align*}
The two $\nu_-(t)$ and $\nu_+(t)$ can be written compactly in vector notation 
\begin{align*}
\nu(t)=
\left(
\begin{array}{c}
\nu_-(t) \\[0.5em] \nu_+(t)
\end{array}
\right). 
\end{align*}
The state probability at time $t+1$ can be expressed in terms of the last time $t$, that is
\begin{align*}
\nu(t+1)&=P_t^\dagger\nu(t)\\[0.5em]
\left(
\begin{array}{c}
\nu_-(t+1) \\ \nu_+(t+1)
\end{array}
\right)
&=
\left(
\begin{array}{cc}
p_{-1-1}(t) & p_{+1-1}(t)\\[0.5em]
p_{-1+1}(t) & p_{+1+1}(t)
\end{array}
\right)
\left(
\begin{array}{c}
\nu_-(t) \\[0.5em] \nu_+(t)
\end{array}
\right)
\end{align*}
where $P_t^\dagger$ denote the transpose of the matrix $P_t$. 
This means if the initial value of the state probability is known at $t=0$, then the future behaviour of the state probability can be described by computing all subsequent values of $\nu(t)$, that is 
\begin{align*}
\nu(t)=\prod_{i=0}^{i=t-1}P_i^\dagger\nu(0). 
\end{align*}
The main aim for the rest of our studies of the Markov Chain is to compute the state probability for various transition matrices. 
When the wells of the potential are oscillating such that one well is higher than the other, we model using $p\neq q $. 
When the wells of the potential are oscillating such that both wells are always at the same height as each other, we model using $p=q$.

\subsection{Discrete Time Markov Chain - Alternating Saddles $p\neq q$}
We want to study a system with periodic elements. 
The transition matrix would change periodically in time. 
Let the period be 
\begin{align*}
T=2m
\end{align*}
where $m$ is an integer. 
Let the time $t$ be written in the form
\begin{align*}
t=NT+n
\end{align*}
where $N$ is an integer number of periods. 
The transition matrix would vary periodically according to 
\begin{align*}
\text{For} \quad n=mod(t,T) \in T_1 \quad P_t=P_1=
\left(
\begin{array}{cc}
1-p & p \\[0.5em]
q & 1-q
\end{array}
\right)\\[0.5em]
\text{For} \quad n=mod(t,T) \in T_2 \quad P_t=P_2=
\left(
\begin{array}{cc}
1-q & q \\[0.5em]
p & 1-p
\end{array}
\right)
\end{align*}
where
\begin{align*}
T_1&=\left\{0,1,\ldots,m-1\right\}\\
T_2&=\left\{m,m+1,\ldots,2m-1\right\}.
\end{align*}
We  interpret $Y_t=-1$ as being in the left well and $Y_t=+1$ as being in the right well. 
We  also interpret $p$ as the probability of escape from a shallow well and $q$ as the probability of escape from a deep well. 
The transition matrix varying periodically in time can be used to model the periodic forcing being applied to the potential. 
The following Theorem derives the state probabilities. 
\begin{Theorem}\label{chap_4_thm:equal}
Let the time be $t=NT+n$. 
Let $\lambda=1-p-q$. 
The state probability at time $t$ is 
\begin{align*}
\text{For} \quad n\in T_1 \quad 
\nu&=\frac{1}{p+q}\left(\begin{array}{c}q\\p\end{array}\right)
-\frac{p-q}{p+q}\times\frac{\lambda^n}{1+\lambda^m}\left(\begin{array}{c}-1\\1\end{array}\right)\\[0.5em]
&\relphantom{=}
+
\frac{\nu_+(0)(p+q\lambda^m)-\nu_-(0)(q+p\lambda^m)}{1+\lambda^m}\lambda^{2mN+n}\left(\begin{array}{c}-1\\1\end{array}\right)\\[0.5em]
\text{For} \quad n\in T_2 \quad 
\nu&=\frac{1}{p+q}\left(\begin{array}{c}p\\q\end{array}\right)
+\frac{p-q}{p+q}\times\frac{\lambda^{n-m}}{1+\lambda^m}\left(\begin{array}{c}-1\\1\end{array}\right)\\[0.5em]
&\relphantom{=}
+
\frac{\nu_+(0)(p+q\lambda^m)-\nu_-(0)(q+p\lambda^m)}{1+\lambda^m}\lambda^{2mN+n}\left(\begin{array}{c}-1\\1\end{array}\right)
\end{align*}
\end{Theorem}

\begin{proof}
Notice that the eigenvectors and eigenvalues of the transpose matrix $P_1^\dagger$ are 
\begin{align*}
\lambda_1&=1-p-q \quad  \quad \quad v_1=\left(\begin{array}{c}-1\\1\end{array}\right)\\[0.5em]
\lambda_2&=1 \relphantom{-p-q} \quad \quad \quad  v_2=\left(\begin{array}{c}q\\p\end{array}\right)
\end{align*}
which also spans the $\mathbb{R}^2$ space. 
For short we call $\lambda=\lambda_1$. 
This means an arbitrary vector can be expressed as a linear combination of the eigenvectors of $P_1^\dagger$. 
This is 
\begin{align*}
\left(
\begin{array}{c}
x\\y
\end{array}
\right)
=
\frac{1}{p+q}
\left\{
(qy-px)
\left(
\begin{array}{c}
-1\\1
\end{array}
\right)
+
(x+y)
\left(
\begin{array}{c}
q\\p
\end{array}
\right)
\right\}. 
\end{align*}
Now consider $m$ application of the $P_1^\dagger$ matrix on the arbitrary vector.
\begin{align*}
(P_1^\dagger)^m
\left(
\begin{array}{c}
x\\y
\end{array}
\right)
&=
\frac{1}{p+q}
\left\{
(qy-px)
\lambda^m
\left(
\begin{array}{c}
-1\\1
\end{array}
\right)
+
(x+y)
\left(
\begin{array}{c}
q\\p
\end{array}
\right)
\right\}\\[0.5em]
&=
\frac{1}{p+q}
\left(
\begin{array}{cc}
q+p\lambda^m & q(1-\lambda^m)\\
p(1-\lambda^m) & p+q\lambda^m
\end{array}
\right)
\left(
\begin{array}{c}
x\\y
\end{array}
\right)\\[0.5em]
&=
\left(
\begin{array}{cc}
1-p' & p'\\
q' & 1-q'
\end{array}
\right)^\dagger
\left(
\begin{array}{c}
x\\y
\end{array}
\right)
\end{align*}
where 
\begin{align*}
p'=\frac{p(1-\lambda^m)}{p+q}
\quad \text{and} \quad 
q'=\frac{q(1-\lambda^m)}{p+q}
\end{align*}
with a similar expression for the other transition matrix
\begin{align*}
(P_2^\dagger)^m
\left(
\begin{array}{c}
x\\y
\end{array}
\right)
=
\left(
\begin{array}{cc}
1-q' & q'\\
p' & 1-p'
\end{array}
\right)^\dagger
\left(
\begin{array}{c}
x\\y
\end{array}
\right). 
\end{align*}
Now denote a new matrix by 
\begin{align*}
P_{tot}&=(P_2^\dagger)^m(P_1^\dagger)^m\\
&=
\left(
\begin{array}{cc}
1-q' & p'\\
q' & 1-p'
\end{array}
\right)
\left(
\begin{array}{cc}
1-p' & q'\\
p' & 1-q'
\end{array}
\right)
\end{align*}
where $P_{tot}$ has eigenvalues and eigenvectors 
\begin{align*}
\xi_1&=1 \relphantom{(-p'-q')^2 } \quad \quad \quad e_1=\left(\begin{array}{c}1-q'\\1-p'\end{array}\right)\\[0.5em]
\xi_2&=(1-p'-q')^2 \quad \quad \quad e_2=\left(\begin{array}{c}-1\\1\end{array}\right)
\end{align*}
and these eigenvectors span the $\mathbb{R}^2$ space
\begin{align*}
\left(\begin{array}{c}x\\y\end{array}\right)
=
\frac{1}{1+\lambda'}
\left\{
(x+y)
\left(
\begin{array}{c}
1-q'\\1-p'
\end{array}
\right)
+
\left[
y(1-q')-x(1-p')
\right]
\left(
\begin{array}{c}
-1\\1
\end{array}
\right)
\right\}
\end{align*}
where we have denoted 
\begin{align*}
\lambda'=1-p'-q'. 
\end{align*}
Now consider $N$ applications of the matrix $P_{tot}$ on the initial value of the state probability
\begin{align*}
P_{tot}^N\,\nu(0)
&=
\frac{1}{1+\lambda'}
\left\{
\left(
\begin{array}{c}
1-q'\\1-p'
\end{array}
\right)
+
\left[
\nu_+(0)(1-q')-\nu_-(0)(1-p')
\right]
\lambda'^{2N}
\left(
\begin{array}{c}
-1\\1
\end{array}
\right)
\right\}
\end{align*}
and express the results in terms of the eigenvectors of $P_1^\dagger$
\begin{align*}
P_{tot}^N\,\nu(0)
&=
\frac{1}{1+\lambda'}
\times
\frac{1}{p+q}
\left\{
\left[
q(1-p')-p(1-q')
\right]
\left(
\begin{array}{c}
-1\\1
\end{array}
\right)
+
(1+\lambda')
\left(
\begin{array}{c}
q\\p
\end{array}
\right)
\right\}\\[0.5em]
&\relphantom{=}
+\frac{\left[\nu_+(0)(1-q')-\nu_-(0)(1-p')\right]\lambda'^{2N}}{1+\lambda'}
\left(
\begin{array}{c}
-1\\1
\end{array}
\right). 
\end{align*}
If $n\in T_1$ consider
\begin{align*}
(P_1^\dagger)^nP_{tot}^N\,\nu(0)
&=
\frac{1}{1+\lambda'}
\times
\frac{1}{p+q}
\left\{
\left[
q(1-p')-p(1-q')
\right]\lambda^n
\left(
\begin{array}{c}
-1\\1
\end{array}
\right)
+
(1+\lambda')
\left(
\begin{array}{c}
q\\p
\end{array}
\right)
\right\}\\[0.5em]
&\relphantom{=}
+\frac{\left[\nu_+(0)(1-q')-\nu_-(0)(1-p')\right]\lambda'^{2N}}{1+\lambda'}\lambda^n
\left(
\begin{array}{c}
-1\\1
\end{array}
\right)\\[0.5em]
&=\frac{1}{p+q}\left(\begin{array}{c}q\\p\end{array}\right)
-\frac{p-q}{p+q}\times\frac{\lambda^n}{1+\lambda^m}\left(\begin{array}{c}-1\\1\end{array}\right)\\[0.5em]
&\relphantom{=}
+
\frac{\nu_+(0)(p+q\lambda^m)-\nu_-(0)(q+p\lambda^m)}{1+\lambda^m}\lambda^{2mN+n}\left(\begin{array}{c}-1\\1\end{array}\right).
\end{align*}
If $n\in T_2$ consider
\begin{align*}
(P_2^\dagger)^{n-m}(P_1^\dagger)^mP_{tot}^N\,\nu(0)
&=(P_2^\dagger)^{n-m}
\left[
\frac{1}{p+q}\left(\begin{array}{c}q\\p\end{array}\right)
-\frac{p-q}{p+q}\times\frac{\lambda^m}{1+\lambda^m}\left(\begin{array}{c}-1\\1\end{array}\right)
\right. 
\\[0.5em]
&\relphantom{=}
+
\left.
\frac{\nu_+(0)(p+q\lambda^m)-\nu_-(0)(q+p\lambda^m)}{1+\lambda^m}\lambda^{2mN+m}\left(\begin{array}{c}-1\\1\end{array}\right)
\right]
\end{align*}
and we express $(q,p)^\dagger$ in terms of the eigenvectors of $P_2^\dagger$
\begin{align*}
(P_2^\dagger)^{n-m}(P_1^\dagger)^mP_{tot}^N\,\nu(0)
&=(P_2^\dagger)^{n-m}
\left[
\frac{p-q}{p+q}\left(\begin{array}{c}-1\\1\end{array}\right)
+\frac{1}{p+q}\left(\begin{array}{c}p\\q\end{array}\right)
-\frac{p-q}{p+q}\times\frac{\lambda^m}{1+\lambda^m}\left(\begin{array}{c}-1\\1\end{array}\right)
\right. 
\\[0.5em]
&\relphantom{=}
+
\left.
\frac{\nu_+(0)(p+q\lambda^m)-\nu_-(0)(q+p\lambda^m)}{1+\lambda^m}\lambda^{2mN+m}\left(\begin{array}{c}-1\\1\end{array}\right)
\right]\\[0.5em]
&=\frac{1}{p+q}\left(\begin{array}{c}p\\q\end{array}\right)
+\frac{p-q}{p+q}\times\frac{\lambda^{n-m}}{1+\lambda^m}\left(\begin{array}{c}-1\\1\end{array}\right)\\[0.5em]
&\relphantom{=}
+
\frac{\nu_+(0)(p+q\lambda^m)-\nu_-(0)(q+p\lambda^m)}{1+\lambda^m}\lambda^{2mN+n}\left(\begin{array}{c}-1\\1\end{array}\right). 
\end{align*}
This completes the proof. 
\end{proof}

\subsection{Discrete Time Markov Chain - Synchronised Saddles $p=q$}
Let the period be 
\begin{align*}
T=4m
\end{align*}
where $m$ is an integer. 
Let the time $t$ be written in the form
\begin{align*}
t=NT+n
\end{align*}
where $N$ is an integer number of periods. 
The transition matrix would vary periodically according to 
\begin{align*}
\text{For} \quad n=mod(t,T) \in T_1 \quad P_1=
\left(
\begin{array}{cc}
1-p & p \\[0.5em]
p & 1-p
\end{array}
\right)\\[0.5em]
\text{For} \quad n=mod(t,T) \in T_2 \quad P_2=
\left(
\begin{array}{cc}
1-q & q \\[0.5em]
q & 1-q
\end{array}
\right)\\[0.5em]
\text{For} \quad n=mod(t,T) \in T_3 \quad P_3=
\left(
\begin{array}{cc}
1-p & p \\[0.5em]
p & 1-p
\end{array}
\right)\\[0.5em]
\text{For} \quad n=mod(t,T) \in T_4 \quad P_4=
\left(
\begin{array}{cc}
1-q & q \\[0.5em]
q & 1-q
\end{array}
\right)
\end{align*}
where 
\begin{align*}
T_1&=\left\{0,1,\ldots,m-1\right\}\\
T_2&=\left\{m,m+1,\ldots,2m-1\right\}\\
T_3&=\left\{2m.2m+1,\ldots,3m-1\right\}\\
T_4&=\left\{3m,3m+1,\ldots,4m-1\right\}
\end{align*}
where again $p$ should be interpreted as the probability of escape from a shallow well and $q$ from a deep well.
The following Theorem derives the state probabilities. 
\begin{Theorem}\label{chap_4_thm:not_equal}
Let the time be $t=NT+n$. The state probability at time $t$ is 
\begin{align*}
\text{For} \quad n \in T_1 \quad 
\nu&=\frac{1}{2}
\left\{
\left(
\begin{array}{c}
1\\1
\end{array}
\right)
+
(1-2p)^{2mN+n}
(1-2q)^{2mN}
[\nu_+(0)-\nu_-(0)]
\left(
\begin{array}{c}
-1\\1
\end{array}
\right)
\right\}\\[0.5em]
\text{For} \quad n \in T_2 \quad 
\nu&=\frac{1}{2}
\left\{
\left(
\begin{array}{c}
1\\1
\end{array}
\right)
+
(1-2p)^{2mN+m}
(1-2q)^{2mN+(n-m)}
[\nu_+(0)-\nu_-(0)]
\left(
\begin{array}{c}
-1\\1
\end{array}
\right)
\right\}\\[0.5em]
\text{For} \quad n \in T_3 \quad 
\nu&=\frac{1}{2}
\left\{
\left(
\begin{array}{c}
1\\1
\end{array}
\right)
+
(1-2p)^{2mN+m+(n-2m)}
(1-2q)^{2mN+m}
[\nu_+(0)-\nu_-(0)]
\left(
\begin{array}{c}
-1\\1
\end{array}
\right)
\right\}\\[0.5em]
\text{For} \quad n \in T_4 \quad 
\nu&=\frac{1}{2}
\left\{
\left(
\begin{array}{c}
1\\1
\end{array}
\right)
+
(1-2p)^{2mN+2m}
(1-2q)^{2mN+m+(n-3m)}
[\nu_+(0)-\nu_-(0)]
\left(
\begin{array}{c}
-1\\1
\end{array}
\right)
\right\}
\end{align*}
\end{Theorem}

\begin{proof}
Notice that the transpose of the matrix $P_1^\dagger$ has the following eigenvalues and eigenvectors
\begin{align*}
\lambda_1&=1-2p \quad v_1=\left(\begin{array}{c}-1\\1\end{array}\right)\\[0.5em]
\lambda_2&=1 \quad \quad \quad \, \, v_2=\left(\begin{array}{c}1\\1\end{array}\right). 
\end{align*}
These eigenvectors span the space, which means any vectors can be expressed as a linear combination of them 
\begin{align*}
\left(
\begin{array}{c}
x \\ y
\end{array}
\right)
=
\frac{1}{2}
\left\{
(y-x)
\left(
\begin{array}{c}
-1\\1
\end{array}
\right)
+(x+y)
\left(
\begin{array}{c}
1\\1
\end{array}
\right)
\right\}. 
\end{align*}
This means the initial values of the state probability $\nu$ can be expressed in terms of the eigenvectors of $P_1^\dagger$. 
Denote the matrix 
\begin{align*}
P_{tot}=(P_2^\dagger)^m(P_1^\dagger)^m(P_2^\dagger)^m(P_1^\dagger)^m
\end{align*}
which is the total transition matrix in one period. 
Proceeding we have 
\begin{align*}
P_{tot}^N\,\nu(0)=
\frac{1}{2}
\left\{
\left(
\begin{array}{c}
1\\1
\end{array}
\right)
+(1-2p)^{2mN}
(1-2q)^{2mN}
[\nu_+(0)-\nu_-(0)]
\left(
\begin{array}{c}
-1\\1
\end{array}
\right)
\right\}
\end{align*}
with the added condition $\nu_-(0)+\nu_+(0)=1$. 
Now note the following
\begin{align*}
\text{For} \quad n\in T_1 \quad \nu(t)&=(P_1^\dagger)^nP_{tot}^N\,\nu(0)\\[0.5em]
\text{For} \quad n\in T_2 \quad \nu(t)&=(P_2^\dagger)^{n-m}(P_1^\dagger)^mP_{tot}^N\,\nu(0)\\[0.5em]
\text{For} \quad n\in T_3 \quad \nu(t)&=(P_1^\dagger)^{n-2m}(P_2^\dagger)^{m}(P_1^\dagger)^mP_{tot}^N\nu(0)\,\\[0.5em]
\text{For} \quad n\in T_4 \quad \nu(t)&=(P_2^\dagger)^{n-3m}(P_1^\dagger)^{m}(P_2^\dagger)^{m}(P_1^\dagger)^mP_{tot}^N\,\nu(0)
\end{align*}
This completes the proof.
\end{proof}

\subsection{Discrete Time Markov Chain - Invariant Measures, Relaxation Time and Fourier Transform}
We consider a discrete Markov Chain on $\left\{-1,+1\right\}$, that is 
\begin{align*}
Y^\epsilon_t=\pm1
\end{align*}
and the time is 
\begin{align*}
t\in\left\{0,1,2,\ldots\right\}
\end{align*}
and the probabilities for being in $Y^\epsilon_t=-1$ or $Y^\epsilon_t=+1$ at time $t$ are given by the state probabilities $\nu_\pm(t)$
\begin{align*}
P\left(Y^\epsilon_t=-1\right)=\nu_-(t)
\quad \text{and} \quad 
P\left(Y^\epsilon_t=+1\right)=\nu_+(t).
\end{align*}
The probabilities of transitions occurring as given in the transition matrices changes with period $T$. 
After a very long time the state probabilities $\nu_\pm(\cdot)$ should not depend on the initial state probabilities $\nu_\pm(0)$. 
At time infinity $\nu_\pm(\cdot)$ should also be cyclic on $[0,T]$. 
Let the time be given by $t=NT+n$ where $N$ is a discrete number of periods.
This leads us to define the invariant measure as the state probabilities in the limit as $N\longrightarrow\infty$
\begin{align*}
\overline{\nu}(n):=\lim_{N\longrightarrow\infty}\nu(NT+n)
\end{align*}
and since $\overline{\nu}(\cdot)$ is periodic on $[0,T]$ it should also satisfy
\begin{align*}
\overline{\nu}(t+T)
=
\prod_{i=t}^{i=t+T-1}P_i^\dagger
\overline{\nu}(t)
\end{align*}
where $P_i$ are the transition matrices, that is to say $\overline{\nu}$ is invariant over one period of application of the transition matrices. 
This brings us to the following.

\begin{Corollary}
For the state probabilities in Theorem \ref{chap_4_thm:equal} the invariant measures are 
\begin{align*}
\text{For} \quad t\in T_1 \quad 
\overline{\nu}(t)&=\frac{1}{p+q}\left(\begin{array}{c}q\\p\end{array}\right)
-\frac{p-q}{p+q}\times\frac{\lambda^t}{1+\lambda^m}\left(\begin{array}{c}-1\\1\end{array}\right)\\[0.5em]
\text{For} \quad t\in T_2 \quad 
\overline{\nu}(t)&=\frac{1}{p+q}\left(\begin{array}{c}p\\q\end{array}\right)
+\frac{p-q}{p+q}\times\frac{\lambda^{t-m}}{1+\lambda^m}\left(\begin{array}{c}-1\\1\end{array}\right)
\end{align*}
\end{Corollary}

\begin{Corollary}\label{cor_half_discrete}
For the state probabilities in Theorem \ref{chap_4_thm:not_equal} the invariant measures are
\begin{align*}
\overline{\nu}(t)=\frac{1}{2}
\left(
\begin{array}{c}
1\\1
\end{array}
\right)
\end{align*} 
\end{Corollary}

\noindent The proof is easy and omitted. 
The fact that the $\overline{\nu}$ are invariant over one period of application of the transition matrices follow from the proof of the Theorems. 

The rate of convergence to the invariant measure would depend on the value of $p$ and $q$ themselves. Define the relaxation time $T_{relax}$ as the first time $t=T_{relax}$ such that
\begin{align*}
\left|
\overline{\nu}\left(T_{relax}\right)
-\nu\left(T_{relax}\right)
\right|
\leq
e^{-1}
\end{align*}
which is a measure of the rate of convergence to the invariant measure. 

Consider the averaged Markov Chain over many realisations. 
This is related to the invariant measure by 
\begin{align*}
\langle Y^\epsilon_t \rangle=\overline{\nu}_-(t)(-1)+\overline{\nu}_+(t)(+1)=\overline{\nu}_+(t)-\overline{\nu}_-(t). 
\end{align*}
The Fourier Transform of the averaged Markov Chain $\langle\tilde{Y}_\omega^\epsilon\rangle$ is often studied (see Chapter \ref{chap_analysis_theory}), that is 
\begin{align*}
\langle\tilde{Y}_\omega^\epsilon\rangle
&=\mathcal{F}\left(\langle Y^\epsilon_t \rangle\right). 
\end{align*}
We can Fourier Transform both the alternating saddle $p\neq q $ case and the synchronised saddle $p=q$ case. 
This brings us to the following. 

\begin{Corollary}
For the Markov Chain in Theorem \ref{chap_4_thm:equal} the Fourier Transform of the averaged trajectory is 
\begin{align*}
\langle\tilde{Y}_\omega^\epsilon\rangle
=\frac{1}{T}
\frac{p-q}{p+q}
\left(
1-e^{-i\pi\omega}
\right)
\left\{
\frac{1-e^{-i\pi\omega}}{1-e^{-i\pi\omega/m}}
-\frac{2}{1+\lambda^m}\frac{1-\lambda^me^{-i\pi\omega}}{1-\lambda e^{-i\pi \omega /m}}
\right\}. 
\end{align*}
\end{Corollary}

\begin{proof}
Notice that
\begin{align*}
\overline{\nu}_-(t+m)=\overline{\nu}_+(t)
\quad \text{and} \quad
\overline{\nu}_+(t+m)=\overline{\nu}_-(t)
\end{align*}
so we have 
\begin{align*}
\langle\tilde{Y}_\omega^\epsilon\rangle
&=\mathcal{F}\left(\langle Y^\epsilon_t \rangle\right)\\
&=\frac{1}{T}\sum_{t=0}^{2m-1}\langle Y^\epsilon_t\rangle e^{-2\pi i\omega t/2m}\\
&=\frac{1}{T}\sum_{t=0}^{m-1}\langle Y^\epsilon_t\rangle e^{-2\pi i\omega t/2m}+\frac{1}{T}\sum_{t=m}^{2m-1}\langle Y^\epsilon_t\rangle e^{-2\pi i\omega t/2m}\\
&=\frac{1}{T}\sum_{t=0}^{m-1}\left[\overline{\nu}_+(t)-\overline{\nu}_-(t)\right]e^{-2\pi i\omega t/2m}+\frac{1}{T}\sum_{t=0}^{m-1}\left[\overline{\nu}_-(t)-\overline{\nu}_+(t)\right]e^{-2\pi i\omega (t+m)/2m}\\
&=\frac{1}{T}\sum_{t=0}^{m-1}\left[\overline{\nu}_+(t)-\overline{\nu}_-(t)\right]
\left(
e^{-2\pi i \omega t /2m}-e^{-2\pi i \omega (t+m)/2m}
\right)\\
&=\frac{1}{T}
\left(
1-e^{-i\pi\omega}
\right)
\sum_{t=0}^{m-1}\left[\overline{\nu}_+(t)-\overline{\nu}_-(t)\right]e^{-i\pi\omega t /m}\\
&=\frac{1}{T}
\left(
1-e^{-i\pi\omega}
\right)
\frac{p-q}{p+q}
\sum_{t=0}^{m-1}
\left(
1-2\frac{\lambda^t}{1+\lambda^m}
\right)
e^{-i\pi\omega t /m}\\
&=\frac{1}{T}
\frac{p-q}{p+q}
\left(
1-e^{-i\pi\omega}
\right)
\left\{
\frac{1-e^{-i\pi\omega}}{1-e^{-i\pi\omega/m}}
-\frac{2}{1+\lambda^m}\frac{1-\lambda^me^{-i\pi\omega}}{1-\lambda e^{-i\pi \omega /m}}
\right\}. 
\end{align*}
This completes the proof. 
\end{proof}

\begin{Corollary}\label{transform_discrete_chain}
For the Markov Chain in Theorem \ref{chap_4_thm:not_equal} the Fourier Transform of the averaged trajectory is 
\begin{align*}
\langle\tilde{Y}_\omega^\epsilon\rangle
=0. 
\end{align*}
\end{Corollary}

\noindent Again the proof is trivial and omitted. 
If we study the Fourier Transform at $\omega=1$, this would be the same as studying the driving frequency, which is the frequency at which the transition matrices are changing. 
The  physical intuition is that one has the most significant response at this frequency.

\section{Continuous Time Markov Chain}
Let the time be continuous. 
This is to say time $t$ belongs to 
\begin{align*}
t\in\mathbb{R}.
\end{align*}
The Markov Chain is a time dependent stochastic process with values $-1$ or $+1$,
\begin{align*}
Y_t=\pm1.
\end{align*}
Let $p$ and $q$ be real functions 
\begin{align*}
p:\mathbb{R}\longrightarrow \mathbb{R}
\quad \text{and} \quad 
q:\mathbb{R}\longrightarrow \mathbb{R}
\end{align*}
and $p$ and $q$ are periodic on $[0,T]$
\begin{align*}
p(t+T)=p(t) 
\quad \text{and} \quad 
q(t+T)=q(t). 
\end{align*}
Let $A\subseteq[0,T]$ be a subset of the interval $[0,T]$. 
The probability of $Y_t$ transiting from $Y_t=-1$ to $Y_t=+1$ for the times in $A$, $t\in A$, is denoted by 
\begin{align*}
p_{-1+1}(A).
\end{align*}
Similarly the probability of $Y_t$ transiting from $Y_t=+1$ to $Y_t=-1$ for the times in $A$, $t\in A$, is denoted by 
\begin{align*}
p_{+1-1}(A). 
\end{align*}
The probability of $Y_t$ staying at $-1$ in the time $t\in A$ is given by 
\begin{align*}
p_{-1-1}(A)=1-p_{-1+1}(A).
\end{align*}
Similarly the probability of $Y_t$ staying at $+1$ in the time $t\in A$ is given by 
\begin{align*}
p_{+1+1}(A)=1-p_{+1-1}(A).
\end{align*}
If $A$ is a small time interval $A=[t,t+\delta t]$ then the following infinitesimal representation can be made 
\begin{align*}
p_{-1+1}([t,t+\delta t])&=p(t)\delta t \\
p_{+1-1}([t,t+\delta t])&=q(t)\delta t.
\end{align*}
Now we consider a small change in the state probabilities at times $t$ and $t+\delta t$. 
\begin{align*}
\left(
\begin{array}{c}
\nu_-(t+\delta t)\\[0.5em] \nu_+(t+\delta t) 
\end{array}
\right)
&=
\left(
\begin{array}{cc}
p_{-1-1}([t,t+\delta t]) & p_{-1+1}([t,t+\delta t])\\[0.5em]
p_{+1-1}([t,t+\delta t]) & p_{+1+1}([t,t+\delta t]) 
\end{array}
\right)^\dagger
\left(
\begin{array}{c}
\nu_-(t)\\[0.5em] \nu_+(t) 
\end{array}
\right)\\[0.5em]
&=
\left(
\begin{array}{cc}
1-p(t)\delta t & p(t)\delta t\\[0.5em]
q(t)\delta t & 1-q(t)\delta t
\end{array}
\right)^\dagger
\left(
\begin{array}{c}
\nu_-(t)\\[0.5em] \nu_+(t) 
\end{array}
\right)
\end{align*}
then 
\begin{align*}
\nu(t+\delta t)-\nu(t)
&=
\left(
\begin{array}{c}
\nu_-(t+\delta t)\\[0.5em] \nu_+(t+\delta t) 
\end{array}
\right)
-
\left(
\begin{array}{c}
\nu_-(t)\\[0.5em] \nu_+(t) 
\end{array}
\right)\\[0.5em]
&=
\left(
\begin{array}{cc}
-p(t)\delta t & p(t)\delta t\\[0.5em]
q(t)\delta t & -q(t)\delta t
\end{array}
\right)^\dagger
\left(
\begin{array}{c}
\nu_-(t)\\[0.5em] \nu_+(t) 
\end{array}
\right)\\[0.5em]
\frac{\nu(t+\delta t)-\nu(t)}{\delta t}
&=
\left(
\begin{array}{cc}
-p(t) & p(t)\\[0.5em]
q(t) & -q(t)
\end{array}
\right)^\dagger
\left(
\begin{array}{c}
\nu_-(t)\\[0.5em] \nu_+(t) 
\end{array}
\right)
\end{align*}
which in the limit of small $\delta t$ leads to a differential equation describing the behaviour of $\nu(t)$ 
\begin{align}
\frac{d\nu}{dt}\label{chap_4:diff:eqn}
&=
Q^\dagger
\nu
\end{align}
where the infinitesimal generator $Q$ is defined as 
\begin{align*}
Q=
\left(
\begin{array}{cc}
-p(t) & p(t)\\[0.5em]
q(t) & -q(t)
\end{array}
\right). 
\end{align*}
Note that the transpose of $Q$ is taken in Equation \ref{chap_4:diff:eqn}. 
The aim now is to derive the state probability by solving this differential equation for various forms of $p$ and $q$. 
The extra conditions we use are 
\begin{align*}
\nu_-(t)+\nu_+(t)&=1\\
\nu'_-(t)+\nu'_+(t)&=0
\end{align*}
for all times $t$ and the initial conditions at $t=0$ are $\nu_-(0)$ and $\nu_+(0)$. 

\subsection{Continuous Time Markov Chain - Alternating Saddles $p\neq q$}
Notice that $p$ may be interpreted as  the probability of escape from the left well and $q$ as  the probability of escape from the right well. 
If $p$ and $q$ are cyclic over $[0,T]$, then this can be interpreted as modelling a potential with periodic forcing in continuous time. 

\begin{Theorem}\label{chap_4_thm:equal_contin}
Let $p\neq q$ and $t\geq0$.
The state probabilities are given by 
\begin{align*}
\nu_-(t)&=\frac{\nu_-(0)+\int_0^tq(s)\exp\left\{\int_0^sp(u)+q(u)\,du\right\}\,ds}{\exp\left\{\int_0^tp(u)+q(u)\,du\right\}}\\[0.5em]
\nu_+(t)&=\frac{\nu_+(0)+\int_0^tp(s)\exp\left\{\int_0^sp(u)+q(u)\,du\right\}\,ds}{\exp\left\{\int_0^tp(u)+q(u)\,du\right\}}
\end{align*}
\end{Theorem}

\begin{proof}
The differential equations we want to solve are given by 
\begin{align*}
\frac{d}{dt}
\left(
\begin{array}{c}
\nu_-(t)\\[0.5em]\nu_+(t)
\end{array}
\right)
&=
\left(
\begin{array}{cc}
-p(t)&p(t)\\[0.5em]q(t)&-q(t)
\end{array}
\right)^\dagger
\left(
\begin{array}{c}
\nu_-(t)\\[0.5em]\nu_+(t)
\end{array}
\right)\\[0.5em]
&=
\left(
\begin{array}{cc}
-p(t)&q(t)\\[0.5em]p(t)&-q(t)
\end{array}
\right)
\left(
\begin{array}{c}
\nu_-(t)\\[0.5em]\nu_+(t)
\end{array}
\right)
\end{align*}
which gives 
\begin{align*}
\frac{d\nu_-}{dt}&=-p\nu_-+q\nu_+\\
\frac{d\nu_+}{dt}&=p\nu_--q\nu_+
\end{align*}
and by using $\nu_-+\nu_+=1$ we get 
\begin{align}
\frac{d\nu_-}{dt}+(p+q)\nu_-&=q \label{chap_4_eqn_neg}\\
\frac{d\nu_+}{dt}+(p+q)\nu_+&=p. \label{chap_4_eqn_pos}
\end{align}
We will only solve for $\nu_-(t)$. 
The case for  $\nu_+(t)$ is similar. 
Equation \ref{chap_4_eqn_neg} can easily be solved with an integrating factor
\begin{align*}
\frac{d}{dt}
\left\{
\nu_-(t)
\exp
\left\{
\int^t_0p(u)+q(u)\,du
\right\}
\right\}
&=
q(t)
\exp
\left\{
\int^t_0p(u)+q(u)\,du
\right\}
\end{align*}
and proceeding we have 
\begin{align*}
\nu_-(t)\exp
\left\{
\int^t_0p(u)+q(u)\,du
\right\}
-\nu_-(0)
&=
\int^t_0
q(s)
\exp
\left\{
\int^s_0p(u)+q(u)\,du
\right\}
\,ds
\end{align*}
which rearranges to give
\begin{align*}
\nu_-(t)&=\frac{\nu_-(0)+\int_0^tq(s)\exp\left\{\int_0^sp(u)+q(u)\,du\right\}\,ds}{\exp\left\{\int_0^tp(u)+q(u)\,du\right\}}.
\end{align*}
This completes the proof. 
\end{proof}

\subsection{Continuous Time Markov Chain - Synchronised Saddles $p=q$}
If $p=q$ for continuous time, then  this can be modelled as both wells of the potential always being at the same height but moving together. 

\begin{Theorem}\label{chap_4_thm:not:equal_contin}
Let $p=q$ and $t\geq0$.
The state probabilities are given by 
\begin{align*}
\nu_-(t)&=\frac{1}{2}-\frac{\nu_+(0)-\nu_-(0)}{2}\exp\left\{-2\int_0^tp(s)\,ds\right\}\\[0.5em]
\nu_+(t)&=\frac{1}{2}+\frac{\nu_+(0)-\nu_-(0)}{2}\exp\left\{-2\int_0^tp(s)\,ds\right\}
\end{align*}
\end{Theorem}

\begin{proof}
The differential equations we have to solve are given by 
\begin{align*}
\frac{d}{dt}
\left(
\begin{array}{c}
\nu_-(t)\\[0.5em]\nu_+(t)
\end{array}
\right)
&=
p(t)
\left(
\begin{array}{cc}
-1&1\\[0.5em]1&-1
\end{array}
\right)
\left(
\begin{array}{c}
\nu_-(t)\\[0.5em]\nu_+(t)
\end{array}
\right)
\end{align*}
which gives 
\begin{align}
\frac{d\nu_-}{dt}&=p\left(\nu_+-\nu_-\right)\label{chap_4_eqn_half_neg}\\
\frac{d\nu_+}{dt}&=p\left(\nu_--\nu_+\right)\label{chap_4_eqn_half_pos}
\end{align}
and subtracting Equation \ref{chap_4_eqn_half_neg} away from Equation \ref{chap_4_eqn_half_pos} leads to 
\begin{align*}
\frac{d}{dt}\left(\nu_+-\nu_-\right)
=-2p\left(\nu_+-\nu_-\right).
\end{align*}
Denote the difference by 
\begin{align*}
z(t)=\nu_+(t)-\nu_-(t)
\end{align*}
which gives the differential equation we need to solve to 
\begin{align*}
\frac{dz}{dt}
&=-2pz\\
\int^{z(t)}_{z(0)}
\frac{dz}{z}
&=
-2\int^t_0p(s)\,ds\\
\ln(z(t))-\ln(z(0))
&=-2\int^t_0p(s)\,ds\\
z(t)&=
z(0)\exp\left\{-2\int^t_0p(s)\,ds\right\}
\end{align*}
and using $z=\nu_+-\nu_-$ and $\nu_-+\nu_+=1$ rearranges the solution to 
\begin{align*}
\nu_-(t)&=\frac{1}{2}-\frac{\nu_+(0)-\nu_-(0)}{2}\exp\left\{-2\int_0^tp(s)\,ds\right\}\\[0.5em]
\nu_+(t)&=\frac{1}{2}+\frac{\nu_+(0)-\nu_-(0)}{2}\exp\left\{-2\int_0^tp(s)\,ds\right\}.
\end{align*}
This completes the proof. 
\end{proof}

\subsection{Continuous Time Markov Chain - Invariant Measures and Fourier Transform}
As in the discrete time case we can compute the corresponding invariant measures.

\begin{Corollary}
\label{corollary_invariant_measure_contin}
For the state probabilities in Theorem \ref{chap_4_thm:equal_contin}
the invariant measures are
\begin{align*}
\overline{\nu}_-(t)
&=
\frac{\int^t_0p(s)g(s)\,ds}{g(t)}
+\frac{\int^T_0p(s)g(s)\,ds}{g(t)\left(g(T)-1\right)}\\[0.5em]
\overline{\nu}_+(t)
&=
\frac{\int^t_0q(s)g(s)\,ds}{g(t)}
+\frac{\int^T_0q(s)g(s)\,ds}{g(t)\left(g(T)-1\right)}
\end{align*}
where
\begin{align*}
g(t)=\exp\left\{\int^t_0p(u)+q(u)\,du\right\}.
\end{align*}
\end{Corollary}

\begin{proof}
We derive the invariant measure for $\overline{\nu}_-(t)$. 
The case for $\overline{\nu}_+(t)$ is similar. 
Consider the fact that $p(\cdot)$ and $q(\cdot)$ are cyclic on $[0,T]$ and let $i$ be an integer, then the following integral can be rewritten as 
\begin{align*}
\int^{(i+1)T}_{iT}
p(s)g(s)\,ds
&=
\int^T_0p(s)g(iT+s)\,ds\\
&=
\int^T_0p(s)g(iT)g(s)\,ds\\
&=
g(iT)\int^T_0p(s)g(s)\,ds\\
&=
g(T)^i\int^T_0p(s)g(s)\,ds.
\end{align*}
Let the time be given by $NT+t$ where $N$ is an integer number of periods. 
This means the following integral can be written as 
\begin{align*}
\int^{NT+t}_0
p(s)g(s)\,ds
&=
\int^{NT+t}_{NT}
p(s)g(s)\,ds
+
\sum_{i=0}^{N-1}\int^{(i+1)T}_{iT}
p(s)g(s)\,ds\\[0.5em]
&=
\int^t_0
p(s)g(NT+s)\,ds
+
\int^T_0p(s)g(s)\,ds
\sum_{i=0}^{N-1}
g(T)^i\\[0.5em]
&=
g(T)^N
\int^t_0
p(s)g(s)\,ds
+
\int^T_0p(s)g(s)\,ds
\sum_{i=0}^{N-1}
g(T)^i. 
\end{align*}
So the state probability is equal to 
\begin{align*}
\nu_-(NT+t)
&=\frac{\nu_-(0)+\int_0^{NT+t}q(s)\exp\left\{\int_0^sp(u)+q(u)\,du\right\}\,ds}{\exp\left\{\int_0^{NT+t}p(u)+q(u)\,du\right\}}\\[0.5em]
&=
\frac{
\nu_-(0)+
g(T)^N
\int^t_0
p(s)g(s)\,ds
+
\int^T_0p(s)g(s)\,ds
\sum_{i=0}^{N-1}
g(T)^i
}
{
g(T)^Ng(t)
}\\[0.5em]
&=
\frac{\nu_-(0)}{g(T)^Ng(t)}
+
\frac{\int^t_0p(s)g(s)\,ds}{g(t)}
+
\frac{\int^T_0p(s)g(s)\,ds}{g(t)}
\frac{1}{g(T)^N}
\sum_{i=0}^{N-1}
g(T)^i\\[0.5em]
&=
\frac{\nu_-(0)}{g(T)^Ng(t)}
+
\frac{\int^t_0p(s)g(s)\,ds}{g(t)}
+
\frac{\int^T_0p(s)g(s)\,ds}{g(t)}
\frac{1}{g(T)-1}
\left(
1-\frac{1}{g(T)^N}
\right). 
\end{align*}
Letting $N\longrightarrow\infty$ gives the required result. 
\end{proof}

\begin{Corollary}\label{cor_half_continuous}
For the state probabilities in Theorem \ref{chap_4_thm:not:equal_contin}
the invariant measures are
\begin{align*}
\overline{\nu}(t)=\frac{1}{2}
\left(
\begin{array}{c}
1\\1
\end{array}
\right). 
\end{align*}
\end{Corollary}

\noindent The proof is trivial and omitted. 
Similar to the discrete time case we can also study the Fourier Transform of the averaged Markov Chain.

\begin{Corollary}
For the Markov Chain in Theorem \ref{chap_4_thm:equal_contin}
the Fourier Transform of the averaged Markov Chain is 
\begin{align*}
\mathcal{F}\left(\left\langle Y^\epsilon_t \right\rangle\right)
=
\int^{+\infty}_{-\infty}
\left(
\frac{\int^t_0\left[q(s)-p(s)\right]g(s)\,ds}{g(t)}
+\frac{\int^T_0\left[q(s)-p(s)\right]g(s)\,ds}{g(t)\left(g(T)-1\right)}
\right)
e^{-i2\pi \omega t}\,dt.
\end{align*}
\end{Corollary}

\begin{Corollary}\label{transform_continuous_chain}
For the Markov Chain in Theorem \ref{chap_4_thm:not:equal_contin}
the Fourier Transform of the averaged Markov Chain is 
\begin{align*}
\mathcal{F}\left(\left\langle Y^\epsilon_t \right\rangle\right)
=0.
\end{align*}
\end{Corollary}

\section{Probability Density Function of Escape Times}
The  escape rates from the left to right  are denoted by $R_{-1+1}(\cdot)$
and right to left escape rates are denoted by
$R_{+1-1}(\cdot)$. 
The PDFs for the escape times are given by the Theorem below. 

\begin{Theorem}\label{chap_4_pdf_thm}
Let $u$ be the time of entry into a well, 
then the PDFs for the escape occurring at time $t>u$ are 
\begin{align*}
p_{-}(t,u)&=R_{-1+1}(t)\exp\left\{-\int^t_uR_{-1+1}(s)\,ds\right\}\\[0.5em] 
p_{+}(t,u)&=R_{+1-1}(t)\exp\left\{-\int^t_uR_{+1-1}(s)\,ds\right\}
\end{align*}
where $p_{-}(t,u)$ is for left to right and $p_{+}(t,u)$ is for right to left. 
\end{Theorem}

\begin{proof}
We consider escaping from the left well.
The right well is similar.
Divide the time interval $[u,t]$ into many small time intervals 
\begin{align*}
\delta t =\frac{t-u}{N}.
\end{align*}
Similar to how we derived the invariant measures we want to derive the probability of escape in a very small time interval $[t,t+\delta t]$. 
This is given by 
\begin{align*}
p_{-1+1}([t,t+\delta t])
&=p(t)\delta t \\
&=1-e^{-R_{-1+1}(t)\delta t}\\
&\approx R_{-1+1}(t)\delta t 
\end{align*}
which is valid for small $\delta t$. 
Large deviations allow us to say even more about the escape time $\tau_{-1+1}$ and $\tau_{+1-1}$. 
Theorem 1 in \cite{oliveri_expoen} shows that it is an exponentially distributed random variable. 
The probability of staying in the left well is given by 
\begin{align*}
p_{-1-1}([t,t+\delta t])
&=1-p_{-1+1}([t,t+\delta t])\\
&=1-p(t)\delta t \\
&=1-\left(1-e^{-R_{-1+1}(t)\delta t}\right)\\
&=e^{-R_{-1+1}(t)\delta t}.
\end{align*}
We want to know the probability of escaping in the time interval $[t,t+\delta t]$ given that the particle has entered at $u$ and stayed up to time $t$.
This is given by 
\begin{align*}
p_{-1-1}([u,t])p_{-1+1}([t,t+\delta t])
&=\prod_{i=1}^{N}p_{-1-1}\left([u+(i-1)\delta t,u+i\delta t]\right)
p_{-1+1}([t,t+\delta t])\\
&=\prod_{i=1}^{N}\exp\left\{-R_{-1+1}\left(u+(i-1)\delta t\right)\delta t\right\}
p_{-1+1}([t,t+\delta t])\\
&=\exp\left\{\sum_{i=1}^{N}-R_{-1+1}\left(u+(i-1)\delta t\right)\delta t\right\}
p_{-1+1}([t,t+\delta t])\\
&=\exp\left\{-\int^t_uR_{-1+1}(s)\,ds\right\}
R_{-1+1}(t)\delta t.
\end{align*}
This completes the proof. 
\end{proof}

\subsection{Normalised Time Probability Density Function of Escape Times}

\noindent The period of the forcing is $T$ and we can make a change of variables to normalised time
\begin{align*}
t^{norm}&=\frac{t^{real}}{T}
\end{align*}
which measures time in how many periods have elapsed. 
This rearranges the PDFs to 
\begin{align*}
p_{-}(t,u)&=TR_{-1+1}(Tt)\exp\left\{-T\int^t_uR_{-1+1}(Ts)\,ds\right\}\\
p_{+}(t,u)&=TR_{+1-1}(Tt)\exp\left\{-T\int^t_uR_{+1-1}(Ts)\,ds\right\}
\end{align*}
where $u$, $t$ and $s$ are in normalised time. 
Note that $R_{-1+1}(\cdot)$ and $R_{+1-1}(\cdot)$ always have their arguments in real time and $R_{-1+1}(\cdot)$ and $R_{+1-1}(\cdot)$ always give the averaged number of transitions per real unit time.  
Expressing the PDF in normalised time can be found in \cite{tran2014}.

\subsection{Perfect Phase Approximation of  Probability Density Function of Escape Times}
\label{chap_approx_pdf}
The PDF for the escape times derived in Theorem \ref{chap_4_pdf_thm} had to differentiate between left and right escapes and are conditioned on the time $u$ of entrance into the well. 
Suppose now that $t$ is the escape time from any well, which does not differentiate between left and right escape. 
Note that $t$ is the actual time it takes to escape from a well and is not a time coordinate. 
The PDF for $t$ is given by 
\begin{align*}
p_{tot}(t)=
\frac{1}{2}
\int_{0}^{T}p_-(t+u,u)m_-(u)+p_+(t+u,u)m_+(u)\,du.
\end{align*}
This is because  after a long time has elapsed
we would expect that many transitions would have occurred between left and right. 
The number of transitions escaping from the left and right should be roughly the same. 
The $m_-(u)$ is a PDF for the time of entrance into the left well and 
the $m_+(u)$ is a PDF for the time of entrance into the right well. 
We may not have explicit expressions for $m_-(u)$ and $m_+(u)$. 
We derive an approximate expression for $p_{tot}$ without an explicit expressions for $m_-(u)$ and $m_+(u)$. 
Let  $m_-(u)$ and $m_+(u)$ be approximated by 
\begin{align*}
m_-(u)&\approx\delta\left(u-T/2\right)\\[0.5em]
m_+(u)&\approx
\frac{1}{2}\delta\left(u\right)
+
\frac{1}{2}\delta\left(u-T\right)
\end{align*}
where $\delta(\cdot)$ is the Dirac delta function. 
This approximation is used because in the SDEs which we will simulate, 
the times when transition into the left well is greatest is at half the period $u=\frac{T}{2}$
and the times when transition into the right well is greatest is at $u=0$ and $u=T$. 
Due to the fact that $m_-(u)$ and $m_+(u)$ are probabilities a factor of $\frac{1}{2}$ is used in $m_+(u)$. 
Progressing we have 
\begin{align*}
p_{tot}(t)&=
\frac{1}{2}
\int_{0}^{T}p_-(t+u,u)m_-(u)+p_+(t+u,u)m_+(u)\,du\\
&\approx
\frac{1}{2}
\int_{0}^{T}
p_-(t+u,u)\delta\left(u-T/2\right)
+p_+(t+u,u)
\left(
\frac{1}{2}\delta\left(u\right)
+
\frac{1}{2}\delta\left(u-T\right)
\right)
\,du\\[0.5em]
&=
\frac{1}{2}
\left\{
p_-(t+T/2,T/2)
+\frac{1}{2}
p_+(t,0)
+\frac{1}{2}
p_+(t+T,T)
\right\}\\[0.5em]
&=
\frac{1}{2}
\left\{
p_-(t+T/2,T/2)
+\frac{1}{2}
p_+(t,0)
+\frac{1}{2}
p_+(t+0,0)
\right\}\\[0.5em]
&=
\frac{1}{2}
\left\{
p_-(t+T/2,T/2)
+
p_+(t,0)
\right\}\\[0.5em]
&=p_+(t,0).
\end{align*}
This is because for the simulations which we are going to do, the Kramers' rate satisfy $R_{-1+1}(t)=R_{+1-1}(t+T/2)$ (see later in Chapter \ref{chap_mexican_hat_toy_model} for the geometry of the Mexican Hat Toy Model which justifies this). 
Thus the following approximation 
\begin{align*}
p_{tot}\approx p_+(t,0)
\end{align*}
is only valid for the simulations we do, and not for a general potential. 
We call this way of approximating $m_-(u)$ and $m_+(u)$ the perfect phase approximation.

\section{Adiabatic Large Deviation}
We have to stress that this thesis is built on three approximations, which form the backbone of all the research presented. These are small noise approximation, adiabatic approximation and perfect phase approximation.

Perfect phase approximation only works for small noise. This is because the noise is so small the particle will only escape when the maximum probability to escape has arrived. When the minimum probability to escape is present it will almost never escape. This is the idea behind the perfect phase approximation. 

Notice one subtlety behind all the theory presented in this Chapter. 
The derivations involved probabilities of escape $p$ and $q$ and the escape rates $R_{-1+1}$ and $R_{+1-1}$.
But it was assumed that $p$, $q$, $R_{-1+1}$ and $R_{+1-1}$ are accurately known no matter how large or small the noise level $\epsilon$ is and no matter how fast or slow the driving frequency $\Omega$ is. 
But such ideal expressions for $p$, $q$, $R_{-1+1}$ and $R_{+1-1}$ are not known. 

When we come to do the analysis in Chapter \ref{sum_chap_results}, the $p_{tot}$ is calculated with the approximation $p_{tot}\approx p_+(t,0)$. When the rates $R_{-1+1}$ and $R_{+1-1}$ are needed they are calculated using Kramers' formula as though it is escape from a static potential in the small noise limit. 
This means an oscillatory potential is being approximated by a static potential which is the adiabatic approximation.

In the paper
\cite{adiabatic_large_deviation_herrmann2006}
the adiabatic approximation was justified in the small noise, slow forcing limit using time dependent large deviation theory, that is, it was shown asymptotically the escape times are given by the adiabatic approximation. This result is only for the leading term, whether the analogue result holds for the Kramers' rate is unknown.

\chapter{Theory of Analysis of Stochastic Resonance}\label{chap_analysis_theory}
\label{sum_chap_analysis_theory}

We present different criteria that have been used to define stochastic resonance. This includes the six measures, which are   linear response, signal-to-noise ratio, energy, out-of-phase measures, relative entropy and entropy.  A new statistical test called the conditional Kolmogorov-Smirnov test is introduced. 

\section{Six Measures of Stochastic Resonance}
We introduce six possible criteria of measuring how close a process is to exhibiting stochastic resonance \cite{pav_thesis,tran2014}.
These six criteria are closely related to 
linear response 
\cite{lin_1_TUS:TUS1787,
lin_2_doi:10.1137/0143037}, 
signal-to-noise ratio
\cite{sign_1_PhysRevA.39.4668,
sign_2_escape_2_PhysRevA.41.4255}
and distribution of escape times
\cite{escape_1_PhysRevA.42.3161,
sign_2_escape_2_PhysRevA.41.4255,
escape_3_PhysRevE.49.4821}. 
We call them the six measures denoted by $M_1$, $M_2$, $M_3$, $M_4$, $M_5$ and $M_6$.  
Recall that the SDE we want to study is 
\begin{align*}
\dot{X}^\epsilon_t=-\nabla V_0 + F\cos\Omega t + \epsilon\, \dot{W}_t
\end{align*}
where $V_0:\mathbb{R}^2\longrightarrow\mathbb{R}$ is the unperturbed potential, $F$ is the forcing, $\Omega$ is the forcing frequency, $\epsilon$ is the noise level and $W_t$ is a Wiener process in two dimensions, which when rewritten into separate components are
\begin{align*}
dx&=\left[-\frac{\partial V_0}{\partial x}+F_x\cos \Omega t \ \right]dt+\epsilon \ dw_x\\
dy&=\left[-\frac{\partial V_0}{\partial y}+F_y\cos \Omega t \ \right]dt+\epsilon \ dw_y
\end{align*}
where $w_x$ and $w_y$ are two independent Wiener processes. 
The solution to these equations is the trajectory in two dimensions
\begin{align*}
X^\epsilon_t=\left(x_t,y_t\right).
\end{align*}
This diffusion can be reduced to a Markov Chain on $\left\{-1,+1\right\}$ denoted by
\begin{align*}
Y^\epsilon_t=\pm1
\end{align*}
where by definition of the Markov Chain the escape times  are the same as the diffusion case (see Chapter \ref{chapter_oscil_times}).
The probability of the Markov Chain being in one state at time $t$ is given by the state probabilities
\begin{align*}
P\left(Y^\epsilon_t=-1\right)=\nu_-(t)
\quad \text{and} \quad
P\left(Y^\epsilon_t=+1\right)=\nu_+(t).
\end{align*}
In what follows we will consider so large times,  that the relaxation time has  effectively elapsed for both the diffusion and Markov Chain, in other words
the state probability would have effectively converged to the invariant measure $\overline{\nu}$.
This means that over one period $T=2\pi/\Omega$ of the forcing, the invariant measures will have the properties
\begin{align*}
\overline{\nu}_\pm(t)=\overline{\nu}_\pm(t+T)
\quad \text{and} \quad 
\overline{\nu}_\pm(t)=\overline{\nu}_\mp(t+T/2).
\end{align*}
We obtain the averaged trajectories given by 
\begin{align*}
\left\langle X^\epsilon_t \right\rangle
=E\left(X^\epsilon_t\right)
\quad \text{and} \quad
\left\langle Y^\epsilon_t \right\rangle
=E\left(Y^\epsilon_t\right)
\end{align*}
which are the trajectories obtained after averaging over many realisations. 
Notice that $\left\langle Y^\epsilon_t \right\rangle$ is related to the invariant measures by 
\begin{align*}
\left\langle Y^\epsilon_t \right\rangle=\overline{\nu}_+(t)-\overline{\nu}_-(t).
\end{align*}
We introduce the Out-of-Phase Markov Chain defined by 
\begin{align*}
\overline{Y}^\epsilon_t=
\left\{
\begin{array}{l}
0 \quad \text{if} \quad Y^\epsilon_t=-1 \quad \text{and} \quad mod(t,T)\leq T/2\\
1 \quad \text{if} \quad Y^\epsilon_t=-1 \quad \text{and} \quad mod(t,T)> T/2\\
1 \quad \text{if} \quad Y^\epsilon_t=+1 \quad \text{and} \quad mod(t,T)\leq T/2\\
0 \quad \text{if} \quad Y^\epsilon_t=+1 \quad \text{and} \quad mod(t,T)> T/2
\end{array}
\right.
\end{align*}
and similarly the averaged Out-of-Phase Markov Chain is defined by 
\begin{align*}
\left\langle \overline{Y}^\epsilon_t \right\rangle=E\left(\overline{Y}^\epsilon_t\right).
\end{align*}
Define two new functions by 
\begin{align*}
\phi^-(t)&=
\left\{
\begin{array}{c}
1 \quad \text{if} \quad mod(t,T)\leq T/2\\
0 \quad \text{if} \quad mod(t,T)> T/2
\end{array}
\right.\\[0.5em]
\phi^+(t)&=
\left\{
\begin{array}{c}
0 \quad \text{if} \quad mod(t,T)\leq T/2\\
1 \quad \text{if} \quad mod(t,T)> T/2.
\end{array}
\right.
\end{align*}
The following trajectories are Fourier transformed
\begin{align*}
\tilde{x}(\omega)&=\mathcal{F}\left(
\langle x_t  \rangle
\right)
=
\langle \mathcal{F} \left(x_t\right) \rangle
\\[0.5em]
\tilde{Y}(\omega)&=\mathcal{F}\left(
\left\langle Y^\epsilon_t \right\rangle
\right)
=
\langle \mathcal{F} \left(x_t\right) \rangle.
\end{align*}
The linear response is defined as the intensity of the Fourier Transform at the driving frequency $\Omega$
\footnote{See Appendix \ref{appendix_linear_response} for how the linear response is calculated numerically.}
\begin{align*}
X_{lin}=\left|\tilde{x}\left(\frac{\Omega}{2\pi}\right)\right|
\quad \text{and} \quad 
Y_{lin}=\left|\tilde{Y}\left(\frac{\Omega}{2\pi}\right)\right|.
\end{align*}
Now we can define the six measures.
For the diffusion case only $M_1$ and $M_2$ are defined. 
\begin{align*}
M_1&=\frac{1}{F}X_{lin}\\
M_2&=\frac{1}{\epsilon F}X_{lin}
\end{align*}
where $F$ is the magnitude of the forcing. 
For the Markov Chain $M_1$, $M_2$, $M_3$, $M_4$, $M_5$ and $M_6$ are all defined as 
\begin{align*}
M_1&=\frac{1}{F}Y_{lin}\\
M_2&=\frac{1}{\epsilon F}Y_{lin}\\
M_3&=\int_0^T \left\langle Y^\epsilon_t \right\rangle^2 dt\\
M_4&=\int_0^T \left\langle \overline{Y}^\epsilon_t \right\rangle dt\\
M_5&=\int_0^T
\phi^-(t)\ln\left(\frac{\phi^-(t)}{\overline{\nu}_-(t)}\right)+
\phi^+(t)\ln\left(\frac{\phi^+(t)}{\overline{\nu}_+(t)}\right)
dt\\
M_6&=\int^T_0
-\overline{\nu}_-(t)\ln\overline{\nu}_-(t)
-\overline{\nu}_+(t)\ln\overline{\nu}_+(t)\,
dt.
\end{align*}
Note that in the definition of the six measures it is assumed that the process has relaxed to equilibrium. 
We give a few physical interpretation of the six measures 
$M_1$, $M_2$, $M_3$, $M_4$, $M_5$ and $M_6$. 
The $M_1$ is the intensity of the driving frequency $\Omega$ in the spectrum of the Fourier transform. 
The $M_2$ is sometimes called signal-to-noise ratio as it compares this intensity to the noise level $\epsilon$. 
The $M_3$ is sometimes called the energy. 
The $M_4$ is sometimes called the out-of-phase measure since it measures the amount of time the Markov Chain spends in the ``wrong'' well. 
The $M_5$ and $M_6$ are sometimes called relative entropy and entropy respectively, since they measure how far away the invariant measures are from being constant. 
If the invariant measures are constant then these six measures will also be constant. 
Thus it can be understood that these six measures is a measure of how far away the invariant measures are from being constant. 
$M_6$ measures how non-constant the invariant measure is.
$M_5$ is extremal if the invariant measure is constant.

\section{Statistical Tests}
We will measure the escape time for many consecutive transitions. 
This will result in a collection of measurements of escape times
\begin{align*}
\tau_1,\tau_2,\ldots,\tau_n. 
\end{align*}
A new method for analysing such a collection of measurements is presented. 

\subsection{Kolmogorov-Smirnov Test}
First we recall results about the Kolmogorov-Smirnov statistic and the Kolmogorov-Smirnov test \cite{ks_test_1}. 
Let
\begin{align*}
\xi_1, \xi_2, \ldots, \xi_n
\end{align*}
be $n$ independently and identically distributed real random variables. 
Each $\xi_i$ is distributed with PDF $f(\cdot)$ as in 
\begin{align*}
P(\xi_i\in A)=\int_A f(s)\,ds
\end{align*}
and distributed with CDF $F(\cdot)$ as in 
\begin{align*}
P(\xi_i\leq x)=F(x)=\int^x_{-\infty} f(s)\,ds. 
\end{align*}
Define a function by $F_n(\cdot)$ by 
\begin{align*}
F_n(x)=
\frac{1}{n}
\sum_{i=1}^n
\mathbf{1}_{(-\infty,x]}
(\xi_i)
\end{align*}
where $\mathbf{1}_{A}$ is the indicator function for a set $A$. 
We may think of $\xi_1,\xi_2,\ldots,\xi_n$ as $n$ empirical or numerical realisations of the same random variable $\xi$. 
The $F_n(x)$ is therefore an approximation to the CDF of $\xi$ that is empirically found using $\xi_1,\xi_2,\ldots,\xi_n$,
therefore $F_n(\cdot)$ is called the empirical CDF. 
Consider the supremum  metric on the space of real continuous functions. 
Consider the distance between the real and the empirical CDF in this metric. 
\begin{align*}
D_n&=
\left\Vert
F_n-F
\right\Vert_\infty\\
&=\sup_{x\in\mathbb{R}}
\left|
\frac{1}{n}
\sum_{i=1}^n
\mathbf{1}_{(-\infty,x]}
(\xi_i)
-F(x)
\right|
\end{align*}
where $D_n$ is called the Kolmogorov-Smirnov statistic or KS statistic. 
Intuitively we would expect $D_n$ to tend to zero as $n$ increases, that is 
\begin{align*}
\lim_{n\longrightarrow\infty}D_n=0
\end{align*}
if the $\xi_i$ are distributed by $F(\cdot)$. 
There are times when we experimentally obtain $n$ values of a random variable $\xi_1,\xi_2,\ldots,\xi_n$, and want to test whether they are distributed by a CDF $F(\cdot)$. 
We define what we mean by the null hypothesis. 

\begin{Definition}
Let $\xi_1,\xi_2,\ldots,\xi_n$ be $n$ real random variables.
The null hypothesis is 
that each $\xi_i$ is independently distributed with CDF $F(x)$.
\end{Definition}

\noindent We want to know how large or small $D_n$ needs to be before deciding whether to reject the null hypothesis. 
The following Theorem offers a remarkable answer to this problem.

\begin{Theorem}\label{chap_4_thm_ks_test}
Suppose the null hypothesis is true,
then the distribution of $D_n$ depends only on $n$. 
\end{Theorem}

\noindent Notice that $D_n$ is in itself a real random variable. 
The PDF and CDF of $D_n$ is a function of $n$ only, and will be the same whatever $F(\cdot)$ is.
This distribution is called the KS distribution and tables are available upto $n=100$. 
There is a Theorem which describes the asymptotic behaviour of the KS distribution \cite{q2_smirnov1948,q2_feller1948}.\footnote{There appears to be topographical errors in the literature for the limiting function.
Some sources cite $Q_1=1-2\sum_{k=1}^\infty (-1)^{k-1}e^{-2k^2x^2}$
(see \cite{q1_monahan1989evaluating,q1_JSSv008i18,q1_JSSv039i11})
and some cite $Q_2=1-2\sum_{k=1}^\infty (-1)^{k-1}e^{-k^2x^2}$ (see \cite{q2_smirnov1948,q2_feller1948}).  
But the proof of Theorem 1 in \cite{q2_feller1948} shows $Q=Q_1$. Nevertheless in this thesis we use $Q=Q_1$ which actually gives smaller and more conservative values of the metric $D_n$ that are needed. 
} 

\begin{Theorem}
In the limit $n \longrightarrow \infty$, $\sqrt{n}D_n$ is asymptotically Kolmogorov distributed with the CDF 
\begin{equation*}
Q(x)=1-2\sum_{k=1}^\infty (-1)^{k-1}e^{-2k^2x^2}
\end{equation*}
that is to say 
\begin{equation*}
\lim_{n\longrightarrow \infty} P(\sqrt{n} D_n \leq x)=Q(x). 
\end{equation*}
\end{Theorem}

\subsection{Conditional Kolmogorov-Smirnov Test}
Let $\zeta_1,\zeta_2,\ldots,\zeta_n$ be $n$ iid real random variables. 
They are $n$ empirical observations of a random variable $\zeta$. 
Now suppose that each of the $\xi_1,\xi_2,\ldots,\xi_n$ is conditioned and dependent on the corresponding $\zeta_1,\zeta_2,\ldots,\zeta_n$.
This means a conditional PDF $f(\cdot,\cdot)$ gives the probability
\begin{align*}
P(\xi_i\in A\,|\,\zeta_i)=\int_Af(s,\zeta_i)\,ds
\end{align*}
and the conditional CDF $F(\cdot,\cdot)$ is
\begin{align*}
P(\xi_i\leq x\,|\,\zeta_i)=F_{\zeta_i}(x)=\int^x_{-\infty}f(s,\zeta_i)\,ds.
\end{align*}
But $\xi_1,\xi_2,\ldots,\xi_n$ are empirical measurements of the same random variable $\xi$. 
The PDF for $\xi$ is given by 
\begin{align*}
P(\xi\in A)=\int_A\int^{+\infty}_{-\infty}f(s,u)m(u)\,ds\,du
\end{align*}
and the CDF for $\xi$ is 
\begin{align*}
P(\xi\leq x)=F(x)=\int^x_{-\infty}\int^{+\infty}_{-\infty}f(s,u)m(u)\,ds\,du
\end{align*}
where $m(\cdot)$ is the PDF for $\zeta$, that is 
\begin{align*}
P(\zeta\in A)=\int_Am(s)\,ds.
\end{align*}
In our context we have the problem that the random variables are not identically distributed under the null hypothesis. 
The  $\xi_1,\xi_2,\ldots,\xi_n$ and $\zeta_1,\zeta_2,\ldots,\zeta_n$ are obtained experimentally and $F_{\zeta_i}(\xi_i)$ can be  calculated but a PDF for $\zeta_i$, that is $m(\cdot)$, has no easy expression. 
We still want to perform a statistics test that is similar to the KS test even in such situations where the distribution $m(\cdot)$ of $\zeta$ is unknown. 
First we define what we call the total null hypothesis and the conditional null hypothesis. 

\begin{Definition}
Let $\xi_1, \xi_2, \ldots, \xi_n$ be $n$ empirical observations of a random variable $\xi$. 
The total null hypothesis is that
$\xi$ is distributed with the CDF $F(\cdot)$.
The conditional null hypothesis is that
 each $\xi_i$ is distributed with the conditional CDF $F_{\zeta_i}(\cdot)$.
\end{Definition}

\noindent A new statistical test is developed, which is similar to the KS test. 

\begin{Theorem}
Suppose the conditional null hypothesis is true. 
Let $F_{\zeta_i}(\cdot)$ be continuous. 
Let $S_n$ be the statistic given by 
\begin{align*}
S_n=\sup_{x\in[0,1]}
\left|
\frac{1}{n}
\sum^n_{i=1}
\mathbf{1}_{[0,x]}
\left(
F_{\zeta_i}(\xi_i)
\right)
-x
\right|
\end{align*}
then $S_n$ is KS distributed. 
\end{Theorem}

\begin{proof}
Denote
\begin{align*}
Y_i=F_{\zeta_i}(\xi_i)
\end{align*}
which means 
\begin{align*}
P\left(Y_i\leq x\right)&=P\left(F_{\zeta_i}(\xi_i)\leq x\right)\\
&=P\left(\xi_i\leq F^{-1}_{\zeta_i}(x)\right)\\
&=F_{\zeta_i}\left(F^{-1}_{\zeta_i}(x)\right)\\
&=x
\end{align*}
and $0\leq Y_i\leq 1$, so $Y_i$ is uniformly distributed on $[0,1]$. 
Note that $F_{\zeta_i}(\cdot)$ is a function of one variable only. 
Let 
\begin{align*}
F_n(x)=
\frac{1}{n}
\sum_{i=1}^n
\mathbf{1}_{[0,x]}
(Y_i)
=
\frac{1}{n}
\sum_{i=1}^n
\mathbf{1}_{[0,x]}
\left(F_{\zeta_i}(\xi)\right)
\end{align*}
where $F_n(\cdot)$ is the empirical CDF of a uniformly distributed random variable, computed using $n$ observations. 
The statistic $S_n$ is the suprenum metric
\begin{align*}
S_n&=\left\Vert F_n-x   \right\Vert_\infty\\
&=\sup_{x\in[0,1]}
\left|
\frac{1}{n}
\sum^n_{i=1}
\mathbf{1}_{[0,x]}
\left(
F_{\zeta_i}(\xi_i)
\right)
-x
\right|.
\end{align*}
So clearly $S_n$ is KS distributed. 
\end{proof}

\noindent We call $S_n$ the conditional KS statistic. 
Compare this to the original KS statistic, which under the assumption of the total null hypothesis can be rewritten as 
\begin{align*}
D_n=\sup_{x\in\mathbb{R}}
\left|
\frac{1}{n}
\sum_{i=1}^n
\mathbf{1}_{(-\infty,x]}
(\xi_i)
-F(x)
\right|
=\sup_{x\in[0,1]}
\left|
\frac{1}{n}
\sum^n_{i=1}
\mathbf{1}_{[0,x]}
\left(
F(\xi_i)
\right)
-x
\right|.
\end{align*}
When both the total and conditional hypothesis are true $D_n$ and $S_n$ are KS distributed, that is 
\begin{align*}
P\left(D_n\in A \right)=P\left(S_n\in A \right)
\quad \text{and} \quad 
P\left(D_n \leq x\right)=P\left(S_n\leq x\right).
\end{align*}
The subtlety here is that
$D_n$ and $S_n$ are different objects, yet they have the same distribution. 
$D_n$ is KS distributed under the total null hypothesis, 
whereas $S_n$ is KS distributed under the conditional null hypothesis. 
This can be explained in another way.
We have $n$ experimental observations of a random variable $\xi$ denoted by $\xi_1,\xi_2, \ldots, \xi_n$ and each are conditioned on observations of another random variable $\zeta$ denoted by $\zeta_1, \zeta_2, \ldots, \zeta_n$. 
The $D_n$ is KS distributed if the random variable $\xi$ is distributed by CDF $F(\cdot)$, 
but the $S_n$ is KS distributed if each $\xi_i$ is conditionally distributed by the CDF $F_{\zeta_i}(\cdot)$.

\chapter{Mexican Hat Toy Model}\label{chap_mexican_hat_toy_model}
\label{sum_chap_mexican_hat_toy_model}

The main object of consideration of this project, which is called the Mexican Hat Toy Model, is now introduced. 
Let $a>0$, $b>0$ and $V_0:\mathbb{R}^2\longrightarrow \mathbb{R}$ be a real function from the plane to the line.
The unperturbed potential is defined as
\begin{align*}
V_0(x,y)=\frac{1}{4}r^4-\frac{1}{2}r^2-ax^2+by^2
\quad \text{where} \quad r=\sqrt{x^2+y^2}. 
\end{align*}
Let $F_x, F_y \in  \mathbb{R}$ be the forcing. The potential with forcing $V_F$ is defined as 
\begin{align*}
V_F(x,y)
&=\frac{1}{4}r^4-\frac{1}{2}r^2-ax^2+by^2+F_xx+F_yy\\
&=\frac{1}{4}r^4-\frac{1}{2}r^2-ax^2+by^2+\mathbf{F}\cdot\mathbf{x}\\
&=V_0+\mathbf{F}\cdot\mathbf{x}
\end{align*}
written more compactly in vector notation.
When $V_F$ is defined using $+\mathbf{F}\cdot\mathbf{x}$ we say positive forcing. Alternatively if $V_F$ is defined using $-\mathbf{F}\cdot\mathbf{x}$ we say negative forcing.
The $V_F$ is defined with a positive forcing because for the rest of this Chapter we will study the critical points which are solutions to the simultaneous equations  
\begin{align*}
\frac{\partial V_F}{\partial x}=0
\quad \text{and} \quad 
\frac{\partial V_F}{\partial y}=0. 
\end{align*}
The properties of the critical points will change as $F$ is increased from zero, therefore it is convenient to define $V_F$ with a positive forcing. 
The behaviour of the critical points are studied for different cases. 
The main aim is to  find the positions and nature of the critical points for a range of parameter values. 
This is complex due to several cases to be considered and previewed in the following Theorem, which is one of the main conclusions of this Chapter. 
Although this Theorem only considers the case for non-negative forcing $F_x\geq0$ and $F_y\geq0$, the case for $F_x<0$ and $F_y<0$ is similar by considering the symmetry of the potential. 

\begin{Theorem}\label{chap_5_case_0}
Let $F_x\geq 0$, $F_y\geq 0$, $a>0$, $b>0$. 
Note that definitions of constants are at the end. 
The positions and nature of the critical points of the Mexican Hat Toy Model $V_F(\cdot)$ are for the following 
range of parameters. \newline

\noindent For $F_x=0$, $F_y=0$ and $b<\frac{1}{2}$
\begin{equation*}
\begin{array}{cc}
 (0,0) & hill \\[0.5em]
 (\pm \sqrt{1+2a},0) & well \\[0.5em]
 (0,\pm\sqrt{1-2b}) & saddle 
\end{array}
\end{equation*}
\newline 

\noindent For $F_x=0$, $F_y=0$ and $b\geq\frac{1}{2}$
\begin{equation*}
\begin{array}{cc}
(0,0) & saddle\\[0.5em]
(\pm \sqrt{1+2a},0) & well 
\end{array}
\end{equation*}
\newline 

\noindent For $F_x>0$, $F_y=0$, $b<\frac{1}{2}$ and the following values of $F_x$
\begin{align*}
\text{For any}\quad F_x>0\relphantom{F_x}&\left\{
\begin{array}{ll}
(x_0,0) & \text{well}
\end{array}
\right.\\[0.7em]
\text{and if}\quad F_x< F_x^{sad}&\left\{
\begin{array}{lll}
(x_1,0) &\text{well}\\
(x_2,0)&\text{hill} \\
(x_{saddle},\pm y_{saddle})&\text{saddle}
\end{array}
\right.\\[0.7em]
\text{or}\quad F_x^{sad} < F_x < F_x^{crit}&\left\{
\begin{array}{llll}
(x_1,0)&\text{saddle} &\text{for} & \sqrt{1-2b}\in R_1\\
(x_1,0)&\text{well} &\text{for}&\sqrt{1-2b}\in R_2\\
(x_2,0)&\text{hill} &\text{for}&\sqrt{1-2b}\in R_1\\
(x_2,0) &\text{saddle} &\text{for}&\sqrt{1-2b}\in R_2\\
(x_{saddle},\pm y_{saddle})&\text{nonexistent}&&
\end{array}
\right.\\[0.7em]
\text{or}\quad F_x= F_x^{sad}<F_x^{crit}&\left\{
\begin{array}{llll}
(x_1,0)&\text{unidentified} &\text{for} & \sqrt{1-2b}\in R_1\\
(x_1,0)&\text{well} &\text{for}&\sqrt{1-2b}\in R_2\\
(x_2,0)&\text{hill} &\text{for}&\sqrt{1-2b}\in R_1\\
(x_2,0) &\text{unidentified} &\text{for}&\sqrt{1-2b}\in R_2\\
(x_{saddle}, \pm y_{saddle})&\text{unidentified}&&
\end{array}
\right.
\\[0.7em]
\text{or}\quad F_x= F_x^{crit}&\left\{
\begin{array}{ll}
(x_{1},0) &\text{unidentified}\\
(x_{2},0)&\text{unidentified}
\end{array}
\right.\\[0.7em]
\text{or}\quad F_x>F_x^{crit}&\left\{
\begin{array}{ll}
(x_{1},0) &\text{nonexistent}\\
(x_{2},0)&\text{nonexistent}
\end{array}
\right.
\end{align*}
\newline

\noindent For $F_x>0$, $F_y=0$, $b\geq\frac{1}{2}$ and the following values of $F_x$
\begin{align*}
F_x< F_x^{crit}&\left\{
\begin{array}{lll}
(x_{0},0)&\text{well}\\
(x_1,0) &\text{well}\\
(x_2,0)&\text{saddle}
\end{array}
\right.\\[0.7em]
F_x> F_x^{crit}&\left\{
\begin{array}{lll}
(x_{0},0)&\text{well}\\
(x_1,0) &\text{nonexistent}\\
(x_2,0)&\text{nonexistent}
\end{array}
\right.\\[0.7em]
F_x= F_x^{crit}&\left\{
\begin{array}{lll}
(x_{0},0)&\text{well}\\
(x_1,0) &\text{unidentified}\\
(x_2,0)&\text{unidentified}
\end{array}
\right.
\end{align*}
\newline 

\noindent For $F_x=0$, $F_y>0$, $b\geq\frac{1}{2}$ and the following values of $F_y$
\begin{align*}
F_y<F_y^{crit}&\quad\left\{
\begin{array}{ll}
(0,y_1)&\text{saddle}\\
(0,y_2)&\text{hill}
\end{array}
\right.\\[0.3em]
F_y<F_y^{sad}&\quad\left\{
\begin{array}{ll}
(0,y_0)&\text{saddle}\\
(\pm x_{well},y_{well})&\text{well}
\end{array}
\right.\\[0.3em]
F_y>F_y^{crit}&\quad\left\{
\begin{array}{ll}
(0,y_1)&\text{nonexistent}\\
(0,y_2)&\text{nonexistent}\\
\end{array}
\right.\\[0.3em]
F_y>F_y^{sad}&\quad\left\{
\begin{array}{ll}
(0,y_0)&\text{well}\\
(\pm x_{well},y_{well})&\text{nonexistent}
\end{array}
\right.\\[0.3em]
F_y=F_y^{crit}&\quad\left\{
\begin{array}{ll}
(0,y_1)&\text{unidentified}\\
(0,y_2)&\text{unidentified}
\end{array}
\right.\\[0.3em]
F_y=F_y^{sad}&\quad\left\{
\begin{array}{ll}
(0,y_0)&\text{unidentified}\\
(\pm x_{well},y_{well})&\text{unidentified}
\end{array}
\right.\\[0.3em]
\end{align*}

\noindent For $F_x=0$, $F_y>0$, $b>\frac{1}{2}$ and the following values of $F_y$
\begin{align*}
F_y<F_y^{sad}&\quad\left\{
\begin{array}{ll}
(0,y_0)&\text{saddle}\\
(\pm x_{well}, y_{well})&\text{well}
\end{array}
\right.\\[0.3em]
F_y=F_y^{sad}&\quad\left\{
\begin{array}{ll}
(0,y_0)=(\pm x_{well}, y_{well})&\text{unidentified}
\end{array}
\right.\\[0.3em]
F_y>F_y^{sad}&\quad\left\{
\begin{array}{ll}
(0,y_0)&\text{well}\\
(\pm x_{well}, y_{well})&\text{nonexistent}
\end{array}
\right.
\end{align*}
where
\begin{align*}
x_k&=-\frac{2}{\sqrt{3}}\sqrt{1+2a}\cos \left\{    
\frac{1}{3}\tan^{-1}\left(   
\frac{\sqrt{4(1+2a)^3-27F_x^2}}{F_x\sqrt{27}}
\right)+\frac{2\pi}{3}k
\right\}\\
y_k&=-\frac{2}{\sqrt{3}}\sqrt{1-2b}\,\cos
\left\{
\frac{1}{3}\tan^{-1}
\left(
\frac{\sqrt{4(1-2b)^3-27F_y^2}}{F_y\sqrt{27}}\right)
+\frac{2\pi}{3}k
\right\}\\
x_{saddle}&=\frac{F_x}{2(a+b)}\\
y_{saddle}&=\sqrt{(1-2b)-\left(\frac{F_x}{2(a+b)}\right)^2}\\
x_{well}&=\sqrt{(1+2a)-\left(\frac{F_y}{2(a+b)}\right)^2}\\
y_{well}&=\frac{-F_y}{2(a+b)}\\
F_x^{crit}&=\sqrt{\frac{4(1+2a)^3}{27}}\\
F_x^{sad}&=2(a+b)\sqrt{1-2b}\\
F_y^{crit}&=\sqrt{\frac{4(1-2b)^3}{27}}\\
F_y^{sad}&=2(a+b)\sqrt{1+2a}\\
R_1&=\left(\frac{1}{\sqrt{3}}\sqrt{1+2a}, \sqrt{1+2a}\right)\\
R_2&=\left(0,\frac{1}{\sqrt{3}}\sqrt{1+2a}\right).
\end{align*}

\end{Theorem}

\noindent The proof is given in a series of Lemmas for each of the six different cases. 
Theorem \ref{chap_5_case_1} proves the case for $F_x=0$ and $F_y=0$,
Theorem \ref{chap_5_case_2} proves the case for $F_x>0$, $F_y=0$ and $b<\frac{1}{2}$,
Theorem \ref{chap_5_case_3} proves the case for $F_x>0$, $F_y=0$ and $b\geq \frac{1}{2}$,
Theorem \ref{chap_5_case_4} proves the case for $F_x=0$, $F_y>0$ and $b<\frac{1}{2}$
and Theorem \ref{chap_5_case_5} proves the case for $F_x=0$, $F_y>0$ and $b\geq\frac{1}{2}$.
All the notation used will be consistent with this current Theorem \ref{chap_5_case_0}. 
The following standard result is used.

\begin{Theorem}\label{chap_5_thm_1}
Let $V:\mathbb{R}^2\longrightarrow \mathbb{R}$ be twice differentiable everywhere. Let $H$ be the Hessian at a critical point $(x_0, y_0)$. The nature of the critical point can be determined by 
\begin{equation*}
\begin{array}{lc}
\det H<0 \ \  \Rightarrow & \text{saddle} \\
\det H>0 \ \  \text{then} & \left\{  
\begin{array}{lc}  
\text{if}\quad\frac{\partial ^2 V}{\partial x^2}>0  \ \ \Rightarrow \ \ & \text{well}\\[0.5em]
\text{if}\quad\frac{\partial ^2 V}{\partial x^2}<0 \ \  \Rightarrow \ \ & \text{hill}
\end{array}\right.
\end{array}
\end{equation*}
\end{Theorem} 

\noindent We recall results about the cubic equation. 

\begin{Theorem}\label{chap_5_thm_cubic_theory}
Let $a_3,a_2,a_1,a_0 \in\mathbb{R}$  where $a_3\neq 0$. Consider the cubic equation
\begin{align*}
a_3x^3+a_2x^2+a_1x+a_0=0
\end{align*}
and its discriminant
\begin{align*}
\Delta&=18a_3a_2a_1a_0
-4a_2^3a_0
+a_2^2a_1^2
-4a_3a_1^3
-27a_3^2a_0^2
\end{align*}
then following statements hold
\begin{align*}
\text{if} \quad \Delta&>0 \quad \text{then the equation has 3 distinct real roots.}\\
\text{if} \quad \Delta&=0 \quad \text{then the equation has a multiple real root and all its roots are real.}\\
\text{if} \quad \Delta&<0 \quad \text{then the equation has 1 real root and 2 complex conjugate roots.}
\end{align*}
and the three roots of the equations are
\begin{equation*}
x_k=-\frac{1}{3a_3}\left(a_2+e^{i\psi_k}C+e^{-i\psi_k}\frac{\Delta_0}{C}   \right)
\end{equation*}
where 
\begin{align*}
\psi_0&=0, \ \ \psi_1=2\pi/3, \ \ \psi_2=4\pi/3\\
C&=\sqrt[3]{\frac{\Delta_1+\sqrt{-27\Delta}}{2}}\\
\Delta_0&=a_2^2-3a_3a_1\\
\Delta_1&=2a_2^3-9a_3a_2a_1+27a_3^2a_0.
\end{align*}
\end{Theorem}

\section{Case $F_x=0$ and $F_y=0$}
The case for no forcing $F=0$  is considered first. 
\begin{Theorem}\label{chap_5:thm:new_1}\label{chap_5_case_1}
When $F_x=F_y=0$ the critical points of the potential have the following properties. 
For $b<\frac{1}{2}$ the critical points and their nature are
\begin{equation*}
\begin{array}{ccc}
b<\frac{1}{2} & (0,0) & hill \\[0.5em]
& (\pm \sqrt{1+2a},0) & well \\[0.5em]
& (0,\pm\sqrt{1-2b}) & saddle 
\end{array}
\end{equation*}
For $b\geq \frac{1}{2}$ the critical points and their nature are
\begin{equation*}
\begin{array}{ccc}
b\geq\frac{1}{2} & (0,0) & saddle\\[0.5em]
& (\pm \sqrt{1+2a},0) & well 
\end{array}
\end{equation*}
\end{Theorem}
\noindent The proof is trivial and omitted.

\section{Case $F_x>0$ and $F_y=0$}
When forcing is only in the $x$ direction two cases are considered separately, that is for $b<\frac{1}{2}$ and $b\geq \frac{1}{2}$. The case for $F_x\leq 0$ is similar.

\subsection{Case $F_x>0$, $F_y=0$ and $b<\frac{1}{2}$}
When there is no forcing there are five critical points. 
Intuitively as forcing is increased the system could gradually start to deviate away from having five critical points. 
The critical points may collide and coincide. The structure of the following proofs are first determining the bounds on the critical points and then determining their nature. 
We have the following consequence which uses the solution and theory of the cubic equation with three real roots. 

\begin{Theorem}\label{chap_5_thm_five_roots}
Let $F_x>0$, $F_y=0$ and $b<\frac{1}{2}$. Let $F_x$ be bounded by 
\begin{align*}
F_x\leq F_x^{sad}
\quad \text{and} \quad 
F_x\leq F_x^{crit}
\end{align*}
where 
\begin{align*}
F_x^{sad}=2(a+b)\sqrt{1-2b}
\quad \text{and} \quad 
F_x^{crit}=\sqrt{\frac{4(1+2a)^3}{27}}
\end{align*}
then there are five critical points given by 
\begin{align*}
&(x_k, 0) \quad k=1,2,3\\
&(x_{saddle},\pm y_{saddle})
\end{align*}
where 
\begin{align*}
x_{saddle}&=\frac{F_x}{2(a+b)}\\
y_{saddle}&=\sqrt{(1-2b)-\left(\frac{F_x}{2(a+b)}\right)^2}\\
x_k&=-\frac{2}{\sqrt{3}}\sqrt{1+2a}\, \cos\left(\phi+k\frac{2\pi}{3}\right)\\
\phi&=\frac{1}{3}\tan^{-1}(l)\\
l&=\frac{\sqrt{4(1+2a)^3-27F_x^2}}{F_x\sqrt{27}}\\
k&=0, \quad k=1, \quad k=2.
\end{align*}
\end{Theorem}

\begin{proof}
The simultaneous equations to be solved are 
\begin{align}
\frac{\partial V_F}{\partial x}&=x(x^2+y^2)-(1+2a)x+F_x=0 \label{chap_5_new_eqn1}\\
\frac{\partial V_F}{\partial y}&=y(x^2+y^2)-(1-2b)y=0.  \label{chap_5_new_eqn2}
\end{align}
Equation \ref{chap_5_new_eqn2} holds if either $(x^2+y^2)-(1-2b)=0$ or $y=0$. 
The $(x^2+y^2)-(1-2b)=0$ case is considered first, which gives 
$(x^2+y^2)=(1-2b)$. 
Substituting this into Equation \ref{chap_5_new_eqn1} gives $x$ as 
\begin{equation*}
x_{saddle}=\frac{F_x}{2(a+b)}
\end{equation*}
which when substituted back into $(x^2+y^2)=(1-2b)$ gives $y$ as
\begin{equation*}
y_{saddle}=\pm\sqrt{(1-2b)-\left(  \frac{F_x}{2(a+b)}\right)^2}. 
\end{equation*}
For the $y=0$ case, Equation \ref{chap_5_new_eqn1} becomes 
\begin{equation*}
x^3-(1+2a)x+F_x=0
\end{equation*}
which is a cubic equation. 
Solving this cubic equation using the notation in Theorem \ref{chap_5_thm_cubic_theory} gives
\begin{align*}
\Delta_0&=3(1+2a), \ \ \Delta_1=27F_x, \ \ \Delta=4(1+2a)^3-27F_x^2\\
\end{align*}
which gives\footnote{After noting that $(x+iy)^{1/3}=(x^2+y^2)^{1/6}\exp\left\{i(1/3)\tan^{-1}(y/x)\right\}$.}
\begin{align*}
C&=\sqrt[3]{\frac{27F_x+i\sqrt{27\Delta}}{2}}\\
&=\left[ \left(\frac{27F_x}{2}\right)^2 +\left( 
\frac{\sqrt{27\Delta}}{2}
 \right)^2\right]^{1/6}\exp \left\{i(1/3)\tan^{-1}\left( \frac{\sqrt{27\Delta}}{27F_x}\right)\right\}
\end{align*}
which simplifies to 
\begin{align*}
\left(\frac{27F_x}{2}\right)^2 +\left( 
\frac{\sqrt{27\Delta}}{2}
 \right)^2&=\frac{1}{4}\left(27 \times 27 F_x^2+27\Delta \right)\\
&=27(1+2a)^3,
\end{align*}
after letting 
\begin{align*}
\phi=(1/3)\tan^{-1}\left( \frac{\sqrt{27\Delta}}{27F_x}\right)
\end{align*}
we have 
\begin{align*}
C&=\sqrt{3(1+2a)}\ e^{i\phi}\\
\frac{\Delta_0}{C}&=\frac{3(1+2a)}{\sqrt{3(1+2a)}} \ e^{-i\phi}\\
&=\sqrt{3(1+2a)}\ e^{-i\phi}
\end{align*}
which gives the 3 solution as 
\begin{align*}
x_k&=-\frac{1}{3}\sqrt{3(1+2a)} \left(e^{i(\phi+\psi_k)} +e^{-i(\phi+\psi_k)}\right)\\
&=-\frac{2}{3}\sqrt{3(1+2a)}\cos(\phi+\psi_k)\\
&=-\frac{2}{3}\sqrt{3(1+2a)}\cos \left\{    
\frac{1}{3}\tan^{-1}\left(   
\frac{\sqrt{4(1+2a)^3-27F_x^2}}{F_x\sqrt{27}}
\right)+\psi_k
\right\}
\end{align*}
where $k=0,1,2$, $\psi_0=0$, $\psi_1=\frac{2\pi}{3}$ and $\psi_2=\frac{4\pi}{3}$. 
Now notice that $y_{saddle}$ requires taking the square root of a real number.
This means $y_{saddle}$ will be real if and only if the argument under the square root is positive
\begin{align*}
y_{saddle}&=\pm\sqrt{(1-2b)-\left(  \frac{F_x}{2(a+b)}\right)^2}\\
 &\in \mathbb{R}\\
\Leftrightarrow 0&\leq(1-2b)-\left(  \frac{F_x}{2(a+b)}\right)^2\\
\Leftrightarrow F_x&\leq F_x^{sad}\\
\text{where} \quad F_x^{sad}&=2(a+b)\sqrt{1-2b}.
\end{align*}
Notice also how the argument inside the $\tan^{-1}(\cdot)$ function contains a square root as well, which is actually the square root of the discriminant.
The three cubic roots $x_k$ would be real if and only if the discriminant is positive
\begin{align*}
\sqrt{\Delta}&=\sqrt{4(1+2a)^3-27F_x^2}\\
&\in \mathbb{R}\\
\Leftrightarrow 0&\leq 4(1+2a)^3-27F_x^2\\
\Leftrightarrow F_x&\leq F_x^{crit}\\
\text{where}\quad F_x^{crit}&=\sqrt{\frac{4(1+2a)^3}{27}}
\end{align*}
which clearly puts bounds on the forces. This completes the proof. 
\end{proof}

\noindent Next there is a simple but useful Lemma. 

\begin{Lemma}\label{chap_5:th:mono}
The function $l$ (as in Theorem \ref{chap_5_thm_five_roots}) is monotone in $F_x$ for $F_x<F_x^{crit}$. 
\end{Lemma}

\begin{proof}
We differentiate $l$ with respect to $F_x$. 
\begin{align*}
\frac{dl}{dF_x}&=\frac{d}{dF_x}\left(\frac{\sqrt{4(1+2a)^3-27F_x^2}}{F_x\sqrt{27}}\right)\\
&=\frac{1}{F_x\sqrt{27}}\left(4(1+2a)^3-27F_x^2\right)^{-\frac{1}{2}}(-2F_x)\ \frac{1}{2}\ 27\  \\
&\relphantom{=}+\frac{\sqrt{4(1+2a)^3-27F_x^2}}{F_x\sqrt{27}} \ \left(\frac{-1}{F_x^2}\right)\\
&\leq0
\end{align*}
which is always negative since we assumed $F_x<F_x^{crit}$ for the square roots to be real and forcing is assumed to be in the positive direction. 
\end{proof}

\noindent Although the monotonicity of $l$ is trivial, it would prove essential for the next series of reasoning. 
It is also easy to see that 
\begin{equation*}
\begin{array}{ll}
F_x=0 & \text{then} \quad l=+\infty\\[0.3em]
F_x=F_x^{crit} &\text{then}  \quad l=0.
\end{array}
\end{equation*}
Since the derivative of $l$ is negative this means that $l$ would decrease from $+\infty$ to $0$ as $F_x$ increase from $0$ to $F_x^{crit}$. This function being monotone means it would decrease to $0$ without any oscillations. 
For short this means 
\begin{equation*}
l = \infty \downarrow 0 \quad \text{as} \quad F_x = 0 \uparrow F_x^{crit}.
\end{equation*}
But $\phi=\frac{1}{3}\tan^{-1}(l)$, with $\tan^{-1}$ is also monotone over $[-\infty,+\infty]$. 
So similarly  we can also say 
\begin{align*}
\phi=\frac{\pi}{6}\downarrow 0  \quad \text{as} \quad F_x = 0 \uparrow F_x^{crit}
\end{align*}
monotonically for increasing $F_x$. 
This means that for $0\leq F_x \leq F_x^{crit}$ we would have 
\begin{equation*}
0\leq  \phi \leq \frac{\pi}{6}. 
\end{equation*}
We note the following values of the $\cos(\cdot)$ function. 
\begin{equation*}
\begin{array}{ll}
\cos(0)=1 & \cos\left(\frac{\pi}{6}\right)=\frac{\sqrt{3}}{2}\\[0.5em]
\cos\left(\frac{2\pi}{3}\right)=\frac{-1}{2}&\cos\left(\frac{2\pi}{3}+\frac{\pi}{6}\right)=\frac{-\sqrt{3}}{2}\\[0.5em]
\cos\left(\frac{4\pi}{3}\right)=\frac{-1}{2}&\cos\left(\frac{4\pi}{3}+\frac{\pi}{6}\right)=0.
\end{array}
\end{equation*}
From this we can bound  $\cos(\cdot)$ for the three values of $k$ for $0\leq F_x\leq F_x^{crit}$. 
\begin{align*}
\begin{array}{l r c l l l}
k=0 & \frac{\sqrt{3}}{2}&\leq&\cos\left(\phi+k\frac{2\pi}{3}\right)&\leq&1\\[0.5em]
k=1 & \frac{-\sqrt{3}}{2}&\leq&\cos\left(\phi+k\frac{2\pi}{3}\right)&\leq&\frac{-1}{2}\\[0.5em]
k=2 & \frac{-1}{2}&\leq&\cos\left(\phi+k\frac{2\pi}{3}\right)&\leq&0. 
\end{array}
\end{align*}
We also note that the $\cos(\cdot)$ function is monotone on 
$\left[0, \pi\right]$ and $\left[\pi, 2\pi\right]$.
But $\left(\phi+k\frac{2\pi}{3}\right)$
is in the intervals where $\cos(\cdot)$ is monotone,
therefore we can say that the three critical points on the $x$-axis $x_k$ are also monotone in $F_x$. 
We can now have bounds on the $x_k$ for $0\leq F_x\leq F_x^{crit}$. 
\begin{align*}
\frac{-2}{\sqrt{3}}\sqrt{1+2a}&\leq x_0 \leq -\sqrt{1+2a}\\
\frac{1}{\sqrt{3}}\sqrt{1+2a}&\leq x_1 \leq \sqrt{1+2a}\\
0&\leq x_2 \leq \frac{1}{\sqrt{3}}\sqrt{1+2a}. 
\end{align*}
Using the monotonicity of $x_k$ we get that
\begin{equation*}
\begin{array}{lrrll}
x_0: & -\sqrt{1+2a}&\longrightarrow& \frac{-2}{\sqrt{3}}\sqrt{1+2a}& \text{as}\quad F_x\rightarrow F_x^{crit}\\[0.5em]
x_1:& \sqrt{1+2a}&\longrightarrow& \frac{1}{\sqrt{3}}\sqrt{1+2a}& \text{as}\quad F_x\rightarrow F_x^{crit}\\[0.5em]
x_2: & 0&\longrightarrow &\frac{1}{\sqrt{3}}\sqrt{1+2a}& \text{as}\quad F_x\rightarrow F_x^{crit}. 
\end{array}
\end{equation*}
The monotonicity of $x_k$ means the movements of the $x_k$ are always in one direction and they will never oscillate. 
We obtain the following
 
\begin{Lemma}\label{chap_5:th:det}
Let $F_x>0$, $F_y=0$ and $b<\frac{1}{2}$. 
These three scenarios hold.

If $F_x<F_x^{crit}$ we have 3 critical points on the $x$-axis: $x_0<x_2<x_1$. 

If $F_x=F_x^{crit}$ we have 2 critical points on the $x$-axis: $x_0<x_2=x_1$. 

If $F_x>F_x^{crit}$ we have 1 critical point on the $x$-axis: $x_0<\frac{-2}{\sqrt{3}}\sqrt{1+2a}$ and $x_0 \rightarrow -\infty $ monotonically with increasing $F_x$. 
\end{Lemma}

\begin{proof}
The three $x_k$ are solutions to a cubic equation. 
This cubic equation was derived assuming $y=0$. The other critical points $(x_{saddle},\pm y_{saddle})$ were derived assuming $y\neq0$. 

If $F_x<F_x^{crit}$ the discriminant of this cubic equation dictates that there should be three distinct real solution. 
The bounds on $x_0$, $x_1$ and $x_2$ show that $x_0<x_2<x_1$. 
If $F_x=F_x^{crit}$ the discriminant of this cubic equation dictates that there should be at least two repeated solution. 
It was shown that $x_1=x_2=\frac{1}{\sqrt{3}}\sqrt{1+2a}$ when $F_x=F_x^{crit}$. 
The bound on $x_0$, $x_1$ and $x_2$ shows that $x_0<x_2=x_1$. 
If $F_x>F_x^{crit}$ the discriminant of this cubic equation dictates that there should only be one real solution. 
If  we can show that $x_0<0$ is real  then we are done. 
For  $F_x>F_x^{crit}$ the $l$ function becomes
\begin{align*}
l=i\frac{\sqrt{27F_x^2-4(1+2a)^3}}{F_x\sqrt{27}}.
\end{align*}
Differentiating the imaginary part gives 
\begin{align*}
\frac{d}{dF_x}(-il)&=\frac{d}{dF_x}\left(\frac{\sqrt{27F_x^2-4(1+2a)^3}}{F_x\sqrt{27}}\right)\\
&=\frac{F_x\sqrt{27}}{\sqrt{27F_x^2-4(1+2a)^3}}\left(1-\frac{27F_x^2-4(1+2a)^3}{27F_x}\right)\\
&>0,
\end{align*}
since we have assumed $F_x>F_x^{crit}$. 
This also shows that the imaginary part of $l$ is monotonically increasing as $F_x$ increases. We also note that 
\begin{align*}
0\leq (-il)\leq1
\end{align*}
because $(-il)=0$ when $F_x=F_x^{crit}$ and $(-il)=1$ when $F_x=\infty$
and the increase in $(-il)$ is monotone.  
Using $\tan^{-1}(\cdot)$ defined for complex arguments gives 
\begin{align*}
\phi&=\frac{1}{3}\tan^{-1}(l)\\
&=\frac{1}{3}\times\frac{1}{2}i\{\ln(1-il)-\ln(1+il)\}\\
&=\frac{1}{6}i\{\ln(1+\epsilon)-\ln(1-\epsilon)\}
\quad\text{where}\quad0\leq\epsilon\leq1
\end{align*}
after letting $\epsilon=-il$. 
Notice that $\phi$ now has zero real part, which means we can denote $\phi$ with a real $\gamma$ by writing
\begin{align*}
\phi=i\gamma.
\end{align*}
We also note that $\gamma$ will be monotonically decreasing as $F_x$ increases because 
\begin{align*}
\frac{d}{d\epsilon}\left[\ln(1+\epsilon)-\ln(1-\epsilon)\right]&=\frac{1}{1+\epsilon}-\frac{1}{1-\epsilon}\\
&=\frac{-2\epsilon}{(1+\epsilon)(1-\epsilon)}\\
&<0\quad \text{for} \quad F_x>F_x^{crit}.
\end{align*}
Because it was shown that as $F_x$ increases from $F_x=F_x^{crit}$ to $F_x=\infty$,
$\epsilon$ would increase from $\epsilon=0$ to $\epsilon=1$, 
which means $\gamma$ would monotonically decrease. 
So the critical point may now be written as 
\begin{align*}
x_0&=-\frac{2}{\sqrt{3}}\sqrt{1+2a}\cos\left(\phi\right)\\
&=-\frac{2}{\sqrt{3}}\sqrt{1+2a}\ \frac{e^{i(i\gamma)}+e^{-i(i\gamma)}}{2}\\
&=-\frac{2}{\sqrt{3}}\sqrt{1+2a}\ \cosh(\gamma)
\end{align*}
which means $x_0<\frac{-2}{\sqrt{3}}\sqrt{1+2a}$ and monotonically decreasing. Note that $\gamma=0$ when $F_x=F_x^{crit}$. 
\end{proof}

\noindent Now we consider a special case of the forcing when $F_x=F_x^{sad}$. 
This gives
\begin{align*}
(x_{saddle},\pm y_{saddle})
&=(x_{saddle},0)\\
&=(\sqrt{1-2b},0)\\
\Rightarrow
x_{saddle}
&=\sqrt{1-2b}
\end{align*} 
which seemingly adds a fourth critical point onto the $x$-axis. 
This brings us to the next Lemma. 

\begin{Lemma}\label{chap_5:th:great}
$F_x^{sad}\leq F_x^{crit}$ holds. 
\end{Lemma}

\begin{proof}
Proof by contradiction.  
Assume that $F_x^{sad}>F_x^{crit}$.  Let $F_x=F_x^{sad}$ which means $(x_{saddle},\pm y_{saddle})=(x_{saddle},0)$ as a new critical point on the $x$-axis. 
Now we have to show that $(x_{saddle},0)$ is not one of 
$(x_0,0)$, $(x_1,0)$ or $(x_2,0)$. 
The expressions for the critical points mean we would always have $x_{saddle}>0$. 
But $F_x=F_x^{sad}$ also implies $F_x>F_x^{crit}$ which by Lemma \ref{chap_5:th:det} means the only critical point is $x_0<0$ which is a contradiction. 
\end{proof}

\noindent The next Lemma will be useful in avoiding complicated  manipulation of trigonometric identities when it comes to proving properties about the critical points. 

\begin{Lemma}\label{chap_5:th:bond}
Let $F_x>0$, $F_y=0$ and $b<\frac{1}{2}$. 
If $F_x=F_x^{sad}$ we must have either $x_{saddle}=x_1$ or $x_{saddle}=x_2$. If $F_x=F_x^{sad}=F_x^{crit}$ then $x_{saddle}=x_1=x_2$. 
\end{Lemma}

\begin{proof}
For strictly positive $F_x>0$ some of the bounds on the critical points would have to be made strict inequalities. 
This means $x_0<0$, $x_1>0$, $x_2>0$ and $x_{saddle}>0$. 
By the time $F_x=F_x^{sad}$, $x_{saddle}$ would be a critical point on the $x$-axis. 
By Lemma \ref{chap_5:th:great} we must always have $F_x^{sad}\leq F_x^{crit}$. 
If $F_x^{sad}<F_x^{crit}$ then by Lemma \ref{chap_5:th:det} there must be three distinct critical points on the $x$-axis, so we must have either $x_{saddle}=x_1$ or $x_{saddle}=x_2$ (as $x_0<0$).   
If $F_x^{sad}=F_x^{crit}$ then again by Lemma \ref{chap_5:th:det} $x_1=x_2$ and there can only be two critical points on the $x$-axis, therefore $x_{saddle}=x_1=x_2$. 
\end{proof}

\noindent Now we are ready for one of the main Theorems of this Chapter. 
The ultimate aim is to find the nature and position of all the critical points under different values of the forcing $F_x$. 

\begin{Theorem}\label{chap_5_case_2}
Let $F_x>0$, $F_y=0$ and $b<\frac{1}{2}$. The positions and nature of the critical points are as follows
\begin{align*}
\text{For any}\quad F_x>0\relphantom{F_x}&\left\{
\begin{array}{ll}
(x_0,0) & \text{well}
\end{array}
\right.\\[0.7em]
\text{and if}\quad F_x< F_x^{sad}&\left\{
\begin{array}{lll}
(x_1,0) &\text{well}\\
(x_2,0)&\text{hill} \\
(x_{saddle},\pm y_{saddle})&\text{saddle}
\end{array}
\right.\\[0.7em]
\text{or}\quad F_x^{sad} < F_x < F_x^{crit}&\left\{
\begin{array}{llll}
(x_1,0)&\text{saddle} &\text{for} & \sqrt{1-2b}\in R_1\\
(x_1,0)&\text{well} &\text{for}&\sqrt{1-2b}\in R_2\\
(x_2,0)&\text{hill} &\text{for}&\sqrt{1-2b}\in R_1\\
(x_2,0) &\text{saddle} &\text{for}&\sqrt{1-2b}\in R_2\\
(x_{saddle},\pm y_{saddle})&\text{nonexistent}&&
\end{array}
\right.\\[0.7em]
\text{or}\quad F_x= F_x^{sad}<F_x^{crit}&\left\{
\begin{array}{llll}
(x_1,0)&\text{unidentified} &\text{for} & \sqrt{1-2b}\in R_1\\
(x_1,0)&\text{well} &\text{for}&\sqrt{1-2b}\in R_2\\
(x_2,0)&\text{hill} &\text{for}&\sqrt{1-2b}\in R_1\\
(x_2,0) &\text{neither} &\text{for}&\sqrt{1-2b}\in R_2\\
(x_{saddle}, \pm y_{saddle})&\text{unidentified}&&
\end{array}
\right.
\\[0.7em]
\text{or}\quad F_x= F_x^{crit}&\left\{
\begin{array}{ll}
(x_{1},0) &\text{unidentified}\\
(x_{2},0)&\text{unidentified}
\end{array}
\right.\\[0.7em]
\text{or}\quad F_x>F_x^{crit}&\left\{
\begin{array}{ll}
(x_{1},0) &\text{nonexistent}\\
(x_{2},0)&\text{nonexistent}
\end{array}
\right.
\end{align*}
\end{Theorem}

\begin{proof}
For the three $x_k$ we note that for $0<F_x\leq F_x^{crit}$ they are elements of the intervals
\begin{align*}
x_0&\in \left[\frac{-2}{\sqrt{3}}\sqrt{1+2a}, -\sqrt{1+2a}\right)\\
x_1&\in\left[\frac{1}{\sqrt{3}}\sqrt{1+2a}, \sqrt{1+2a}\right)\\
x_2&\in\left(0,\frac{1}{\sqrt{3}}\sqrt{1+2a}\right].
\end{align*}
Note also that at $F_x=F_x^{sad}$ the associated value of $x_{saddle}(a,b,F_x^{sad})=\sqrt{1-2b}$ can only ever be elements of certain intervals. 
\begin{align*}
\text{For} \quad F_x^{sad}<F_x^{crit} \quad \text{either} \quad 
\sqrt{1-2b}&\in R_1:=\left(\frac{1}{\sqrt{3}}\sqrt{1+2a}, \sqrt{1+2a}\right)\\[0.5em]
\quad \text{or} \quad
\sqrt{1-2b}&\in R_2:=\left(0,\frac{1}{\sqrt{3}}\sqrt{1+2a}\right)\\[0.5em]
\text{and if} \quad F_x^{sad}=F_x^{crit} \quad
\text{then} \quad 
\sqrt{1-2b}&=\frac{1}{\sqrt{3}}\sqrt{1+2a}
\end{align*}
where $R_1$ and $R_2$ are defined as above. 
These statements above can be justified as follows. 
Lemma \ref{chap_5:th:great} says $F_x^{sad}\leq F_x^{crit}$. If $F_x=F_x^{sad}=F_x^{crit}$ then  Lemma \ref{chap_5:th:bond} says $x_{saddle}=x_1=x_2=\frac{1}{\sqrt{3}}\sqrt{1+2a}$. 
Also, $\sqrt{1-2b}$ must always live in the regions specified, because we must always have $0<\sqrt{1-2b}<\sqrt{1+2a}$ for $0<b<\frac{1}{2}$. 
Or, to justify it in another way, if $\sqrt{1-2b}>0$ live beyond the regions $R_{1}$ or $R_2$ then we would have four critical points on the $x$-axis which is not possible. 

Now we see how the critical points collide. 
If $\sqrt{1-2b}\in R_1$ then we got to have $x_{saddle}=x_1$. 
If  $\sqrt{1-2b}\in R_2$ then we got to have $x_{saddle}=x_2$. 
This is because for $F_x=F_x^{sad}<F_x^{crit}$ there have to be three distinct critical points on the $x$-axis as argued by Lemma \ref{chap_5:th:det} and Lemma \ref{chap_5:th:bond}. 
From this information new bounds on the $x_k$ critical points may be derived. 
The bounds for the critical points on the $x$-axis written  compactly, concisely and definitively  for $F_x>0$ are
\begin{align*}
F_x\leq F_x^{crit}&\quad\left\{
\begin{array}{rrrrr}
\frac{-2}{\sqrt{3}}\sqrt{1+2a}&\leq & x_0& <&-\sqrt{1+2a}\\
\frac{1}{\sqrt{3}}\sqrt{1+2a}&\leq &x_1& <& \sqrt{1+2a}\\
0&<& x_2 &\leq &\frac{1}{\sqrt{3}}\sqrt{1+2a}
\end{array}
\right.\\[0.3em]
F_x< F_x^{sad}&\quad\left\{
\begin{array}{l}
x_1>\sqrt{1-2b}\quad \text{for} \quad \sqrt{1-2b}\in R_1\\
x_2<\sqrt{1-2b}\quad \text{for} \quad \sqrt{1-2b}\in R_2
\end{array}
\right.\\[0.3em]
F_x> F_x^{sad}&\quad\left\{
\begin{array}{l}
x_1<\sqrt{1-2b}\quad \text{for} \quad \sqrt{1-2b}\in R_1\\
x_2>\sqrt{1-2b}\quad \text{for} \quad \sqrt{1-2b}\in R_2
\end{array}
\right.\\[0.3em]
F_x= F_x^{sad}&\quad\left\{
\begin{array}{l}
x_{1}=x_{saddle}=\sqrt{1-2b}\\
x_{2}=x_{saddle}=\sqrt{1-2b}
\end{array}
\right.
\begin{array}{l}
\text{not necessarily $x_1=x_2$ }
\end{array}
\\[0.3em]
F_x= F_x^{crit}&\quad\left\{
\begin{array}{l}
x_0=\frac{-2}{\sqrt{3}}\sqrt{1+2a}\\
x_1=x_2=\frac{1}{\sqrt{3}}\sqrt{1+2a} 
\end{array}
\right.\\[0.3em]
F_x>F_x^{crit}&\quad\left\{
\begin{array}{l}
x_0<\frac{-2}{\sqrt{3}}\sqrt{1+2a}
\end{array}
\right.\\[0.3em]
F_x=F_x^{sad}=F_x^{crit}&\quad\left\{
\begin{array}{l}
x_1=x_2=x_{saddle}=\sqrt{1-2b}=\frac{1}{\sqrt{3}}\sqrt{1+2a} 
\end{array}
\right.
\end{align*}

\noindent Now that the bounds on the critical points for various forces are known, we can deduce their nature. 
The $(x_{saddle},\pm y_{saddle})$ is the easiest to prove, taking into account of $x_{saddle}^2+y_{saddle}^2=(1-2b)$ we have for the determinant of the Hessian at $(x_{saddle},\pm y_{saddle})$
\begin{align*}
\det H =-4y_{saddle}^2(a+b)<0
\end{align*}
which is definitely a saddle.
The Hessian for the critical points $(x_k,0)$  is already diagonal even with forcing.  
It is 
\begin{align*}
H(x_k,0)&=
\left(
\begin{array}{cc}
\frac{\partial ^2V_F}{\partial x^2}&0\\
0&\frac{\partial^2V_F}{\partial y^2}\\
\end{array}
\right)=
\left(
\begin{array}{cc}
3x_k^2-(1+2a)&0\\
0&x_k^2-(1-2b)\\
\end{array}
\right)
\end{align*}
and so by using the bounds we derived, and by considering whether the eigenvalues are both positive (well), both negative (hill) or opposite signs (saddle) we can finally deduce the nature of all five critical points. 
\end{proof}

\noindent The situation can be represented graphically as 

\begin{figure}[H]
\centerline{\includegraphics[scale=0.75]{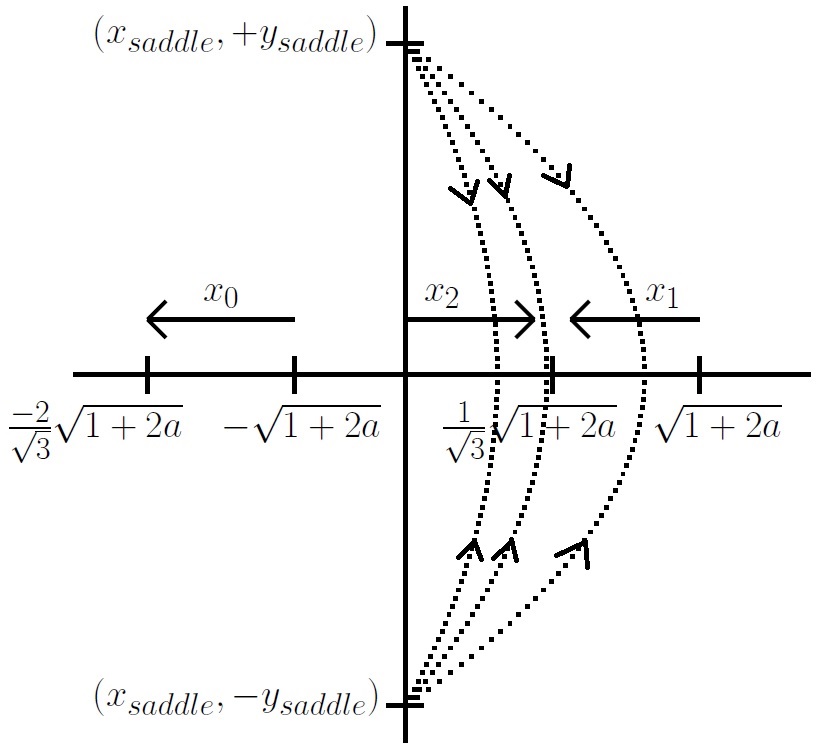}}
\caption{
As $F_x$ increases from $0$ to $F^{crit}_x$, the $x_{1,2,3}$ move as shown in the diagram. 
As $F_x$ increases from $0$ to $F^{sad}_x$ the $(x_{saddle},\pm y_{saddle})$ meet each other on the $x$-axis. 
There are three possible paths for $(x_{saddle},\pm y_{saddle})$. 
If $\sqrt{1-2b}\in R_1$, then the two $(x_{saddle},\pm y_{saddle})$ would meet in the interval $\left(\frac{1}{\sqrt{3}}\sqrt{1+2a},\sqrt{1+2a}\right)$ and collide into $(x_2,0)$. 
If $\sqrt{1-2b} \in R_2$, then the two $(x_{saddle},\pm y_{saddle})$ would meet in the interval $\left(0,\frac{1}{\sqrt{3}}\sqrt{1+2a}\right)$ and collide into $(x_1,0)$.  
If $\sqrt{1-2b}=\frac{1}{\sqrt{3}}\sqrt{1+2a}$, which is also when $F_x^{sad}=F^{crit}_x$, then the two $(x_{saddle},\pm y_{saddle})$ would meet at $x=\frac{1}{\sqrt{3}}\sqrt{1+2a}$ and collide simultaneously into $(x_1,0)$ and $(x_2,0)$. 
}
\end{figure}

\begin{figure}[H]
\centerline{\includegraphics[scale=0.75]{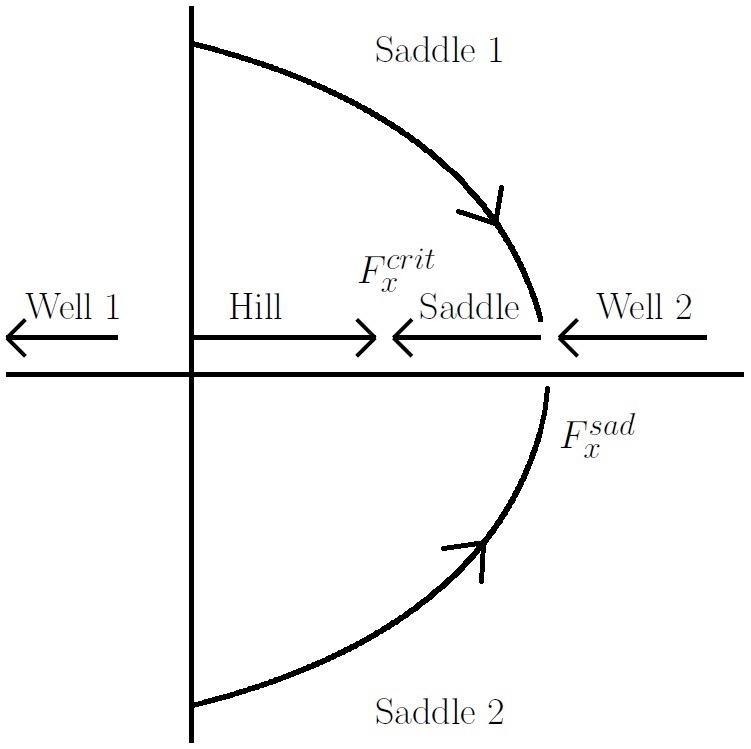}}
\caption{
This is the case for when $\sqrt{1-2b}\in R_1$. When $F=F^{sad}_x$ the two saddles collide into the right well and turns into a new saddle. 
At $F^{crit}_x$, this newly created saddle collides into the hill and both disappears. When Well 2 turns into a Saddle here, it is like creating a new path for the particle to transit to Well 1. }
\end{figure}

\begin{figure}[H]
\centerline{\includegraphics[scale=0.75]{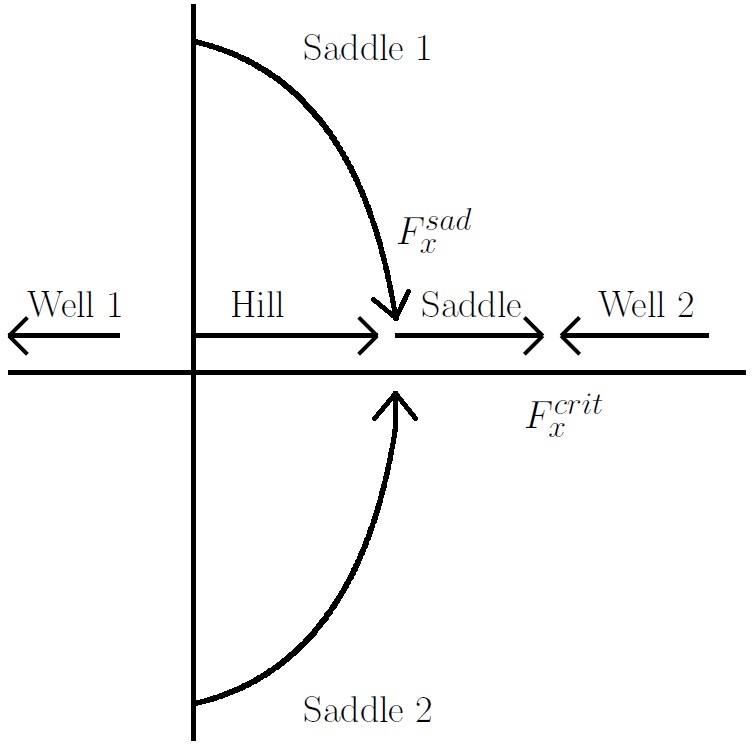}}
\caption{
This is the case for $\sqrt{1-2b}\in R_2$.  
When $F=F^{sad}_x$ the two saddles collide into the hill and turns into a new saddle. 
At $F^{crit}_x$, this newly created saddle collides into the right well and both disappears.
This system behaves in a similar way to a One Dimensional Potential.
}
\end{figure}

\begin{figure}[H]
\centerline{\includegraphics[scale=0.75]{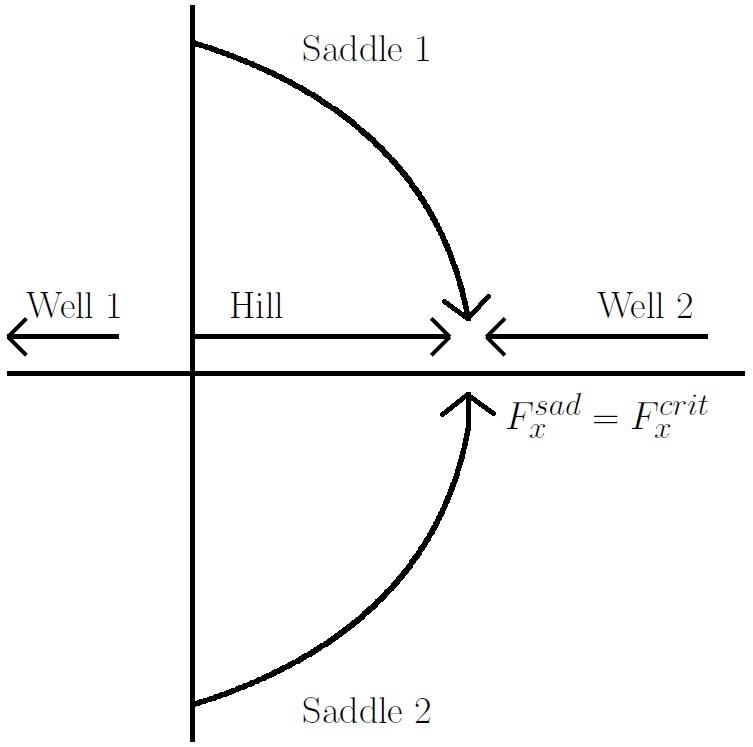}}
\caption{
This is the case for when $\sqrt{1-2b}=\frac{1}{\sqrt{3}}\sqrt{1+2a}$, which is also when $F_x^{sad}=F^{crit}_x$. 
At $F=F^{sad}_x=F^{crit}_x$ the two saddles, hill and right well mutually collide at the same place and disappears. 
}
\end{figure}

\subsection{Case $F_x>0$, $F_y=0$ and $b\geq \frac{1}{2}$ }
For $b\geq\frac{1}{2}$ the reasoning is similar to the $b<\frac{1}{2}$ case, but $(x_{saddle},\pm y_{saddle})$ does not exist. 
We have the following Theorem. 

\begin{Theorem}\label{chap_5_case_3}
Let $F_x>0$, $F_y=0$ and $b\geq \frac{1}{2}$. 
The positions and nature of the critical points are as follows
\begin{align*}
F_x< F_x^{crit}&\left\{
\begin{array}{lll}
(x_{0},0)&\text{well}\\
(x_1,0) &\text{well}\\
(x_2,0)&\text{saddle}
\end{array}
\right.\\[0.7em]
F_x> F_x^{crit}&\left\{
\begin{array}{lll}
(x_{0},0)&\text{well}\\
(x_1,0) &\text{nonexistent}\\
(x_2,0)&\text{nonexistent}
\end{array}
\right.\\[0.7em]
F_x= F_x^{crit}&\left\{
\begin{array}{lll}
(x_{0},0)&\text{well}\\
(x_1,0) &\text{neither}\\
(x_2,0)&\text{neither}
\end{array}
\right.
\end{align*}
\end{Theorem}

\noindent This can be graphically conveyed as

\begin{figure}[H]
\begin{center}
\centerline{\includegraphics[scale=0.35]{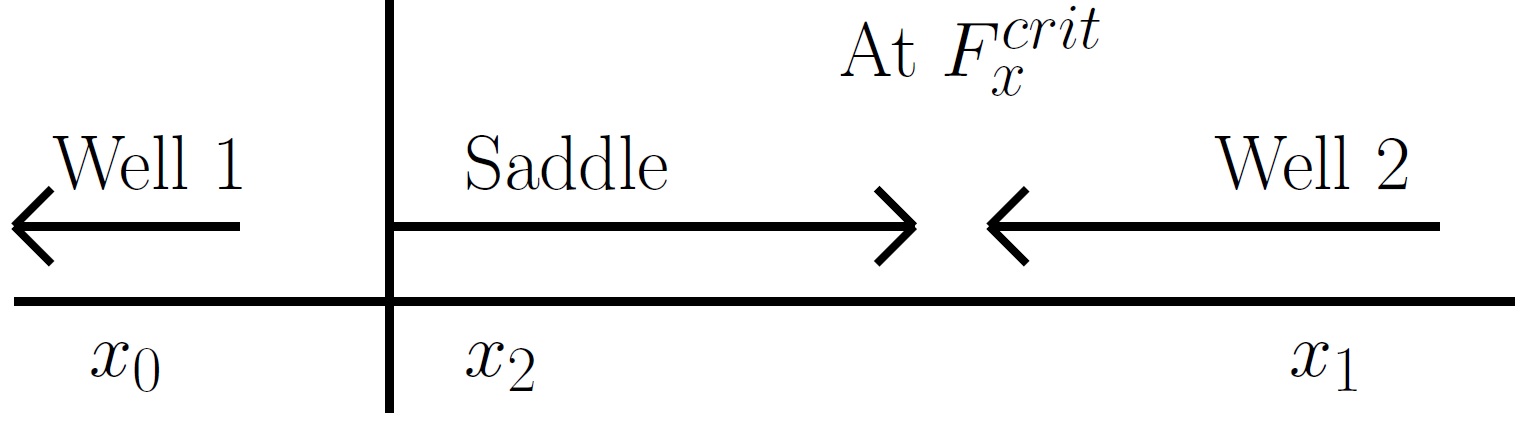}}
\caption{As $F_x$ increases from $0$ to $F^{crit}_x$, the saddle collides into the right well and disappears.}
\end{center}
\end{figure}

\section{Case $F_x=0$ and $F_y>0$}
Similarly when forcing is only in the $y$ direction, the cases for $b<\frac{1}{2}$ and $b\geq\frac{1}{2}$ have to be considered separately. 
The case for $F_y\leq0$ is similar. 

\subsection{Case $F_x=0$, $F_y>0$ and $b<\frac{1}{2}$}
We have some Lemmas and Theorems which are almost analogous to the case for $F_x>0$, $F_y=0$ and $b<\frac{1}{2}$. 
Their proofs are very similar and are omitted. 

\begin{Theorem}\label{chap_5_y_roots}
Let $F_x=0$, $F>0$ and $b<\frac{1}{2}$. 
Let $F_y$ be bounded by 
\begin{align*}
F_y\leq F_y^{sad}
\quad \text{and} \quad 
F_y\leq F_y^{crit}
\end{align*}
where 
\begin{align*}
F_y^{sad}=2(a+b)\sqrt{1+2a}
\quad \text{and} \quad 
F_y^{crit}=\sqrt{\frac{4(1-2b)^3}{27}}\\
\end{align*}
then there are five critical points given by 
\begin{align*}
&(0,y_k) \quad k=1,2,3\\
&(\pm x_{well}, y_{well})
\end{align*}
where 
\begin{align*}
 x_{well}&=\sqrt{(1+2a)-\left(\frac{F_y}{2(a+b)}\right)^2}\\
y_{well}&=\frac{-F_y}{2(a+b)}\\
y_k&=-\frac{2}{\sqrt{3}}\sqrt{1-2b}\,\cos\left(\psi+k\frac{2\pi}{3}\right)\\
\psi&=\frac{1}{3}\tan^{-1}(p)\\
p&=\frac{\sqrt{4(1-2b)^3-27F_y^2}}{F_y\sqrt{27}}\\
k&=0, \quad k=1, \quad k=2.
\end{align*}

\end{Theorem}

\begin{Lemma}\label{chap_5:th:mono:y}
The function $p$ (as in Theorem \ref{chap_5_y_roots}) is monotone in $F_y$ for $F_y<F_y^{crit}$. 
\end{Lemma}

\begin{Lemma}\label{chap_5:th:det:y}
Let $F_x=0$, $F_y>0$ and $b<\frac{1}{2}$. 
These three scenarios hold.

If $F_y<F_y^{crit}$ we have 3 critical points on the $x$-axis: $y_0<y_2<y_1$. 

If $F_y=F_y^{crit}$ we have 2 critical points on the $y$-axis: $y_0<y_2=y_1$. 

If $F_y>F_y^{crit}$ we have 1 critical point on the $y$-axis: $y_0<\frac{-2}{\sqrt{3}}\sqrt{1-2b}$ and $y_0 \rightarrow -\infty $ monotonically with increasing $F_y$. 
\end{Lemma}

\noindent Using the same reasoning as for the $x$-direction case we have bounds on the three $y_k$ for $0\leq F_y\leq F_y^{crit}$. 
\begin{align*}
\frac{-2}{\sqrt{3}}\sqrt{1-2b}&\leq y_0 \leq -\sqrt{1-2b}\\
\frac{1}{\sqrt{3}}\sqrt{1-2b}&\leq y_1 \leq \sqrt{1-2b}\\
0&\leq y_2 \leq \frac{1}{\sqrt{3}}\sqrt{1-2b}.
\end{align*}
The monotonicity of $y_k$ means 
\begin{equation*}
\begin{array}{lrrll}
y_0: & -\sqrt{1-2b}&\longrightarrow& \frac{-2}{\sqrt{3}}\sqrt{1-2b}& \text{as}\quad F_y\rightarrow F_y^{crit}\\
y_1:& \sqrt{1-2b}&\longrightarrow& \frac{1}{\sqrt{3}}\sqrt{1-2b}& \text{as}\quad F_y\rightarrow F_y^{crit}\\
y_2: & 0&\longrightarrow &\frac{1}{\sqrt{3}}\sqrt{1-2b}& \text{as}\quad F_y\rightarrow F_y^{crit}
\end{array}
\end{equation*}
without any oscillations. 
Now consider the special case when $F_y=F_y^{sad}$. This gives 
\begin{align*}
(\pm x_{well}, y_{well})
&=(0,y_{well})\\
&=(0,-\sqrt{1+2a})\\
\Rightarrow
y_{well}&=-\sqrt{1+2a}
\end{align*}
which definitely satisfies $-\sqrt{1+2a}<-\sqrt{1-2b}$ for $0<b<\frac{1}{2}$. 
This seemingly adds a fourth  critical point onto the $y$-axis. 
We have a Lemma whose method of proof is similar to Lemma \ref{chap_5:th:great}.
But now it does not take long to find numerical examples such that $F_y^{sad}<F_x^{crit}$, $F_y^{crit}<F_y^{sad}$ and $F_y^{sad}=F_y^{crit}$. 
These would form the separate sub-cases we would have to consider.

\begin{Lemma}\label{chap_5:th:bound:y}
The following statements hold. 
\begin{enumerate}
\item If  $F_y^{sad}<F_y^{crit}$, then $-\sqrt{1+2a}>\frac{-2}{\sqrt{3}}\sqrt{1-2b}$
\item If  $F_y^{sad}>F_y^{crit}$, then $-\sqrt{1+2a}<\frac{-2}{\sqrt{3}}\sqrt{1-2b}$
\item If $F_y^{sad}=F_y^{crit}$, then $-\sqrt{1+2a}=\frac{-2}{\sqrt{3}}\sqrt{1-2b}$
\end{enumerate}
\end{Lemma}

\begin{proof}
If  $F_y^{sad}<F_y^{crit}$, assume that  $-\sqrt{1+2a}\leq\frac{-2}{\sqrt{3}}\sqrt{1-2b}$. Let $F_y=F_y^{sad}$. 
But this means $F_y<F_y^{crit}$ and by Lemma \ref{chap_5:th:det:y} there must be three critical points on the $y$-axis.  
But monotonicity implies $y_0>\frac{-2}{\sqrt{3}}\sqrt{1-2b}$ for $F_y<F_y^{crit}$
meaning there would be 4 critical points on the $y$-axis, which is a contradiction. 

If  $F_y^{sad}>F_y^{crit}$ assume, that  $-\sqrt{1+2a}\geq\frac{-2}{\sqrt{3}}\sqrt{1-2b}$.  Let $F_y=F_y^{sad}$. 
But this means $F_y>F_y^{crit}$ and Lemma \ref{chap_5:th:det:y} implies that there should only be one critical point on the $y$-axis. 
We know that $y_0=\frac{-2}{\sqrt{3}}\sqrt{1-2b}$ at $F_y=F_y^{crit}$ and yet the monotonicity of $y_0$ means $y_0<\frac{-2}{\sqrt{3}}\sqrt{1-2b}$ for $F_y>F_y^{crit}$.  
This would mean 2 critical points on the $y$-axis which is a contradiction.\footnote{Just like in the $x$-direction case it can be shown that for $F_y>F_y^{crit}$, $y_0=\frac{-2}{\sqrt{3}}\sqrt{1-2b}\ \cosh (z)$ where $z$ is a real number which monotonically decreases with increasing $F_y$, and yet $z=0$ when $F_y=F_y^{crit}$, which justifies the idea of this proof.}

If   $F_y^{sad}=F_y^{crit}$ then  let $F_y=F_y^{sad}=F_y^{crit}$. 
But Lemma \ref{chap_5:th:det:y} says there can only be 2 critical points on the $y$-axis. 
But $y_1=y_2>0$ and $y_0<0$. 
This means $y_{well}$ must collide into $y_0$, hence the statement of the Theorem. 
\end{proof}

\noindent Again we are ready for another main Theorem of this Chapter. 
It is finding the positions and nature of all the critical points for different values of $F_y$. 

\begin{Theorem}\label{chap_5_case_4}
Let $F_x=0$, $F_y>0$ and $b<\frac{1}{2}$. 
The positions and nature of the critical points are as follows 
\begin{align*}
F_y<F_y^{crit}&\quad\left\{
\begin{array}{ll}
(0,y_1)&\text{saddle}\\
(0,y_2)&\text{hill}
\end{array}
\right.\\[0.3em]
F_y<F_y^{sad}&\quad\left\{
\begin{array}{ll}
(0,y_0)&\text{saddle}\\
(\pm x_{well},y_{well})&\text{well}
\end{array}
\right.\\[0.3em]
F_y>F_y^{crit}&\quad\left\{
\begin{array}{ll}
(0,y_1)&\text{nonexistent}\\
(0,y_2)&\text{nonexistent}\\
\end{array}
\right.\\[0.3em]
F_y>F_y^{sad}&\quad\left\{
\begin{array}{ll}
(0,y_0)&\text{well}\\
(\pm x_{well},y_{well})&\text{nonexistent}
\end{array}
\right.\\[0.3em]
F_y=F_y^{crit}&\quad\left\{
\begin{array}{ll}
(0,y_1)&\text{unidentified}\\
(0,y_2)&\text{unidentified}
\end{array}
\right.\\[0.3em]
F_y=F_y^{sad}&\quad\left\{
\begin{array}{ll}
(0,y_0)&\text{unidentified}\\
(\pm x_{well},y_{well})&\text{unidentified}
\end{array}
\right.
\end{align*}
\end{Theorem}

\begin{proof}
\noindent Just like in the $x$-direction case, monotonicity of $p$ is  essential in justifying the following bounds on the critical points for $F_y>0$. 
Note that Lemma \ref{chap_5:th:bound:y} is used to determine the bounds on the $y_k$ critical points
\begin{align*}
F_y\leq F_y^{crit}&\quad\left\{
\begin{array}{rrrrr}
\frac{-2}{\sqrt{3}}\sqrt{1-2b}&\leq & y_0& <&-\sqrt{1-2b}\\
\frac{1}{\sqrt{3}}\sqrt{1-2b}&\leq &y_1& <& \sqrt{1-2b}\\
0&<& y_2 &\leq &\frac{1}{\sqrt{3}}\sqrt{1-2b}
\end{array}
\right.\\[0.3em]
F_y<F_y^{sad}&\quad\left\{
\begin{array}{l}
y_0>-\sqrt{1+2a}
\end{array}
\right.\\[0.3em]
F_y>F_y^{sad}&\quad\left\{
\begin{array}{l}
y_0<-\sqrt{1+2a}
\end{array}
\right.\\[0.3em]
F_y>F_y^{crit}&\quad\left\{
\begin{array}{l}
y_0<\frac{-2}{\sqrt{3}}\sqrt{1-2b}
\end{array}
\right.\\[0.3em]
F_y=F_y^{sad}&\quad\left\{
\begin{array}{l}
y_0=y_{well}=-\sqrt{1+2a}
\end{array}
\right.\\[0.3em]
F_y=F_y^{crit}&\quad\left\{
\begin{array}{l}
y_1=y_2=\frac{1}{\sqrt{3}}\sqrt{1-2b}
\end{array}
\right.\\[0.3em]
F_y=F_y^{sad}=F_y^{crit}&\quad\left\{
\begin{array}{l}
y_0=y_{well}=-\sqrt{1+2a}=\frac{-2}{\sqrt{3}}\sqrt{1-2b}
\end{array}\right.
\end{align*}
Now that the bounds on the critical points are found we can determine their nature. 
Similarly $(\pm x_{well}, y_{well})$ is the one whose nature is easiest to prove. 
After noting that $x_{well}^2+y_{well}^2=(1+2a)$, the determinant of the Hessian at $(\pm x_{well}, y_{well})$ gives 
\begin{equation*}
\det H=4x_{well}^2(a+b)>0
\end{equation*}
and the second partial derivative in $x$ gives 
\begin{equation*}
\frac{\partial^2V_F}{\partial x^2}(\pm x_{well},y_{well})=2x_{well}^2>0
\end{equation*}
which is a well by Theorem \ref{chap_5_thm_1}.  
The Hessian matrix for the three critical points on the $y$-axis $(0,y_k)$ is already diagonal even with forcing
\begin{equation*}
H(0,y_k)=\left(\begin{array}{cc}\frac{\partial^2V_F}{\partial x^2}&0\\ 0&\frac{\partial^2V_F}{\partial y^2}\end{array}\right)=\left(\begin{array}{cc}y_k^2-(1+2a)&0\\ 0&3y_k^2-(1-2b)\end{array}\right).
\end{equation*}
Lemma \ref{chap_5:th:bound:y} has to be used in conjunction with the bounds on $y_0$, $y_1$, $y_3$ and $y_{well}$ (as derived in the proof of this Theorem) to determine the nature of the critical points. 
\end{proof}

\noindent This can be shown graphically. 

\begin{figure}[H]
\begin{center}
\centerline{\includegraphics[scale=1]{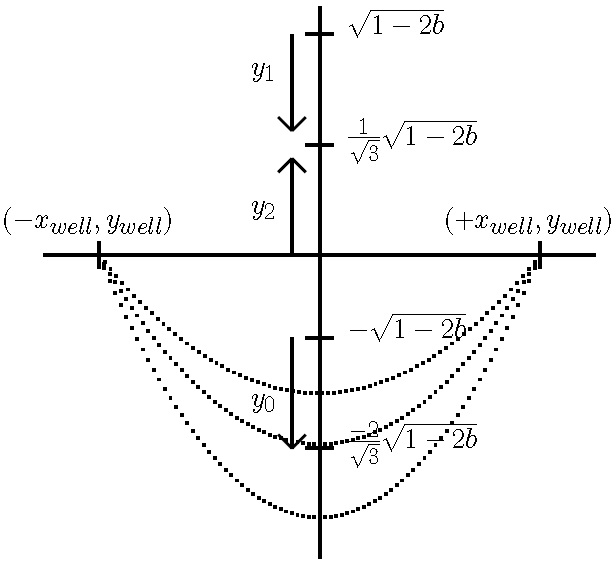}}
\caption{As $F_y$ increases from $0$ to $F^{crit}_y$ the $y_{0}$, $y_{1}$ and $y_{2}$ move as shown in the diagram. 
As $F_y$ increases from $0$ to $F^{sad}_y$, the two $(\pm x_{well}, y_{saddle})$ meet each other on the $y$-axis. There are three possible paths for $(\pm x_{well},  y_{well})$. 
If $F^{sad}_y<F^{crit}_y$, then the two $(\pm x_{well},  y_{well})$ meet between in the interval $\left(-\frac{2}{\sqrt{3}}\sqrt{1-2b},-\sqrt{1-2b}\right)$. 
If $F^{sad}_y=F^{crit}_y$, then the two $(\pm x_{well},  y_{well})$ meet at $y=-\frac{2}{\sqrt{3}}\sqrt{1-2b}$.  
If $F^{sad}_y>F^{crit}_y$ then the two $(\pm x_{well},  y_{well})$ meet in the interval $\left(-\infty,-\frac{2}{\sqrt{3}}\sqrt{1-2b}\right)$.
}
\end{center}
\end{figure}

\begin{figure}[H]
\begin{center}
\centerline{\includegraphics[scale=0.75]{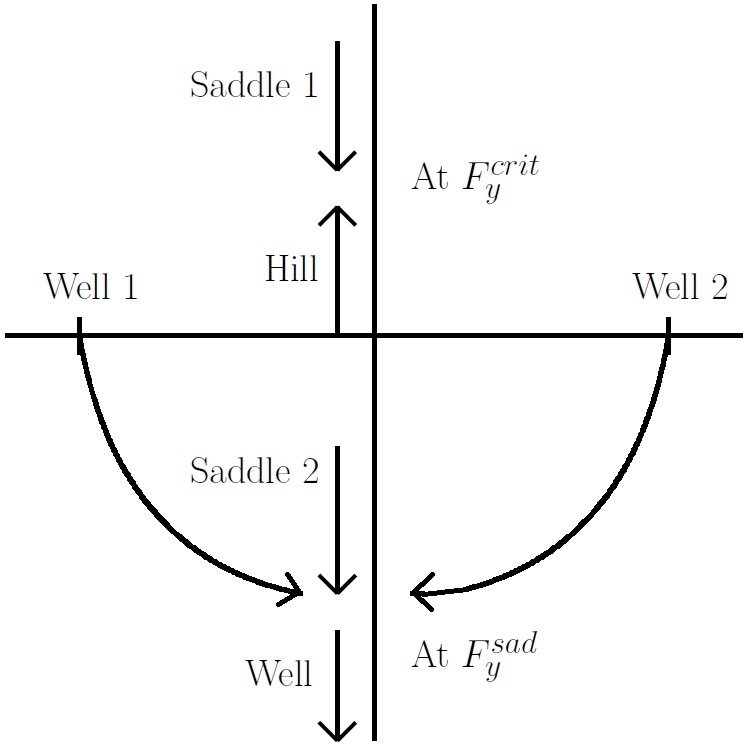}}
\caption{At $F=F^{crit}_y$ the top saddle collides into the hill and both then disappears. At $F=F^{sad}_y$ the bottom saddle collides with the two wells and turns into a new well. These two collisions can occur simultaneously or occur one after the other, depending on whether we have $F^{sad}_y<F^{crit}_y$, $F^{sad}_y=F^{crit}_y$ or $F^{sad}_y>F^{crit}_y$.}
\end{center}
\end{figure}

\subsection{Case $F_x=0$, $F_y>0$ and $b\geq\frac{1}{2}$}
The case for $F_y>0$, $b\geq\frac{1}{2}$ is slightly different in the sense that we have to consider the discriminant of the cubic equation with one real solution for the potential. 
We have the last main Theorem in this Chapter.

\begin{Theorem}\label{chap_5_case_5}
Let $F_x=0$, $F_y>0$ and $b\geq \frac{1}{2}$. 
The positions and nature of the critical points are as follows
\begin{align*}
F_y<F_y^{sad}&\quad\left\{
\begin{array}{ll}
(0,y_0)&\text{saddle}\\
(\pm x_{well}, y_{well})&\text{well}
\end{array}
\right.\\[0.3em]
F_y=F_y^{sad}&\quad\left\{
\begin{array}{ll}
(0,y_0)=(\pm x_{well}, y_{well})&\text{unidenitified}
\end{array}
\right.\\[0.3em]
F_y>F_y^{sad}&\quad\left\{
\begin{array}{ll}
(0,y_0)&\text{well}\\
(\pm x_{well}, y_{well})&\text{nonexistent}
\end{array}
\right.
\end{align*}
\end{Theorem}

\begin{proof}
The simultaneous equations we have to solve are 
\begin{align}
\frac{\partial V_F}{\partial x}&=x(x^2+y^2)-(1+2a)x=0 \label{chap_5_sim_eqn_y1}\\
\frac{\partial V_F}{\partial y}&=y(x^2+y^2)-(1-2b)y+F_y=0. \label{chap_5_sim_eqn_y2}
\end{align}
Equation \ref{chap_5_sim_eqn_y1} holds if either $x=0$ or $(x^2+y^2)-(1+2a)=0$.
The case for $(x^2+y^2)-(1+2a)=0$ gives $(\pm x_{well}, y_{well})$ as a critical point in similar way as before. 
The case for $x=0$ reduces Equation \ref{chap_5_sim_eqn_y2} to 

\begin{equation*}
y^3-(1-2b)y+F_y=0.
\end{equation*}
It is this resulting cubic equation which forms the next series of discussions. 
The required expressions in solving this cubic equation are
\begin{equation*}
\Delta_0=3(1-2b),\quad\Delta_1=27F_y,\quad\Delta=4(1-2b)^3-27F_y^2. 
\end{equation*}
Since $b\geq\frac{1}{2}$ the discriminant of this cubic equation is strictly negative meaning $\Delta<0$, which means there can only be one real solution. 
All the solutions whether complex or real are given by 
\begin{align*}
y_k&=-\frac{1}{3}\left(e^{i\psi_k}C+e^{-i\psi_k}\frac{\Delta_0}{C}\right)\quad \text{where}\\
\psi_0&=0, \quad \psi_1=2\pi/3, \quad \psi_2=4\pi/3\\
C&=\sqrt[3]{\frac{\Delta_1+\sqrt{-27\Delta}}{2}}.
\end{align*}
Notice that the  $\Delta<0$, $C$ and $\Delta_0$ are all real numbers. 
Since there can only be one real $y_k$ solution this has to be $y_0$, as $y_0$ would just be a sum of real numbers. 
This then means $y_1$ and $y_2$ would be complex conjugate solutions. 
Now written explicitly we have 
\begin{align*}
C=\sqrt[3]{\frac{27F_y+\sqrt{(27F_y^2-4(1-2b)^3)\times27}}{2}}
\end{align*}
which for $b\geq\frac{1}{2}$ is clearly monotonically increasing in $F_y$ for $F_y>0$. 
This is because the $\sqrt{\cdot}$ and $\sqrt[3]{\cdot}$ are both monotone with respect to their own argument.  
This means we can say 
\begin{align*}
\frac{dC}{dF_y}&\geq0\\
\frac{dy_0}{dF_y}&=-\frac{1}{3}\left(1-\frac{\Delta_0}{C^2}\right)\frac{dC}{dF_y}\\
&\leq0
\end{align*}
because $(-\Delta_0)>0$ for $b\geq\frac{1}{2}$. 
This means $y_0$ would always monotonically decrease  with increasing $F_y$. 
We also know that at $F_y=F_y^{sad}$ the two wells on the sides become 
\begin{equation*}
(\pm x_{well}, y_{well})=(0,-\sqrt{1+2a}).
\end{equation*}
Note that earlier when $x=0$ was imposed in Equation \ref{chap_5_sim_eqn_y1} and \ref{chap_5_sim_eqn_y2} we were reduced with a cubic equation that only admits one real solution in $y$. 
This means there can only be one critical point on the $y$-axis so we must have 
\begin{equation*}
y_0=-\sqrt{1+2a}
\end{equation*}
at $F_y=F_y^{sad}$. 
This justifies the following bounds 
\begin{align*}
F_y<F_y^{sad}&\quad\left\{
\begin{array}{l}
y_0>-\sqrt{1+2a}
\end{array}
\right.\\[0.3em]
F_y=F_y^{sad}&\quad\left\{
\begin{array}{l}
y_0=-\sqrt{1+2a}
\end{array}
\right.\\[0.3em]
F_y>F_y^{sad}&\quad\left\{
\begin{array}{l}
y_0<-\sqrt{1+2a}
\end{array}
\right.
\end{align*}
which can be justified by the monotonicity of $y_0$. 
This together with the Hessian can allow us to identify the nature of the critical points.
\end{proof}

\noindent Again the situation can be represented graphically. 

\begin{figure}[H]
\begin{center}
\centerline{\includegraphics[scale=0.5]{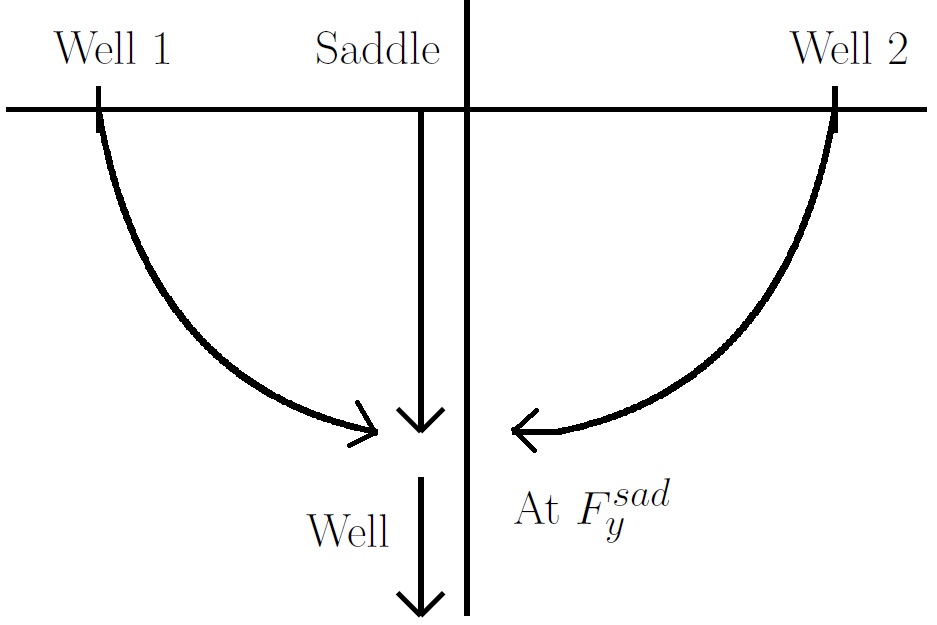}}
\caption{At $F_y=F^{sad}$ the saddle collides with the two wells and turns into a new well.}
\end{center}
\end{figure}

\section{Case $F_x\neq0$ and $F_y\neq0$}

Consider the forcing 
\begin{equation*}
\mathbf{F}=
\left(
\begin{array}{c}
F_x \\
F_y
\end{array}
\right)
=
\left(
\begin{array}{c}
F\cos \phi\\
F\sin \phi
\end{array}
\right)
\quad 
\text{where}
\quad 
F=\sqrt{F_x^2+F_y^2}
\end{equation*}
so far we have only studied the case when $\phi=0^\circ, 90^\circ, 180^\circ, 360^\circ$. 
Now we consider the case for forcing in a general direction. 
The critical points are given by solutions to 
\begin{align}
\frac{\partial V_F}{\partial x}=x(x^2+y^2)-(1+2a)x+F_x=0 \label{chap_5:sum:eqn:1}\\
\frac{\partial V_F}{\partial y}=y(x^2+y^2)-(1-2b)y+F_y=0.
\label{chap_5:sum:eqn:2}
\end{align}
The arguments presented next can actually apply for any $a$, $b$, $F_x$ and $F_y$ regardless of whether they are positive or negative.
Notice that if $\phi\notin\{0^\circ, 90^\circ, 180^\circ, 360^\circ\}$ then $F_x\neq0$ and $F_y\neq0$. This means $x=0$ or $y=0$ must not appear in any solution. This is the assumption  used. 
Solving Equation \ref{chap_5:sum:eqn:2} for $x^2$ gives 
\begin{align}
y(x^2+y^2)-(1-2b)y+F_y&=0\nonumber\\
(x^2+y^2)-(1-2b)+\frac{F_y}{y}&=0\nonumber\\
(x^2+y^2)&=(1-2b)-\frac{F_y}{y}\nonumber\\
x^2&=(1-2b)-\frac{F_y}{y}-y^2 \label{chap_5:sum:eqn:3}
\end{align}
and substituting Equation \ref{chap_5:sum:eqn:3} into Equation \ref{chap_5:sum:eqn:1} gives 
\begin{align*}
\sqrt{(1-2b)-\frac{F_y}{y}-y^2}
\left((1-2b)-\frac{F_y}{y}\right)
-(1+2a)
\sqrt{(1-2b)-\frac{F_y}{y}-y^2}
+F_x
&=0\\[0.5em]
\sqrt{(1-2b)-\frac{F_y}{y}-y^2}
\left((1-2b)-\frac{F_y}{y}-(1+2a)\right)
+F_x&=0\\[0.5em]
\left[(1-2b)-\frac{F_y}{y}-y^2\right]
\left[-2b-\frac{F_y}{y}-2a\right]^2
-F_x^2&=0\\[0.5em]
\left[(1-2b)-\frac{F_y}{y}-y^2\right]
\left[2(a+b)+\frac{F_y}{y}\right]^2
-F_x^2&=0\\[0.5em]
\left[y(1-2b)-F_y-y^3\right]
\left[2(a+b)+\frac{F_y}{y}\right]^2
-F_x^2y&=0\\[0.5em]
\left[y(1-2b)-F_y-y^3\right]
\left[4(a+b)^2+4(a+b)\frac{F_y}{y}+\frac{F_y^2}{y^2}\right]
-F_x^2y&=0\\[0.5em]
\left[y(1-2b)-F_y-y^3\right]
\left[4(a+b)^2y^2+4y(a+b)F_y+F_y^2\right]
-F_x^2y^3&=0
\end{align*}
which rearranges into a fifth degree polynomial in $y$
\begin{align}
&\relphantom{+}y^5[-4(a+b)^2]\nonumber\\[0.5em]
&+y^4[-4(a+b)F_y]\nonumber\\[0.5em]
&+y^3[(1-2b)4(a+b)^2-F_y^2-F_x^2]\nonumber\\[0.5em]
&+y^2[(1-2b)4(a+b)F_y-F_y4(a+b)^2]\nonumber\\[0.5em]
&+y^1[(1-2b)F_y^2-F_y4(a+b)F_y]\nonumber\\[0.5em]
&+y^0[-F_y^3]\nonumber\\[0.5em]
&=0\label{chap_5:sum:eqn:4}
\end{align}
which can only be solved numerically.
The five solutions for Equation \ref{chap_5:sum:eqn:4} are denoted by
\begin{equation*}
y_i \quad \text{where} \quad i=1,2,3,4,5.
\end{equation*}
Now consider Equation \ref{chap_5:sum:eqn:1}
\begin{align*}
x(x^2+y^2)-(1+2a)x+F_x&=0\\
x&=\frac{-F_x}{x^2+y^2-(1+2a)}\\
&=\frac{-F_x}{(1-2b)-\frac{F_y}{y}-(1+2a)}
\end{align*}
where we have used Equation \ref{chap_5:sum:eqn:3} for the last line. This means the $x$ part of the final set of solutions is 
\begin{equation*}
x_i=\frac{-F_x}{(1-2b)-\frac{F_y}{y_i}-(1+2a)}
\quad \text{where} \quad 
i=1,2,3,4,5.
\end{equation*}
Notice how the calculation is not straightforward if one  feeds $y_i$ into Equation \ref{chap_5:sum:eqn:3} by means of 
\begin{equation*}
x_i\neq\sqrt{(1-2b)-\frac{F_y}{y_i}-y_i^2}.
\end{equation*}
We do not know whether to take the positive or negative solution.
Any easy counting shows that one obtains too many solutions. 
These five solutions $(x_i, y_i)$ are sorted into wells, hills and saddles. 
Although  an explicit value for the critical forcing cannot be given analytically, an educated guess can be made 
\begin{align}
F^{crit}=\min\left\{F_x^{sad}, F_x^{crit}, F_y^{sad}, F_y^{crit}\right\}
\label{chap_5_critical_forcing}
\end{align}
that is because a critical force in a general direction must encompass all the other directions. 
Here are some real examples of the critical points, as the force is being changed during half a period of an oscillatory potential. 

\begin{figure}[H]
\centerline{\includegraphics[scale=0.35]{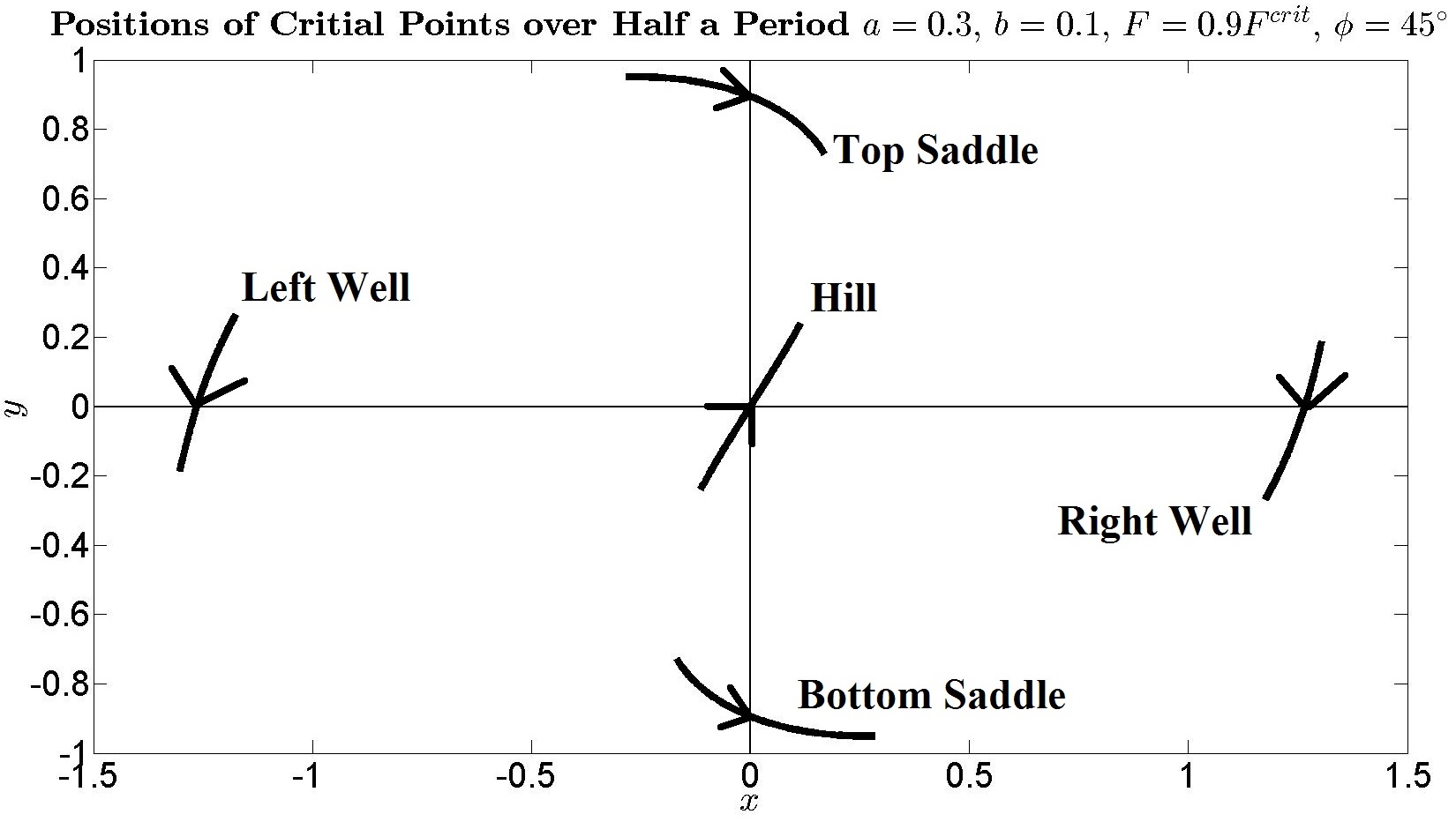}}
\caption{The critical points move very little here. }
\end{figure}

\begin{figure}[H]
\centerline{\includegraphics[scale=0.35]{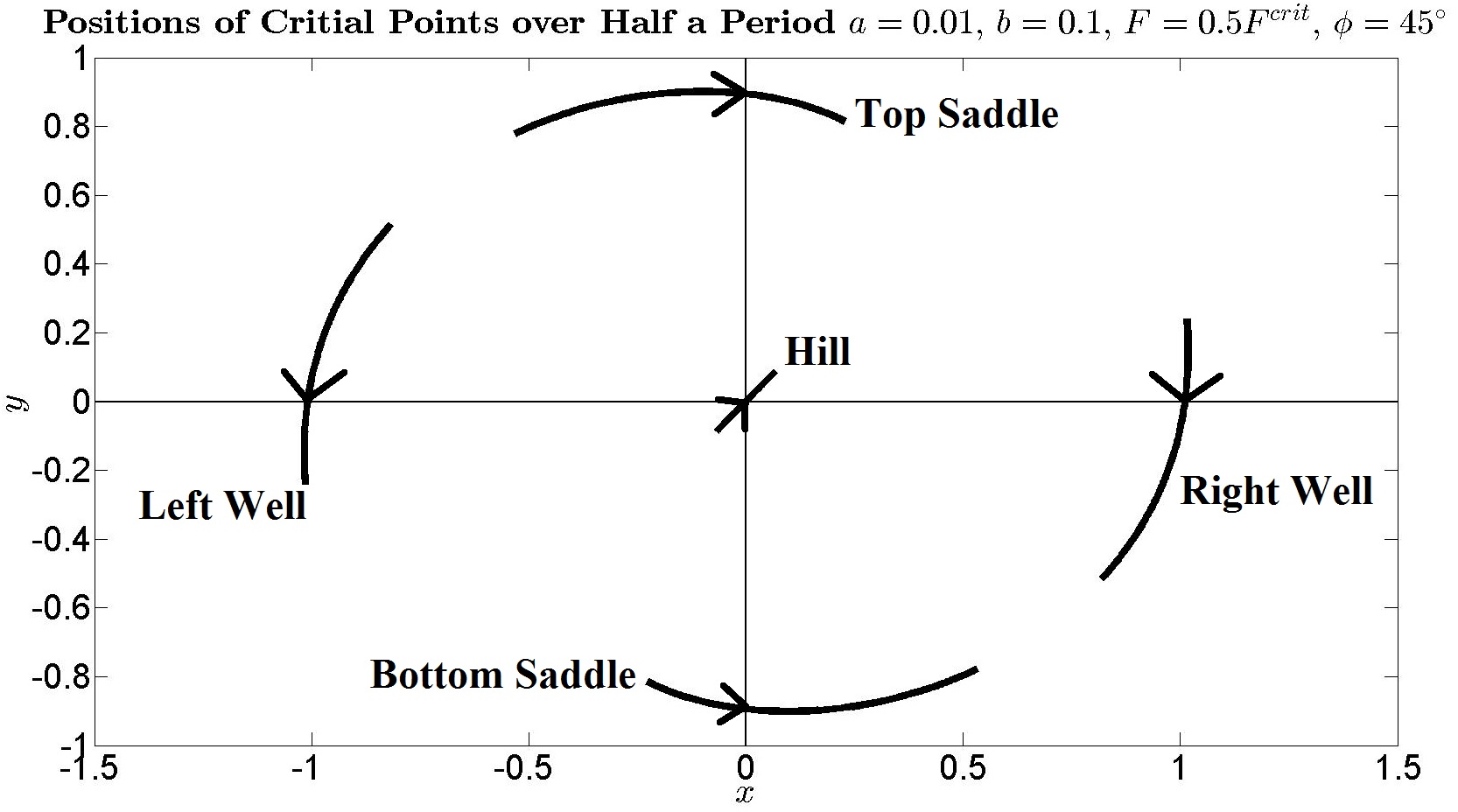}}
\caption{The critical points have a more extreme trajectory here.}
\end{figure}

\begin{figure}[H]
\centerline{\includegraphics[scale=0.35]{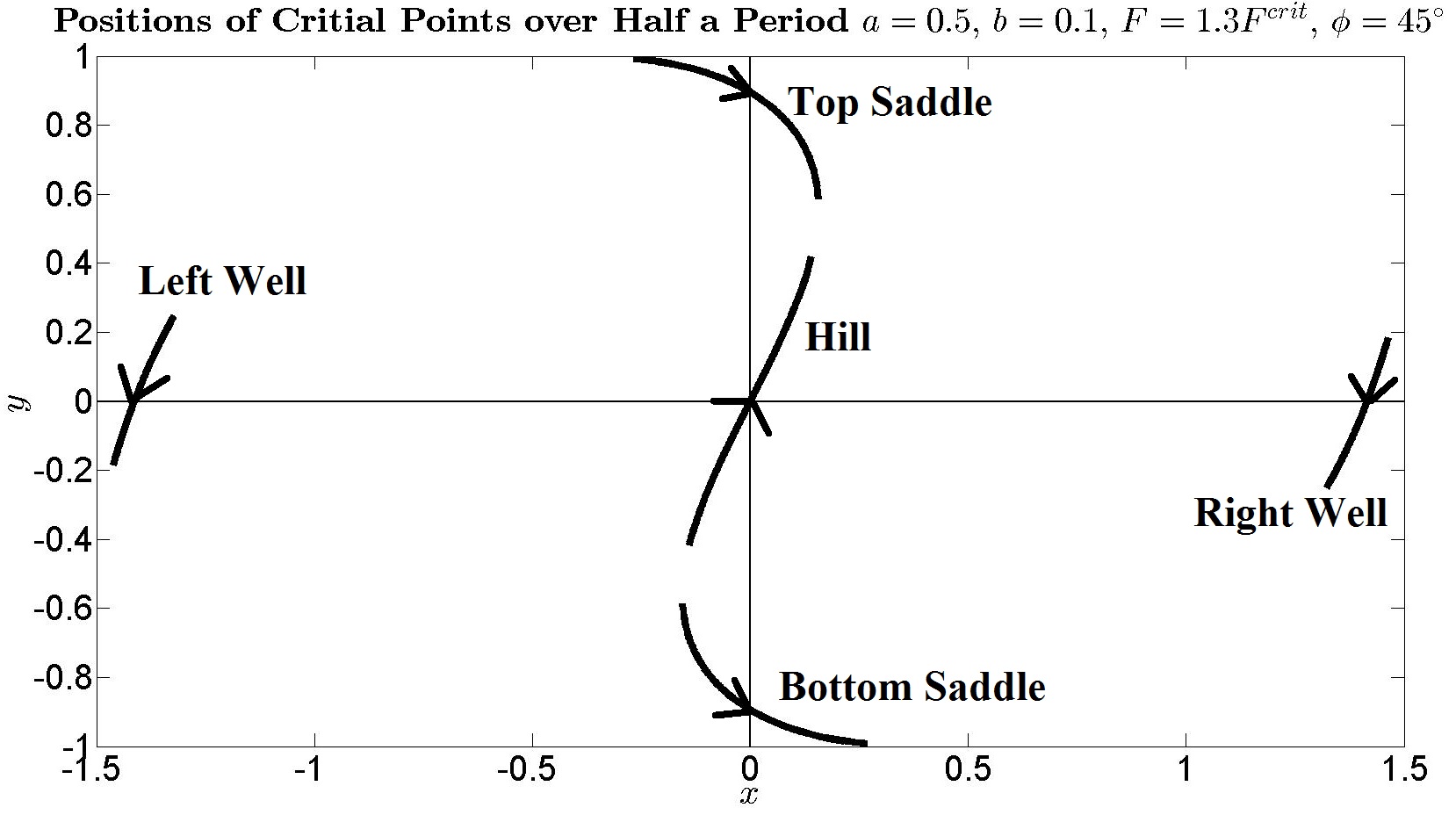}}
\caption{Notice that the use of $F^{crit}$ as a critical force is just an educated guess. Here the system is so close to criticality the saddle is almost colliding with the hill.}
\end{figure}

\section{Remarks on Mexican Hat}

\noindent We give an example of how the Mexican Hat Toy Model look like. 
\begin{figure}[H]
\centerline{\includegraphics[scale=0.36]{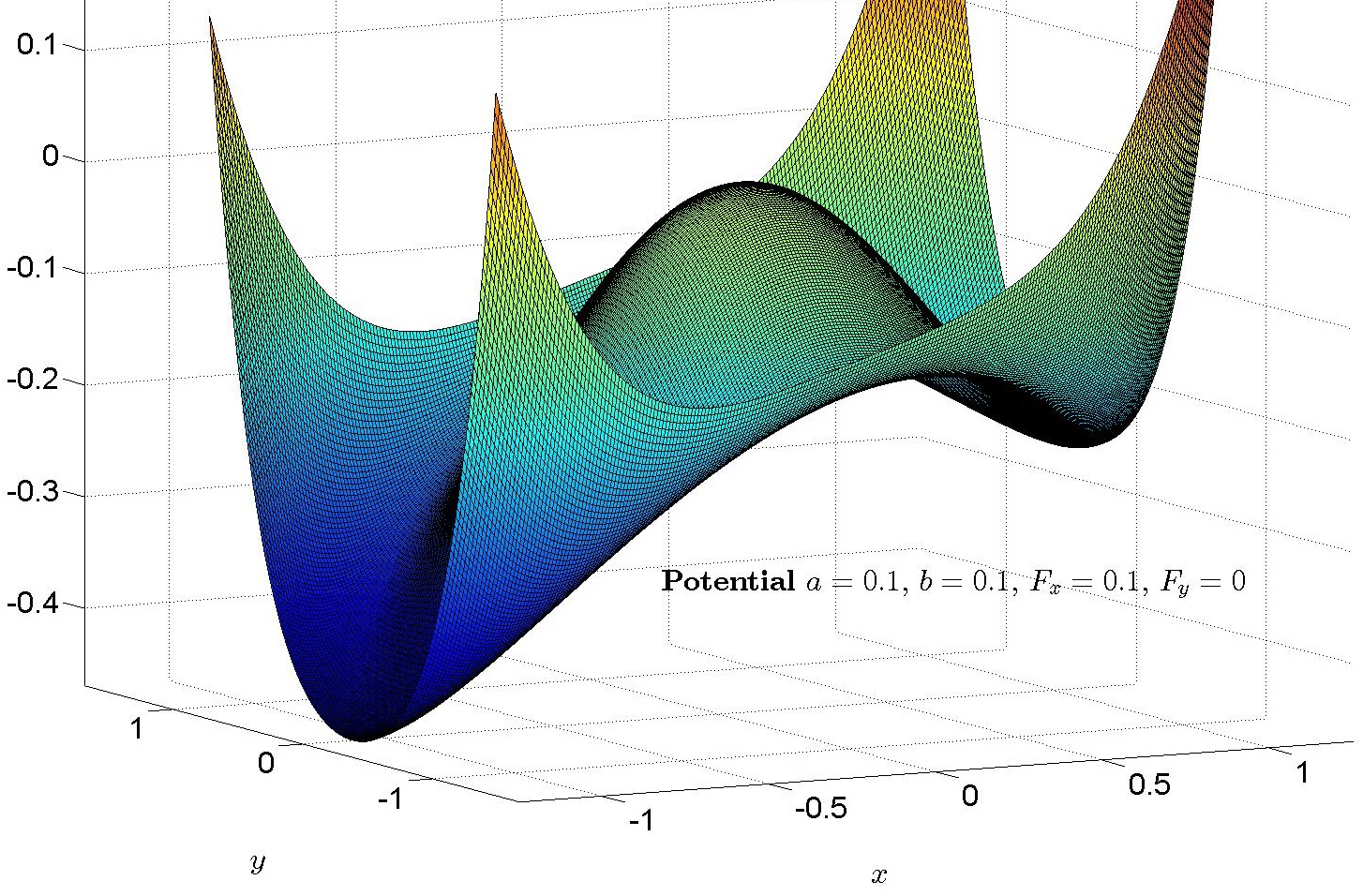}}
\caption{
An example of the potential $V_F(x,y)=\frac{1}{4}r^4-\frac{1}{2}r^2-ax^2+by^2+F_xx+F_yy$ where $r=\sqrt{x^2+y^2}$. 
Here $a=0.1$, $b=0.1$, $F_x=0.1$ and $F_y=0$. 
Notice there are two saddles just ahead of the hill.
The well on the right is higher than the well on the left.
}
\end{figure}

\subsection{Beyond Criticality}
At first glance since the critical points can all be found numerically, one may ask why studied the cubic formula. 
This actually provided exact analytic information about the system near and beyond criticality and in the extremal cases as well.  
Besides, stochastic resonance is studied when the forcing is small enough such that the topology of the potential does not change significantly.
This is because if the forcing is too large (beyond criticality) then transitions are almost certain, and there is little point to consider stochastic resonance in this case. 

\subsection{Numerical Problems}
\label{stable_critical_points}
When the critical points are numerically found, they were fed back into Equations \ref{chap_5:sum:eqn:1} and \ref{chap_5:sum:eqn:2} and were correct to $10^{-9}$. 
But a few problems remain. 
In simulations when the angle of the forcing 
\begin{align*}
\phi=\tan^{-1}\left(\frac{F_y}{F_x}\right)
\end{align*}
is changed from $\phi=0^\circ$ to $\phi=90^\circ$ the potential was continuously changing from a system needing to solve third order roots to a system needing to solve fifth order roots. 
This means the numerical algorithms for solving the quintic polynomial (Equation \ref{chap_5:sum:eqn:4}) became very unstable when the system is close 
to solving a cubic equation.\footnote{Algorithms used in the roots($\cdot$) function in MatLab.}
No further numerical investigation is necessary as analytic results are available to interpolate the correct solution. 

\subsection{Comparison with One Dimensional Case}
The one dimensional potential is 
\begin{align*}
V_F=\frac{x^4}{4}-a\frac{x^2}{2}+Fx
\end{align*}
and their critical points are given by solutions to the equation
\begin{align*}
\frac{\partial V_F}{\partial x}=x^3-ax+F=0
\end{align*}
which when compared to the solutions in the Mexican Hat yields the critical points as 
\begin{align*}
x_k=-\frac{2}{3}\sqrt{a}\cos
\left\{
\frac{1}{3}
\tan^{-1}
\left(
\frac{\sqrt{4a^3-27F^2}}{F\sqrt{27}}+\frac{2\pi}{3}k
\right)
\right\}.
\end{align*}
Similarly if we want three critical points, the $\tan^{-1}(\cdot)$ must take real arguments, which means 
\begin{align*}
F<F^{crit}=\sqrt{\frac{4a^3}{27}}
\end{align*}
and the nature of the critical points are
\begin{align*}
F< F^{crit}&\left\{
\begin{array}{lll}
x_{0}&\text{well}\\
x_1 &\text{well}\\
x_2&\text{hill}
\end{array}
\right.\\[0.7em]
F> F^{crit}&\left\{
\begin{array}{lll}
x_{0}&\text{well}\\
x_1 &\text{nonexistent}\\
x_2&\text{nonexistent}
\end{array}
\right.\\[0.7em]
F= F^{crit}&\left\{
\begin{array}{lll}
x_{0}&\text{well}\\
x_1 &\text{unidentified}\\
x_2&\text{unidentified}
\end{array}
\right.
\end{align*}
which is similar to the $F_x>0$, $F_y=0$ and $b\geq \frac{1}{2}$ case. These critical points are bounded by 
\begin{align*}
x_0<x_2<x_1.
\end{align*}
This was a calculation  not studied in the paper by Benzi et al \cite{benzi81} and other literature. 
In the paper \cite{benzi81} it was assumed that the forcing is so small the hill is very near to $x_2=0$. 
In \cite{benzi81}, the escape times were defined by Equations \ref{chap_1_benzi_escape_time_1} and \ref{chap_1_benzi_escape_time_2}, where having reached the hill (which is assumed to be at $x_2=0$) is sufficient for escape. 
Our calculations show that the hill actually moves as the potential oscillate.
Taking this into account may give a better approximation of the exit times than in \cite{benzi81}.

\chapter{Numerical Methods}
\label{sum_chap_numerical_methods}

Let us remind ourselves of the unperturbed potential of the Mexican Hat Toy Model. 
\begin{align*}
V_0(x,y)=\frac{1}{4}r^4-\frac{1}{2}r^2-ax^2+by^2
\quad \text{where} \quad r=\sqrt{x^2+y^2}.
\end{align*}
The SDE we want to study and simulate is 
\begin{align*}
\dot{X}^\epsilon_t=-\nabla V_0 + F\cos\Omega t + \epsilon\, \dot{W}_t
\end{align*}
where $W_t$ is a two dimensional Wiener process. 
When this SDE is expressed for the separate $x$ and $y$ components we have 
\begin{align*}
dx&=\left[-\frac{\partial V_0}{\partial x}+F_x\cos \Omega t \ \right]dt+\epsilon \ dw_x\\
dy&=\left[-\frac{\partial V_0}{\partial y}+F_y\cos \Omega t \ \right]dt+\epsilon \ dw_y
\end{align*}
where  $\epsilon$ is the noise level and $w_x$ and $w_y$ are two independent Wiener processes. 
When this SDE is numerically approximated by the Euler method we have
\begin{align*}
t_n&=t_{n-1}+t_{step}\\
x_n&=x_{n-1}+\left[-\frac{\partial V_0}{\partial x}(x_{n-1},y_{n-1})+F_x \cos (\Omega t_{n-1})\right]t_{step}+\epsilon \sqrt{t_{step}}\ \xi_x\\
y_n&=y_{n-1}+\left[-\frac{\partial V_0}{\partial y}(x_{n-1},y_{n-1})+F_y \cos (\Omega t_{n-1})\right]t_{step}+\epsilon \sqrt{t_{step}}\ \xi_y
\end{align*}
for the iterative scheme, where $\xi_x$ and $\xi_y$ are two independent normal random variables. 
The level of precision for the numerics are estimated assuming the Euler method is being used in simulations. 
More on numerical solutions to SDEs can be found in \cite{sde_mumeric_platen2010numerical}. 

When the escape times are measured for an oscillatory potential, a fixed radius $R$ is defined around the well which moves with the well. The two parameters we need to consider are 
\begin{align*}
t_{step} 
\quad \text{and} \quad 
R
\end{align*}
and derive appropriate values for them.
The particle is defined as having entered a well when it enters the area covered by the radius $R$ around the well. 
The time difference between entering the first well and the second is measured as the escape time from the first well. 
The idea behind the considerations is to identify possible sources of error for $t_{step}$ and $R$ and eliminate them. 
The consequence is that $t_{step}$ is bounded by six bounds on the time step
\begin{align*}
t_{step}
\leq 
\min
\left\{
t_1, t_2, t_3, t_4,t_5, t_6
\right\}
\end{align*}
where the time step $t_{step}$ has to meet six different conditions. 
Similarly 
the radius is bounded in the following way
\begin{align*}
R
\leq 
\min
\left\{
R_1, R_2
\right\}.
\end{align*}
Although these theories are not rigorous, it gives a fair idea of the level of precision that is needed. 
It is assumed in the following that the same precision needed for measuring escape times is also precise enough for studying the six measures of resonance. 
Although the Mexican Hat is a two dimensional system, part of the numerical theory is derived on a general number of dimensions $\mathbb{R}^r$.

\section{Basic Conditions - Estimating $t_{step}\leq t_1$, $t_{step}\leq t_2$ and $t_{step}\leq t_3$}
When the potential oscillates, one period $T$ is achieved when
\begin{align*}
\Omega T =2\pi
\quad 
\Rightarrow
\quad 
T=\frac{2\pi}{\Omega}
\end{align*}
and the time step $t_{step}$ has to be precise enough such that the potential is well presented.  
Thus it is reasonable to take as the first bound 
\begin{align*}
t_1=\frac{2\pi}{\Omega N_1}
\end{align*}
where $N_1$ is an appropriately large number say $N_1=1000$.
Denote by $t_{end}$ for the end time. 
This is the time the simulation is being run for. 
At least 1000 transitions need to be detected, which makes the following a reasonable choice
\begin{align*}
t_{end}=1000\times\left(\max_{\substack{w_j(t)\\0\leq t \leq T}}\tau+\min_{\substack{w_j(t)\\0\leq t \leq T}}\tau\right)
\end{align*}
where $\tau$ is the predicted escape time as given by Kramers' formula, $w_j(t)$ where $j=1,2,\ldots$ are the positions of the wells at time $t$. 
Thus the $t_{end}$ is 1000 times the minimum and maximum predicted escape times over all wells over one period. 
Trivially, the time step has to be smaller than the shortest predicted escape time so the second bound is 
\begin{align*}
t_2=\frac{1}{N_2}\min_{\substack{w_j(t)\\0\leq t \leq T}}\tau
\end{align*}
where  we use  $N_2=1000$. 
Most of the time, the particle is near the bottom of the well, then one can approximate the iteration scheme to 
\begin{align*}
t_n&=t_{n-1}+t_{step}\\
x_n&=x_{n-1}+\epsilon \sqrt{t_{step}}\ \xi_x\\
y_n&=y_{n-1}+\epsilon \sqrt{t_{step}}\ \xi_y
\end{align*}
which allows the distance travelled by the particle in one increment, in the time of one time step $t_{step}$ to be given as 
\begin{align*}
\Delta z=
&\sqrt{
\left(x_n-x_{n-1}\right)^2
+
\left(y_n-y_{n-1}\right)^2
}.
\end{align*}
Let the critical points be given by 
\begin{align*}
c_1(t), \, c_2(t), \, \ldots 
\end{align*}
If the particle starts at the well and ever reaches  a hill or saddle, then transition to the other well is almost certain. 
Thus travelling from the well to another critical point  should be almost impossible in a single time step $t_{step}$, that is one increment $\Delta z$. 
The length of this forbidden jump is 
\begin{align*}
l_1=\min_{\substack{w_i(t)\neq c_j(t) \\ 0\leq t \leq T}}\left|w_i(t)-c_j(t)\right|
\end{align*}
which is the minimal distance from the wells to any other critical points (which are not wells) over one period. 
But the normal random variables $\xi_i$ are normal distributed in $N(0,1)$. 
This means 
\begin{equation*}
P(\xi_x>r)=\int_r^\infty \frac{e^{-x^2/2}}{\sqrt{2\pi}}\ dx
\end{equation*} 
and the joint distribution is given by 
\begin{align*}
P
\left(
\xi_x>r_1, \,
\xi_y>r_2
\right)
=
\int_{x=r_1}^{\infty}
\int_{y=r_2}^{\infty}
n(x,y)\,
dx \, dy
\end{align*}
where 
\begin{align*}
n(x,y)=\frac{1}{2\pi}\exp\left\{-\frac{1}{2}(x^2+y^2)\right\}. 
\end{align*}
This gives 
\begin{align*}
P(\Delta z >l_1)
&=P\left(\sqrt{\xi_x^2+\xi_y^2}>\frac{l_1}{\epsilon \sqrt{t_{step}}}\right)\\
&=\iint_{\left\{(x,y):\sqrt{x^2+y^2}>\frac{l_1}{\epsilon \sqrt{t_{step}}}\right\}}n(x,y)\ dx \ dy\\
&=\int_{\theta=0}^{\theta=2\pi}\int_{\frac{l_1}{\epsilon \sqrt{t_{step}}}}^\infty \frac{1}{2\pi}e^{-r^2/2}\ r\ dr \ d\theta\\
&=\exp\left\{-\frac{1}{2}\left(\frac{l_1}{\epsilon \sqrt{t_{step}}}\right)^2\right\}
\end{align*}
and the number of increments achieving such a direct jump must be almost zero in one session of the simulation. So
\begin{align*}
\frac{t_{end}}{t_{step}}\,
P(\Delta z > l_1)
&<N_3
\end{align*}
where $N_3$ needs to be smaller than one for example $N_3=0.1$. 
Rearranging the expression to 
\begin{align*}
\frac{t_{end}}{t_{step}}\,
\exp\left\{-\frac{1}{2}\left(\frac{l_1}{\epsilon \sqrt{t_{step}}}\right)^2\right\}
&=N_3
\end{align*}
and solving for the time step $t_{step}$ we get 
\begin{align*}
t_{step}=J(\epsilon,l_1,N_3,t_{end})
\end{align*}
where $J$ is a function which numerically inverts the expressions to give the time step required. This gives the third bound as 
\begin{align*}
t_3=J(\epsilon, l_1, N_3, t_{end}).
\end{align*}

\section{Increment Conditions}

The bound $t_3$ on the time step hinges on finding bounds on the length of a single increment
\begin{align*}
\Delta z=
&\sqrt{
\left(x_n-x_{n-1}\right)^2
+
\left(y_n-y_{n-1}\right)^2
}
\end{align*}
that is the distance travelled in the time of one time step $t_{step}$. 
This  is now considered again but in a more general setting in $\mathbb{R}^r$ a general number of dimensions. 

\subsection{Increment Theory - Developing $W(\mathcal{S},l)$}
Let $V_0:\mathbb{R}^r\longrightarrow \mathbb{R}$ be a real function from $\mathbb{R}^r$ to $\mathbb{R}$. 
Its gradient with a periodic forcing and noise gives rise to the SDE 
\begin{equation*}
\dot{X}_t=-\nabla V_0+F\cos(\Omega t)+\epsilon \dot{w}_t
\end{equation*}
where $X_t$ is a trajectory in $\mathbb{R}^r$, $F$ is the force in $\mathbb{R}^r$ and $w_t$ is a vector of $r$ independent  Wiener processes. 
Thus 
\begin{align*}
X_t&=(x_1(t), x_2(t),  \ldots, x_r(t))\\
F&=(F_1, F_2, \ldots, F_r)\\
w_t&=(w_1(t), w_2(t),  \ldots, w_r(t)).
\end{align*}
This trajectory can be numerically approximated with the Euler scheme
\begin{align}
t_{n+1}&=t_{n}+t_{step}\nonumber\\
x^{n+1}_i&=x^n_i
+\left[-\frac{\partial V_0}{\partial x_i}
(x^n_1, x^n_2,  \ldots,  x^n_r)
+F_i\cos(\Omega t_{n})\right]t_{step}
+\epsilon\sqrt{t_{step}}\ \xi_i\label{chap_6:eqn:7}
\end{align}
where the partial derivative is evaluated at the previous iteration step $(x^n_1, x^n_2,  \ldots,  x^n_r)$ and $\xi_i$ is a normal random variable. 
Rearranging Equation \ref{chap_6:eqn:7} gives 
\begin{align}
\left(x^{n+1}_i-x^n_i\right)
&\leq
\left[
-\frac{\partial V_0}{\partial x_i}
(x^n_1, x^n_2,  \ldots,  x^n_r)
+F_i\cos(\Omega t_{n})
\right]
t_{step}
+\epsilon\sqrt{t_{step}}\ \xi_i.\label{chap_6:eqn:10}
\end{align}
Notice that we can make the following bound 
\begin{align}
\left[
-\frac{\partial V_0}{\partial x_i}
(x^n_1, x^n_2,  \ldots,  x^n_r)
+F_i\cos(\Omega t_{n})
\right]t_{step}
&\leq
\left|
-\frac{\partial V_0}{\partial x_i}
(x^n_1, x^n_2,  \ldots,  x^n_r)
+F_i\cos(\Omega t_{n})
\right|t_{step}\nonumber\\[0.5em]
&\leq
\left[\,
\left|-\frac{\partial V_0}{\partial x_i}
(x^n_1, x^n_2,  \ldots,  x^n_r)\right|
+\left|F_i\cos(\Omega t_{n})\right|
\,\right]t_{step}\nonumber\\[0.5em]
&\leq
\left[\,
\left|-\frac{\partial V_0}{\partial x_i}
(x^n_1, x^n_2,  \ldots,  x^n_r)\right|
+\left|F_i\right|
\,\right]t_{step}\nonumber\\[0.5em]
&\leq
\left[
\max_{\mathcal{S}}\left|\frac{\partial V_0}{\partial x_i}\right|+|F_i|
\right]t_{step} \label{chap_6:eqn:8}
\end{align}
where  $\mathcal{S}$ is a set that is large enough such that
\begin{align*}
(x_1^n, x_2^n, \ldots, x_r^n) \in \mathcal{S}\subset \mathbb{R}^n. 
\end{align*}
We define
\begin{align}
\Delta x_i &= \left[
\max_{ \mathcal{S}}\left|\frac{\partial V_0}{\partial x_i}\right|+|F_i|
\right]t_{step}+\epsilon\sqrt{t_{step}}\ \xi_i
=A_i+B\xi_i.\label{chap_6:eqn:11}
\end{align}
By using the triangle inequality we can bound Equation 
\ref{chap_6:eqn:10} and \ref{chap_6:eqn:11} in the following way
\begin{align*}
\left|x^{n+1}_i-x^n_i\right|
&\leq
\left|-\frac{\partial V_0}{\partial x_i}
(x^n_1, x^n_2,  \ldots,  x^n_r)
+F_i\cos(\Omega t_{n})\right|t_{step}
+\left|\epsilon\sqrt{t_{step}}\ \xi_i\right|\\
\left|\Delta x_i\right|
&\leq \left[
\max_{ \mathcal{S}}\left|\frac{\partial V_0}{\partial x_i}\right|+|F_i|
\right]t_{step}+\left|\epsilon\sqrt{t_{step}}\ \xi_i\right|
\end{align*}
and by using \ref{chap_6:eqn:8} we know that
\begin{align*}
\left|x^{n+1}_i-x^n_i\right|\leq \left|\Delta x_i\right|.
\end{align*}
This means the total increment in the time of one iteration is bounded by 
\begin{align*}
\sqrt{
\left|x^{n+1}_1-x^n_1\right|^2
+\left|x^{n+1}_2-x^n_2\right|^2
+\ldots
+\left|x^{n+1}_r-x^n_r\right|^2
}
&\leq 
\sqrt{\Delta x_1^2 + \Delta x_2^2 + \ldots + \Delta x^2_r}
\end{align*}
leading us to define 
\begin{equation*}
\Delta z =\sqrt{\Delta x_1^2 + \Delta x_2^2 + \ldots + \Delta x^2_r}.
\end{equation*}
Now introduce a new variable 
\begin{equation*}
\eta_i=\frac{\Delta x_i}{B}\sim N\left(\frac{A_i}{B},1 \right)
\end{equation*}
which is a normal random variable with mean $A_i/B$ and variance one. Let 
\begin{equation*}
\eta=\eta_1^2+\eta_2^2+\ldots+\eta_r^2
\quad \text{and} \quad 
\lambda=\frac{1}{B^2}\left(A_1^2+A_2^2+\ldots+A_r^2\right)
\end{equation*}
which means $\eta$ is a sum of the squares of $r$ normal random variables with variance one and $\lambda$ is the sum of their means. This means $\eta$ is noncentral chi-squared distributed with $r$ degrees of freedom. Its CDF is 
\begin{equation*}
P(\eta \leq x )=1-Q_{\frac{r}{2}}(\sqrt{\lambda}, \sqrt{x})
\end{equation*}
where $Q_m$ is the Marcum $Q$-function. Thus the total increment is distributed by 
\begin{align*}
P(\Delta z > l)&=P(\Delta z^2 >l^2)\\
&=P(\eta B^2 >l^2)\\
&=P\left(\eta >\frac{l^2}{B^2}\right)\\
&=Q_{\frac{r}{2}}
\left(
\sqrt{\lambda},\frac{l}{\epsilon \, \sqrt{t_{step}}}
\right)
\end{align*}
where $l$ would be the length of a forbidden increment. The number of such forbidden jumps must stay below an appropriate number $N_3$, that is 
\begin{equation}
\frac{t_{end}}{t_{step}}P(\Delta z > l)
=N_3
\label{chap_6:sum:eqn1}
\end{equation}
and Equation \ref{chap_6:sum:eqn1} has to be numerically inverted to give the required time step $t_{step}$
\begin{equation*}
t_{step}=W(\mathcal{S}, l)
\end{equation*}
where a different region $\mathcal{S}$ and length $l$ fulfilling different criteria is used to calculate a different bound on the time step. 
Note that we choose $N_3=0.1$. 

\subsection{Absence of Large Jumps - Estimating $t_{step}\leq t_4$ and $t_{step}\leq t_5$}

Let the set $\mathcal{S}_1$ be given by 
\begin{align*}
\mathcal{S}_1=
\left\{
\left(- \sqrt{1+2a},0\right),
\,
\left(+ \sqrt{1+2a},0\right),
\,
\left(0, - \sqrt{1-2b}\right),
\,
\left(0, + \sqrt{1-2b}\right),
\,
\left(0,0\right)
\right\}
\end{align*}
which is the positions of all the critical points as they would be when the forcing is zero $F=0$. 
This means 
\begin{align*}
\max_{\mathcal{S}_1}\left|\frac{\partial V_0}{\partial x}\right|=0
\quad \text{and} \quad 
\max_{\mathcal{S}_1}\left|\frac{\partial V_0}{\partial y}\right|=0.
\end{align*}
The set $\mathcal{S}_1$ can be used as an approximation for small forcing. 
Equation \ref{chap_6:eqn:11} now becomes 
\begin{align*}
\Delta x&=\left|F_x\right|t_{step}+\epsilon\sqrt{t_{step}}\,\xi_x\\
\Delta y&=\left|F_y\right|t_{step}+\epsilon\sqrt{t_{step}}\,\xi_y
\end{align*}
and we want a time step $t_{step}$ small enough such that almost every value of $\Delta z=\sqrt{\Delta x^2 + \Delta y^2}$ is bounded by 
\begin{align*}
\Delta z=\sqrt{\Delta x^2 + \Delta y^2} \leq l_1
\end{align*}
and this time step is given by $t_4$ below
\begin{align*}
t_4=W(\mathcal{S}_1,l_1).
\end{align*}
Now we remind ourselves that if $\zeta$ is an exponentially distributed random variable, its PDF, CDF, mean and variance are given by 
\begin{align*}
P(\zeta\in A)&=\int_A \lambda e^{-\lambda x}\,dx\\
P(\zeta \leq x)&=\int_{-\infty}^x \lambda e^{-\lambda x}\,dx
=1-e^{-\lambda x}=F(x)\\
\langle \zeta \rangle &=\frac{1}{\lambda}\\
\text{var}(\zeta)&=\frac{1}{\lambda^2}
\end{align*}
where $\lambda$ is the parameter associated with the exponential distribution. 
Now define the height
\begin{align*}
h_0=\max_{\substack{w_j(t)\\ 0\leq t \leq T}}V_t(w_j(t))
\end{align*}
which is the maximum height any well can ever reach. Now define the expression 
\begin{align*}
\Delta V_h= h_1-h_0
\end{align*}
where $h_1$ is chosen so high that the particle will probably never reach there. 
Even if it starts from the highest possible well the chances are still very slim. 
Now we try to estimate what this height $h_1$ may be. 
From Freidlin-Wentzell we know that the escape time from $V=h_0$ to $V=h_1$ is roughly
\begin{align*}
\tau_h\approx e^{2\Delta V_h/\epsilon^2}
\end{align*}
and we want the time it takes to reach $V=h_1$ to be significantly more than the duration of the simulation
\begin{align*}
\tau_h\gg t_{end}
\end{align*}
which means a reasonable estimate would be  
\begin{align*}
N_2t_{end}&=e^{2\Delta V_h/\epsilon^2}\\
2\Delta V_h&= \epsilon^2 \ln (N_2t_{end})\\
\Rightarrow h_1&= \frac{1}{2}\epsilon^2 \ln (N_2t_{end})+h_0
\end{align*} 
where as before we choose $N_2=1000$. 
The escape times leaving $V=h_0$ and arriving at $V=h_1$ is exponentially distributed. 
The average of them would be 
\begin{align*}
\left\langle 
\tau_h
\right\rangle
=
N_2t_{end}
=e^{2\Delta V_h /\epsilon^2}
\end{align*}
and so by using the CDF of the exponential distribution we can say 
\begin{align*}
P(\tau_h<t_{end})&=1-e^{-1/N_2}\\
&=0.632 \quad \text{for} \quad N_2=1\\
&=0.095 \quad \text{for} \quad N_2=10\\
&=0.01\,\ \quad \text{for} \quad N_2=100\\
&=0.001 \quad \text{for} \quad N_2=1000.
\end{align*}
So if we choose $N_2\geq1000$ the chances of reaching $V=h_1$ are less then one in a thousand. 
Define the set 
\begin{align*}
\mathcal{S}_2&=\left\{x \in \mathbb{R}^r : [V_0(x)-Fx\cos(\Omega t)]\leq h_1 : 0\leq t \leq T \right\}\\
&=\left\{x \in \mathbb{R}^r : V_t\leq h_1 : 0\leq t \leq T \right\}
\end{align*}
which is the set of all the points below $V_t\leq h_1$ over the time of one period. 
The forbidden increment is taken as the same as last time 
\begin{align*}
l_2=l_1
\end{align*}
so the second bound is 
\begin{align*}
t_5=W(\mathcal{S}_2, l_2).
\end{align*}

\section{Stability and Radius Conditions}
Stability of the trajectory in the context of this thesis is for the  time step $t_{step}$ to be small enough such that 
the simulated discrete trajectory is a good enough approximation of a physical continuous trajectory. 
For example consider the following trajectories for a particle falling down to the well of the Mexican Hat starting at $(x_{start},y_{start})=(-0.75,-0.75)$.
\begin{figure}[H]
\centerline{\includegraphics[scale=0.34]{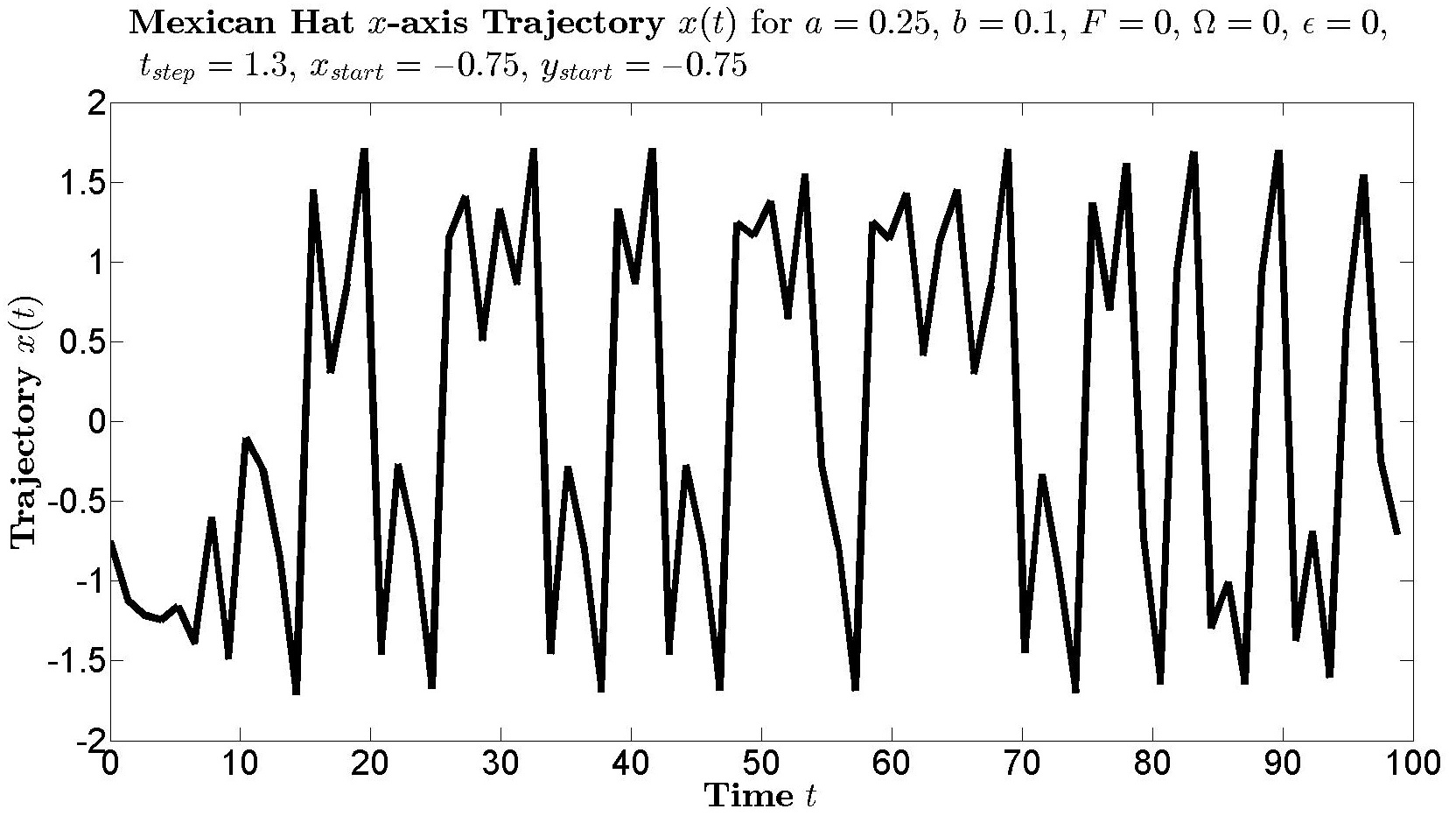}}
\caption{The trajectory is so unstable the particle even transits to the other well.}
\end{figure}
\begin{figure}[H]
\centerline{\includegraphics[scale=0.34]{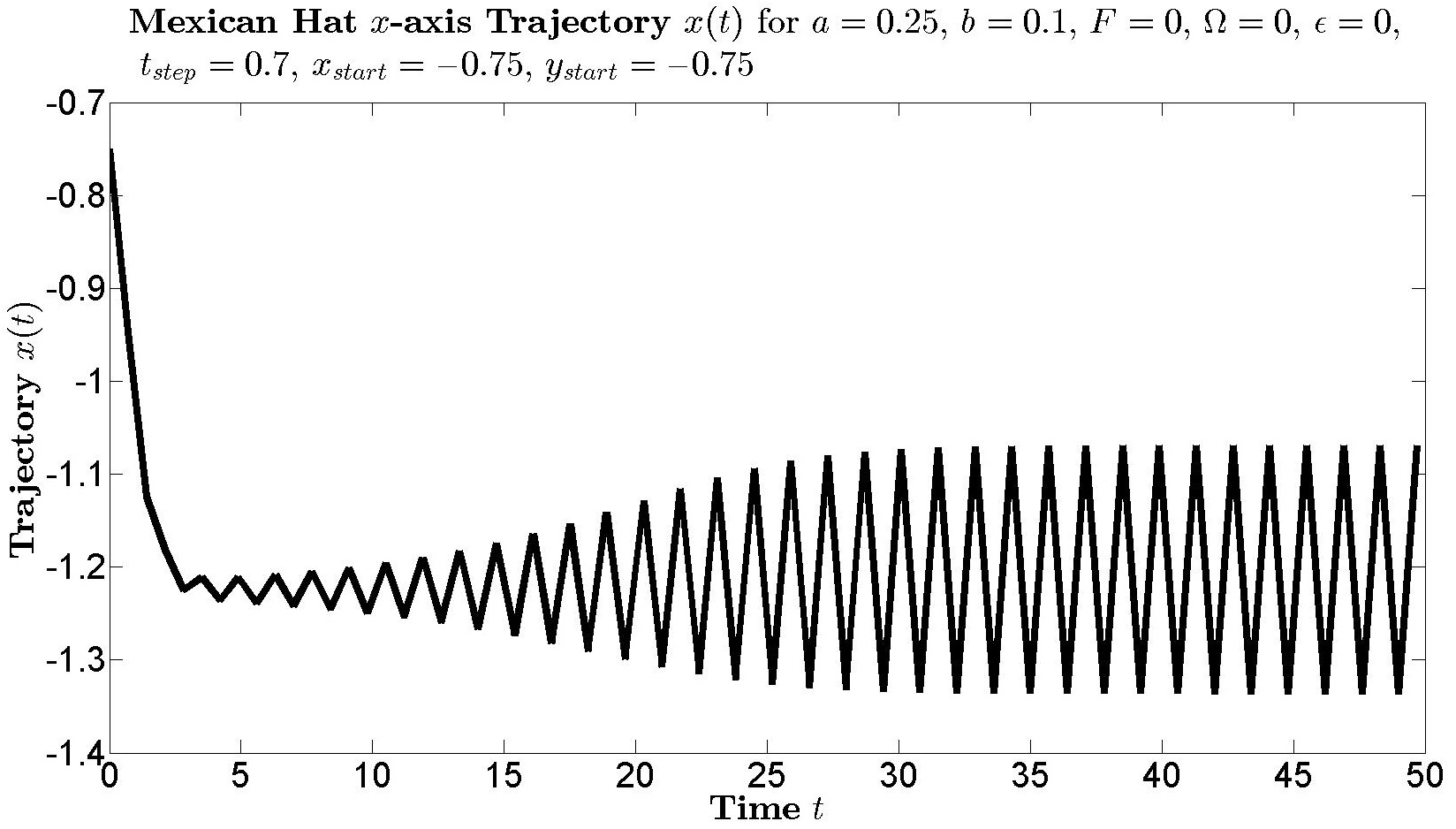}}
\caption{The trajectory is more stable but the particle now oscillates near the well.}
\end{figure}
\begin{figure}[H]
\centerline{\includegraphics[scale=0.34]{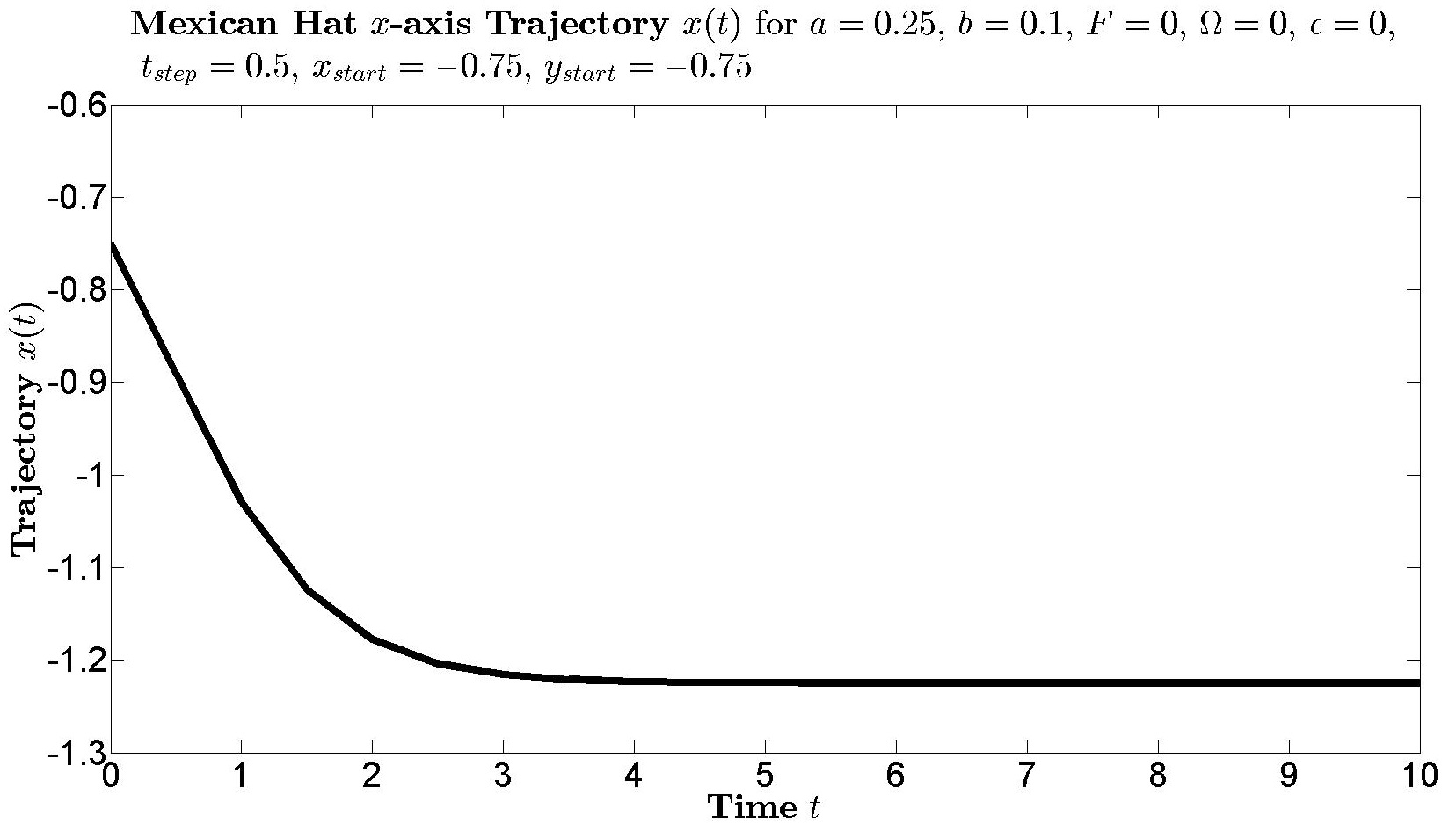}}
\caption{The trajectory is sufficiently stable here.}
\end{figure}

\noindent It is the aim of this section to study these stability problems. 
The Euler method is effectively a discrete iterative map with some operations. 
We have the following.

\begin{Lemma}\label{chap_6:thm:iter}
Let $(X, \Vert \cdot \Vert )$ be a vector space endowed with the norm $\Vert \cdot \Vert$ over the complex scalar field $\mathbb{F}$. 
Let $M$ be an operator $M:X\longrightarrow X$ with the property $\Vert Mx \Vert \leq \Vert M \Vert \Vert x \Vert$ and $\Vert M \Vert \geq 0$. 
Let $x\in X$ be part of an iterative scheme
\begin{equation*}
x_n=Mx_{n-1}+\epsilon\sqrt{t}\,\xi_{n-1}
\end{equation*}
where $\xi_{n-1}\in X$ is a term which depends on the iterative step $n$. The entire term $x_n$ is then bounded by 
\begin{equation*}
\left\Vert x_n \right\Vert \leq 
\Vert M \Vert ^n \Vert x_0 \Vert +
\epsilon\sqrt{t}\sum_{i=0}^{n-1}\Vert M \Vert ^i \
\Vert \xi_{n-1-i} \Vert 
\end{equation*}
where $x_0$ and $\xi_0$ are the starting (first) steps. 
\end{Lemma}

\begin{proof}
Rewrite the iterative scheme with a new operator 
\begin{align*}
x_n
&=Mx_{n-1}+\epsilon\sqrt{t}\,\xi_{n-1}\\
&=M'_nx_{n-1}
\end{align*}
where $M'_n$ is the total operator which depends on the $n$th step. In general $M'_n$ is not commutative so we write 
\begin{align*}
x_n&=M'_{n}M'_{n-1}\ldots M'_1 x_0
\end{align*}
and consider just one operation on $M'_n$
\begin{align*}
M'_nx
&=Mx+\epsilon\sqrt{t}\,\xi_{n-1}\\
\left\Vert M'_nx \right\Vert
&=\left\Vert Mx+\epsilon\sqrt{t}\,\xi_{n-1} \right\Vert\\
&\leq \left \Vert Mx \right \Vert + \epsilon\sqrt{t}\,\left\Vert\xi_{n-1}\right\Vert\\
&\leq \left \Vert M \right \Vert \left\Vert x \right \Vert + \epsilon\sqrt{t}\,\left\Vert\xi_{n-1}\right\Vert.
\end{align*}
So we have an iterative expression
\begin{align}
\left\Vert M'_nx \right\Vert
&\leq \left \Vert M \right \Vert \left\Vert x \right \Vert + \epsilon\sqrt{t}\,\left\Vert\xi_{n-1}\right\Vert \label{chap_6:sum:eqn2}
\end{align}
and iterating Equation \ref{chap_6:sum:eqn2} gives 
\begin{align*}
\left\Vert x_n\right\Vert
&=\left\Vert M'_{n}M'_{n-1}\ldots M'_1 x_0\right\Vert\\
&\leq \Vert M \Vert ^n \Vert x_0 \Vert +
\epsilon\sqrt{t}\sum_{i=0}^{n-1}\Vert M \Vert ^i \
\Vert \xi_{n-1-i} \Vert 
\end{align*}
which completes the proof. 
\end{proof}

\begin{Remark}\label{chap_6:tru:remark}

There are a lot of remarks to say about this simple Lemma. 

\begin{enumerate}
\item Nowhere was it assumed $M$ is bounded, linear or commutative. 
\item Notice that $\Vert M \Vert$ was not defined. 
It could be the usual operator norm or something else. 
\item Notice that we are only truncating the bound as in Equation \ref{chap_6:sum:eqn2} and NOT truncating the original iterative Euler scheme. 
\item Notice that $M0=0$ for the zero vector. 
Notice also that $M(\lambda x )=\lambda M(x)$ where $\lambda$ is a scalar was not assumed. 
\item Notice that a norm on $\Vert M'_n \Vert$ was not needed. This is because such a  norm would have to satisfy $\left\Vert M'_nx \right \Vert \leq \left \Vert M'_n \right \Vert \Vert x \Vert$ which implies $M'_n 0=0$. But with the random vector $\xi_{n-1}$ in $M'_n$ this cannot be achieved. 
\end{enumerate}

\end{Remark}

\noindent Consider the two dimensional Mexican Hat system with $z=(x,y)\in \mathbb{R}^2$. 
Let the potential be stationary by freezing it as it would be at a certain fixed point in time $t=t_{fix}$. 
The gradient then becomes
\begin{align*}
\nabla V_{t=t_{fix}}(z)=
\left(
\begin{array}{c}
\frac{\partial V_0}{\partial x}-F_x\cos \Omega t_{fix}\\[0.5em]
\frac{\partial V_0}{\partial y}-F_y\cos \Omega t_{fix}
\end{array}
\right)
\end{align*}
where $t_{fix}$ is a constant point in time.  
The Euler method can now be rewritten as 
\begin{align}
t_{n}&=t_{n-1}+t_{step}\nonumber\\
x_{n}&=x_{n-1}+\left[-\frac{\partial V_0}{\partial x}+F_x\cos\Omega t_{fix}\right]t_{step}+\epsilon\sqrt{t_{step}}\, \xi_x \nonumber\\
y_{n}&=y_{n-1}+\left[-\frac{\partial V_0}{\partial y}+F_y\cos\Omega t_{fix}\right]t_{step}+\epsilon\sqrt{t_{step}}\, \xi_y \nonumber
\end{align}
and we recast it into vector notation by writing 
\begin{align}
t_{n}&=t_{n-1}+t_{step}\nonumber\\
z_{n}&=z_{n-1}-t_{step}\nabla V_{t=t_{fix}}(z_{n-1})+\epsilon\sqrt{t_{step}}\, \xi_{n-1} \label{chap_6:eqn:3}
\end{align}
where 
\begin{align*}
z_n=\left(
\begin{array}{c}
x_n\\y_n
\end{array}
\right). 
\end{align*}
Now we apply  Lemma \ref{chap_6:thm:iter} to Equation \ref{chap_6:eqn:3}. 
Clearly the operator as mentioned in Theorem \ref{chap_6:thm:iter} is 
\begin{align*}
M(z)=z-t_{step}\nabla V_{t=t_{fix}}(z)
\end{align*}
and this operator has to satisfy $\left\Vert M z \right\Vert\leq\left\Vert M \right\Vert \left\Vert x \right\Vert$. The usual operator norm would suffice with some restrictions
\begin{align*}
\left\Vert M \right\Vert_{\mathcal{S}}
&=\sup_{\substack{z \in \mathcal{S}\\ z\neq0}}
\frac{\left\Vert Mz \right\Vert}{\left\Vert z \right\Vert}
\end{align*} 
where $\mathcal{S}\subset \mathbb{R}^2$ is a strict and suitable subset of the whole space. 
This is because in general, the expression $\Vert M z \Vert$ is unbounded on $\mathbb{R}^2$. 
For example in our case $\nabla V$ is unbounded. 
This would give
\begin{align*}
\left\Vert Mz \right\Vert 
=\Vert z \Vert 
\frac{\left\Vert Mz \right\Vert}{\left\Vert z \right\Vert}
\leq \left\Vert z \right\Vert
\sup_{\substack{z \in \mathcal{S}\\ z\neq0}}
\frac{\left\Vert Mz \right\Vert}{\left\Vert z \right\Vert}
\leq \left\Vert M \right \Vert
_{\mathcal{S}}\, \left\Vert z \right\Vert.
\end{align*}
But we also need the condition $M0=0$ for the zero vector (see Remark \ref{chap_6:tru:remark}). 
This means we have to shift the coordinates to
\begin{align*}
z_{new}=z_{old}-z_{well}
\end{align*}
where the well is now the new origin. 
Now we can apply Lemma \ref{chap_6:thm:iter} to Equation \ref{chap_6:eqn:3} and get 
\begin{align}
\left\Vert z_n \right\Vert 
&\leq 
\left\Vert M \right\Vert ^n_{\mathcal{S}}\, \left\Vert z_0 \right\Vert +
\epsilon\sqrt{t_{step}}\sum_{i=0}^{n-1}\left\Vert M \right\Vert ^i_{\mathcal{S}} \
\left\Vert \xi_{n-1-i} \right\Vert \nonumber 
\end{align}
where the starting position is 
\begin{align*}
z_0&=\left(x_{start},y_{start}\right) 
-\left(x_{well}(t_{fix}),y_{well}(t_{fix})\right)
\end{align*}
where 
$\left(x_{start},y_{start}\right)$ is the starting position and 
$\left(x_{well}(t_{fix}),y_{well}(t_{fix})\right)$ is the position of one well at time $t=t_{fix}$

\subsection{Stability Problems}
Lemma \ref{chap_6:thm:iter} has allowed us to rewrite the Euler method, expressed in Equation \ref{chap_6:eqn:3}, as
\begin{align}
\left|z_{n}\right|&=\left|M(z_{n-1})+\epsilon\sqrt{t_{step}}\,\xi_{n-1}\right|\nonumber\\
&\leq 
\left\Vert M \right\Vert ^n_{\mathcal{S}}\, \left\Vert z_0 \right\Vert +
\epsilon\sqrt{t_{step}}\sum_{i=0}^{n-1}\left\Vert M \right\Vert ^i_{\mathcal{S}} \
\left\Vert \xi_{n-1-i} \right\Vert\label{chap_6:eqn:4}
\end{align}
where if $\epsilon=0$ we would reduce back to the deterministic system 
\begin{align*}
\left|z_{n}\right|\leq 
\left\Vert M \right\Vert ^n_{\mathcal{S}}\, \left\Vert z_0 \right\Vert. 
\end{align*}
It is known that if the time step $t_{step}$ is too large then even the deterministic trajectory is unstable. 
A stable time step $t_{step}$ is one which gives 
\begin{align}
\left\Vert M \right\Vert _{\mathcal{S}}
\leq (1-\delta)\label{chap_6:eqn:5}
\end{align}
where $0<\delta<1$ and the particle would settle at the bottom of the well as $n\longrightarrow \infty$. 
There are two approaches to the problem here.
\begin{enumerate}
\item Fix the set $\mathcal{S}$ and solve for the time step $t_{step}$ such that $\left\Vert M \right\Vert_{\mathcal{S}}\leq (1-\delta)$ holds. This is  solving for time. 
\item Fix the time step $t_{step}$ and solve for the set $\mathcal{S}$ such that $\left\Vert M \right\Vert_{\mathcal{S}}\leq (1-\delta)$ holds. This is  solving for space. 
\end{enumerate}
and each approach hinges on the assumption that the $\mathcal{S}$ and $t_{step}$ we fixed to begin with is a good and stable choice. 
We can only solve for time or space but not both. 
This is also complicated by the fact that 
the following estimate is too rough
\begin{align}
\frac{\left\Vert Mz \right\Vert}{\left\Vert z \right\Vert}
\leq \frac{\left\Vert z-t_{step}\nabla V_{t=t_{fix}}(z) \right\Vert}{\left\Vert z \right\Vert}
\leq \frac{\left\Vert z \right\Vert +\left\Vert t_{step}\nabla V_{t=t_{fix}}(z)\right\Vert}{\left\Vert z \right\Vert}
\leq 1+\delta'\label{chap_6:eqn:6}
\end{align}
where $\delta'>0$.

\subsection{Estimating $R\leq R_1$ and $R\leq R_2$}
Now two bounds on the radius $R$ is derived. The first bound is 
\begin{align*}
R_1=\min_{\substack{w_j(t),c_i(t)\\ w_j\neq c_i  \\ 0\leq t \leq T}}
\left\{
\frac{1}{2}
\left|
w_j(t)-c_i(t)
\right|
\right\}
\end{align*}
which is half the distance from the wells to all critical points, for all wells, over one period. 
This is such that there is always exactly one critical point inside the region covered by the radius $R$. 
The second bound comes from considering Equation \ref{chap_6:eqn:4}
\begin{align*}
z_{n}\leq 
\left\Vert M \right\Vert ^n_{\mathcal{S}}\, \left\Vert z_0 \right\Vert +
\epsilon\sqrt{t_{step}}\sum_{i=0}^{n-1}\left\Vert M \right\Vert ^i_{\mathcal{S}} \
\left\Vert \xi_{n-1-i} \right\Vert
\end{align*}
and seek a bound on the variance of the random part. 
The variance of the random part is 
\begin{align*}
\text{var}
\left(
\epsilon\sqrt{t_{step}}\sum_{i=0}^{n-1}\left\Vert M \right\Vert ^i_{\mathcal{S}} \
\left\Vert \xi_{n-1-i} \right\Vert
\right)
=\epsilon^2t_{step}
\sum_{i=0}^{n-1}\left\Vert M \right\Vert ^{2i}_{\mathcal{S}} \
\text{var} \left(\Vert \xi_{n-1-i} \Vert \right).
\end{align*}
Notice how each of the $\Vert \xi_{n-1-i}\Vert$ is $\chi$ distributed with $r=2$ degrees of freedom.
Its variance is
\footnote{
Let $\xi_x$ and $\xi_y$ be independently and normally distributed in $N(0,1)$.
Let $\zeta=\sqrt{\xi_x^2+\xi_y^2}$ and $\eta=\xi_x^2+\xi_y^2$,
then $\zeta$ is $\chi$ distributed and $\eta$ is $\chi^2$ distributed. 
Notice that we want the variance of $\zeta$ and NOT the variance of $\eta$, therefore only the $\chi$ distribution is needed.
} 
\begin{align*}
\sigma^2
&=\text{var} \left(\Vert \xi_{n-1-i}\Vert\right)\\
&=r-
\left(\sqrt{2}
\frac{\Gamma\left(\frac{r+1}{2}\right)}{\Gamma\left(\frac{r}{2}\right)}\right)^2\\
&=0.4292 \quad \text{for} \quad r=2
\end{align*}
which means the variance of the random part is bounded by 
\begin{align*}
\text{var}
\left(
\epsilon\sqrt{t_{step}}\sum_{i=0}^{n-1}\left\Vert M \right\Vert ^i_{\mathcal{S}} \
\left\Vert \xi_{n-1-i} \right\Vert
\right)
&=
\epsilon^2t_{step}
\sum_{i=0}^{n-1}\left\Vert M \right\Vert ^{2i}_{\mathcal{S}} \
\sigma^2\\
&\leq
\epsilon^2\sigma^2t_{step}
\sum_{i=0}^{\infty}\left\Vert M \right\Vert ^{2i}_{\mathcal{S}}\\
&=\epsilon^2\sigma^2t_{step}\,
\frac{1}{1-\left\Vert M \right \Vert_{\mathcal{S}}^2}
\end{align*}
which means the random part has bounded variance, even if one considers infinite time. 
The error is constantly cancelling out with itself, and the Euler method can run for a very long time and still be stable. 
The radius must be similar to the size of this variance.  
Define
\begin{align*}
t_{step}^{osc}
=
\min
\left\{
t_1, t_2, t_3, t_4, t_5
\right\}
\end{align*}
which is the smallest of all the time steps we have derived so far. 
Assume that the operator can indeed be bounded by $
\left\Vert M \right \Vert_{\mathcal{S}}
\leq (1-\delta)$, then we can define a second bound $R_2$ on the radius as 
\begin{align*}
R_2^2&=
\text{var}
\left(
\epsilon\sqrt{t_{step}}\sum_{i=0}^{n-1}\left\Vert M \right\Vert ^i_{\mathcal{S}} \
\left\Vert \xi_{n-1-i} \right\Vert
\right)\\
&\leq
\epsilon^2\sigma^2t_{step}\,
\frac{1}{1-\left\Vert M \right \Vert_{\mathcal{S}}^2}\\
&=
\epsilon^2\sigma^2t_{step}\,
\frac{1}{1-(1-\delta)^2}
\end{align*}
where we have used $
\left\Vert M \right \Vert_{\mathcal{S}}
\leq (1-\delta)$.
This gives $R_2$ as 
\begin{align*}
R_2=
\sqrt{
\frac{\epsilon^2\sigma^2t_{step}^{osc}}{2\delta-\delta^2}
}
\end{align*}
after choosing a suitable value for $\delta$ say $\delta=0.01$, then a reasonable choice on the radius $R$ could be 
\begin{align*}
R=\min\left\{R_1, R_2\right\}.
\end{align*}

\subsection{Estimating $t_{step}\leq t_6$}

Define the set 
\begin{equation*}
\mathcal{S}_3=\left\{x\in\mathbb{R}^r: \left|w_j(t)-x\right|\leq R:0\leq t \leq T, \; \forall w_j(t)\right\}
\end{equation*}
which is the set of points where the distance to a well is less than the radius over all times and all wells. 
The new jump size we do not want to see during our simulation is $l_3=R$
\begin{equation*}
l_3=R
\end{equation*}
which gives another condition on the time step as
\begin{equation*}
t_6=W(\mathcal{S}_3, l_3).
\end{equation*}
The idea behind this final condition is so that the time step $t_{step}$ is small and precise enough such that the region around the well can capture it. 

\section{Selection of Parameters}
\label{conclusion_parameter_selection}

Six conditions on the time step and two bounds on the radius were developed. 
These give the recommended values for $R$ and $t_{step}$ as 
\begin{align*}
R&=\min\left\{R_1, R_2\right\}\\
t_{step}
&=
\min
\left\{
t_1, t_2, t_3, t_4,t_5, t_6
\right\}. 
\end{align*}
These are not rigorous estimates and hence can only be used as a guideline. 
We performed several checks of consistency of our simulations at a different level of precision before comparing them with the theoretical result. 
Notice $t_{step}$ and $R$ are just some of the considerations we have to make when choosing a set of parameters to use for the simulations. 
More details are given below.

\subsection{Selection of Parameters - Simulation}
\label{select_parameter}

In Chapter \ref{sum_chap_results} we will introduce the simulations which we are going to do. 
Notice that the SDE we want to simulate can be rewritten as 
\begin{align*}
dx&=\left[-\frac{\partial V_0}{\partial x}+F\cos\phi\cos \Omega t \ \right]dt+\epsilon \ dw_x\\
dy&=\left[-\frac{\partial V_0}{\partial y}+F\sin\phi\cos \Omega t \ \right]dt+\epsilon \ dw_y.
\end{align*}
The following parameters will be fixed with the values 
\begin{align*}
a=0.15 \quad
b=0.1 \quad 
F=0.7F^{crit} \quad 
\Omega=0.001
\end{align*}
and $\epsilon$ and $\phi$ will systemically vary by going through all possible combinations of 
\begin{align*}
\epsilon=0.15, 0.16, \ldots, 0.30 
\quad \text{and} \quad 
\phi=0^\circ,75^\circ,78^\circ,81^\circ,84^\circ,87^\circ,90^\circ
\end{align*}
and the following value of the time step and radius is used
\begin{align*}
t_{step}=0.014 \quad R=0.19. 
\end{align*}
We give some reasons as to why these parameters were chosen.

\subsection{Selection of Parameters - Validity of Kramers' Formula}
There is the validity of Kramers' formula. 
Consider the graphs below. 

\begin{figure}[H]
\centerline{\includegraphics[scale=0.32]{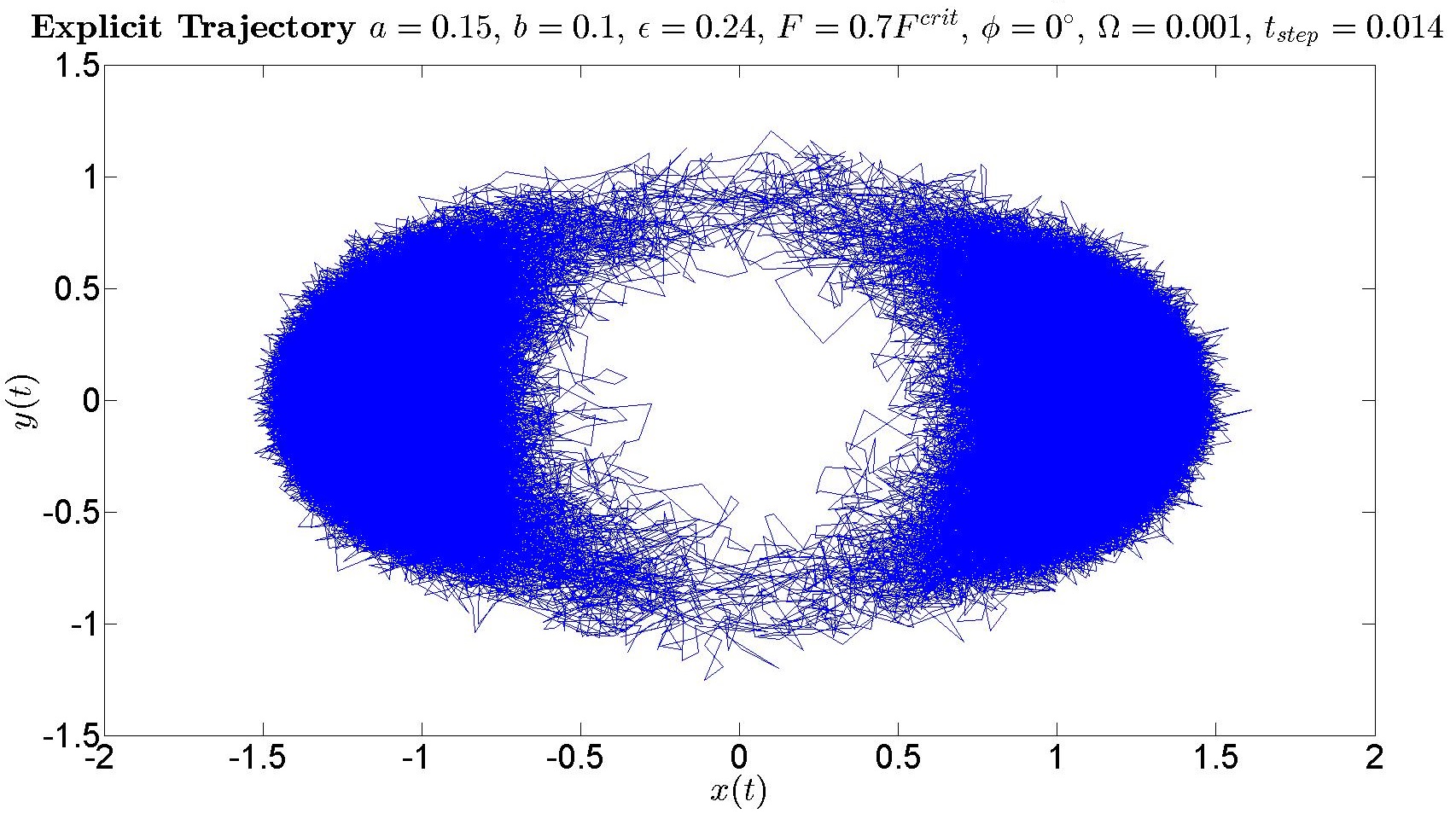}}
\caption{Notice that transitions tend to occur near the saddles. 
This is when Kramers' formula gives a good approximation for the escape rates and escape times.
}
\label{chap_8_path_explicit}
\end{figure}

\begin{figure}[H]
\centerline{\includegraphics[scale=0.32]{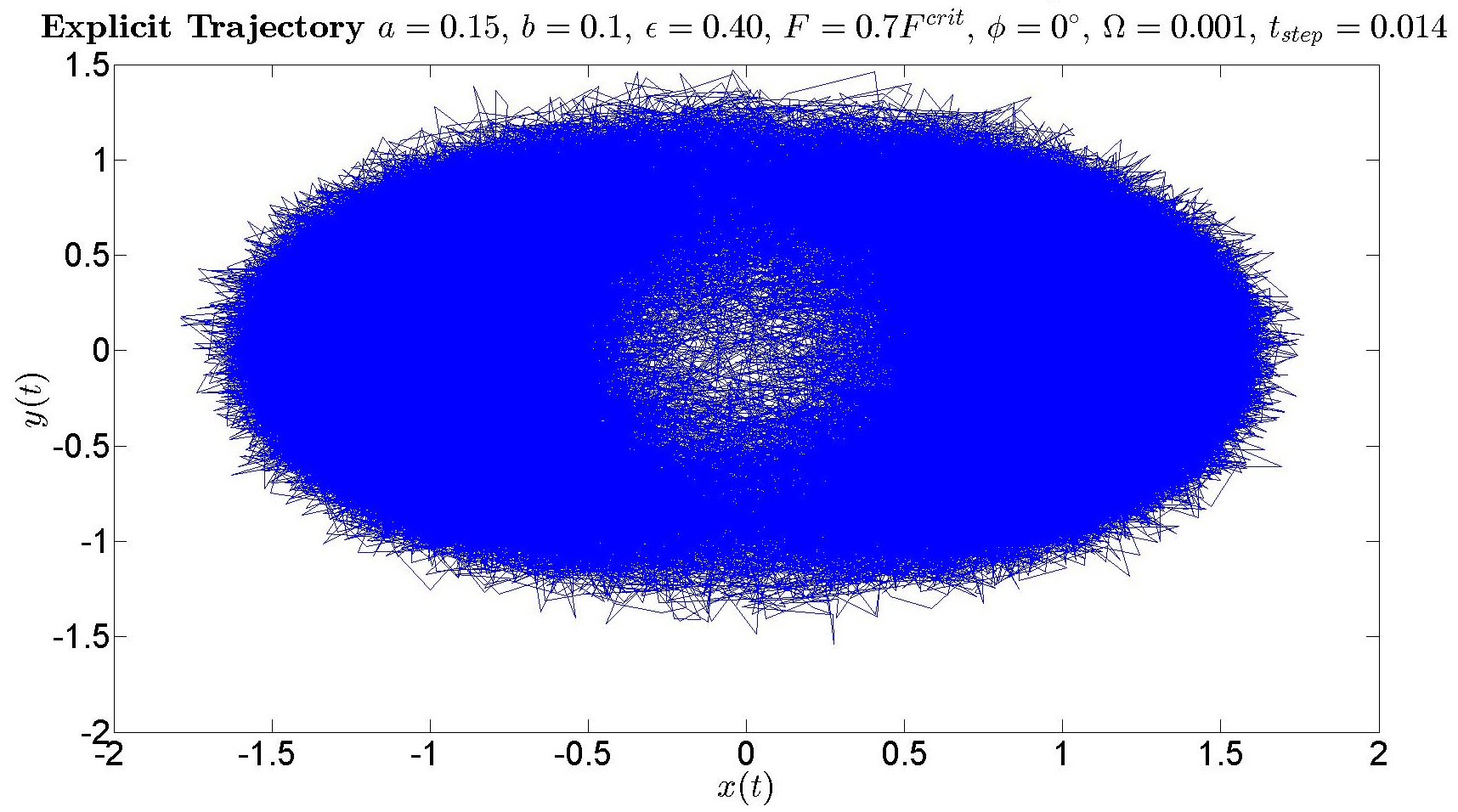}}
\caption{For higher noise levels transitions would occur near the hill, which is close to the origin.
Kramers' formula is not a good approximation here.
}
\label{chap_8_path_explicit_2}
\end{figure}

\noindent 
Note that in Figures \ref{chap_8_path_explicit} and \ref{chap_8_path_explicit_2} the position of the hill is near the origin and the positions of the saddles are near the $y$-axis. 
For very high noise levels Kramers' formula would start to fail as an approximation to the escape times and rates. This is when the particle tend to transit through both the saddles and the hill.
A level of subjective judgement is required to gauge how good an approximation Kramers' formula is. 
Nevertheless it has been checked that at $\epsilon=0.30$, that Kramers' formula is good enough an approximation for all angles. 
This checking was also done for an unperturbed static potential, a static potential with maximal forcing and an oscillating potential.\footnote{This footnote also applies for graphs later in the thesis. 
Note that Figures 
\ref{chap_8_path_explicit}, \ref{chap_8_path_explicit_2},
\ref{chap_8_path_p0},
\ref{chap_8_path_p84} and 
\ref{chap_8_path_p90}
have use the following parameters in the simulations. 
\begin{align*}
t_{start}=0 \quad t_{step}=0.014 \quad t_{end}=100000
\end{align*}
although a time step size of $t_{step}=0.014$ was used, only the data for every ten time steps were shown. 
This was done to avoid handling a very large graph in MatLab.}

\subsection{Selection of Parameters - Adiabatic Approximation}
\label{adiabatic_parameters}
There is a reason why we need Kramers' formula to be valid, that is a good approximation to the escape times. 
Recall that in Chapter \ref{sum_chap_escape_oscill} theories about escape times from an oscillatory potential was developed. 
One of the main results were the continuous time invariant measures for an oscillatory potential (see Corollary \ref{corollary_invariant_measure_contin}) and the PDF for the escape times (see Theorem \ref{chap_4_pdf_thm}). 
The invariant measures are
\begin{align*}
\overline{\nu}_-(t)
&=
\frac{\int^t_0p(s)g(s)\,ds}{g(t)}
+\frac{\int^T_0p(s)g(s)\,ds}{g(t)\left(g(T)-1\right)}
\\[0.5em]
\overline{\nu}_+(t)
&=
\frac{\int^t_0q(s)g(s)\,ds}{g(t)}
+\frac{\int^T_0q(s)g(s)\,ds}{g(t)\left(g(T)-1\right)}
\end{align*}
where
\begin{align*}
g(t)=\exp\left\{\int^t_0p(u)+q(u)\,du\right\}
\end{align*}
and the PDFs are
\begin{align*}
p_{-}(t,u)&=R_{-1+1}(t)\exp\left\{-\int^t_uR_{-1+1}(s)\,ds\right\}\\[0.5em] 
p_{+}(t,u)&=R_{+1-1}(t)\exp\left\{-\int^t_uR_{+1-1}(s)\,ds\right\}
\end{align*}
where we will make the approximation 
\begin{align*}
p_{tot}(t)&=
\frac{1}{2}
\int_{0}^{T}p_-(t+u,u)m_-(u)+p_+(t+u,u)m_+(u)\,du
\approx p_+(t,0)
\end{align*}
because we do not have expressions for $m_-(u)$ and $m_+(u)$. 
Notice one subtlety about all the theories developed in Chapter \ref{sum_chap_escape_oscill}. 
It was assumed that the probabilities for transit, that is $p(t)$ and  $q(t)$, were known for whatever the driving frequency $\Omega$ may be. 
In the PDFs it was also assumed that the escape rates 
$R_{-1+1}(t)$ and $R_{+1-1}(t)$ were known for however fast or slow the driving frequency $\Omega$ may be.

But such ideal expressions for $p(t)$, $q(t)$, $R_{-1+1}(t)$ and $R_{+1-1}(t)$ are not known.
Note that the escape rates as given by Kramers' formula in Chapter \ref{chap_sect_kram} are only valid for small noise for a static potential.  
When we do analysis in Chapter \ref{sum_chap_results} the rates $R_{-1+1}(t)$ and $R_{+1-1}(t)$ are calculated by Kramers' formula as though it is a static potential. 
This means an oscillatory potential is being approximated by a static potential. 
This is the adiabatic approximation. 
The following conditions are proposed to decide if the adiabatic approximation is valid 
\begin{align}
\min_{
\substack{
t\in[0,T]}
}
\left\{
\tau^{kram}_{-1+1}(t),\tau^{kram}_{+1-1}(t)
\right\}\leq \frac{2\pi}{\Omega}\label{chap_7_adia_1}\\[0.5em]
\max_{
\substack{
t\in[0,T]}
}
\left\{
\tau^{kram}_{-1+1}(t),\tau^{kram}_{+1-1}(t)
\right\}\leq \frac{2\pi}{\Omega}\label{chap_7_adia_2}
\end{align}
that is we consider the minimum and maximum escape times over one period  as given by Kramers' formula for a static potential. 
If this is less than the period of the driving frequency $T=2\pi/\Omega$, then the adiabatic approximation may be valid. 
This was checked for all the parameters and Equations 
\ref{chap_7_adia_1} and \ref{chap_7_adia_2} only hold for the following range of parameters
\begin{align*}
\phi\geq75^\circ 
\quad \text{and} \quad 
\epsilon\geq0.28. 
\end{align*}
This is a compromise we made. 
Nevertheless  Equations 
\ref{chap_7_adia_1} and \ref{chap_7_adia_2} do not define the adiabatic approximation, but give an idea of what range the parameters need to be in.

\subsection{Selection of Parameters - Stability of Deterministic Trajectory}
There is also the stability of the deterministic trajectory (when $\epsilon=0$) to be concerned about. 
The following starting positions were chosen for four  particles
\begin{align*}
x_{start}=+2 \quad y_{start}=+2\\
x_{start}=+2 \quad y_{start}=-2\\
x_{start}=-2 \quad y_{start}=+2\\
x_{start}=-2 \quad y_{start}=-2
\end{align*}
and their trajectories falling through an unperturbed static potential, a static potential with maximal forcing and an oscillating potential were all shown to be stable for all angles. 
Note that these values of $x_{start}$ and $y_{start}$ were chosen because they are at a place where the potential is so high the particle will probably never go there.
Notice how in Figures  \ref{chap_8_path_explicit} and \ref{chap_8_path_explicit_2} the trajectory almost never reaches any of the four corners $(2,2)$, $(2,-2)$, $(-2,2)$ and $(-2,-2)$. 
This is the reason for checking the trajectories there.

\subsection{Selection of Parameters - Random Number Generator}
A few words may also be said about the random number generator we are using. 
Note that these are pseudo random numbers. 
They are deterministic sequences of numbers with a very long period. 
We use the randn() function in MatLab. 
It is very random and will almost certainly not repeat itself for many years. 
This is because the period is $2^{1492}$. 
Even with the computer generating 60 million random numbers per second it would still take $10^{434}$ years to reach the end of the cycle \cite{matlab_random}. 
The function rng('shuffle') was also used, which picks a seed for the random number generator according to the time of the computer clock. 
When the Parallel ToolBox is used in MatLab, each worker randomly picks a seed for itself.

\subsection{Selection of Parameters - Calculating Positions of Critical Points}
As mentioned in Chapter \ref{stable_critical_points}, the numerical algorithms can be very unstable for calculating the positions of the critical points. 
In the simulations which we are going to conduct, a table of the positions of all the critical points within one period are calculated first, then stored in the temporary memory of the computer, and looked up every time the position of a critical point is needed. 
This table is calculated in the following way. 
Define what we call the pseudo parameters to be 
\begin{align*}
u_{start}=0 \quad u_{step}=0.001 \quad u_{end}=2\pi
\quad \Omega=1
\end{align*}
and then numerically find the positions of the critical points of the equation
\begin{align*}
V_t=V_0-F_x\cos\Omega t -F_y\cos\Omega t
\end{align*}
where 
\begin{align*}
t=0, \quad
t=u_{step},\quad 
2u_{step}, \quad 
3u_{step}, \quad 
\ldots \quad 
t\approx 2\pi.
\end{align*}
Due to the way a matrix is define in MatLab, the last value of $t$ is not exactly at $t=2\pi$. 
This table was checked for all angles, and the pseudo parameters we have chosen are stable.

\subsection{Selection of Parameters - Higher Precision Numerics}
We also have some remarks about the time step we have chosen. 
As we shall see in Chapter \ref{sum_chap_results}, one of the main effects which we have observed in this thesis is what we call the Single, Intermediate and Double Frequency in the histograms of escape times. 
These effects were first observed for the values of $a=0.15$, $b=0.1$, $F=0.7F^{crit}$ and $\Omega=0.001$  when 
\begin{align*}
t_{step}\geq 0.0286
\quad 
R\geq 0.3218
\end{align*} 
and was observed again when $t_{step}=0.014$ and $R=0.19$.
Note that  $t_{step}=0.014$ and $R=0.19$ was used for the results of this thesis. 
Thus we have confidence in believing that the data we have collected is reliable. 
Consider the graphs below. 
They are histograms of escape times from both the left and right wells combined.
They are also normalised to give an empirical PDF. 

\begin{figure}[H]
\centerline{\includegraphics[scale=0.35]{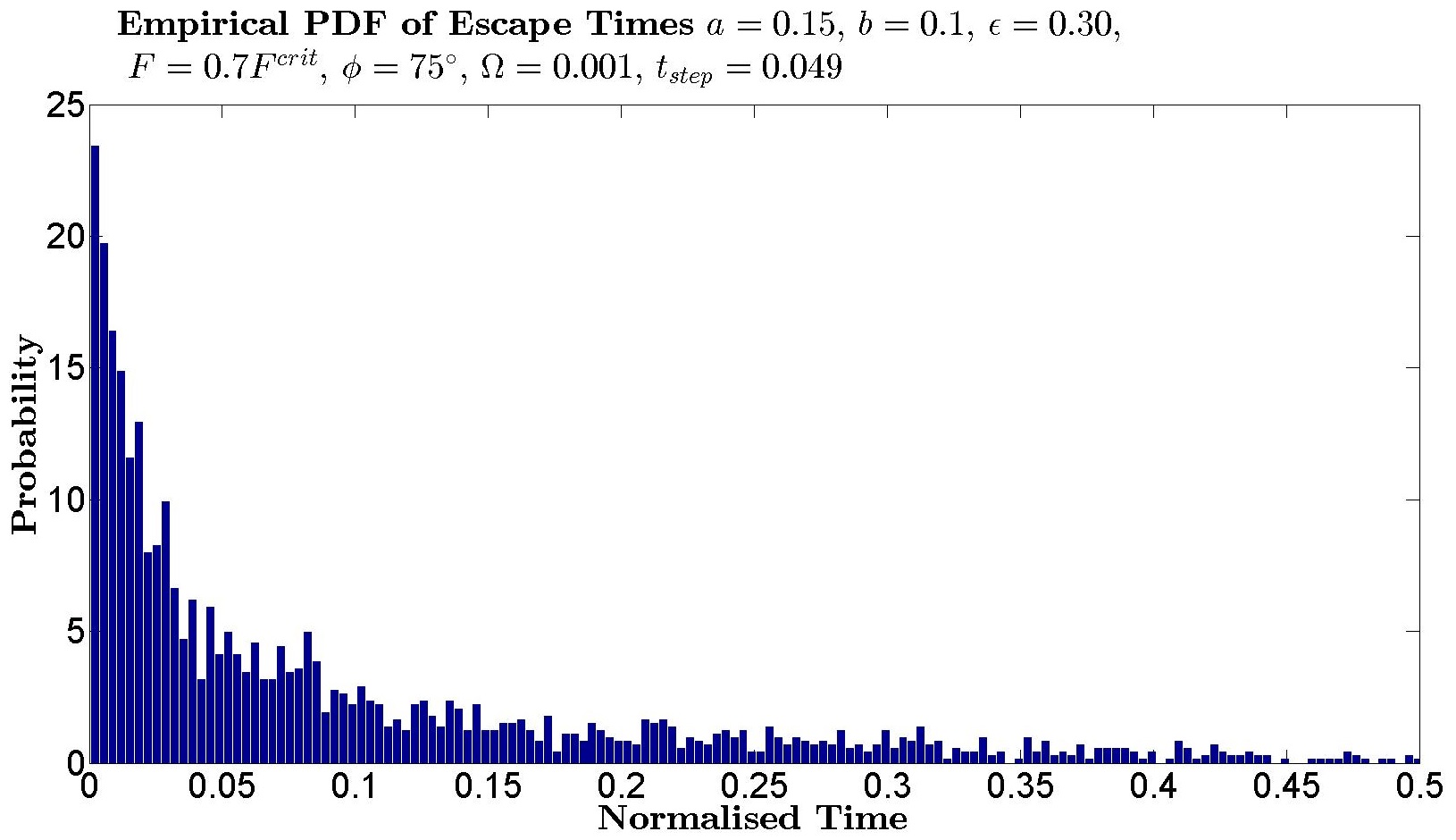}}
\caption{
Here 2239 transitions were used. 
The averaged measured escape time is $0.0977T$. 
The radius used was $R= 0.5386$. 
}
\label{chap_7_g56f_p75_e30_all_data}
\end{figure}

\begin{figure}[H]
\centerline{\includegraphics[scale=0.35]{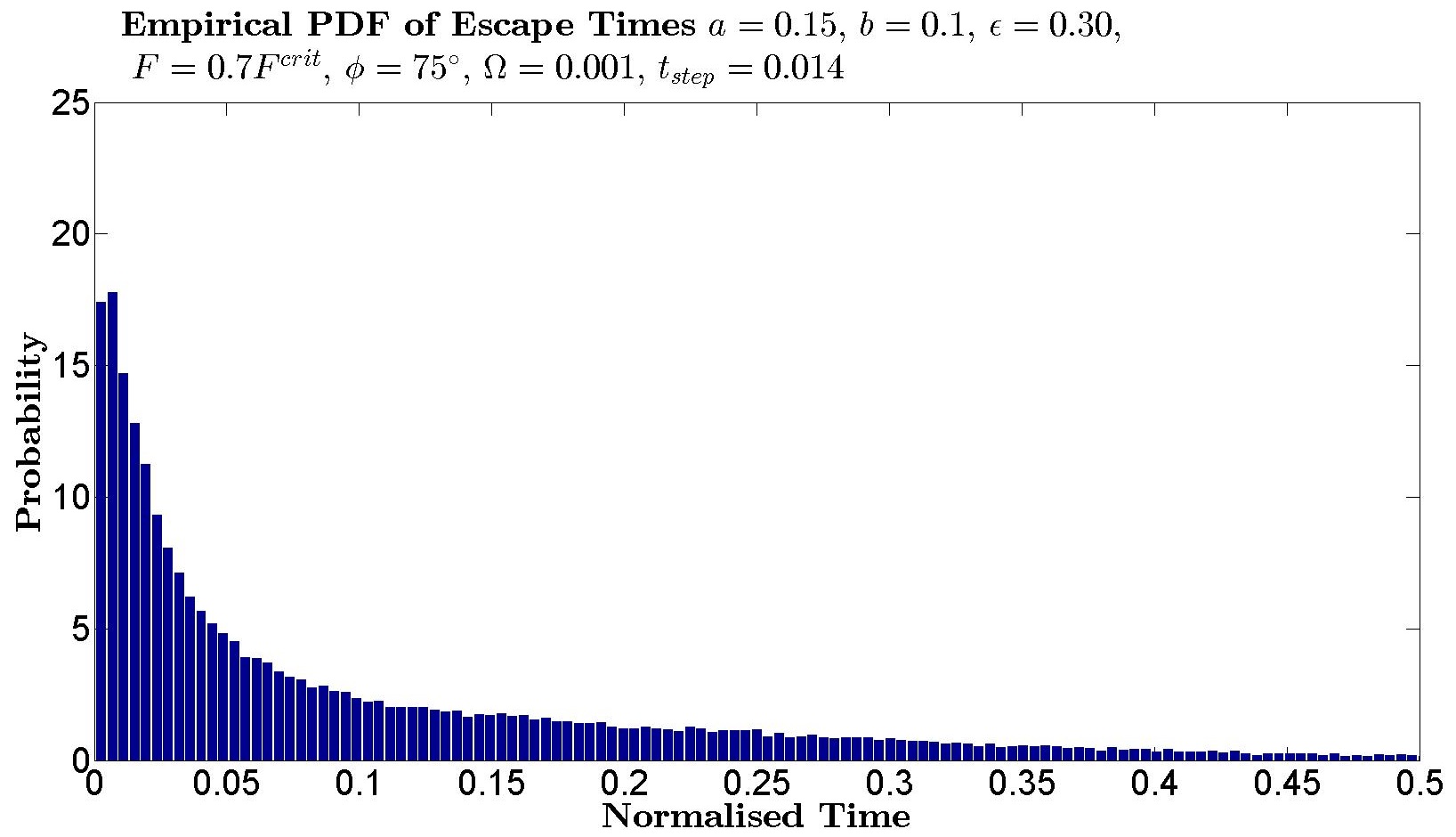}}
\caption{Here 56244 transitions were used. 
The averaged measured escape time is $ 0.1064T$.
The radius used was $R=0.19$.  
}
\label{chap_7_g75_p75_e30_pdf}
\end{figure}

\noindent The higher the noise level $\epsilon$ is the more susceptible to errors would the Euler method be. 
This is why the highest level of noise $\epsilon=0.30$ is chosen for these examples. 
One may argue that having more transitions would give a better measurement of the escape times in Figure \ref{chap_7_g75_p75_e30_pdf}. 
But the difference in real time between the measured averaged escape times is only 54 seconds out of a period of $T=2\pi/\Omega=6283$ seconds.

\chapter{Simulations, Results and Analysis}
\label{sum_chap_results}
This Chapter presents the main results of this thesis. 
The six measures $M_1$, $M_2$, $M_3$, $M_4$, $M_5$, $M_6$, the distributions of escape times and the newly developed conditional Kolmogorov-Smirnov test are used to analyse simulations of the SDE with the Mexican Hat Toy Model used as the potential. The six measures are shown to be insensitive to the saddles changing from alternating to synchronised. 
This is shown to be due to the fact that the invariant measure is constant for synchronised saddles. 
The distribution of escape times shows new signatures as the saddles change from alternating to synchronised and the conditional Kolmogorov-Smirnov test is demonstrated to be an appropriate way to analyse the escape times collected from many transitions.

We simulate a series of stochastic trajectories for the Mexican Hate Toy Model and analyse them. 
We remind ourselves of the unperturbed Mexican Hat potential
\begin{align*}
V_0(x,y)=\frac{1}{4}r^4-\frac{1}{2}r^2-ax^2+by^2
\quad \text{where} \quad r=\sqrt{x^2+y^2}
\end{align*}
and the SDE we want to simulate is 
\begin{align*}
dx&=\left[-\frac{\partial V_0}{\partial x}+F_x\cos \Omega t \ \right]dt+\epsilon \ dw_x\\
dy&=\left[-\frac{\partial V_0}{\partial y}+F_y\cos \Omega t \ \right]dt+\epsilon \ dw_y
\end{align*}
where $F_x$ and $F_y$ are the $x$ and $y$ components of the forcing, $\Omega$ is the forcing frequency, $\epsilon$ is the noise level and $w_x$ and $w_y$ are two independent Wiener processes. 
We can define the magnitude and angle of the forcing by 
\begin{align*}
F=\sqrt{F_x^2+F_y^2}
\quad \text{and} \quad 
\phi=\tan^{-1}\left(\frac{F_y}{F_x}\right). 
\end{align*}
This means the SDE can be written alternatively as 
\begin{align*}
dx&=\left[-\frac{\partial V_0}{\partial x}+F\cos\phi\cos \Omega t \ \right]dt+\epsilon \ dw_x\\
dy&=\left[-\frac{\partial V_0}{\partial y}+F\sin\phi\cos \Omega t \ \right]dt+\epsilon \ dw_y.
\end{align*}
The critical forcing is defined by (see Equation \ref{chap_5_critical_forcing})
\begin{align*}
F^{crit}=\min\left\{F_x^{sad}, F_x^{crit}, F_y^{sad}, F_y^{crit}\right\}. 
\end{align*}

\section{Details of the Simulations}
The Euler method was used to simulate this SDE with the following  parameters being fixed at the following values 
\begin{align*}
a=0.15 \quad
b=0.1 \quad 
F=0.7F^{crit} \quad 
\Omega=0.001.
\end{align*}
The angle of the forcing $\phi$ and the noise level $\epsilon$ were varied. 
The values used were 
\begin{align*}
\epsilon=0.15, 0.16, \ldots, 0.30 
\quad \text{and} \quad 
\phi=0^\circ,75^\circ,78^\circ,81^\circ,84^\circ,87^\circ,90^\circ.
\end{align*}
The averaged diffusion trajectories $\langle x_t \rangle$  and $\langle y_t \rangle$ were collected. 
The averaged Markov Chain $\langle Y_t^\epsilon \rangle$ and the averaged Out-of-Phase Markov Chain $\langle \overline{Y}_t^\epsilon \rangle$ were collected as well. 
This would allow for the calculation of the invariant measures $\overline{\nu}_-(\cdot)$ and $\overline{\nu}_+(\cdot)$. 
The time coordinates of the entrance and exit to and from the left and right wells were also collected. 
This would allow for the calculation of the escape times. 
We use the following values for the time step and the radius around the wells
\begin{align*}
t_{step}=0.014 
\quad \text{and} \quad 
R=0.19.
\end{align*}
Note that the period of the forcing is denoted by 
\begin{align*}
T=\frac{2\pi}{\Omega}.
\end{align*}
The averaged trajectories were simulated by taking the averaged of 200 realisations. 
Each realisation was 30 periods long, that is a trajectory over the interval $[0,30T]$. 
The initial value of the state probabilities were set at 
\begin{align*}
\nu_-(0)=\nu_+(0)=\frac{1}{2}
\end{align*}
which assists in giving a faster convergence to the invariant measures (see Theorems \ref{chap_4_thm:equal},
\ref{chap_4_thm:not_equal},  
and \ref{chap_4_thm:not:equal_contin}). 
We should also stress that a lot of the data and results presented in this Chapter is just a selection out of a much wider range of results. 
All 112 combinations of the parameters were simulated and analysed.

\section{Six Measures Analysis}
\label{conclusion_six_measures}
The six measures are calculated for the diffusion and Markov Chain case for all angles of the forcing $\phi$ and all noise levels $\epsilon$ used in the simulations. 
When $\phi=90^\circ$ the wells were moving up and down but they were always at the same height as each other. 
The distance from either wells to the saddles, which is a gateway for escape, is the same in both wells at all times. 
This means the $\phi=90^\circ$ can be modelled by a synchronised Markov Chain with $p=q$. 
The invariant measures for the $\phi=90^\circ$ case as predicted by Corollary \ref{cor_half_discrete} and \ref{cor_half_continuous} is $\overline{\nu}_-=\overline{\nu}_+=\frac{1}{2}$. 
The Fourier Transform of the averaged Markov Chain is predicted to be zero by Corollary \ref{transform_discrete_chain} and \ref{transform_continuous_chain}. 
This predicts the six measures at $\phi=90^\circ$ to be 
\begin{align*}
M_1&=0\\
M_2&=0\\
M_3&=\int_0^T \left\langle Y^\epsilon_t \right\rangle^2 dt
=\int_0^T\left(\nu_+(t)-\nu_-(t)\right)^2\,dt
=0\\
M_4&=\int_0^T \left\langle \overline{Y}^\epsilon_t \right\rangle dt\\
&=\int_0^{T/2}
0\times\nu_-(t)+1\times\nu_+(t)\,dt
+
\int_{T/2}^T
1\times\nu_-(t)+0\times\nu_+(t)\,dt\\
&=\frac{1}{2}T\\
M_5&=\int_0^T
\phi^-(t)\ln\left(\frac{\phi^-(t)}{\overline{\nu}_-(t)}\right)+
\phi^+(t)\ln\left(\frac{\phi^+(t)}{\overline{\nu}_+(t)}\right)
dt\\
&=\int_0^{T/2}
\ln\left(\frac{1}{\overline{\nu}_-(t)}\right)\,dt
+
\int_{T/2}^T
\ln\left(\frac{1}{\overline{\nu}_+(t)}\right)
\,dt\\
&=+T\ln(2)\\
M_6&=\int^T_0
-\overline{\nu}_-(t)\ln\overline{\nu}_-(t)
-\overline{\nu}_+(t)\ln\overline{\nu}_+(t)\,
dt\\
&=+T\ln(2).
\end{align*}
Note that $\ln(2)=0.6931\approx0.7$. 
Notice that for very low noise level $\epsilon\approx0$ the probabilities of escape from either well is so small it may be approximately modelled by a synchronised Markov Chain with $p\approx q$. 
The results below confirm the predictions for the case of $\phi=90^\circ$.

\begin{figure}[H]
\centerline{\includegraphics[scale=0.35]{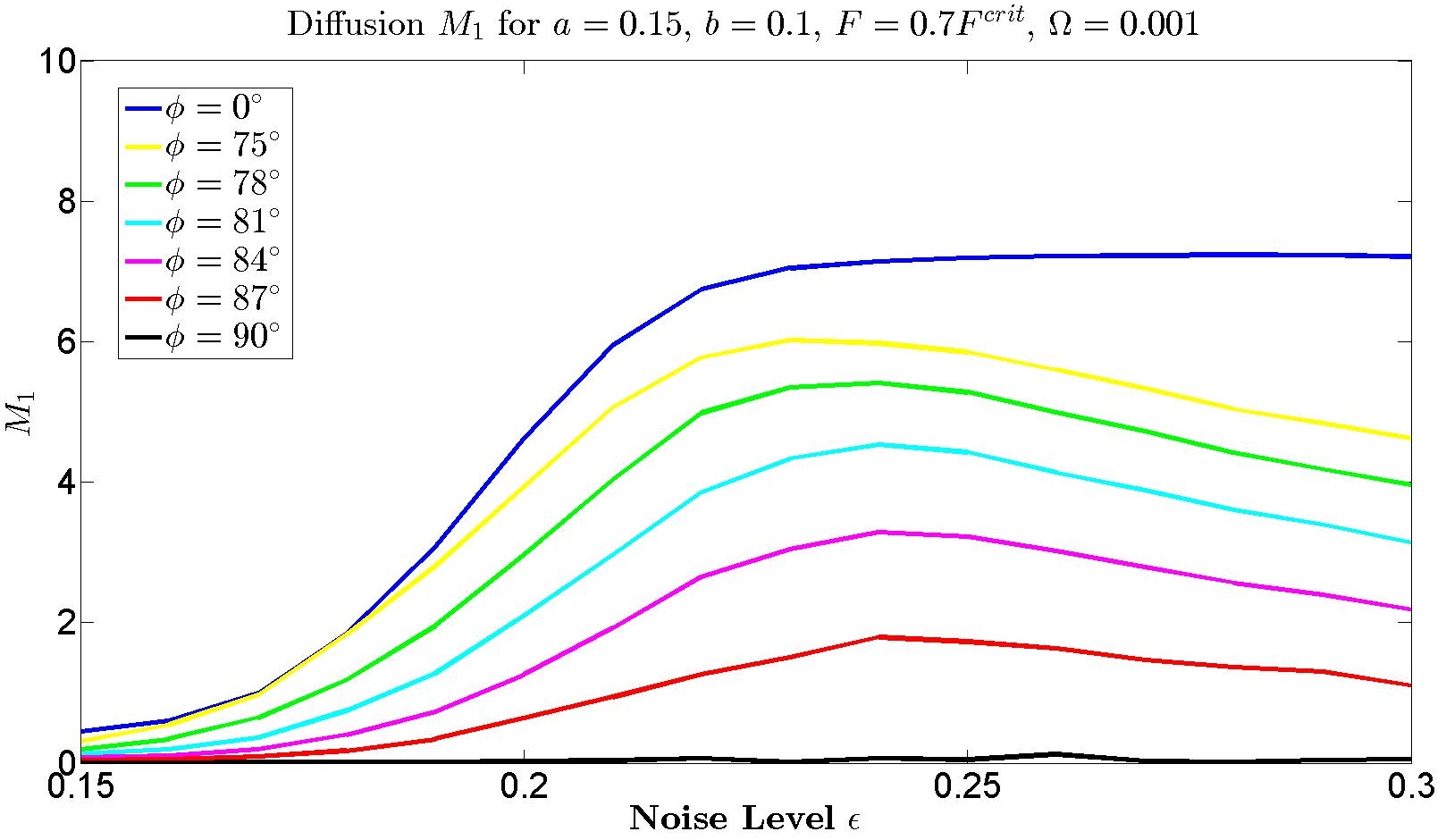}}
\caption{The measure $M_1$ for the diffusion case for various angles and noise levels.}
\label{chap_8_g75_diff_m_1_x_single_measures}
\end{figure}

\begin{figure}[H]
\centerline{\includegraphics[scale=0.35]{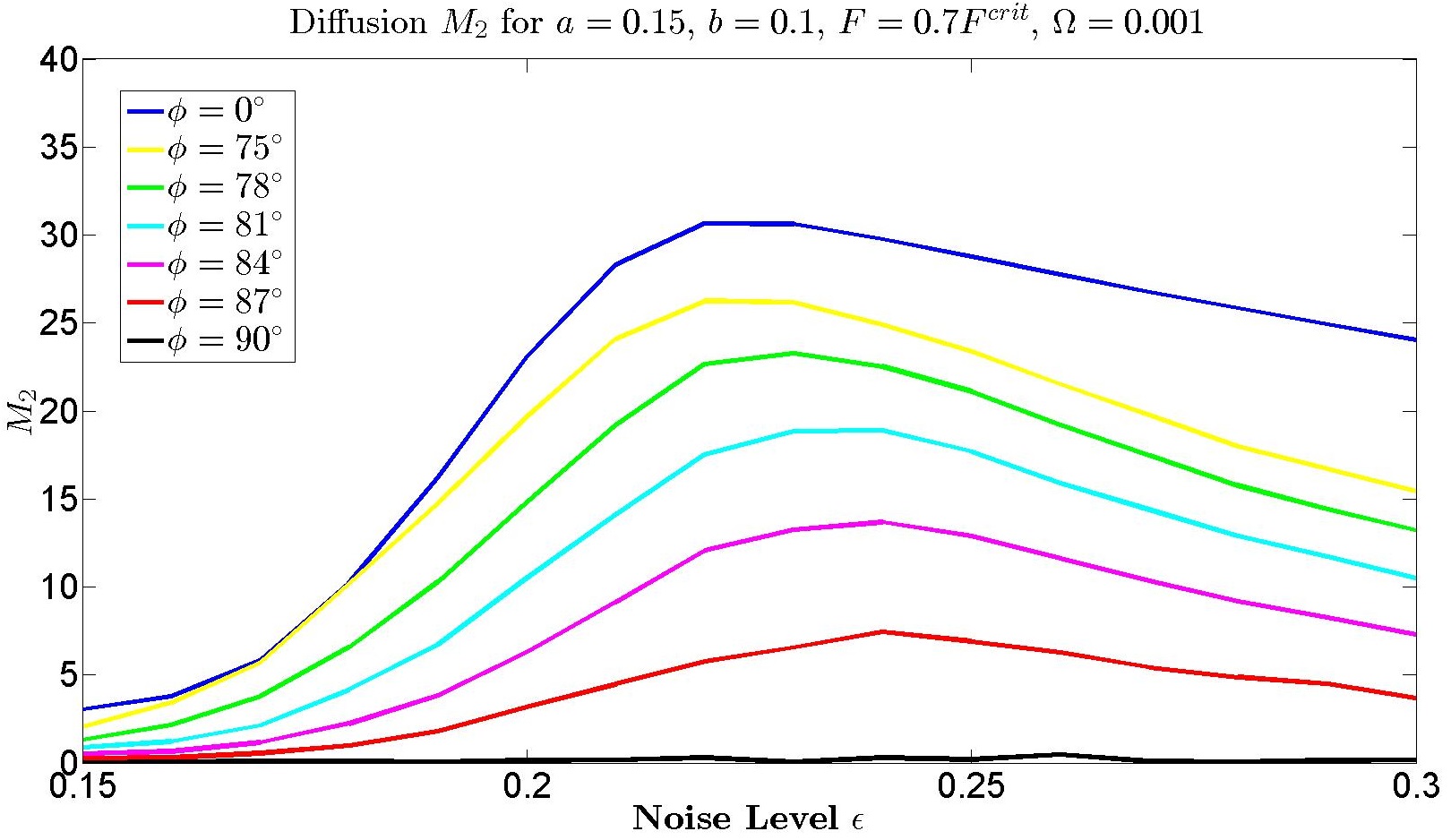}}
\caption{The measure $M_2$ for the diffusion case for various angles and noise levels.}
\label{chap_8_g75_diff_m_2_x_single_measures}
\end{figure}

\begin{figure}[H]
\centerline{\includegraphics[scale=0.35]{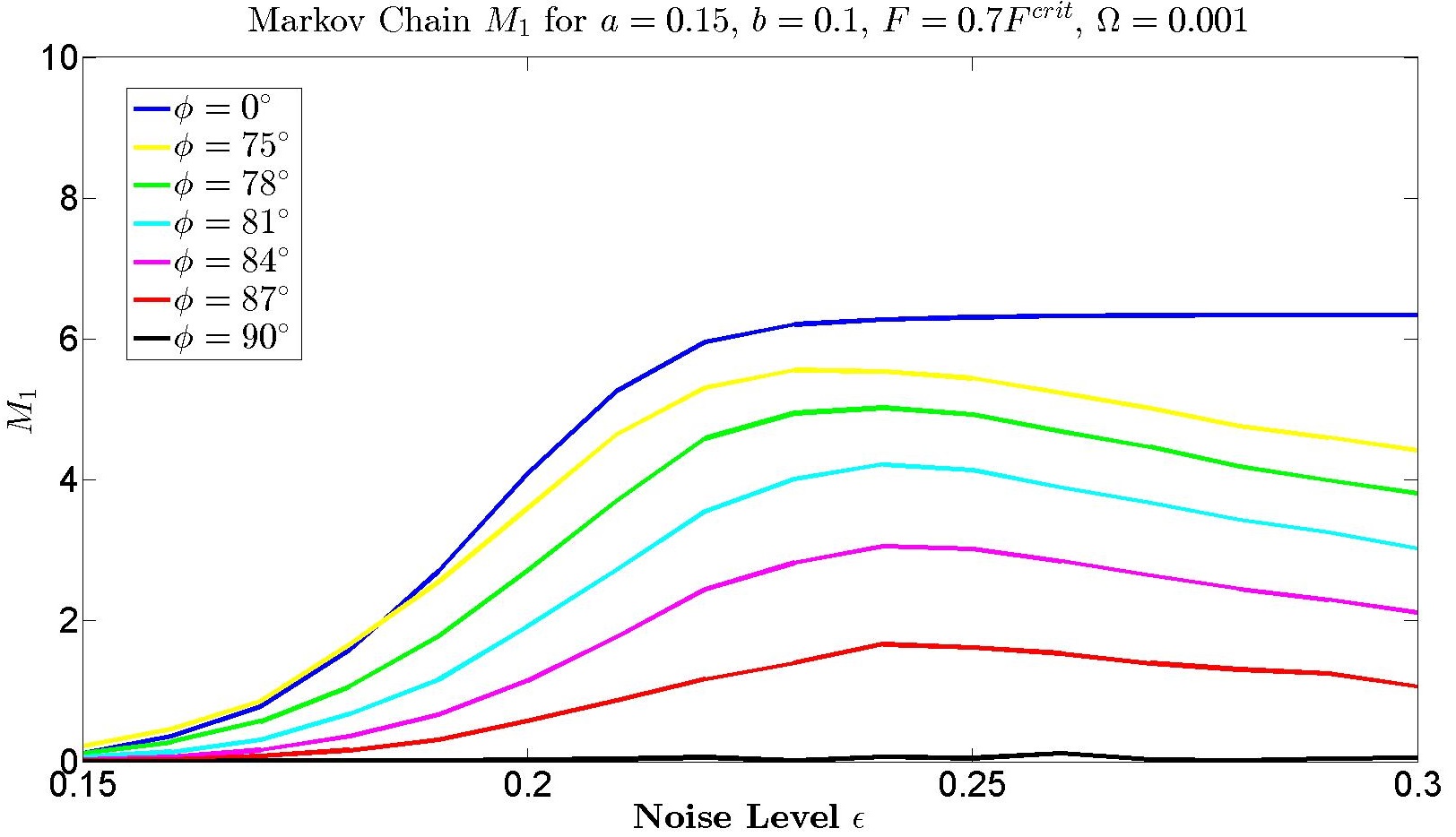}}
\caption{The measure $M_1$ for the Markov Chain for various angles and noise levels.}
\label{chap_8_g75_markov_m_1_single_measures}
\end{figure}

\begin{figure}[H]
\centerline{\includegraphics[scale=0.35]{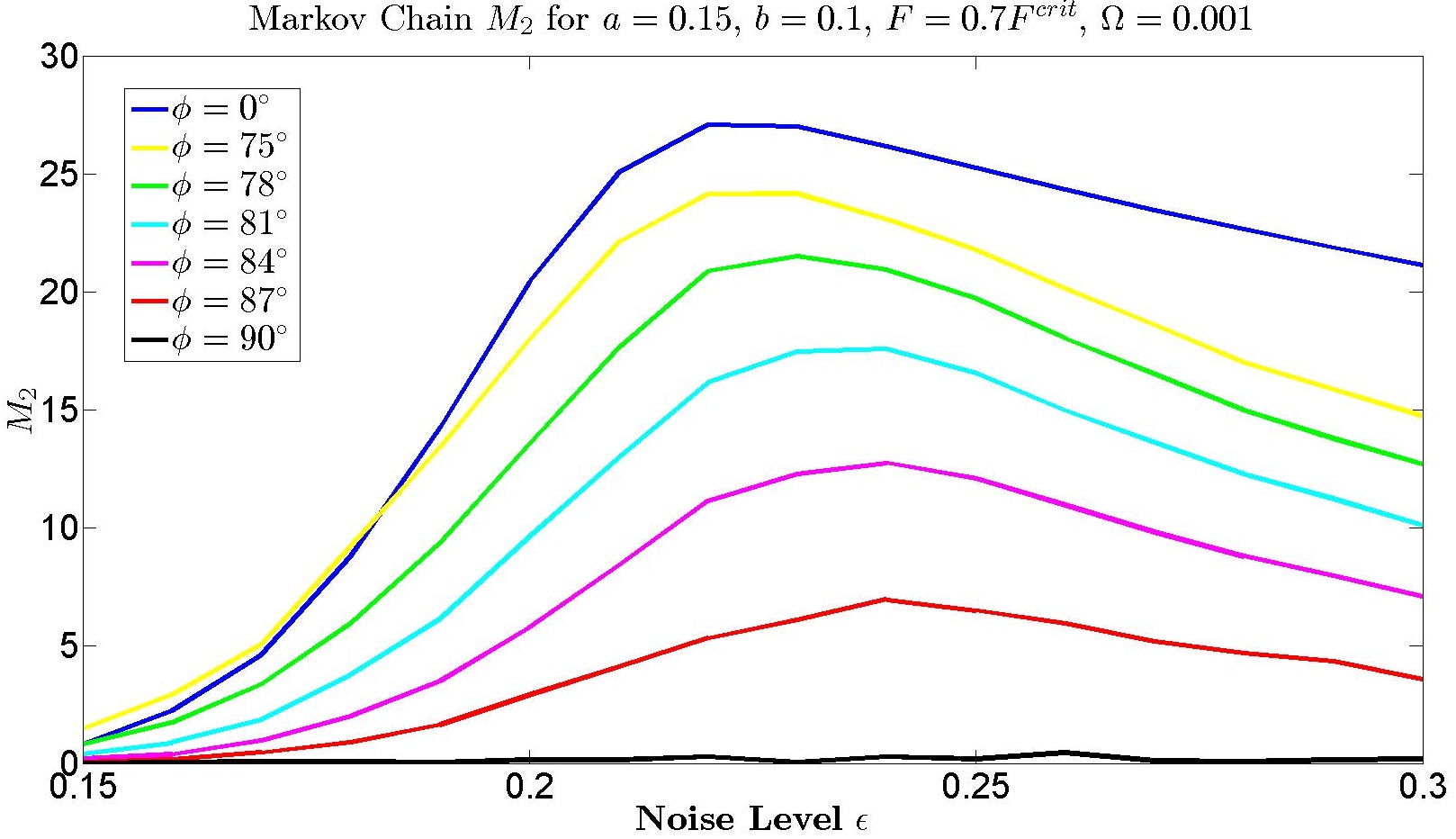}}
\caption{The measure $M_2$ for the Markov Chain for various angles and noise levels.}
\label{chap_8_g75_markov_m_2_single_measures}
\end{figure}

\begin{figure}[H]
\centerline{\includegraphics[scale=0.35]{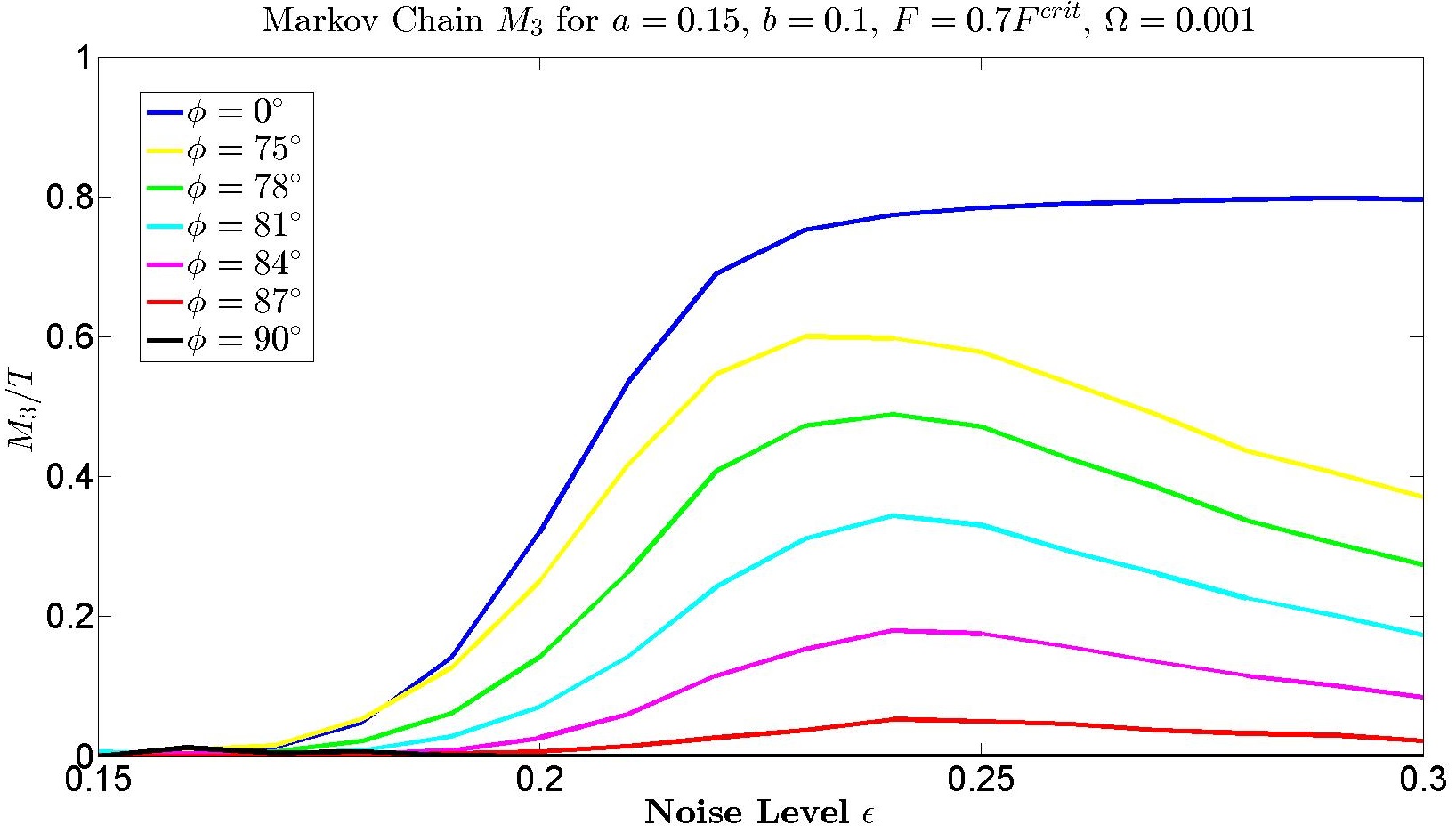}}
\caption{The measure $M_3$ for the Markov Chain for various angles and noise levels.}
\label{chap_8_g75_markov_m_3_measures}
\end{figure}

\begin{figure}[H]
\centerline{\includegraphics[scale=0.35]{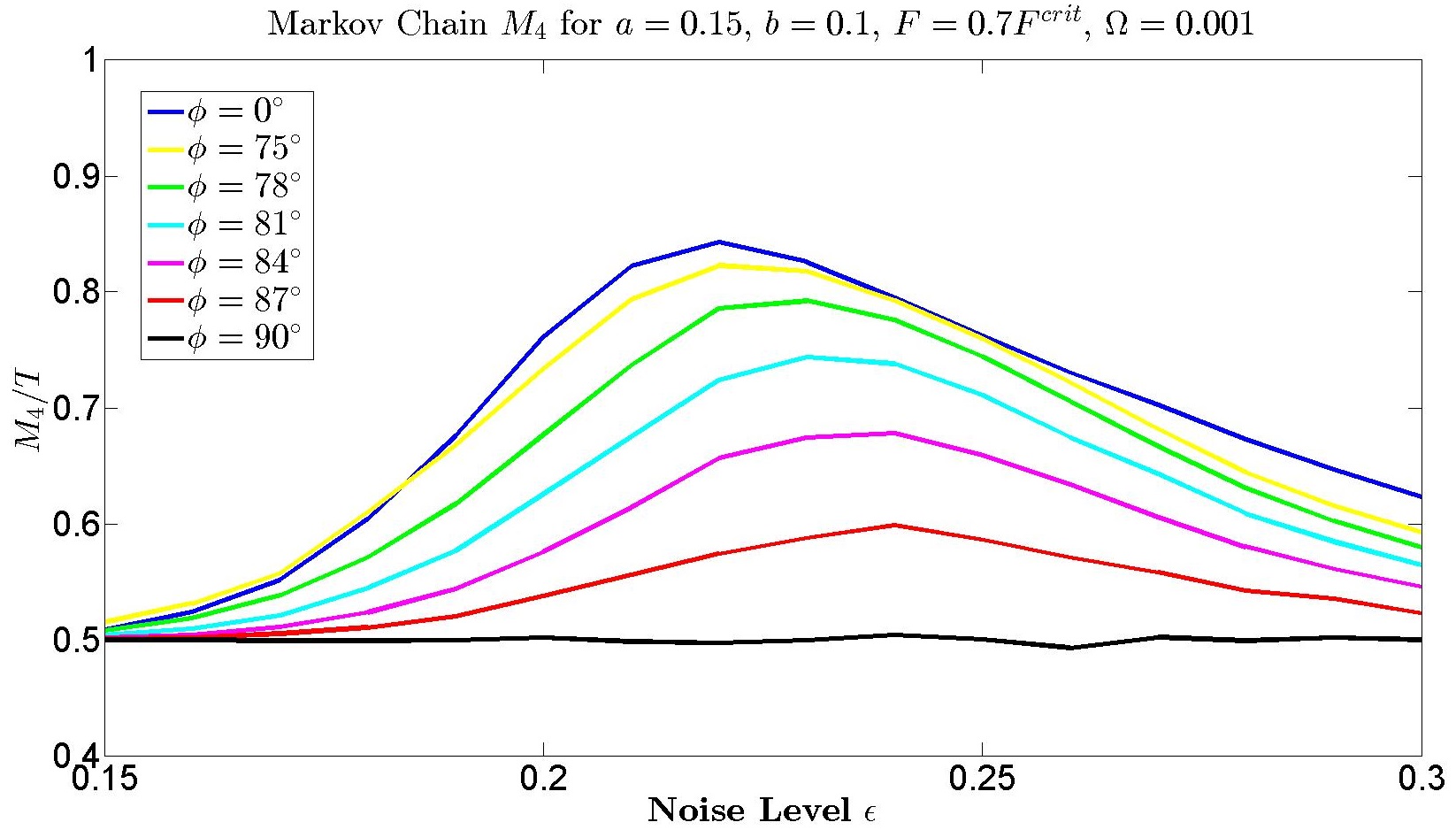}}
\caption{The measure $M_4$ for the Markov Chain for various angles and noise levels.}
\label{chap_8_g75_markov_m_4_measures}
\end{figure}

\begin{figure}[H]
\centerline{\includegraphics[scale=0.35]{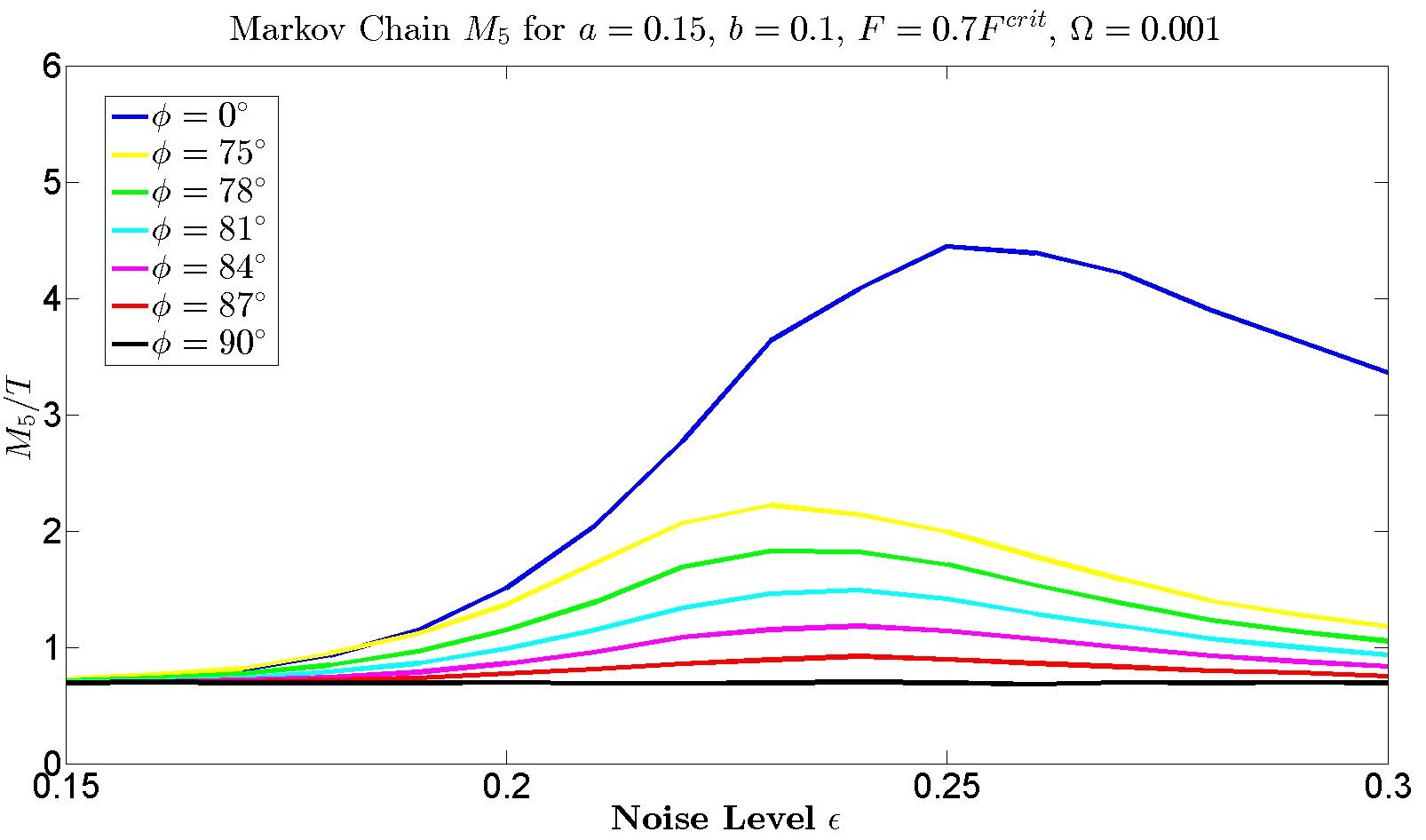}}
\caption{The measure $M_5$ for the Markov Chain for various angles and noise levels.}
\label{chap_8_g79_markov_m_5_measures}
\end{figure}

\begin{figure}[H]
\centerline{\includegraphics[scale=0.35]{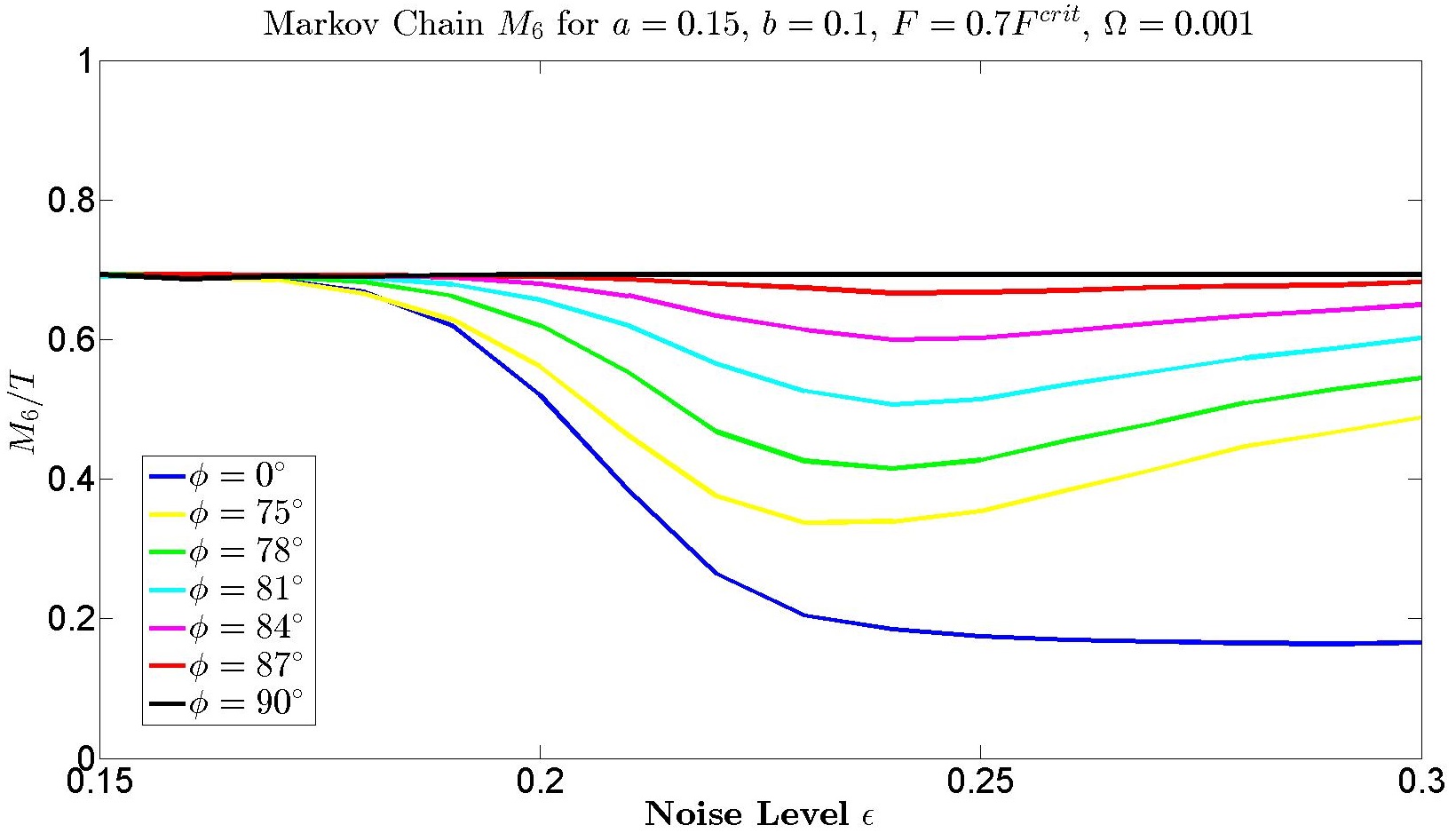}}
\caption{The measure $M_6$ for the Markov Chain for various angles and noise levels.}
\label{chap_8_g75_markov_m_6_measures}
\end{figure}

\subsection{Interpretation of the Six Measures Analysis}
The six measures $M_1$, $M_2$, $M_3$, $M_4$, $M_5$ and $M_6$ were plotted as a function of the noise level $\epsilon$. 
The six measures show a regular systematic behaviour in the angle $\phi$. 
The shape of the graphs of the six measures  were very similar for all the angles. 
As the angle  increased from $\phi=0^\circ$ to $\phi=90^\circ$ the six measures  gradually tended to being nearly constant in $\epsilon$. 

This effect can be explained with the invariant measures. 
When $\phi=0^\circ$ the probabilities for escaping from left to right $p_{-1+1}$ was different to  the probabilities for escaping from right to left $p_{+1-1}$.
But in the $\phi=90^\circ$ case they are the same, that is
\begin{align*}
\phi=0^\circ\,\,\, \quad p_{-1+1}&\neq p_{+1-1}\\
\phi=90^\circ \quad p_{-1+1}&= p_{+1-1}.
\end{align*}
The is can be understood geometrically. 
For $\phi=0^\circ$ we have $F_x\neq0$ and $F_y=0$. 
The two wells in the Mexican Hat potential move up and down and are alternating with each other. 
When one well is high the other is low. 
For $\phi=90^\circ$ we have $F_x=0$ and $F_y\neq0$.
The two wells are always at the same height as each other and the distance to the saddles (which is a gateway to escape) is also the same in both wells. 

Recall our discussions on the Markov Chain in Chapter \ref{chapter_oscil_times}. 
The $p$ is related to left to right escape $p_{-1+1}$
and $q$ was related to right to left escape $p_{+1-1}$. 
For $\phi=0^\circ$ the Markov Chain can be modelled with $p\neq q$ and for $\phi=90^\circ$ the Markov Chain can be modelled with $p=q$. 
In the case of $\phi=0^\circ$ the invariant measure was cyclically changing in time. 
In the case of $\phi=90^\circ$ the invariant measure was constant at $\overline{\nu}_-(\cdot)=\overline{\nu}_+(\cdot)=\frac{1}{2}$. 
This explains why the six measures $M_1$, $M_2$, $M_3$, $M_4$, $M_5$ and $M_6$ were nearly constant for angle $\phi=90^\circ$. 
As $\phi$ changed from $\phi=0^\circ$ to $\phi=90^\circ$, the Markov Chain changed from being modelled by $p\neq q $ to being modelled by $p=q$. 
This explains the change in the six measures tending to being constant in $\epsilon$ as $\phi$ was varied. 
The six measures can be thought of as a way of measuring how far away the invariant measures are from being constant. 
If the invariant measures are constant then the six measures will also be constant.\footnote{See Appendix \ref{appendix_six_measures} for how $M_5$ and $M_6$ were numerically calculated. The ideas were not that trivial.}

For fixed $\phi$ near $\phi=90^\circ$ there is no pronounced maximum of any measure for varying $\epsilon$. 
Hence the six measures indicate the absence of a pronounced stochastic resonance near $\phi=90^\circ$. 
But consider the trajectories at a range of angles. 

\begin{figure}[H]
\centerline{\includegraphics[scale=0.35]{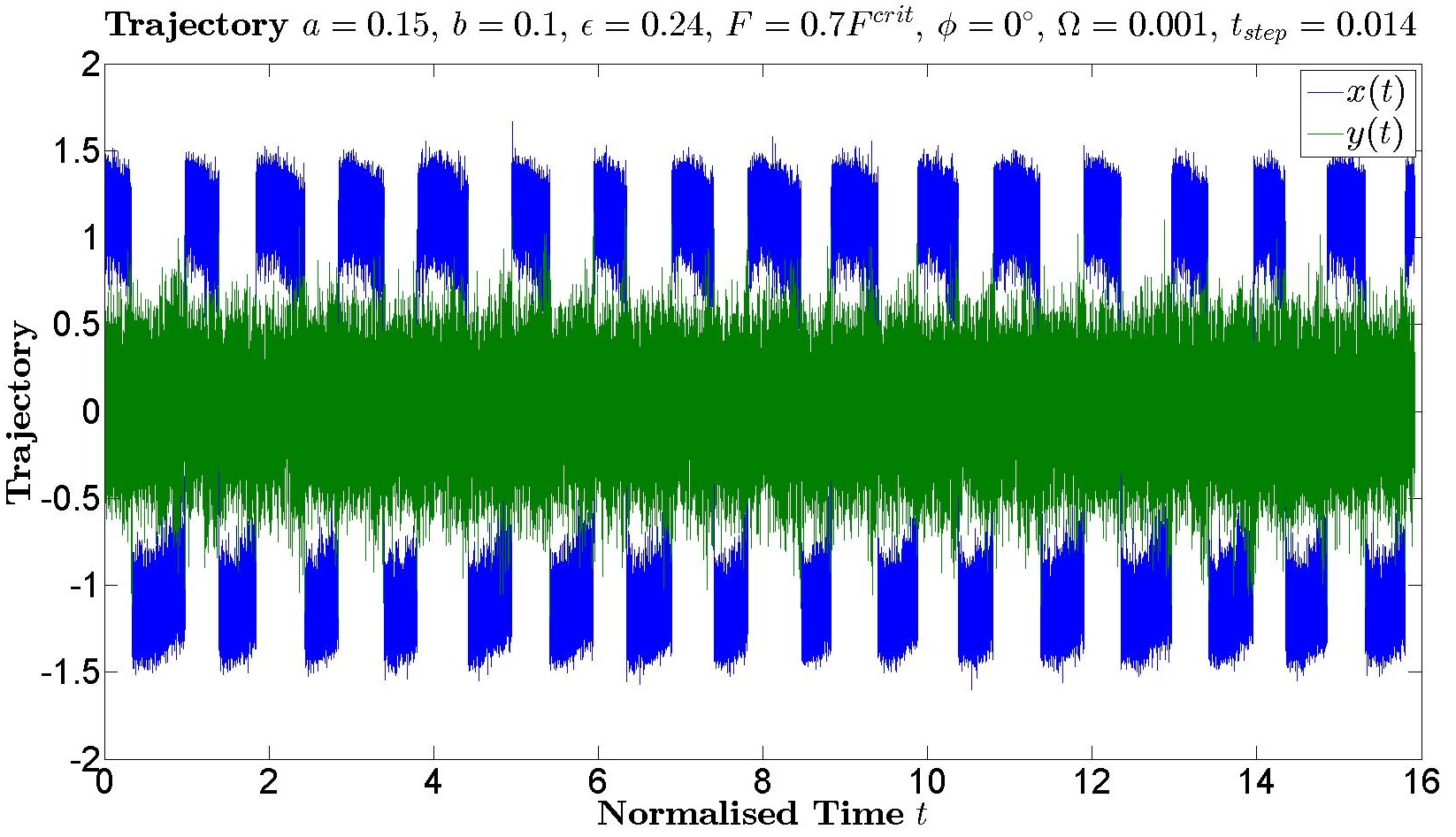}}
\caption{The blue trajectory is $x(t)$ and the green trajectory is $y(t)$.}
\label{chap_8_path_p0}
\end{figure}

\begin{figure}[H]
\centerline{\includegraphics[scale=0.35]{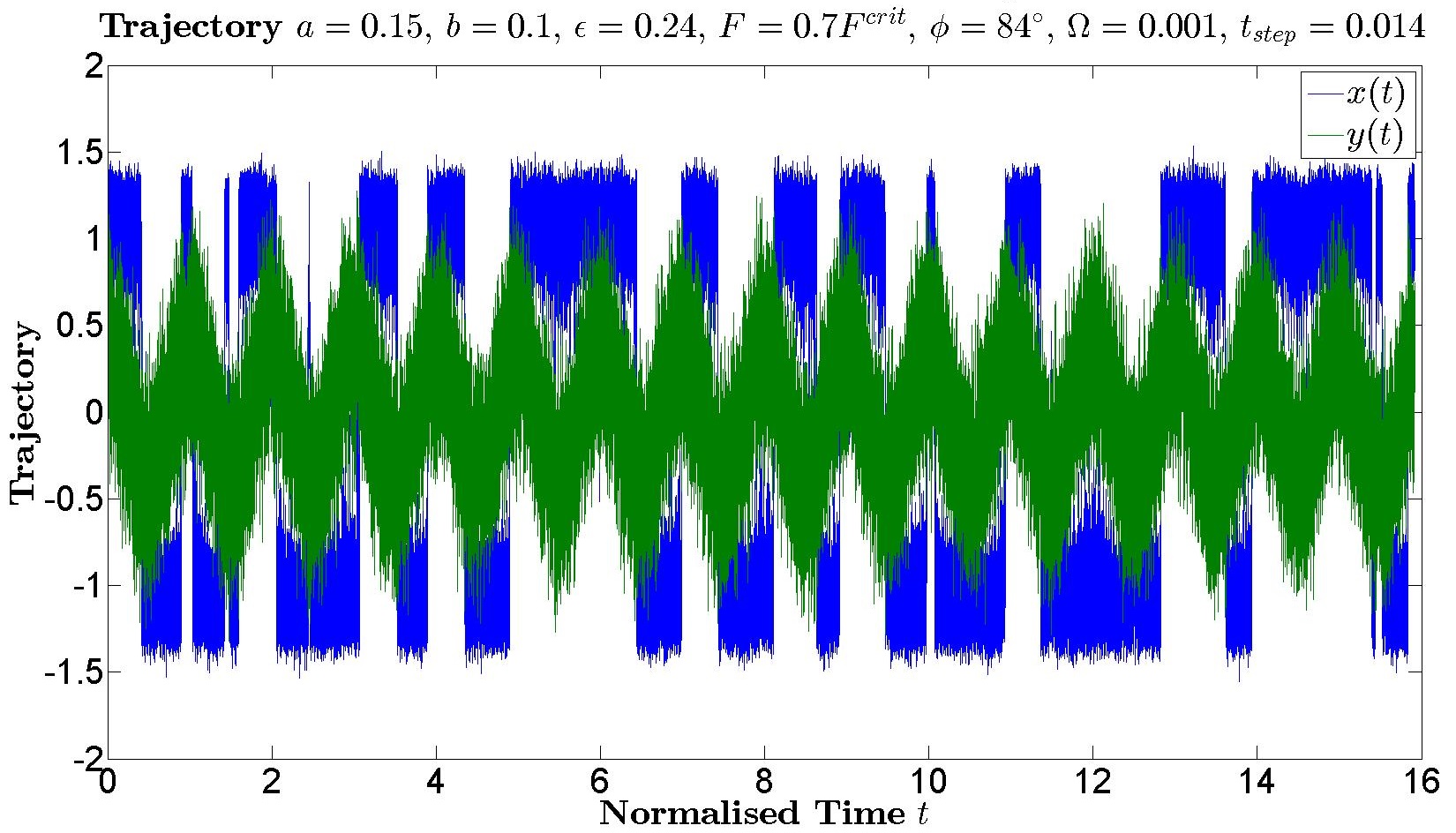}}
\caption{The blue trajectory is $x(t)$ and the green trajectory is $y(t)$.}
\label{chap_8_path_p84}
\end{figure}

\begin{figure}[H]
\centerline{\includegraphics[scale=0.35]{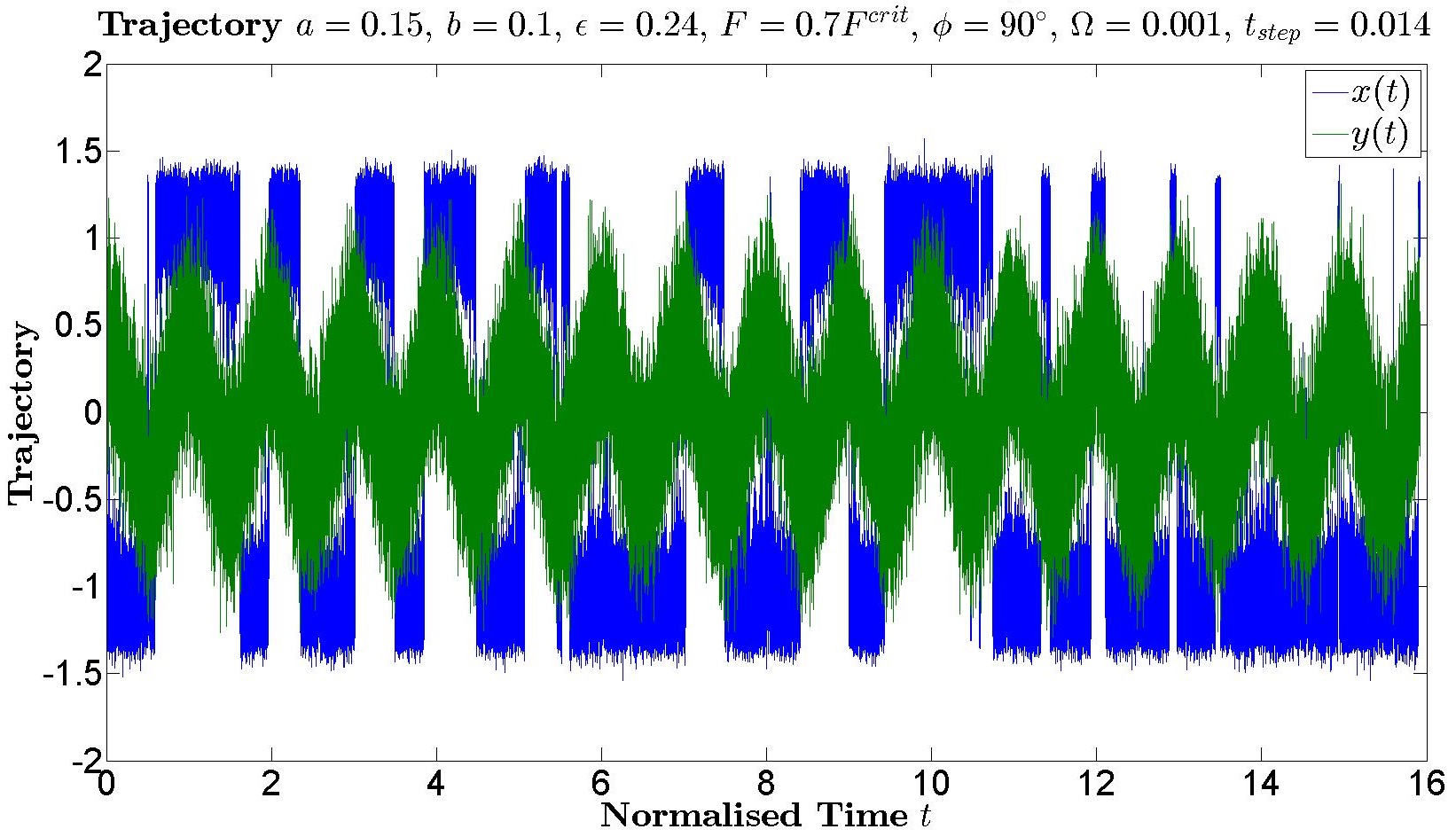}}
\caption{The blue trajectory is $x(t)$ and the green trajectory is $y(t)$.}
\label{chap_8_path_p90}
\end{figure}

\noindent When $\phi=0^\circ$ the $x(t)$ shows quasi-deterministic behaviour. 
The transitions are very regular and $y(t)$ fluctuates around zero. 
As the angle varies the transitions become less regular and $y(t)$ starts to oscillate. 
This suggests that there is some regularity in the behaviour of the trajectories but the six measures are not detecting it. 
Further studies with the escape times would tell us more.

\section{Escape Time and Conditional KS Test Analysis}
We remind ourselves of the PDF of escape times and the way the conditional KS test can be applied in our context. 
The conditional PDF of the escape times are
\begin{align*}
p_{-}(t,u)&=R_{-1+1}(t)\exp\left\{-\int^t_uR_{-1+1}(s)\,ds\right\}\\[0.5em] 
p_{+}(t,u)&=R_{+1-1}(t)\exp\left\{-\int^t_uR_{+1-1}(s)\,ds\right\}
\end{align*} 
where $R_{-1+1}$ and $R_{+1-1}$ are the Kramers' escape rate from left to right and right to left. 
In the case of $p_-(t,u)$, $t$ is the time coordinate of escape from the left well and $u$ is the time coordinate of entrance into the left well. 
In the case of $p_+(t,u)$, $t$ is the time coordinate of escape from the right well and $u$ is the time coordinate of entrance into the right well.
If we do not differentiate between escaping from the left or right then the PDF for an escape time $t$ is 
(note that $t$ here is an escape time as it is and not a time coordinate)
\begin{align*}
p_{tot}(t)=
\frac{1}{2}
\int_{0}^{T}p_-(t+u,u)m_-(u)+p_+(t+u,u)m_+(u)\,du
\end{align*}
where $m_-(\cdot)$ and $m_+(\cdot)$ are PDFs of the time of entrance into the left and right well respectively. 
We do not have  explicit expressions for  $m_-(\cdot)$ and $m_+(\cdot)$.
The  $p_{tot}(t)$ is approximated by
\begin{align*}
p_{tot}(t)\approx p_+(t,0).
\end{align*}
The times it took to escape from both the left or right wells are plotted in  histograms. 
This is an empirical approximation to the PDF $p_{tot}(t)\approx p_+(t,0)$. 
A selection of some of the  results are given below for various angles of the forcing $\phi$ and noise level $\epsilon$. 
They are examples of the Singles, Intermediate and Double Frequencies which we will explain later. 
Note that the escape times are given in units of normalised time, which is in the number of periods $T$. 

\begin{figure}[H]
\centerline{\includegraphics[scale=0.34]{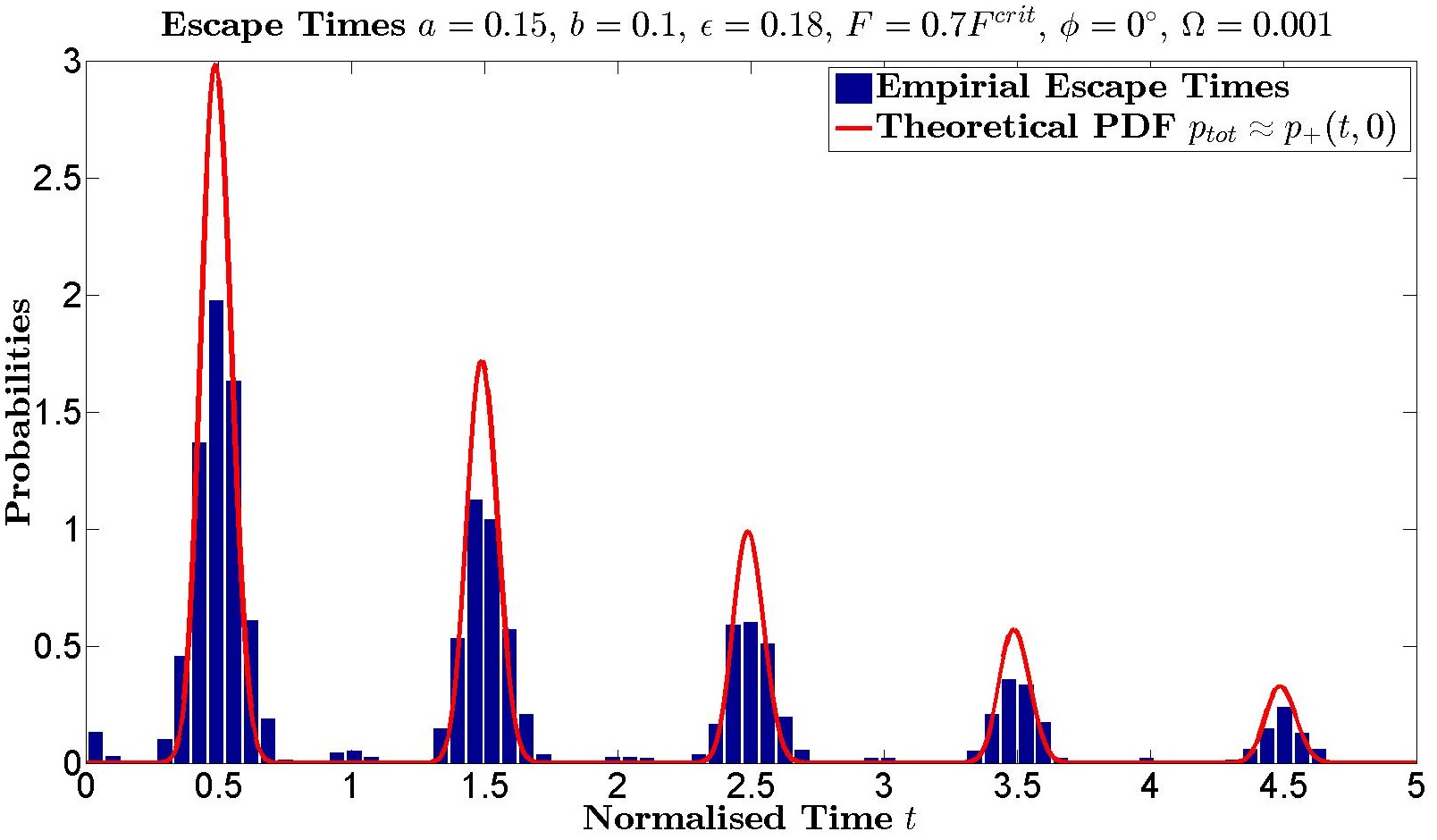}}
\caption{This is an example of the Single Frequency.
}
\label{chap_8_g75_p0_e18_pdf}
\end{figure}

\begin{figure}[H]
\centerline{\includegraphics[scale=0.34]{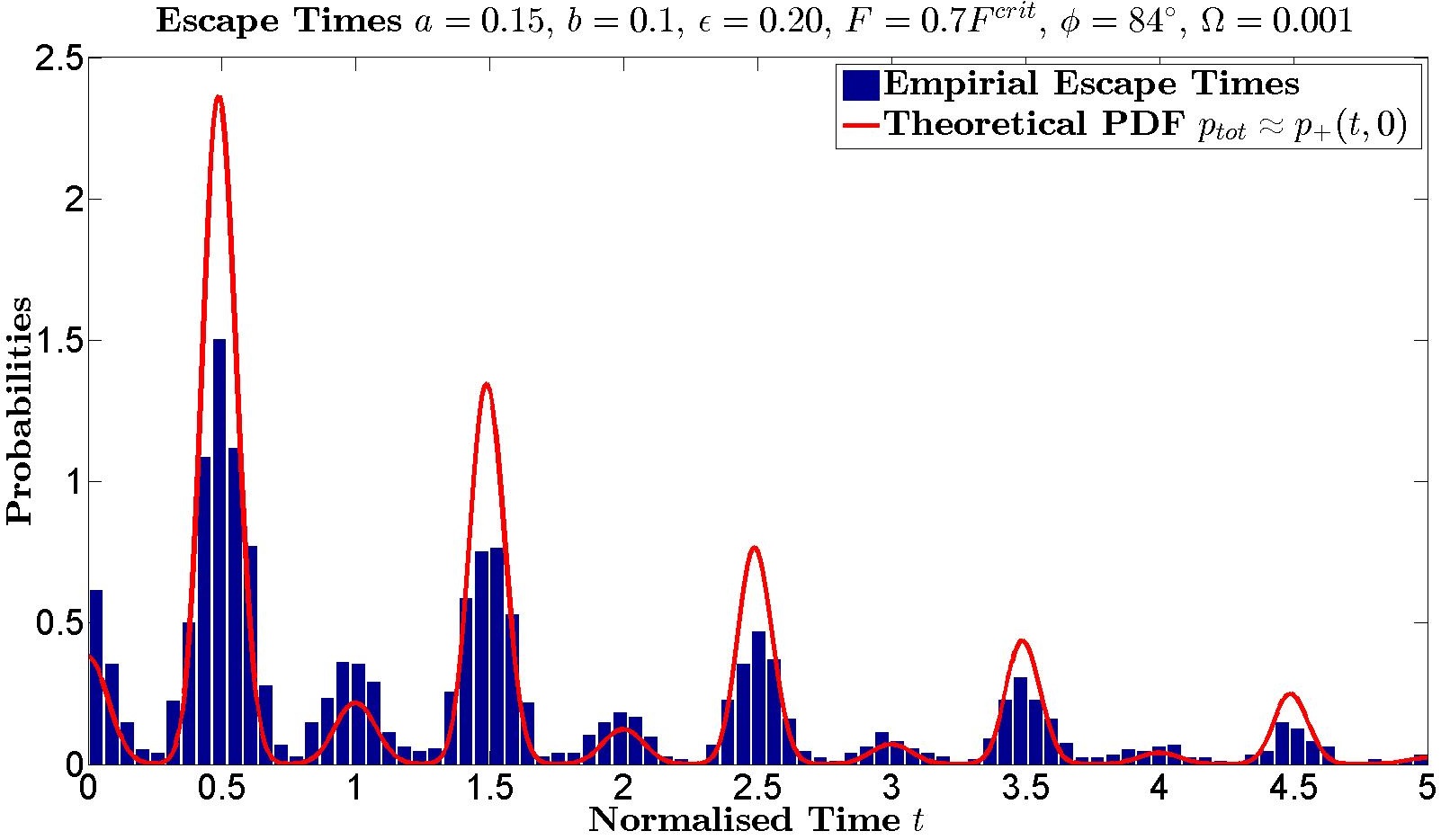}}
\caption{
This is an example of the Intermediate Frequency.
}
\label{chap_8_g75_p84_e20_pdf}
\end{figure}

\begin{figure}[H]
\centerline{\includegraphics[scale=0.34]{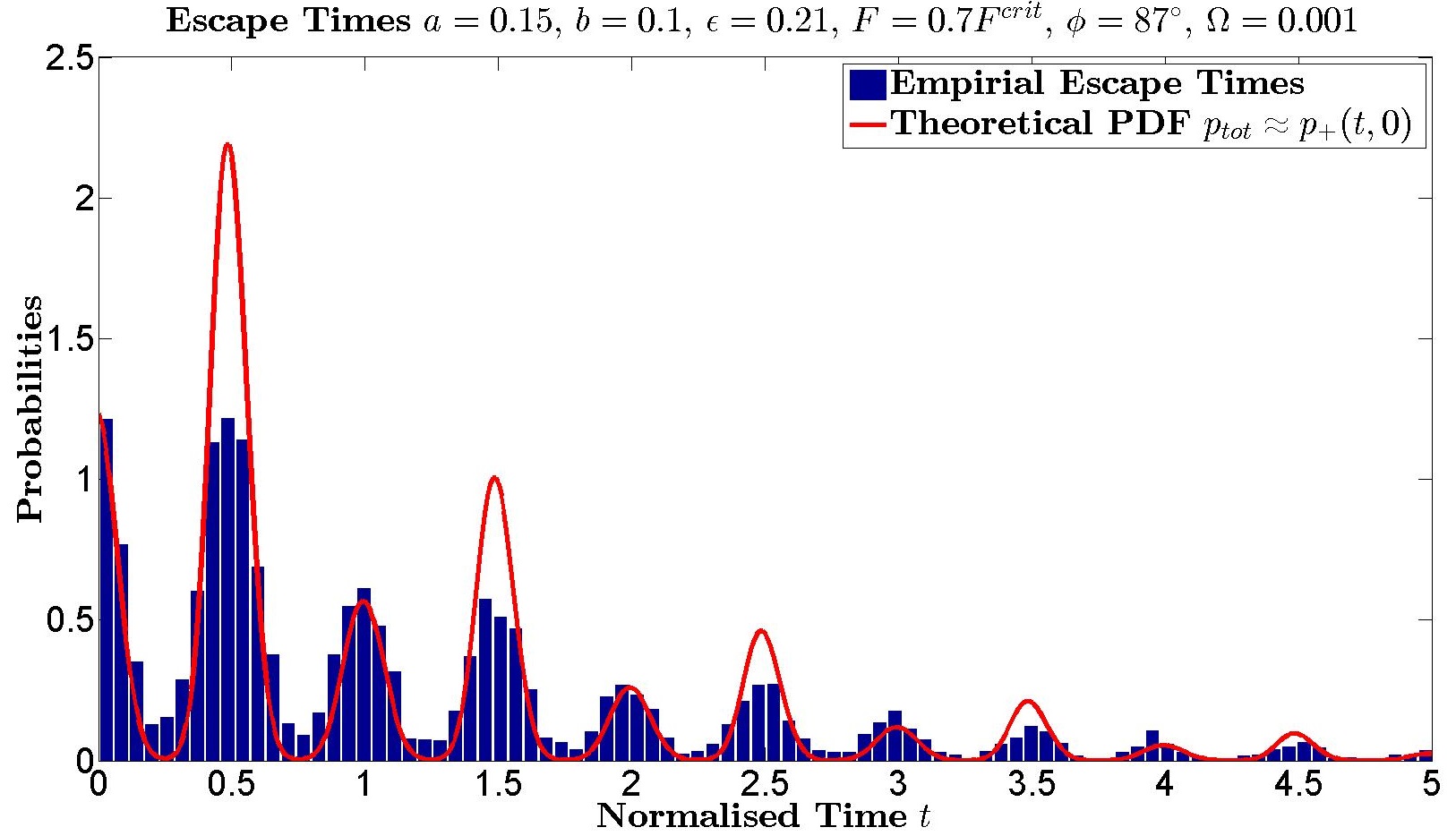}}
\caption{
This is an example of the Intermediate Frequency tending closer to the Double Frequency.
}
\label{chap_8_g75_p87_e21_pdf}
\end{figure}

\begin{figure}[H]
\centerline{\includegraphics[scale=0.34]{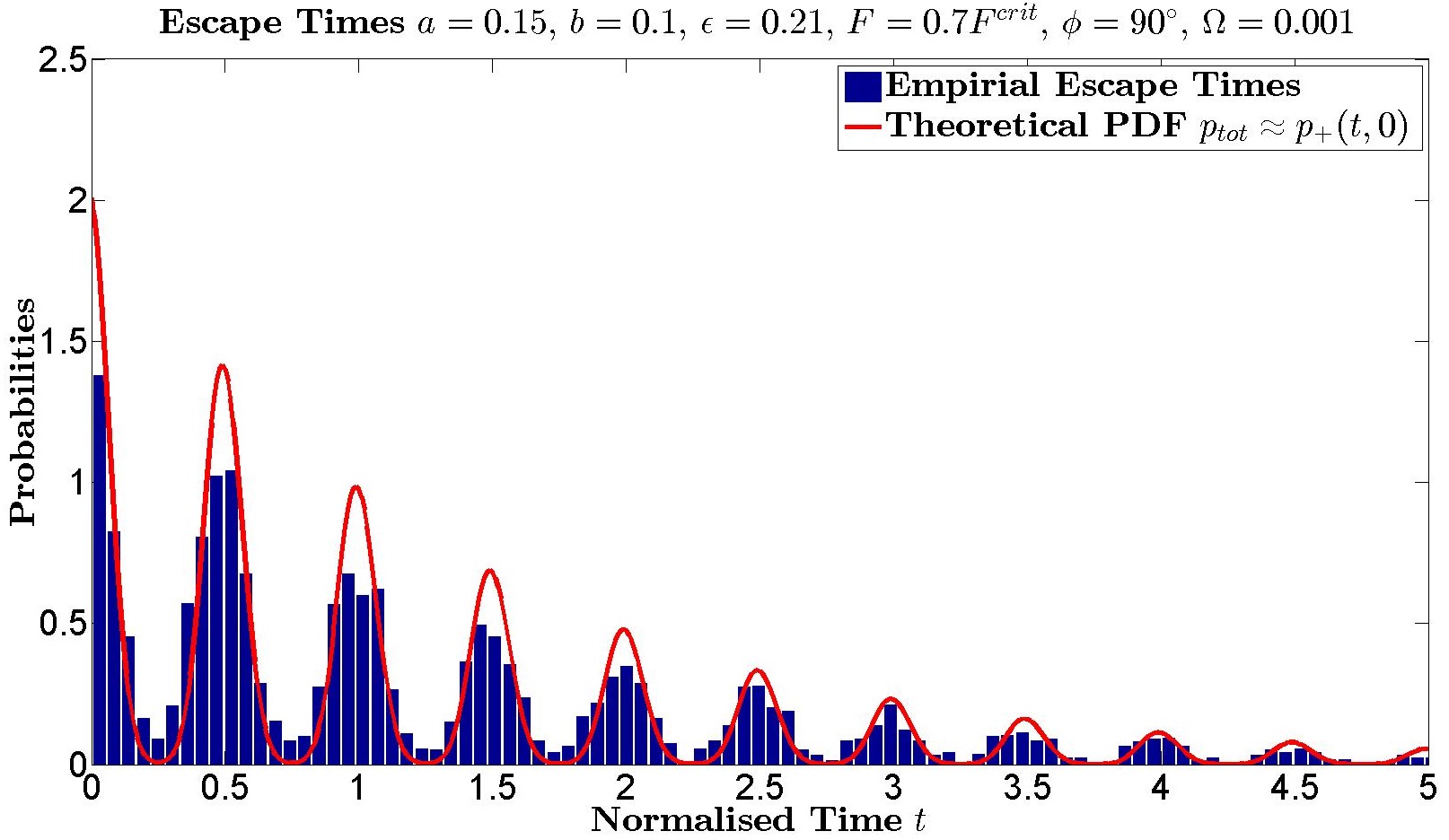}}
\caption{This is an example of the Double frequency.}
\label{chap_8_g75_p90_e21_pdf}
\end{figure}

\noindent It is important to note that Figures \ref{chap_8_g75_p0_e18_pdf}, \ref{chap_8_g75_p84_e20_pdf}, 
\ref{chap_8_g75_p87_e21_pdf} and 
\ref{chap_8_g75_p90_e21_pdf}
are histograms 
of the actual times it took to escape from either wells without differentiation between wells on the left or right. 
The times of entrance into the wells are not shown.
The PDF used is $p_{tot}(\cdot)$ which is being approximated by 
$p_{tot}(t)\approx p_+(t,0)$.

These escape times can be analysed in a different way. 
Let $u$ be the time of entrance into a well and $t$ the time of exit from a well. 
Figures \ref{chap_8_g75_p0_e18_pdf}, \ref{chap_8_g75_p84_e20_pdf}, 
\ref{chap_8_g75_p87_e21_pdf} and 
\ref{chap_8_g75_p90_e21_pdf}
are therefore histograms of the $(t-u)$ for both left and right escapes combined.
Thus $0\leq mod(u,T)\leq1$ is the phase of entrance into a well and $mod(t-u,T)$ is the escape time itself in normalised time. 
Such an analysis is done for the times in Figures \ref{chap_8_g75_p0_e18_pdf}, \ref{chap_8_g75_p84_e20_pdf}, 
and
\ref{chap_8_g75_p90_e21_pdf}
for both the left and right wells respectively. 

\begin{figure}[H]
\centerline{\includegraphics[scale=0.34]{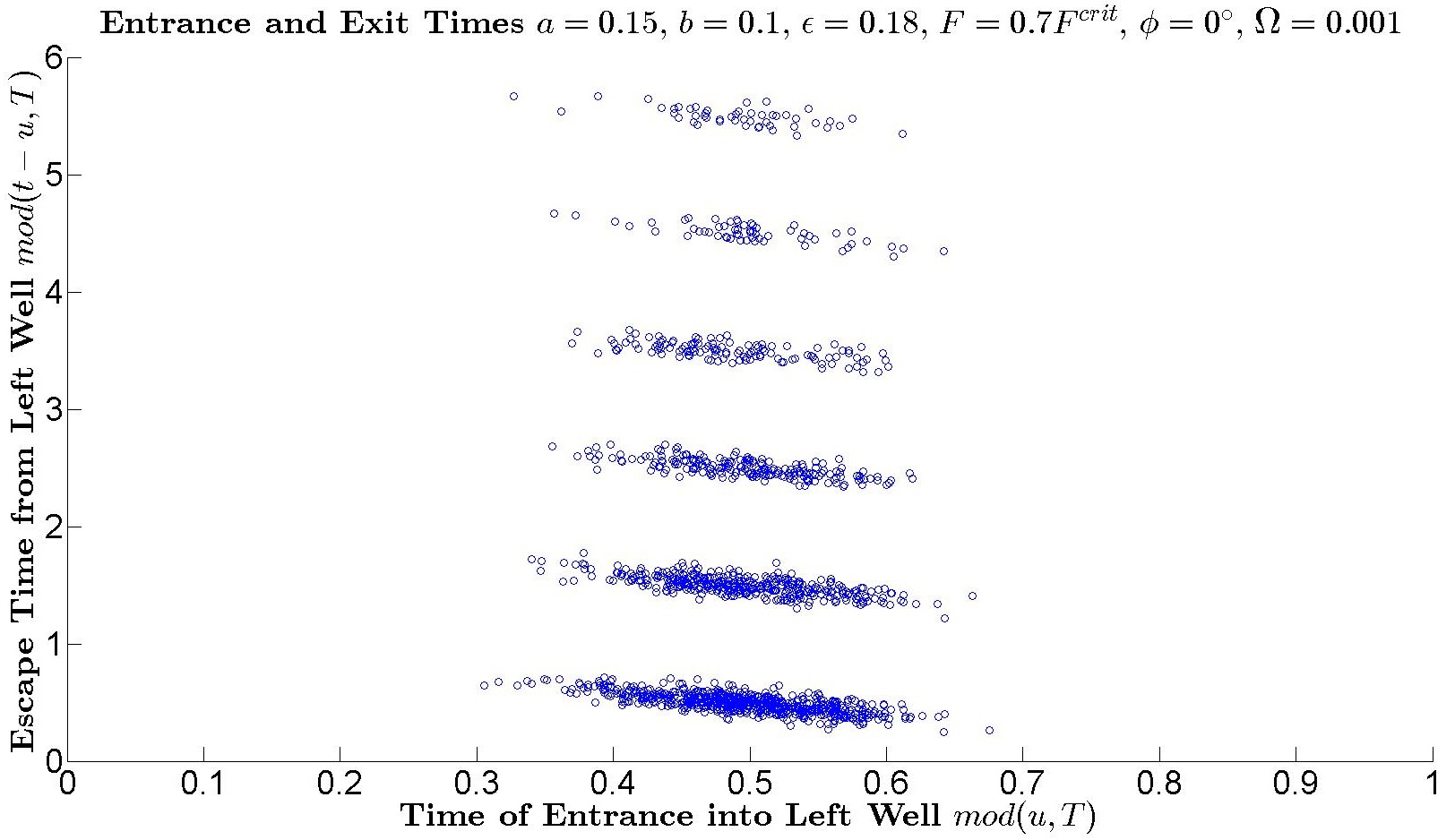}}
\caption{The $u$ is the time of entrance into the well and $t$ is the time of exit from the well.}
\label{chap_8_scatter_g75_p0_e18_pmaT}
\end{figure}

\begin{figure}[H]
\centerline{\includegraphics[scale=0.34]{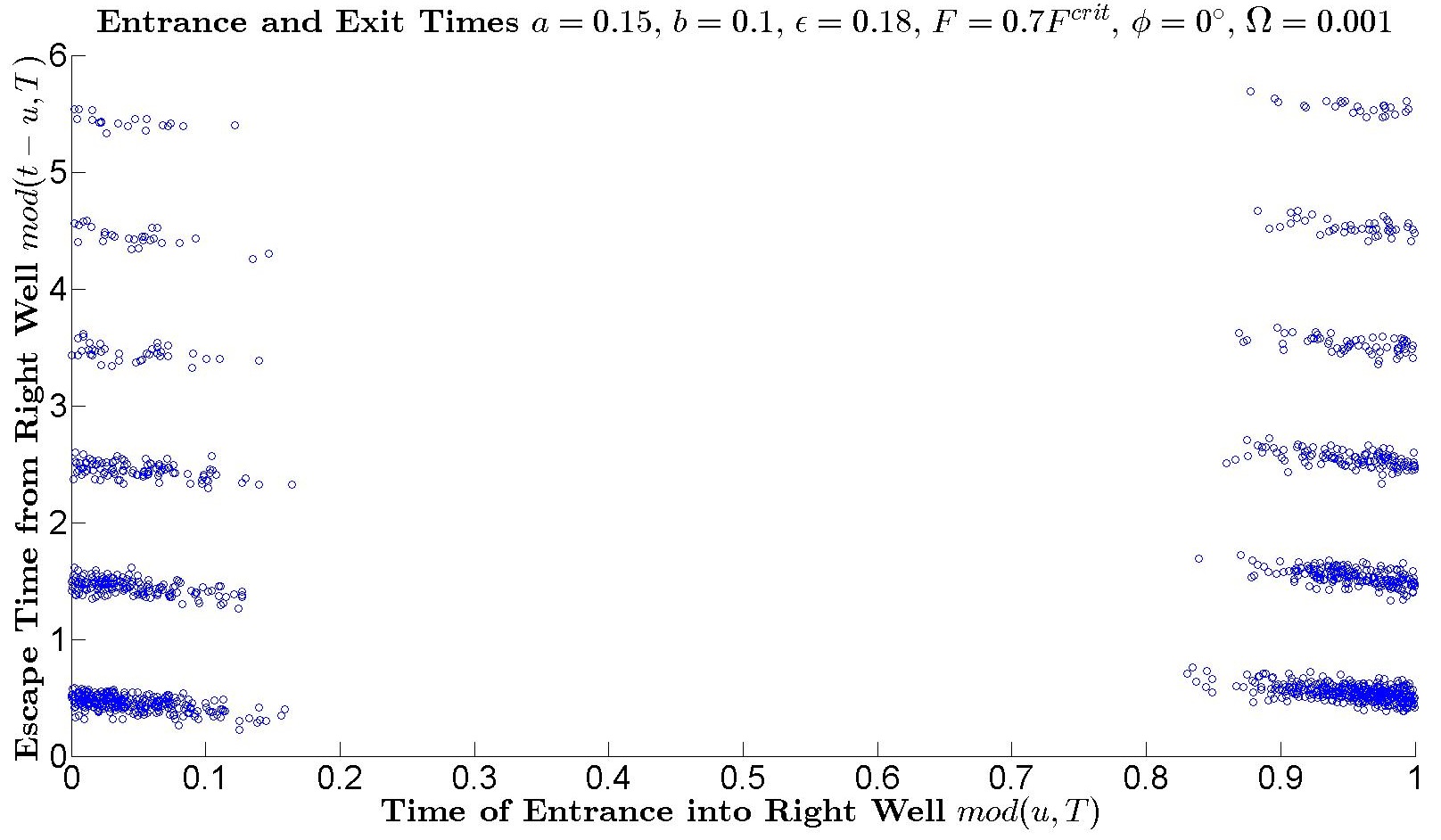}}
\caption{The $u$ is the time of entrance into the well and $t$ is the time of exit from the well.}
\label{chap_8_scatter_g75_p0_e18_ppaT}
\end{figure}

\begin{figure}[H]
\centerline{\includegraphics[scale=0.34]{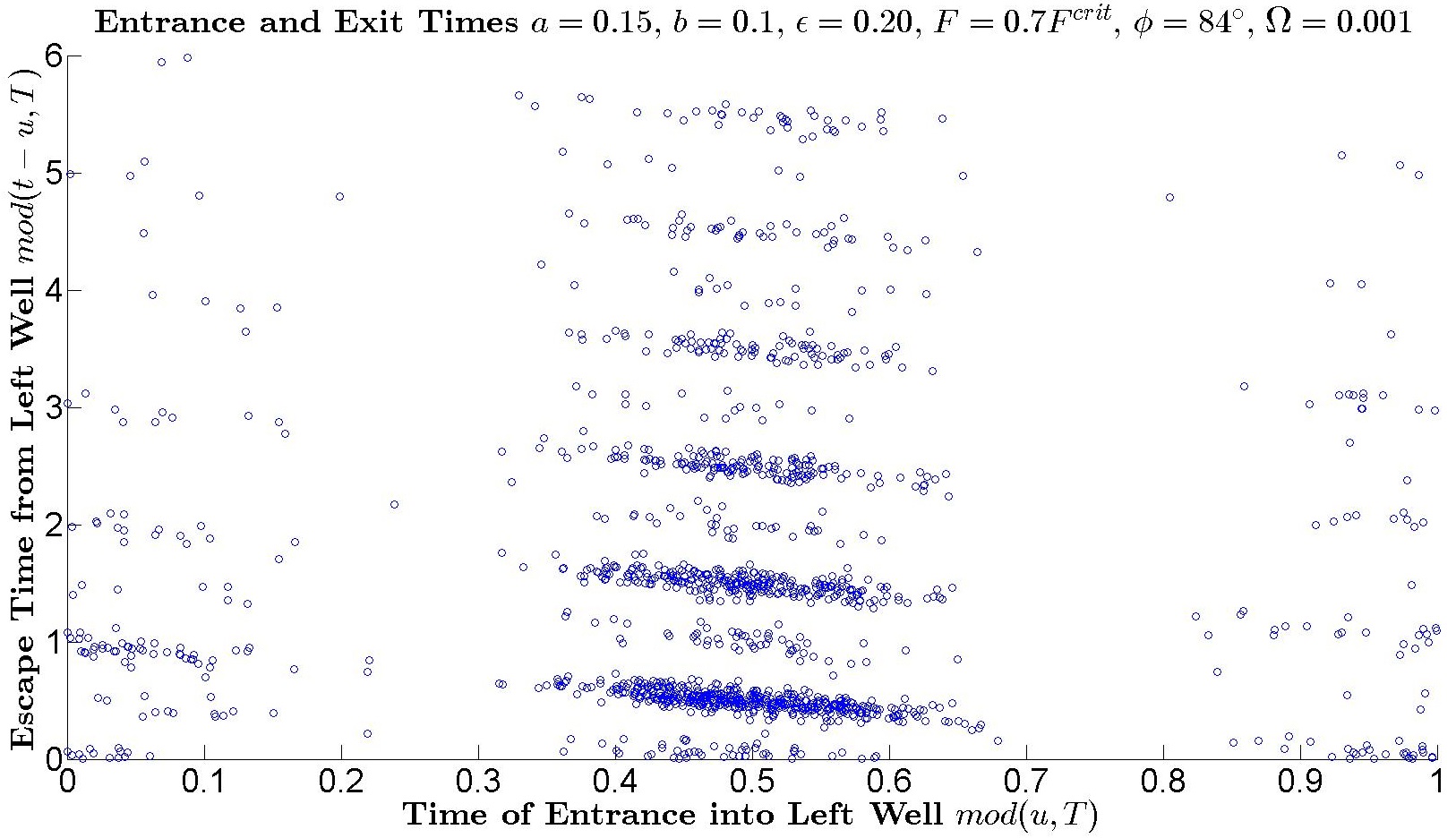}}
\caption{The $u$ is the time of entrance into the well and $t$ is the time of exit from the well.}
\label{chap_8_scatter_g75_p84_e20_pmaT}
\end{figure}

\begin{figure}[H]
\centerline{\includegraphics[scale=0.34]{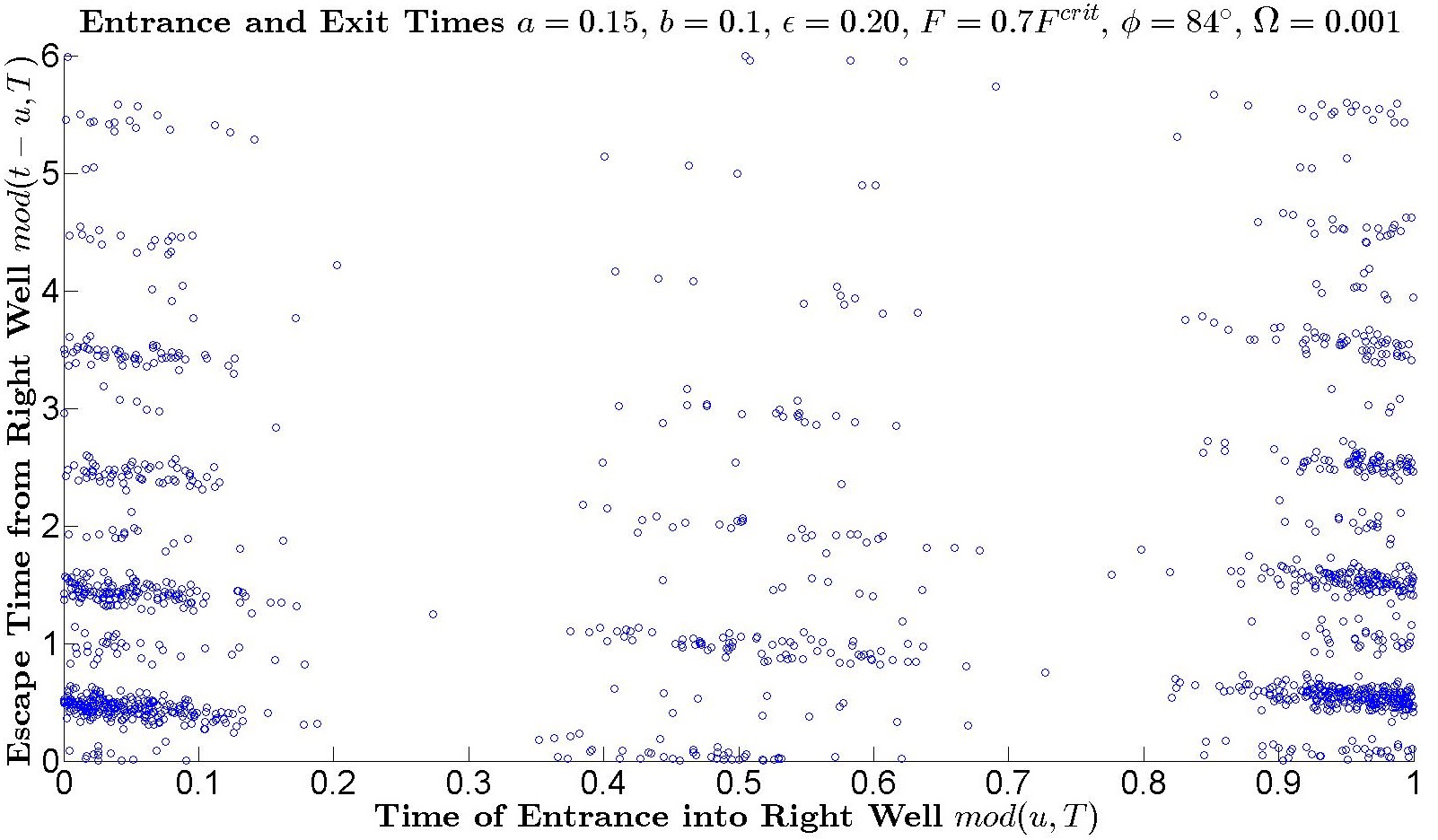}}
\caption{The $u$ is the time of entrance into the well and $t$ is the time of exit from the well.}
\label{chap_8_scatter_g75_p84_e20_ppaT}
\end{figure}

\begin{figure}[H]
\centerline{\includegraphics[scale=0.34]{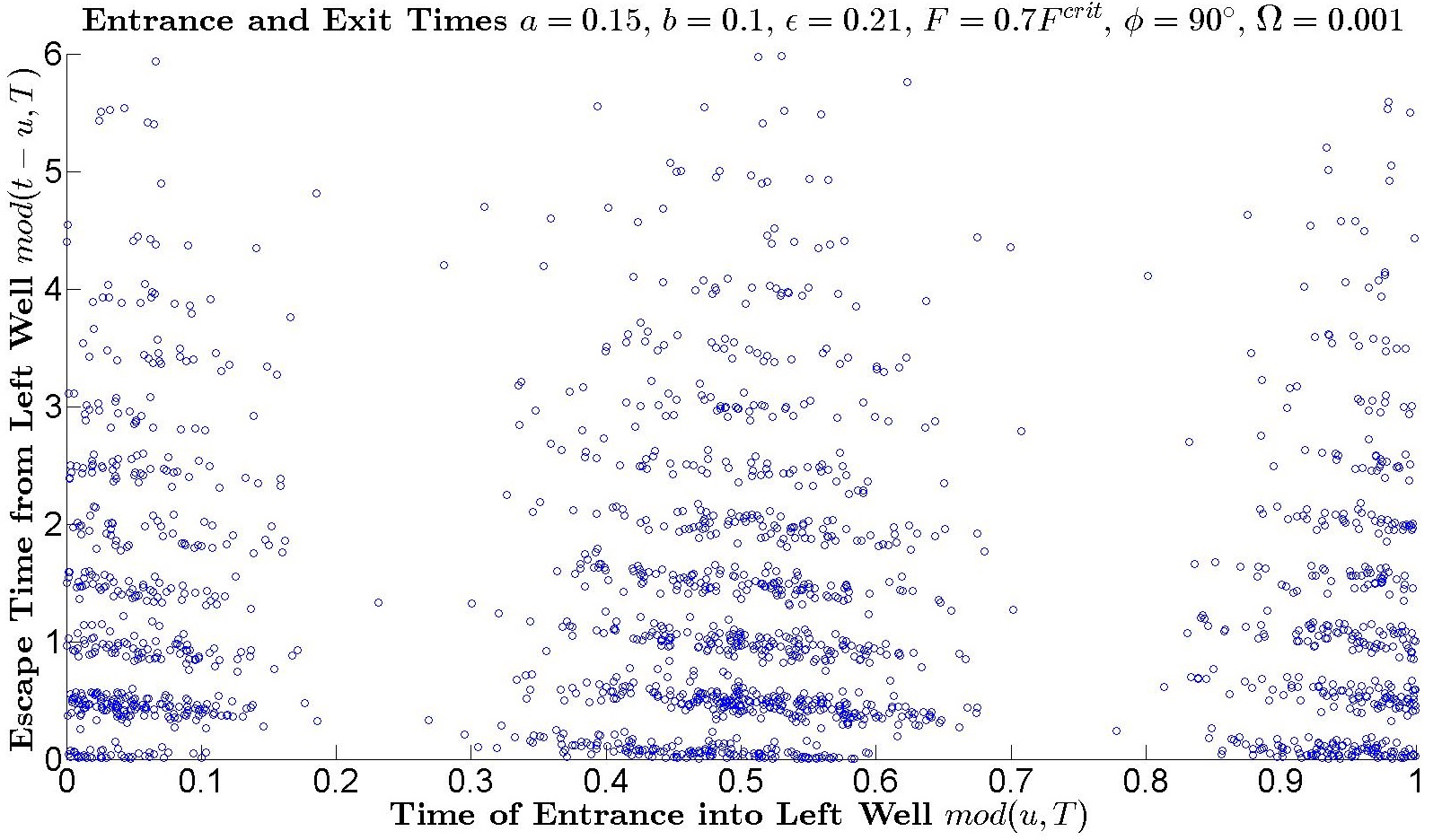}}
\caption{The $u$ is the time of entrance into the well and $t$ is the time of exit from the well.}
\label{chap_8_scatter_g75_p90_e21_pmaT}
\end{figure}

\begin{figure}[H]
\centerline{\includegraphics[scale=0.34]{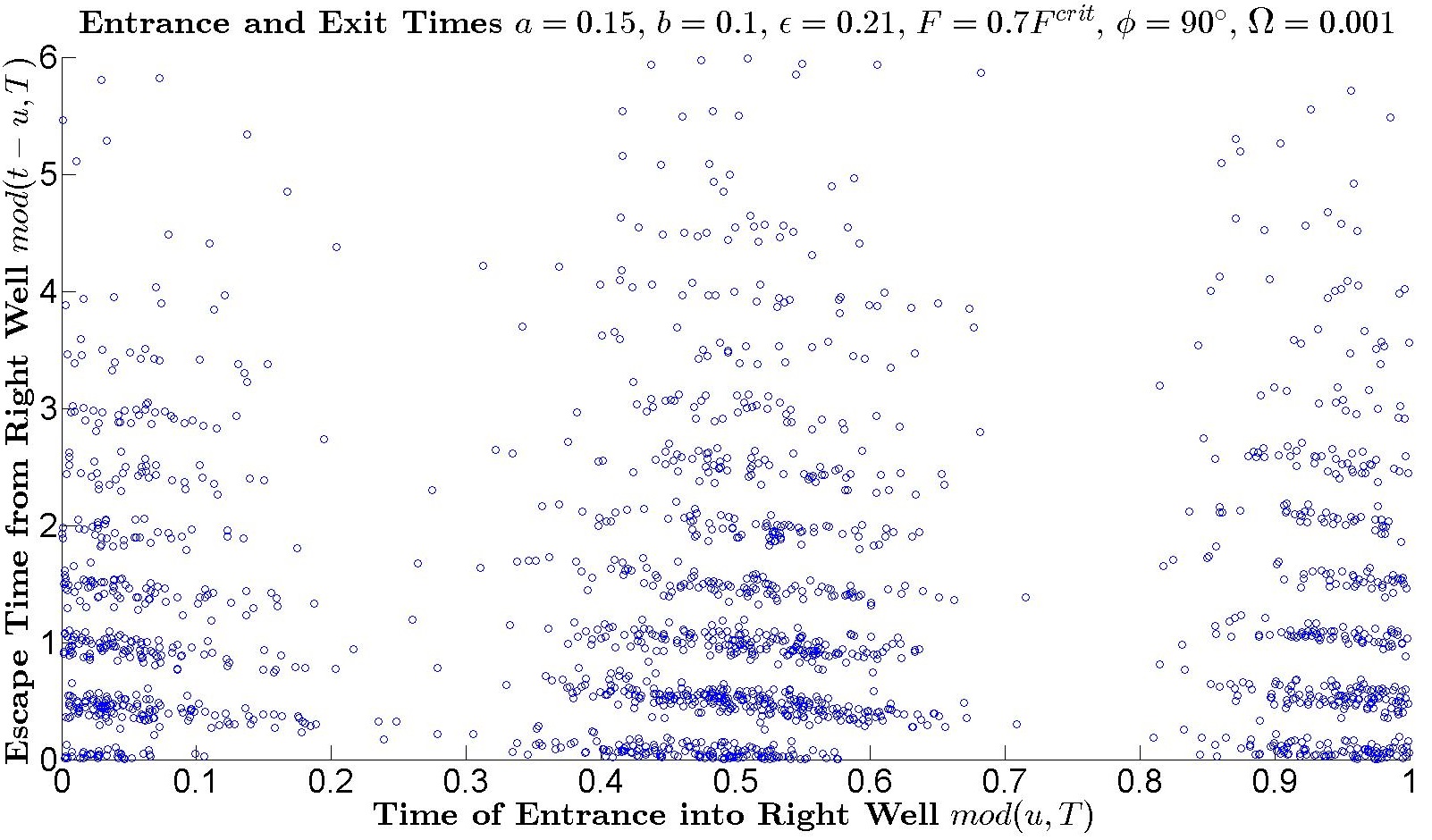}}
\caption{The $u$ is the time of entrance into the well and $t$ is the time of exit from the well.}
\label{chap_8_scatter_g75_p90_e21_ppaT}
\end{figure}

\noindent 
Notice the general behaviour of the data for $mod(u,T)$ and $mod(t-u,T)$. 
For the $\phi=0^\circ$ case the wells are alternating and one well is higher than the other. 
Entrance into the left well tend to occur near $u=0.5$ and entrance into the right well tend to occur near $u=0$ and $u=1$. 
For $\phi=90^\circ$ the wells are synchronised and are always at the same height as each other. 
Entrance and exit to and from either well tend to occur at $u=0$, $u=0.5$ and $u=1$. 
Notice that the Single, Intermediate and Double Frequencies can be seen in Figures
\ref{chap_8_scatter_g75_p0_e18_pmaT},
\ref{chap_8_scatter_g75_p0_e18_ppaT},
\ref{chap_8_scatter_g75_p84_e20_pmaT},
\ref{chap_8_scatter_g75_p84_e20_ppaT},
\ref{chap_8_scatter_g75_p90_e21_pmaT} and  
\ref{chap_8_scatter_g75_p90_e21_ppaT}.

Notice also  in Figure \ref{chap_8_scatter_g75_p0_e18_pmaT}
the data points are tiled near $0.5$. 
This seems to suggest that 
 the use of the Dirac delta function to approximate $p_{tot}\approx p_+(t,0)$
(see Chapter \ref{chap_approx_pdf}) may not be very good. 
The main problem here is the fact that 
we do not have an explicit formula for a probability measure of the time of entrance into a well, that is we do not have expressions for $m_-(u)$ and $m_+(u)$. 
This motivates us into developing the conditional KS test. 

We want to test whether the escape times we have measured are really distributed by the conditional PDFs $p_-(t,u)$ and $p_+(t,u)$. 
This is testing the conditional null hypothesis. 
Define the conditional CDFs by 
\begin{align*}
F^-_u(t)&=\int_u^tp_-(s,u)\,ds=1-\exp\left\{-\int_u^tR_{-1+1}(s)\,ds\right\}\\[0.5em]
F^+_u(t)&=\int_u^tp_+(s,u)\,ds=1-\exp\left\{-\int_u^tR_{+1-1}(s)\,ds\right\}.
\end{align*}
The time coordinates of the entrance and exit from the  wells are collected. 
These are 
\begin{align*}
\left(
\begin{array}{cccc}
u_1&u_2&\ldots&u_n\\
t_1&t_2&\ldots&t_n
\end{array}
\right)
\end{align*}
where $u_i$ is the time coordinate of the $i$th entrance into a well and $t_i$ is the time coordinate of the $i$th exit from a well. 
The conditional KS statistic is calculated by 
\begin{align*}
S_n^-&=\sup_{x\in[0,1]}
\left\Vert 
\frac{1}{n}
\sum_{i=1}^n\mathbf{1}_{[0,x]}
\left(
F^-_{u_i}(t_i)
-x
\right)
\right\Vert\\[0.5em] 
S_n^+&=\sup_{x\in[0,1]}
\left\Vert 
\frac{1}{n}
\sum_{i=1}^n\mathbf{1}_{[0,x]}
\left(
F^+_{u_i}(t_i)
-x
\right)
\right\Vert
\end{align*}
where in $S_n^-$ we sum over the time coordinates of entrance and exit to and from the left well
and in $S_n^+$ we sum over the time coordinates of entrance and exit to and from the right well. 
Recall that if the conditional null hypothesis is true then $S_n^-$ and $S_n^+$ are asymptotically distributed by 
\begin{align*}
\lim_{n\longrightarrow \infty} P(\sqrt{n} S_n \leq x)=Q(x)
\quad \text{where} \quad  
Q(x)=1-2\sum_{k=1}^\infty (-1)^{k-1}e^{-2k^2x^2}. 
\end{align*}
We want  99\% confidence. 
Note that 
\begin{align*}
P\left(
\sqrt{n}S_n\leq 1.6920
\right)
=Q(1.6920)
=0.99. 
\end{align*}
The $Q\left(\sqrt{n}S_n\right)$ is also calculated. 
The smaller $Q\left(\sqrt{n}S_n\right)$ is the more certain we are in accepting the null hypothesis. 
A selection of some of the data being implemented with the conditional KS test are given below for various angles of the forcing $\phi$ and noise level $\epsilon$. 
These are examples of the KS test being implemented for the histograms of escape times just given in Figures \ref{chap_8_g75_p0_e18_pdf}, \ref{chap_8_g75_p84_e20_pdf}, 
\ref{chap_8_g75_p87_e21_pdf} and 
\ref{chap_8_g75_p90_e21_pdf}

\begin{figure}[H]
\centerline{\includegraphics[scale=0.35]{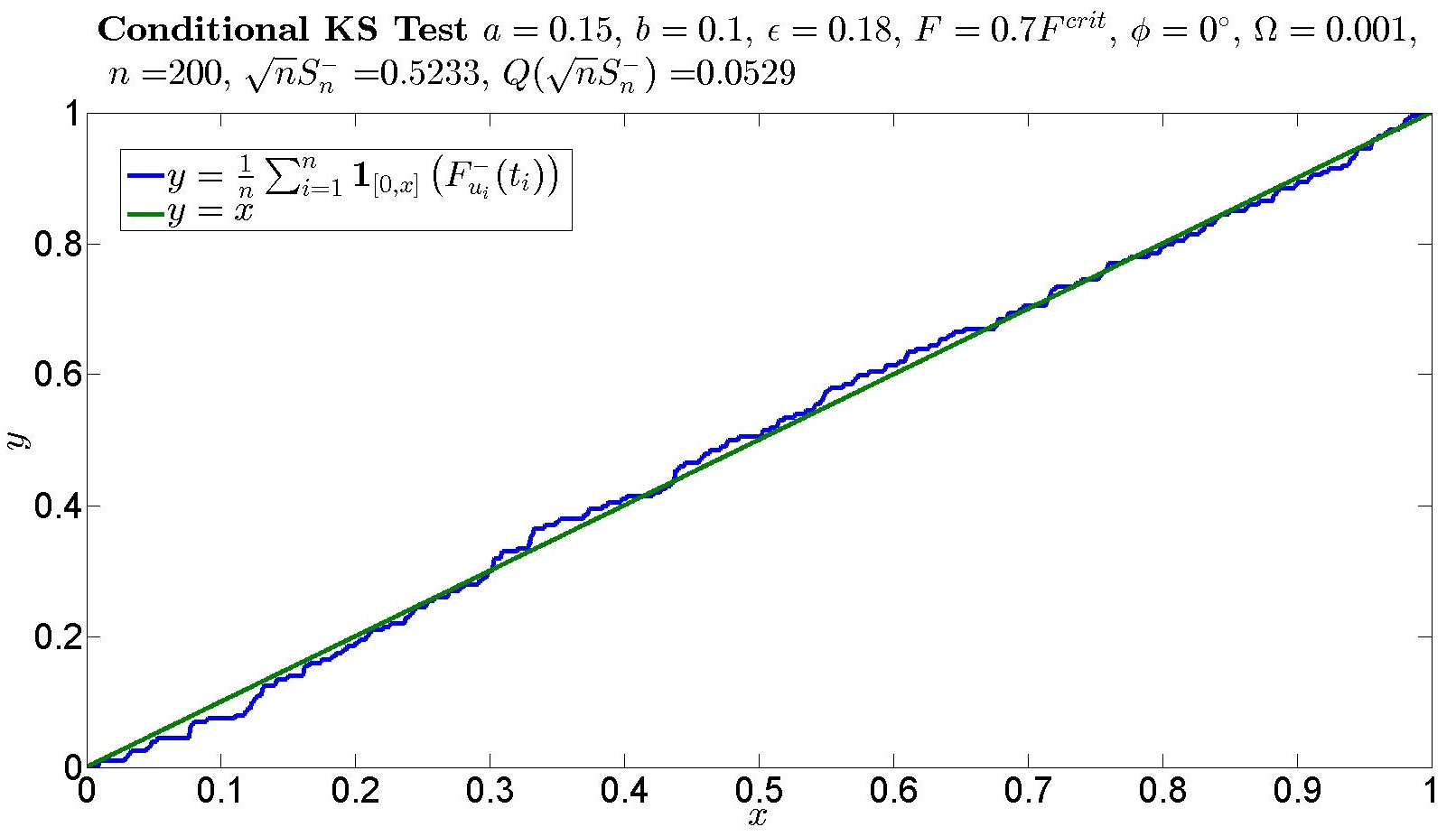}}
\caption{
This is an example of the conditional KS test being implemented for the data in Figure \ref{chap_8_g75_p0_e18_pdf}.
Note that $\epsilon=0.18$, $\phi=0^\circ$, $n=200$, $\sqrt{n}S^-_n=0.5233$ and $Q\left(\sqrt{n}S^-_n\right)=0.0529$.
}
\end{figure}

\begin{figure}[H]
\centerline{\includegraphics[scale=0.35]{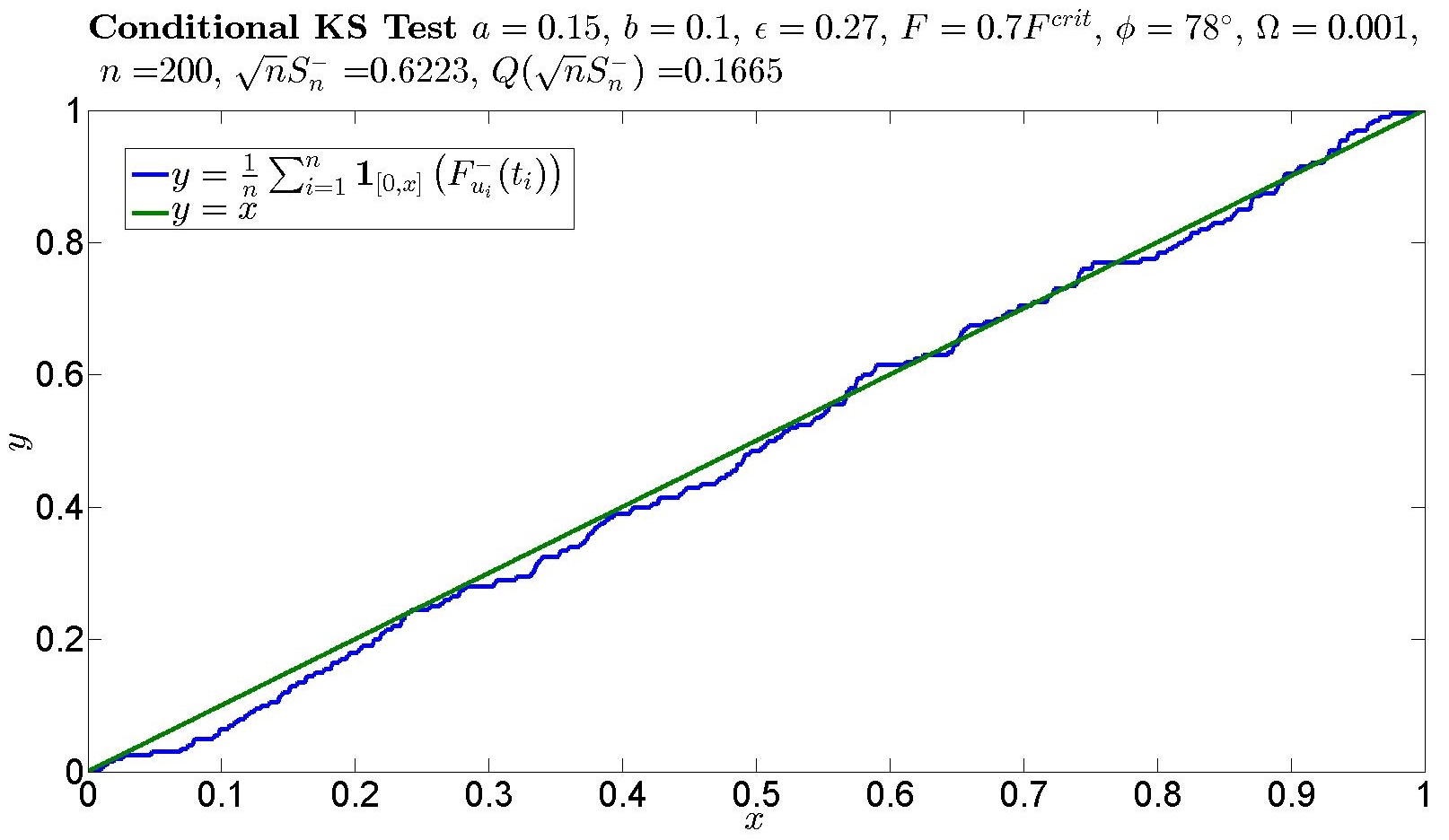}}
\caption{
This is an example of the conditional KS test being implemented for the data in Figure \ref{chap_8_g75_p84_e20_pdf}.
Note that $\epsilon=0.20$, $\phi=84^\circ$, $n=200$, $\sqrt{n}S^-_n=0.6223$ and $Q\left(\sqrt{n}S^-_n\right)=0.1665$.
}
\end{figure}

\begin{figure}[H]
\centerline{\includegraphics[scale=0.35]{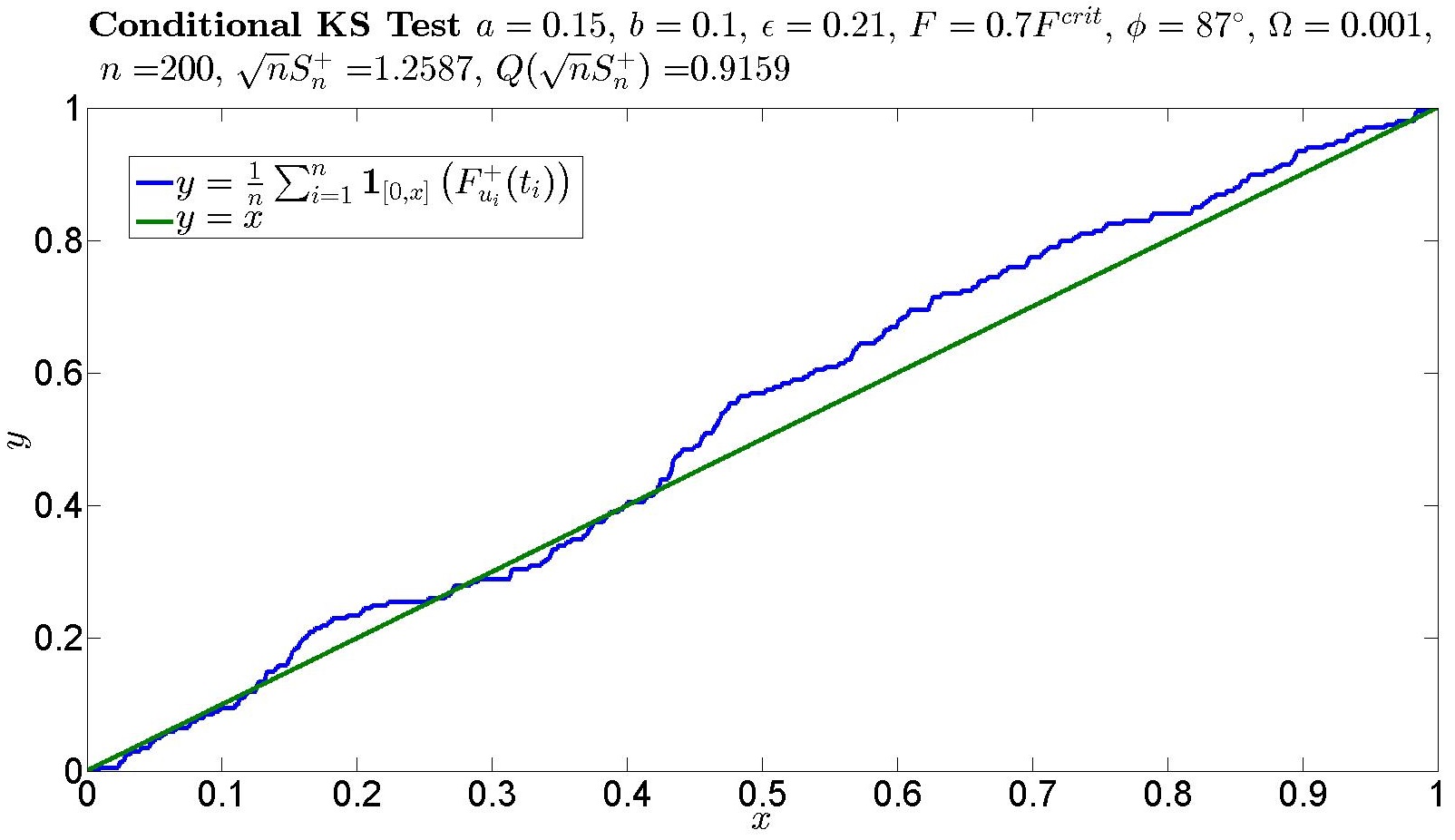}}
\caption{
This is an example of the conditional KS test being implemented for the data in Figure \ref{chap_8_g75_p87_e21_pdf}.
Note that $\epsilon=0.21$, $\phi=87^\circ$, $n=200$, $\sqrt{n}S^-_n=1.2587$ and $Q\left(\sqrt{n}S^+_n\right)=0.9159$.
}\label{chap_8_g80_p87_e21_m2}
\end{figure}

\begin{figure}[H]
\centerline{\includegraphics[scale=0.35]{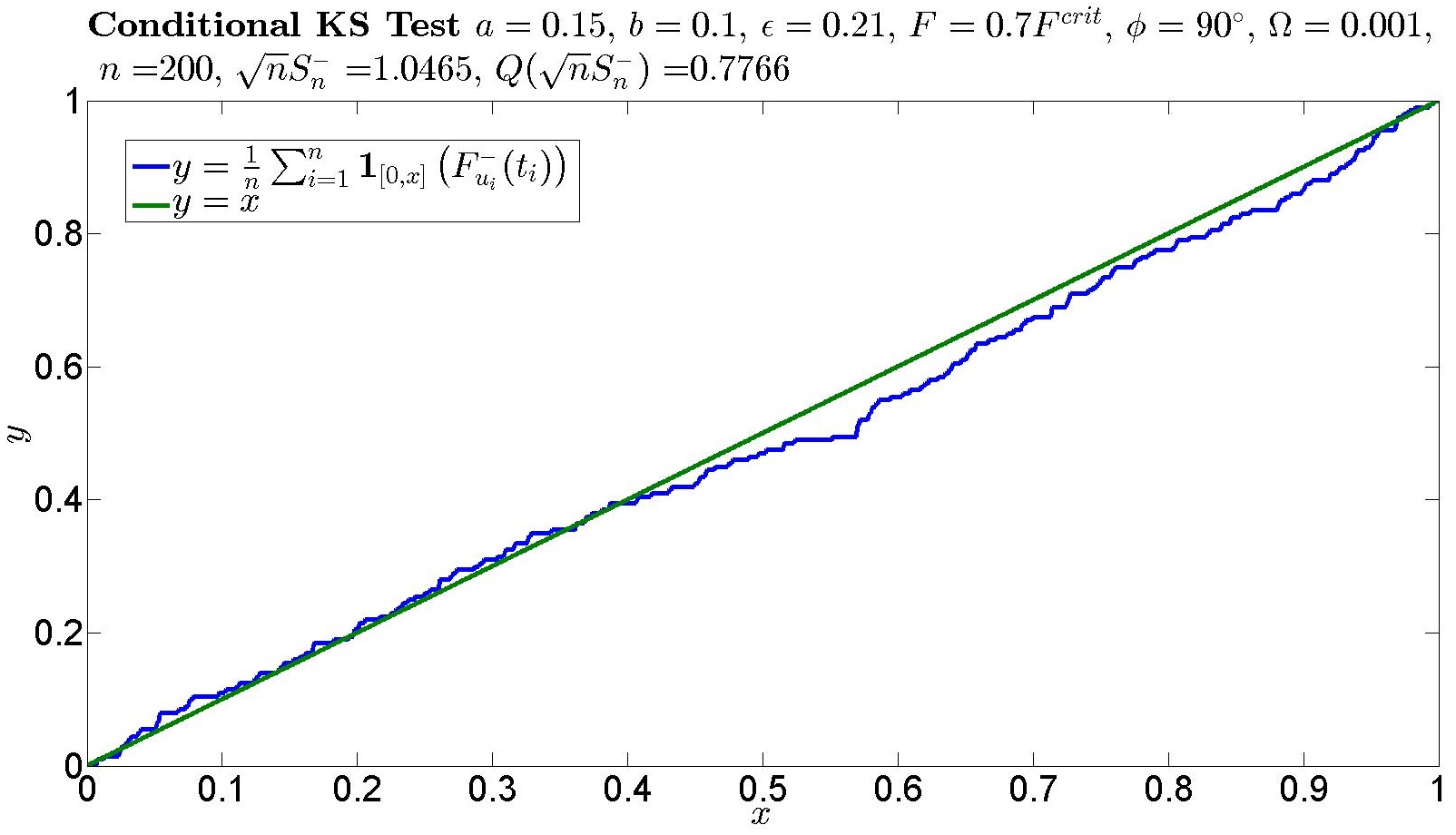}}
\caption{
This is an example of the conditional KS test being implemented for the data in Figure \ref{chap_8_g75_p90_e21_pdf}.
Note that $\epsilon=0.21$, $\phi=90^\circ$, $n=200$, $\sqrt{n}S^-_n=1.0465$ and $Q\left(\sqrt{n}S^-_n\right)=0.7766$.
}\label{chap_8_g80_p90_e21_m1}
\end{figure}

\subsection{Interpretation of the Escape Time and Conditional KS Test Analysis}
When $\phi=0^\circ$ there were peaks in the empirical PDF of the escape times. 
These occurred at times $\frac{1}{2}T$, $\frac{3}{2}T$, $\frac{5}{2}T$, \ldots. 
This effect we call the Single frequency. 
When $\phi=90^\circ$ the peaks occurred  at 
$\frac{1}{2}T$, $\frac{3}{2}T$, $\frac{5}{2}T$, \ldots
and $0$, $T$, $2T$, $3T$, $4T$, \ldots. 
This effect we call the Double Frequency. 
When $0^\circ<\phi<90^\circ$ an intermediate effect is seen. 
There were major peaks at 
$\frac{1}{2}T$, $\frac{3}{2}T$, $\frac{5}{2}T$, \ldots
and minor peaks at 
$0$, $T$, $2T$, $3T$, $4T$.

The behaviour of the Single, Intermediate and Double Frequencies can be explained geometrically. 
When the height between a well and a saddle is minimum, the optimal probability of escape has occurred. 
When $\phi=0^\circ$ the  frequency of the return of the optimal probability of escape is the same as the driving frequency  $\Omega$. 
This optimal probability comes back very $T$ which is once in a period. 
When $\phi=90^\circ$ the  frequency of the return of the optimal probability of escape is double the driving frequency at $2\Omega$. 
This optimal probability comes back very $\frac{T}{2}$ which is twice in a period. 
This explains why the peaks in the Single and Double Frequencies are seen where they are.

As the angle changed from $\phi=0^\circ$ to $\phi=90^\circ$ the Single Frequency gradually changes into the Double Frequency with the Intermediate Frequency seen in between. 
Thus  the angle of the forcing is leaving a mark in the PDFs of escape times.

When the conditional KS test was implemented, the functions
\begin{align*}
y_0(x)=x, 
\quad 
y_-(x)=\sum_{i=1}^n\mathbf{1}_{[0,x]}
\left(
F^-_{u_i}(t_i)
\right)
\quad \text{and} \quad 
y_+(x)=\sum_{i=1}^n\mathbf{1}_{[0,x]}
\left(
F^+_{u_i}(t_i)
\right)
\end{align*}
were used to calculate the following distances which are the conditional KS statistics
\begin{align*}
S^-_n=\left\Vert y_0-y_-\right\Vert_\infty
\quad \text{and} \quad 
S^+_n=\left\Vert y_0-y_+\right\Vert_\infty.
\end{align*}
It is reasonable to say that $y_-(\cdot)$ and $y_+(\cdot)$ were close enough to $y_0(\cdot)$ that we can accept the conditional null hypothesis. 
This can be seen and judged graphically with $S^-_n$ and $S^+_n$ calculated as well. 
This is an example of the conditional KS test giving a reasonable result.\footnote{See Appendix \ref{appendix_over_sample} for discussions as to how some of our implementation of the conditional KS test are examples of oversampling.}

\section{Sparse Data Analysis} 
We do the same analysis with the escape time and the conditional KS test. 
But now we artificially make the data sparse by only implementing the conditional KS test for 20 transitions. 
We want 99\% confidence.
Thus with $n=20$ tables for the KS distribution show that 
\begin{align*}
P\left(
S_{20}\leq 0.356
\right)
=0.99
\end{align*}
and there are two particular examples we want to focus on.
These are when $p_{tot}\approx p_+(t,0)$ is not a good approximation and the conditional KS test is performed in such a situation.

\begin{figure}[H]
\centerline{\includegraphics[scale=0.35]{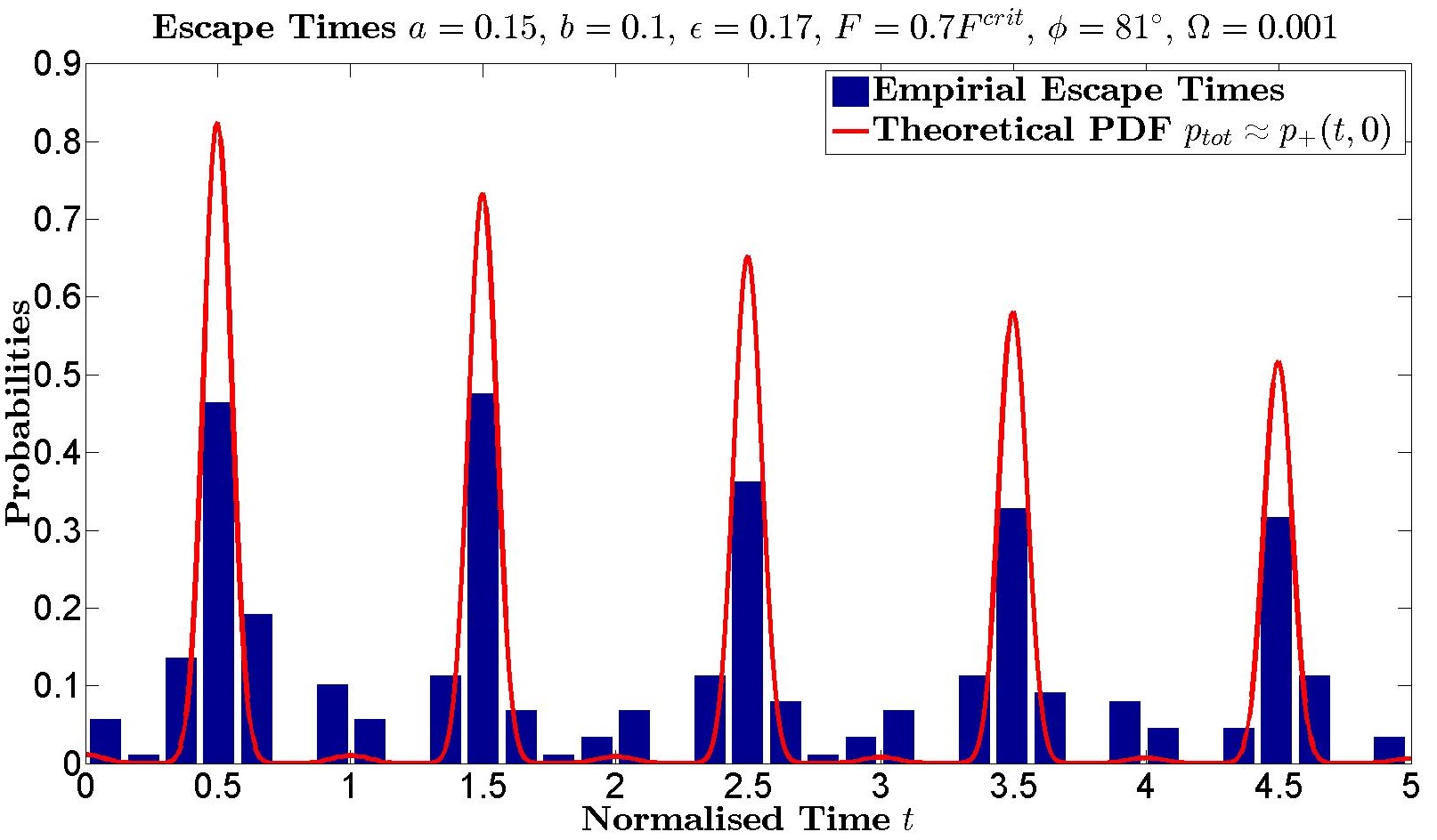}}
\caption{The $p_{tot}\approx p_+(t,0)$ is not a good approximation here.}
\label{chap_8_g75_p81_e17_pdf}
\end{figure}

\begin{figure}[H]
\centerline{\includegraphics[scale=0.35]
{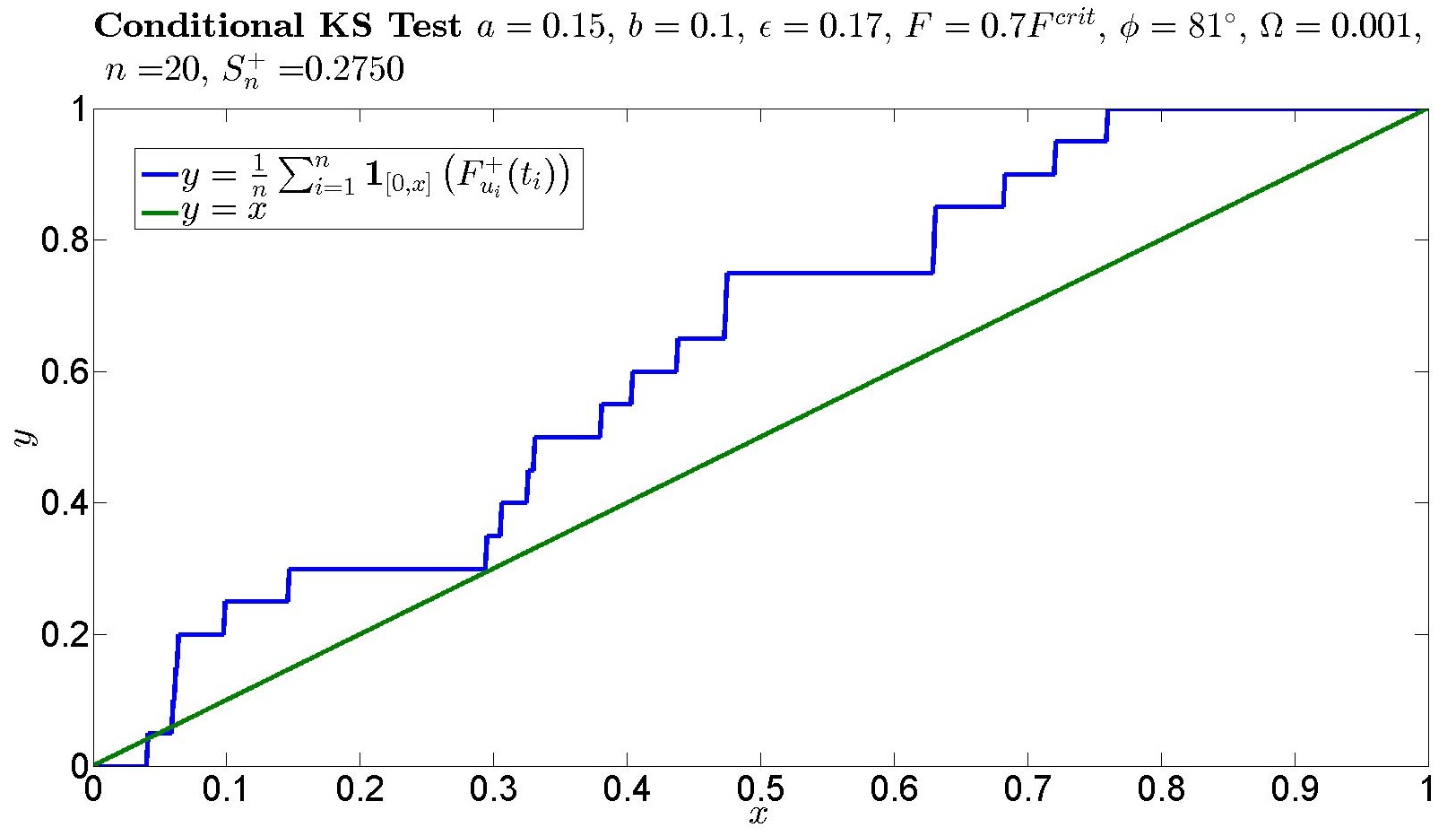}}
\caption{
This is a KS test on the data in Figure \ref{chap_8_g75_p81_e17_pdf}.
The conditional null hypothesis can be reasonably accepted.
Note that $\epsilon=0.17$, $\phi=81^\circ$, $n=20$ and $S^+_n=0.2750$. $Q(\sqrt{n}S_n^+)=0.9029$. 
}
\label{chap_8_g81_p81_e17_m2}
\end{figure}

\begin{figure}[H]
\centerline{\includegraphics[scale=0.35]{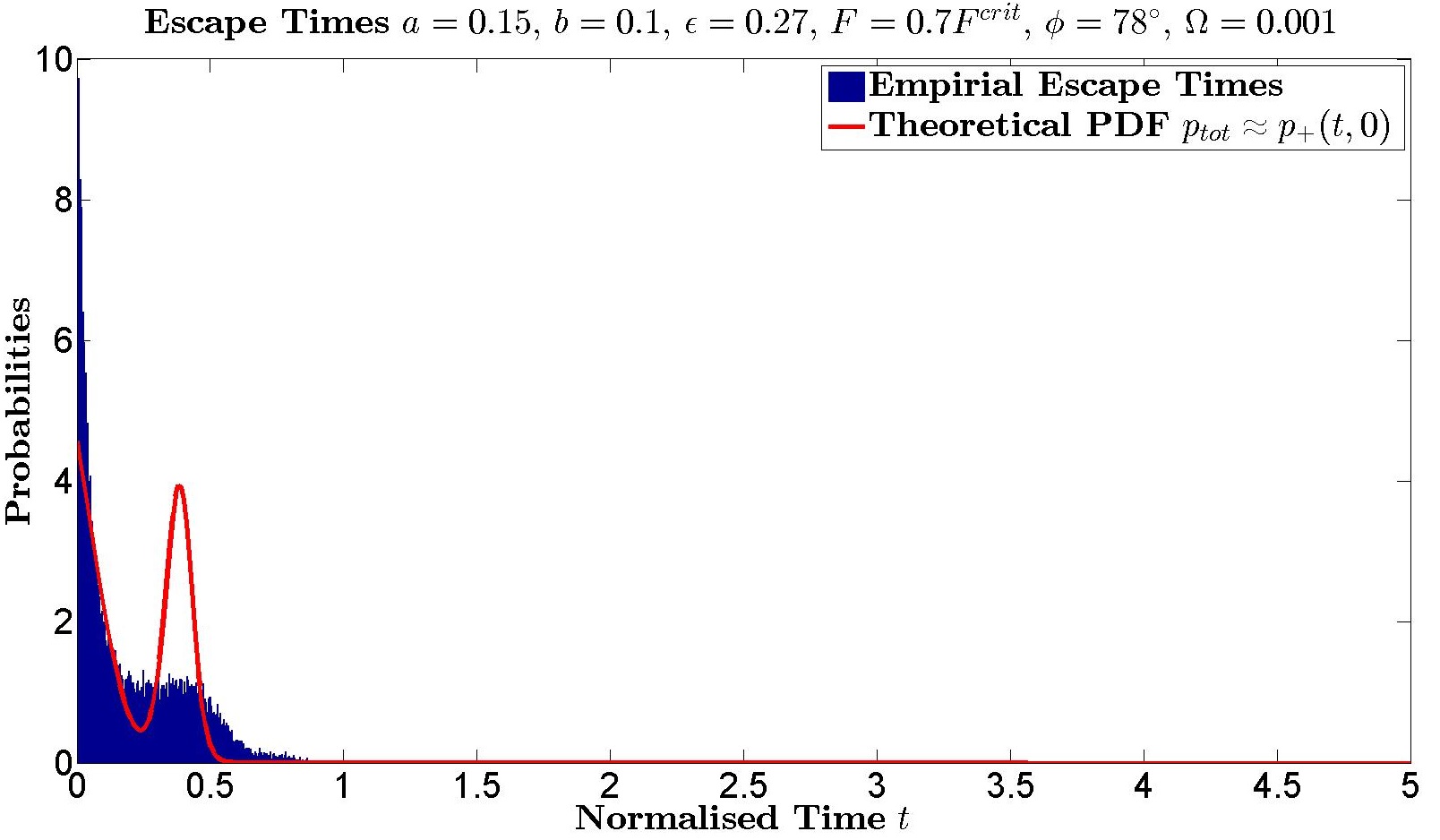}}
\caption{The $p_{tot}\approx p_+(t,0)$ is not a good approximation here.}
\label{chap_8_g75_p78_e27_pdf}
\end{figure}

\begin{figure}[H]
\centerline{\includegraphics[scale=0.35]
{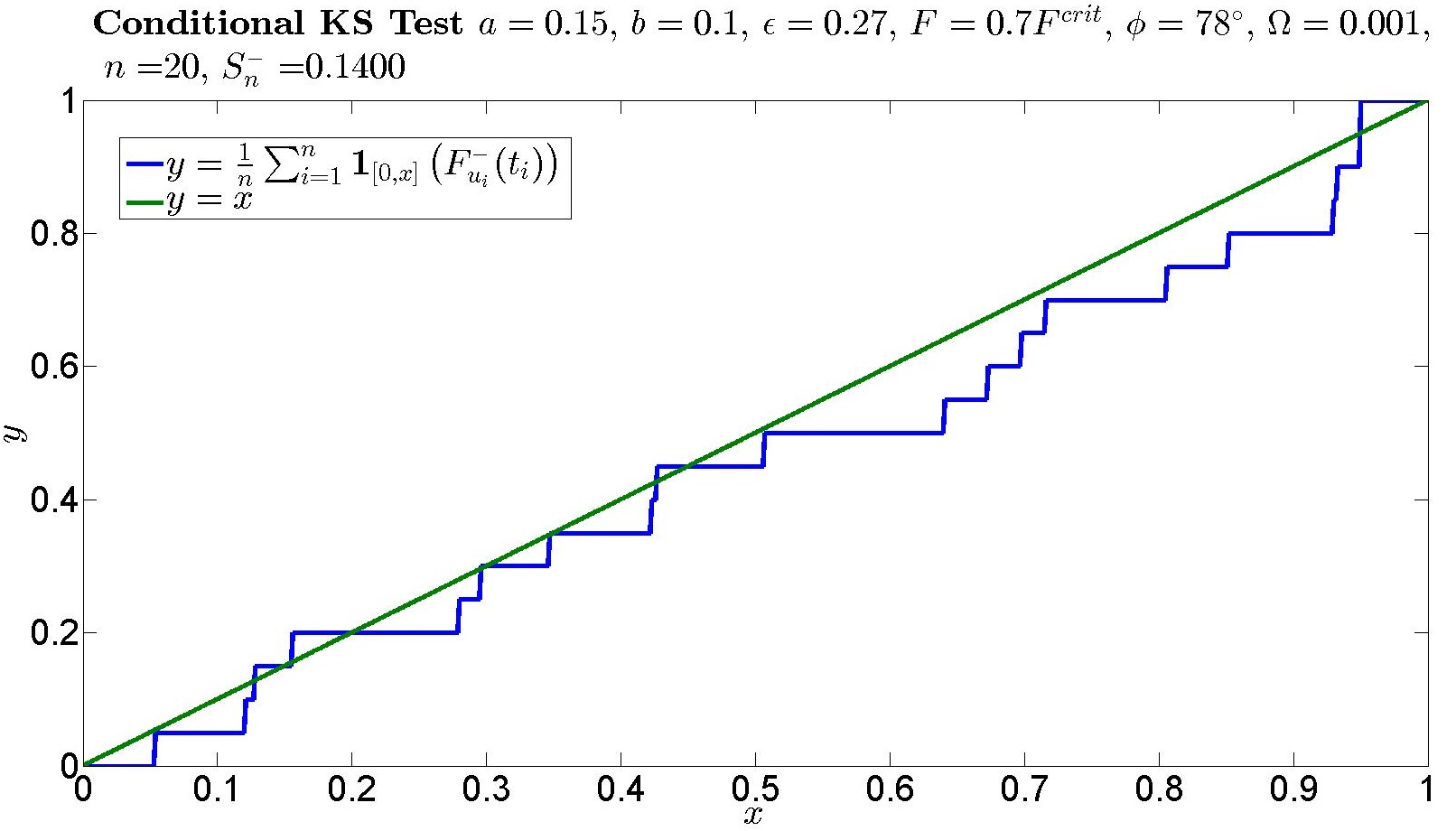}}
\caption{
This is a conditional KS test on the data in Figure \ref{chap_8_g75_p78_e27_pdf}.
The conditional null hypothesis can be reasonably accepted.
Note that $\epsilon=0.27$, $\phi=78^\circ$, $n=20$ and $S^-_n=0.1400$.
Also note that $Q(\sqrt{n}S_n^-)= 0.1720$. 
}
\label{chap_8_g81_p78_e27_m1}
\end{figure}

\subsection{Interpretation of the Sparse Data Analysis}
\label{conclusion_sparse_no_good_pdf}
The aim of Sparse Data Analysis is to see how the conditional KS test performs even if less data is available. 
This is done by looking at two cases where $p_{tot}\approx p_+(t,0)$ is not a good approximation and implementing the conditional KS test on them after artificially making the data sparse.

Consider the case for the parameters in Figure \ref{chap_8_g75_p81_e17_pdf}. 
Figure \ref{chap_8_g75_p81_e17_pdf} is an example of when the noise is so small there is very little escape times being detected in the range $[0,5T]$. 
The $p_{tot}\approx p_+(t,0)$ is not a good approximation here. 
In Figure \ref{chap_8_g81_p81_e17_m2} the conditional KS test was performed on the data in 
\ref{chap_8_g75_p81_e17_pdf} and the distance between the two functions is small. 
This means we can accept the conditional null hypothesis even when there is fewer data and $p_{tot}\approx p_+(t,0)$ is not a good approximation. 

Now consider the case for the parameters in Figure \ref{chap_8_g75_p78_e27_pdf}. 
The noise is so large $p_{tot}\approx p_+(t,0)$ is no longer a good approximation. 
But in Figure \ref{chap_8_g81_p78_e27_m1} the conditional KS test was performed on the data in Figure \ref{chap_8_g75_p78_e27_pdf}. 
Again this is an example of us being able to accept the conditional null hypothesis even if $p_{tot}\approx p_+(t,0)$ is not a good approximation. 

Only 20 escape times were implemented in the conditional KS test and the conditional null hypothesis can still be accepted with a reasonable degree of certainty. 
But the $p_{tot}\approx p_+(t,0)$ was not a good approximation for the empirical PDF of escape times. 
These are examples of the conditional KS test giving reasonable conclusions even when there are sparse data.
It also shows that the conditional KS test can still be used even if there is no  good approximation of the PDF of escape times.

Back in Chapter \ref{chap_approx_pdf} we approximated $m_-(u)$ and $m_+(u)$ by 
\begin{align*}
m_-(u)&\approx\delta\left(u-T/2\right)\\[0.5em]
m_+(u)&\approx
\frac{1}{2}\delta\left(u\right)
+
\frac{1}{2}\delta\left(u-T\right).
\end{align*}
Although the escape times (represented as dots on a scatter graph) tend to cluster around $u=0$, $u=0.5$ and $u=1$ there are spread around them. 
As the noise levels $\epsilon$ increases the spread around $u=0$, $u=0.5$ and $u=1$ would increase and $p_{tot}\approx p_+(t,0)$ would stop to be a good approximation. 
Despite this the conditional KS test still shows sensible results, in that we can accept the conditional null hypothesis.

\section{Remarks on Analysis of Stochastic Resonance}
There are a few subtleties and setbacks to the analysis which is worth mentioning here. 

\subsection{Remarks on Implementing the Conditional KS Test}
Notice that all the theories developed about the KS Test were based on the assumption that the null hypothesis is true. 
This means strictly speaking a small KS statistic, that is a small $S_n^-$ or $S_n^+$ does not immediately allow us to accept the null hypothesis but good reasons not to reject it. 
Also when there were many transitions, that is for large $n$, the terms $Q(\sqrt{n}S_n^-)$ and $Q(\sqrt{n}S_n^+)$ were also calculated.
The smaller  $Q(\sqrt{n}S_n^-)$ and $Q(\sqrt{n}S_n^+)$ are the more confidence we have in not rejecting the null hypothesis. 
This is because 
for very large $n$, we would expect
\begin{align*}
\lim_{n\rightarrow\infty}\sqrt{n}S_n^-=0
\quad \text{and} \quad 
\lim_{n\rightarrow\infty}\sqrt{n}S_n^+=0
\end{align*}
so the smaller $Q(\sqrt{n}S_n^-)$ and $Q(\sqrt{n}S_n^+)$ are the more certain we are in not rejecting the null hypothesis. 

\subsection{Remarks on Adiabatic Approximation}
Notice that in the PDFs $p_-(t,u)$, $p_+(t,u)$ and $p_{tot}(t)$ expressions for the escape rates $R_{-1+1}(t)$ and $R_{+1-1}(t)$ were required. 
These rates were also required for the conditional KS test. 
Strictly speaking these rates are dependent on the driving frequency $\Omega$, 
but we stress that these rates were calculated using Kramers' formula as though the particle is escaping from a static potential. 
This is the adiabatic approximation where an oscillatory potential is approximated by a static potential. 
Considerations for whether the adiabatic approximation would fail in our calculations were done back in Chapter \ref{adiabatic_parameters}

It is worth summarising all the approximations which the analysis of the data have been based.
There is the small noise approximation and slow forcing approximation from Kramers' formula, the adiabatic approximation and the perfect phase approximation where $p_{tot}$ is approximated by  $p_{tot}\approx p_{tot}(t,0)$.

\chapter*{Conclusion}
\addcontentsline{toc}{chapter}{Conclusion}

\section*{Outline of Results}
\addcontentsline{toc}{section}{Outline of Results}

In this thesis we have considered the following problem. 
Let $X^\epsilon_t$ be a stochastic process in $\mathbb{R}^2$ which is described by  the SDE 
\begin{align*}
dX^\epsilon_t=b\left(X^\epsilon_t,t\right)dt+\epsilon\,dW_t
\end{align*}
and the drift term $b(\cdot,\cdot)$ is expressed by  
\begin{align*}
b(x,t)=-\nabla V_0 (x) + F\cos \Omega t 
\end{align*}
where $V_0:\mathbb{R}^2\longrightarrow \mathbb{R}$ is a time independent function, the unperturbed potential, with two metastable states, and two pathways between these states. The $F\in\mathbb{R}^r$ is the magnitude of the forcing and $\Omega$ is the driving frequency. 
Our aim was to see characteristics of the trajectory $X^\epsilon_t$ which only depends on the qualitative structure of $V_0$, that is the existence of two metastable states and two pathways. 

For concreteness we considered a model, which we call the  Mexican Hat Toy Model
\begin{align*}
V_0(x,y)=\frac{1}{4}r^4-\frac{1}{2}r^2-ax^2+by^2
\quad \text{where} \quad r=\sqrt{x^2+y^2}.
\end{align*}
The magnitude and angle of the forcing are given by
\begin{align*}
F=\sqrt{F_x^2+F_y^2}
\quad \text{and} \quad 
\phi=\tan^{-1}\left(\frac{F_y}{F_x}\right).
\end{align*}
The angle $\phi$ and noise level $\epsilon$ were varied. 
At $\phi=0$ the wells were alternating, that is one well is higher than the other, in the sense that it is easer to jump from one well to the other than vice versa.
At $\phi=90^\circ$ the wells are synchronised, that is both wells are always at the same height but the heights of the barrier for the two paths is alternating.

A potential with two pathways has never been considered before in the context of stochastic resonance. 
We studied it using approximation techniques and direct simulations. 
In an adiabatic regime the Freidlin-Wentzell theory allows one to give analytical solutions of the jump type distributions asymptotically in this regime. 
This theory predicted the appearance of additional resonance peaks at half the frequency when the angle approaches $\phi=90^\circ$.

We simulated $X_t^\epsilon$ for different values of $\phi$ and $\epsilon$
and computed for the values of angle increasing from  $\phi=0$ to $\phi=90^\circ$  the six measures $M_1$, $M_2$, $M_3$, $M_4$, $M_5$ and $M_6$  as functions of the noise level. 
The first major surprise was that the graphs showed less and less pronounced minima (or maxima) and hence suggests that the phenomena of stochastic resonance gets less and less pronounced, see Chapter \ref{conclusion_six_measures}.
The effect of resonance seems to disappear overall.

However, considering the path $X_t^\epsilon$ itself, one sees  that there may be nevertheless some synchronisation, see Figure        \ref{chap_8_path_p0},
\ref{chap_8_path_p84} and 
\ref{chap_8_path_p90}.
We carefully controlled our simulation and checked it for consistency, see Chapter \ref{conclusion_parameter_selection}.
To properly quantify synchronisation we considered the histograms of the escape times, which to our knowledge has been not considered thoroughly before. 
The histograms showed a clear periodicity and also the emergence of peaks at the Double Frequency for increasing angle. 
For a quantitative consideration we assume that the entrance time is in perfect phase 
(this is when $m_-(u)$ and $m_+(u)$ \emph{can} be approximated by Dirac delta functions). 
This gives for several cases good quantitative and in general good qualitative agreement with the combined adiabatic and small noise approximation. A more sophisticated analysis based on a Kolmogorov-Smirnov test developed here shows that this approximation works for a larger range of parameters where the approximation of the perfect phase of the entrance time is not appropriate
(this is when $m_-(u)$ and $m_+(u)$ \emph{cannot} be approximated by Dirac delta functions)
see Chapter \ref{conclusion_sparse_no_good_pdf}. 
Summarizing, the theoretical and the simulation results are in very good agreement.
We want to stress that in the comparison no free parameters were present and so no fitting took place. 

The fact that the six measures are blind can be explained using Markov chain models approximating the SDE. As one expects from large deviation theory, for small noise and in an adiabatic regime the SDE can be approximated by a continuous time Markov chain.
In this Markov chain model we showed that the invariant measures are constant when $\phi=90^\circ$. 
Hence we expect that the invariant measure gives in the diffusion case equal weights to the left and the right well. 
Together, this gives us the following qualitative picture of the dynamics for any angle. 
At a fixed time the probability that one sees a jump from the left to the right well or vice versa has the same probability. 
However, conditioned on the phase and the direction of the last jump, for concreteness assume that it was at phase $u$ and from the left to the right
(that is to say the particle entered the well at time $u$)
the next jump will be at  phase which is near to a multiple of $T/2$
(that is to say the particle will leave the well near the times $t=nT/2$ where $n$ is an integer). 
The jump rates will be given by the height of the potential barriers.

At $\phi=90^\circ$, the path $X_t^\varepsilon$ and $-X_t^\varepsilon$ will appear with the same probability if one starts in the invariant measure. 
This explains why the six measures are all insensitive in this case. 
The equilibration happens because the process will skip some of the jump opportunities and in this way the left-right synchronization will get lost quickly.

This new phenomena we discovered has added an additional motivation to the observation of Hermann, Imkeller, Pavlyukevich, Berglund and Gentz that the appropriate 
consideration has to be on the path level. 
Averaged quantities like the six measures can be very misleading and masking the real behaviour of the system. 
The escape time distribution shows a clear signal of stochastic resonance in accordance with the theoretical consideration. 
The presence of a two pathways manifests itself in an appearance of peaks at the Double Frequency. 
We showed that adiabatic small noise approximation gives a good statistical model.
We demonstrated that this appearance can be detected also when only a limited number of transitions is available. 
Our analysis provides us with a clear footprint indicating the existence of a second pathway. 
The angle dependence of our result should also allow us to predict the orientation of the saddles with respect to the wells.

\section*{Further Studies}
\addcontentsline{toc}{section}{Further Studies}

The invariant measures studied in this thesis are for a simple two state model.
One could try to generalise this to continuous states, that is a space-time phase PDF for the position of the particle could be derived.

The conditional KS test gives us confidence that one could develop statistical inference, using maximum likelihood for example, to develop a statistic test to estimate the basic parameters of the system, if they are unknown to us.
Instead of the approximation $p_{tot}\approx p_+(t,0)$ used in parts of the consideration, a better approximation may be found by studying the PDFs of $m_-(u)$ and $m_+(u)$ theoretically and statistically. 
 

A theory beyond adiabatic approximation may be developed for very slow to fast frequencies. 
Higher order of approximation to the escape times than Kramers' formula could be studied. 
Analytic and theoretical developments to go beyond adiabatic approximation and potential theory may be a real mathematical challenge.
But experimental simulations may provide an idea of what this new theory may be like.

\appendix 


\chapter{Conventions in Defining the SDEs, Potential, Time Dependency and Forcing}\label{appendix_potential}
This is more of a clarification on the notation being used. 
In this thesis only two toy model potentials are studied. These in their most unperturbed forms are denoted by 
\begin{align*}
 V_0(x)&=\frac{x^4}{4}-a\frac{x^2}{2}\\
V_0(x,y)&=\frac{1}{4}r^4-\frac{1}{2}r^2-ax^2+by^2
\end{align*}
where $r=\sqrt{x^2+y^2}$, $a>0$ and $b>0$ 
which when perturbed by a force are denoted by 
\begin{align*}
V_F(x)&=\frac{x^4}{4}-a\frac{x^2}{2}+Fx\\
&=V_0(x)+Fx\\
V_F(x,y)&=\frac{1}{4}r^4-\frac{1}{2}r^2-ax^2+by^2+F_xx+F_yy\\
&=V_0(x,y)+F_xx+F_yy\\
&=V_0(x,y)+\mathbf{F}\cdot\mathbf{x}
\end{align*}
and when given a periodic forcing are denoted by 
\begin{align*}
V_t(x,F)&=\frac{x^4}{4}-a\frac{x^2}{2}-Fx\cos\Omega t\\
&=V_0(x)-Fx\cos\Omega t\\
V_t(x,y,F_x,F_y)&=\frac{1}{4}r^4-\frac{1}{2}r^2-ax^2+by^2-F_xx\cos \Omega t-F_yy\cos \Omega t\\
&=V_0(x,y)-F_xx\cos \Omega t-F_yy\cos \Omega t\\
&=V_0(x,y)-\mathbf{F}\cdot\mathbf{x}\cos \Omega t.
\end{align*}
This is so that the SDEs can be written in the form
\begin{align*}
dX^\epsilon_t=-\nabla V_t\,dt+\epsilon dw 
\end{align*}
which when expanded can be written as 
\begin{align*}
dx&=\left[-\frac{\partial V_0}{\partial x}+F_x\cos \Omega t \ \right]dt+\epsilon \ dw_x\\
dy&=\left[-\frac{\partial V_0}{\partial y}+F_y\cos \Omega t \ \right]dt+\epsilon \ dw_y
\end{align*}
meaning more details about the system can be quickly seen in the notation. 
This also implies that the SDEs are always defined with a negative forcing. 
When the critical points of the system are being studied (in Chapter \ref{chap_mexican_hat_toy_model} for example) we can study the critical points with a positive force and $V_F$ would be an appropriate notation to use. 
Using $V_0$, $V_F$ and $V_t$ may seem like an abuse of notation, but if anything specific is being referred to, we can denote
$V_{F=F^{crit}}$ for example. 
When the most general expression for a potential $V$ is being used, it should be deduced from context whether $V=V_0$, $V=V_F$ or $V=V_t$ is being referred to. 

Note also that for a stochastic process in $\mathbb{R}^r$ which is described by the SDE
\begin{align*}
\dot{X^\epsilon_t}=-\nabla V+F\cos(\Omega t)+\epsilon \dot{W_t}
\end{align*}
and the magnitude of the forcing is sometimes denoted by 
\begin{align*}
F=\sqrt{F_1^2+F_2^2+\ldots+F_r^2}. 
\end{align*}
Again this may seem like an abuse of notation, but it should be clear from context whether $F$ is a vector or scalar.

\chapter{Further Numerical Methods}
\label{append_further_methods}

\section{Numerical Methods for measuring Escape Times}\label{appendix_escape_times}
The Markov Chain takes the values $Y^\epsilon_t=\pm1$. 
But in simulations time is discrete with a time step $t_{step}$, that is 
\begin{align*}
0, t_{step}, 2t_{step}, \ldots, Nt_{step}.
\end{align*}
The reduction from the diffusion $X^\epsilon_t$ to the Markov Chain at the $(n+1)$th time step is actually given by 
\begin{align*}
Y^\epsilon_{(n+1)t_{step}}
=
\left\{
\begin{array}{ccc}
-1 & \text{if} & \left|X^\epsilon_{nt_{step}}-w_l(nt_{step})\right|<R\\[0.8em]
+1 & \text{if} & \left|X^\epsilon_{nt_{step}}-w_r(nt_{step})\right|<R\\[0.8em]
Y_{nt_{step}}^\epsilon&\quad\text{if otherwise}
\end{array}
\right. 
\end{align*}
which is slightly different from the way $Y^\epsilon
_t$ was defined in Chapter \ref{chapter_oscil_times} (see page \pageref{chapter_oscil_times_R}). 
This is so that the definition of $Y^\epsilon_t$ was easier to write down theoretically, such that the sets
\begin{align*}
\left\{t:\left|X^\epsilon_t-w_l(t)\right|\leq R\right\}
\quad \text{and} \quad 
\left\{t:\left|X^\epsilon_t-w_r(t)\right|\leq R\right\}
\end{align*}
are compact sets given the continuity of $X^\epsilon
_t$, $w_l(t)$ and $w_r(t)$.  
This meant 
\begin{align*}
Y^\epsilon_t=\left\{
\begin{array}{ccc}
-1&\text{if}&\left|X^\epsilon_t-w_l(t)\right|\leq R\\[0.5em]
+1&\text{if}&\left|X^\epsilon_t-w_r(t)\right|\leq R\\[0.5em]
Z&\text{if neither}&
\end{array}
\right.
\end{align*}
then $Y^\epsilon_t$ would be easier to define for $t\notin \left\{t:\left|X^\epsilon_t-w_l(t)\right|\leq R\right\} \cup \left\{t:\left|X^\epsilon_t-w_r(t)\right|\leq R\right\}$. 
But alternatively if we had 
\begin{align*}
Y^\epsilon_t=\left\{
\begin{array}{ccc}
-1&\text{if}&\left|X^\epsilon_t-w_l(t)\right|< R\\[0.5em]
+1&\text{if}&\left|X^\epsilon_t-w_r(t)\right|< R\\[0.5em]
Z&\text{if neither}&
\end{array}
\right.
\end{align*}
then 
\begin{align*}
\left\{t:\left|X^\epsilon_t-w_l(t)\right|< R\right\}
\quad \text{and} \quad 
\left\{t:\left|X^\epsilon_t-w_r(t)\right|< R\right\}
\end{align*}
would be open sets and $Y^\epsilon_t$ would be harder to define for $t\notin \left\{t:\left|X^\epsilon_t-w_l(t)\right|< R\right\} \cup \left\{t:\left|X^\epsilon_t-w_r(t)\right|< R\right\}$ which is the neither case.  
Nevertheless the simulations should gloss out all these details.

\section{Numerical Methods for calculating Fourier Transform and Linear Response}\label{appendix_linear_response}
Fourier Transforms are involved in finding the linear response. 
The trajectory of the particle is in theory a continuous object, but in practice when simulations are done it is a finite discrete object.
The exact mechanism of obtaining the linear response from a simulated trajectory is now being discussed. 

When the  trajectory is being numerically realised it is a finite discrete set. Let the $x$ (or $y$) coordinate of the particle at time $nt_{step}$ where $0\leq n \leq (N-1)t_{step}$ be denoted by $X_{nt_{step}}$. This gives rise to the set 
\begin{align*}
X&=\left\{
X_{0}, X_{t_{step}}, X_{2t_{step}}, X_{3t_{step}}, \ldots, X_{(N-1)t_{step}}
\right\}\\
&=\left\{
x_0, x_1, x_2, x_3, \ldots, x_{N-1}
\right\}
\end{align*}
where $x_n=X_{nt_{step}}$ etc. 
Notice that time is discrete here. 
When this is Discrete Fourier Transformed (being quickly implemented by the Fast Fourier Transform algorithm) it is denoted by 
\begin{align*}
\tilde{X}&=\left\{
\tilde{X}_{0}, \tilde{X}_{\omega_{step}}, \tilde{X}_{2\omega_{step}}, \tilde{X}_{3\omega_{step}}, \ldots, \tilde{X}_{(N-1)\omega_{step}}
\right\}\\
&=\left\{
\tilde{x}_0, \tilde{x}_1, \tilde{x}_2, \tilde{x}_3, \ldots, \tilde{x}_{N-1}
\right\}
\end{align*}
where $\tilde{x}_n=\tilde{X}_{n\omega_{step}}$ etc and the transform is given by 
\begin{align*}
\tilde{x}_k=\sum_{n=0}^{N-1}x_n e^{-2\pi i kn/N}
\end{align*}
and the following relation is used
\begin{align*}
\omega_{step}=\frac{1}{(N-1)t_{step}}
\end{align*}
which is the highest detectable frequency divided by the number of steps. If we want to find the linear response at driving frequency $\Omega$, then $\Omega$ needs to be approximated by a finite number of $\omega_{step}$ as in 
\begin{align*}
\frac{\Omega}{2\pi}\approx n \omega_{step}
\end{align*}
and the linear response at this driving frequency is then given by 
\begin{align*}
X^\Omega_{lin}=2\times\left|\tilde{X}_{n\omega_{step}}\right|.
\end{align*}
Notice the factor of $2$ being used here.
Suppose that the trajectory can be approximated by
\begin{align*}
X_t^\epsilon\approx A\cos(\Omega t +\phi)
\end{align*}
then a good approximate expression for $A$ and $\phi$ would be 
\begin{align*}
A\approx X^\Omega_{lin}
\quad\text{and}\quad
\phi\approx\text{arg} \left( \tilde{X}_{n\omega_{step}} \right )=\tan^{-1}
\left\{
\frac{\text{Im}\left(\tilde{X}_{n\omega_{step}}\right)}{\text{Re}\left(\tilde{X}_{n\omega_{step}}\right)}
\right\}
\end{align*}
where $\phi$ is the angle of the complex number $\tilde{X}_{n\omega_{step}}$.

\section{Numerical Methods for calculating $M_5$ and $M_6$}
\label{appendix_six_measures}
Here we present how we computed $M_5$ and $M_6$ numerically.
This is how $M_5$ and $M_6$ are calculated in theory 
\begin{align*}
M_5&=\int_0^T
\phi^-(t)\ln\left(\frac{\phi^-(t)}{\overline{\nu}_-(t)}\right)+
\phi^+(t)\ln\left(\frac{\phi^+(t)}{\overline{\nu}_+(t)}\right)
dt\\
M_6&=\int^T_0
-\overline{\nu}_-(t)\ln\overline{\nu}_-(t)
-\overline{\nu}_+(t)\ln\overline{\nu}_+(t)\,
dt
\end{align*}
where 
\begin{align*}
\phi^-(t)&=
\left\{
\begin{array}{c}
1 \quad \text{if} \quad mod(t,T)\leq T/2\\
0 \quad \text{if} \quad mod(t,T)> T/2
\end{array}
\right.\\[0.5em]
\phi^+(t)&=
\left\{
\begin{array}{c}
0 \quad \text{if} \quad mod(t,T)\leq T/2\\
1 \quad \text{if} \quad mod(t,T)> T/2. 
\end{array}
\right.
\end{align*}
When the invariant measures are generated numerically they are finite discrete objects described by 
\begin{align*}
\nu_-&=\left\{\nu^-_1, \nu^-_2, \cdots, \nu^-_N\right\}\\
\nu_+&=\left\{\nu^+_1, \nu^+_2, \cdots, \nu^+_N\right\}.
\end{align*}
The real invariant measure were close to zero sometimes and in the numerical approximation they became actually zero or even negative which lead to numerical artefacts.
Note that 
\begin{align*}
\lim_{x\longrightarrow0}\ln\left(\frac{1}{x}\right)=\infty
\quad \text{and} \quad 
\lim_{x\longrightarrow0}x\ln\left(x\right)=0.
\end{align*}
Define
\begin{align*}
\nu_-^{lim}&=\min_{\substack{i=1,2,\ldots,N\\\nu^-_i>0}}\left\{\nu^-_1, \nu^-_2, \cdots, \nu^-_N\right\}\\
\nu_+^{lim}&=\min_{\substack{i=1,2,\ldots,N\\\nu^+_i>0}}\left\{\nu^+_1, \nu^+_2, \cdots, \nu^+_N\right\}. 
\end{align*}
The quantities $M_5$ and $M_6$ are computed numerically in the following way  
\begin{align*}
M_5&=\sum_{\substack{i\leq\frac{N}{2}\\\nu^-_i>0}}t_{step}\ln\left(\frac{1}{\nu^-_i}\right)
+
\sum_{\substack{i\leq\frac{N}{2}\\\nu^-_i\leq0}}t_{step}\ln\left(\frac{1}{\nu_-^{lim}}\right)
+
\sum_{\substack{i>\frac{N}{2}\\\nu^+_i>0}}t_{step}\ln\left(\frac{1}{\nu^+_i}\right)
+
\sum_{\substack{i>\frac{N}{2}\\\nu^+_i\leq0}}t_{step}\ln\left(\frac{1}{\nu^{lim}_+}\right)\\[0.5em]
M_6&=\sum_{\substack{i=1,2,\cdots,N\\\nu^-_i>0}}\nu^-_i\ln(\nu^-_i)(-t_{step})\quad+\quad
\sum_{\substack{i=1,2,\cdots,N\\\nu^+_i>0}}\nu^+_i\ln(\nu^+_i)(-t_{step}).
\end{align*}

\chapter{Further Commentary on Sparse Data Analysis}

\section{Examples of Oversampling}
\label{appendix_over_sample}
Subjectively one may think that Figures 
\ref{chap_8_g80_p87_e21_m2} and \ref{chap_8_g80_p90_e21_m1}
are so bad the conditional null hypothesis may be rejected. 
This is actually an example of oversampling, where too many transitions were used in the implementation of the conditional KS test. 
We know the PDF we are fitting is not the real PDF but an approximation in the limit of small noise and adiabatic forcing. Hence if one has enough data points this should be picked up and the conditional KS test will refuse the approximate PDF as it will pick up even slight deviation from the real PDF.
When $n=20$ are used we have the following. 
\begin{figure}[H]
\centerline{\includegraphics[scale=0.35]{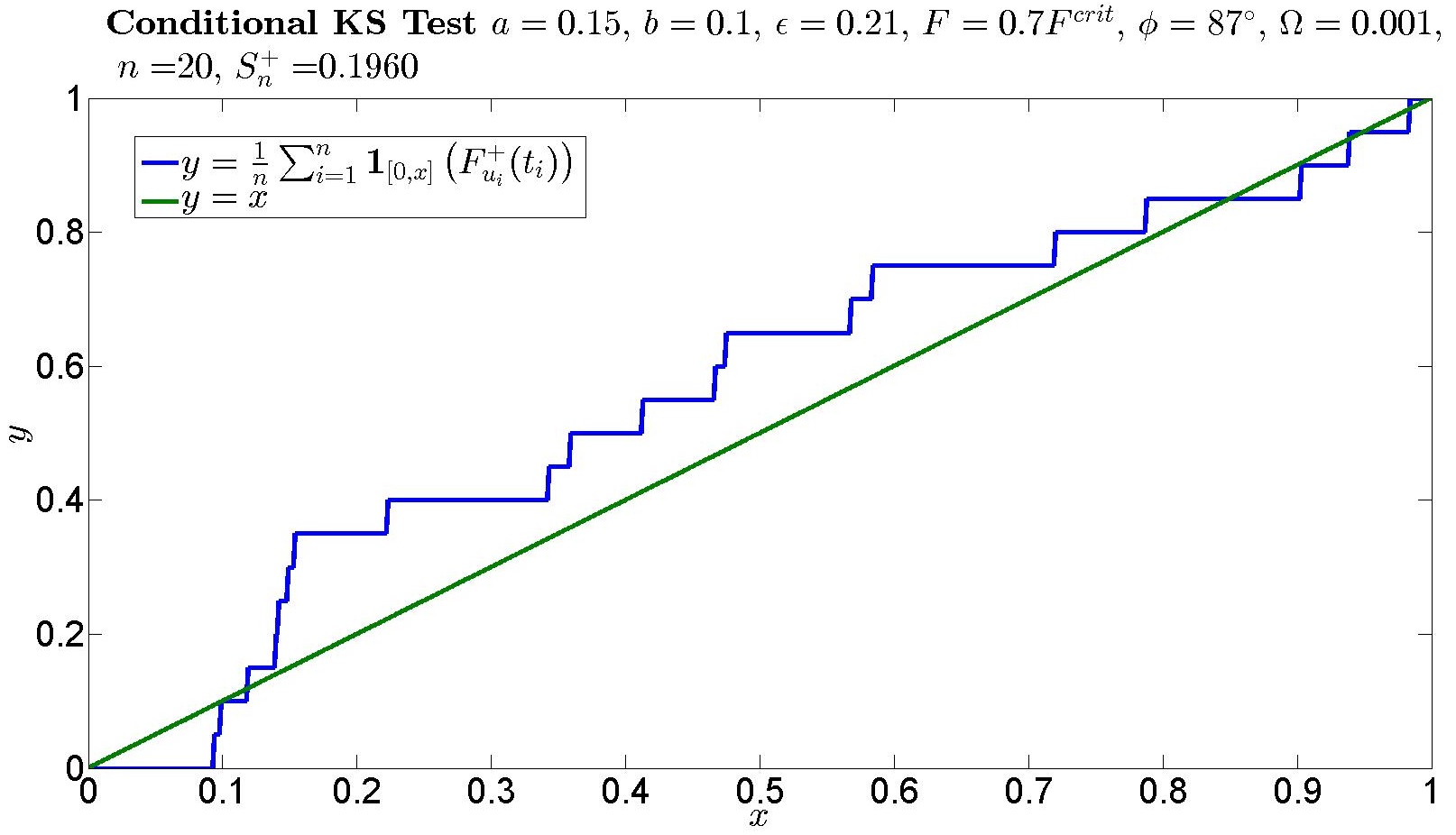}}
\caption{
This is Figure \ref{chap_8_g80_p87_e21_m2} redone with 20 transitions. 
Note that $\epsilon=0.21$, $\phi=87^\circ$, $n=20$, $S^+_n=0.1960$.
}\label{chap_8_g81_p87_e21_m2}
\end{figure}

\begin{figure}[H]
\centerline{\includegraphics[scale=0.35]{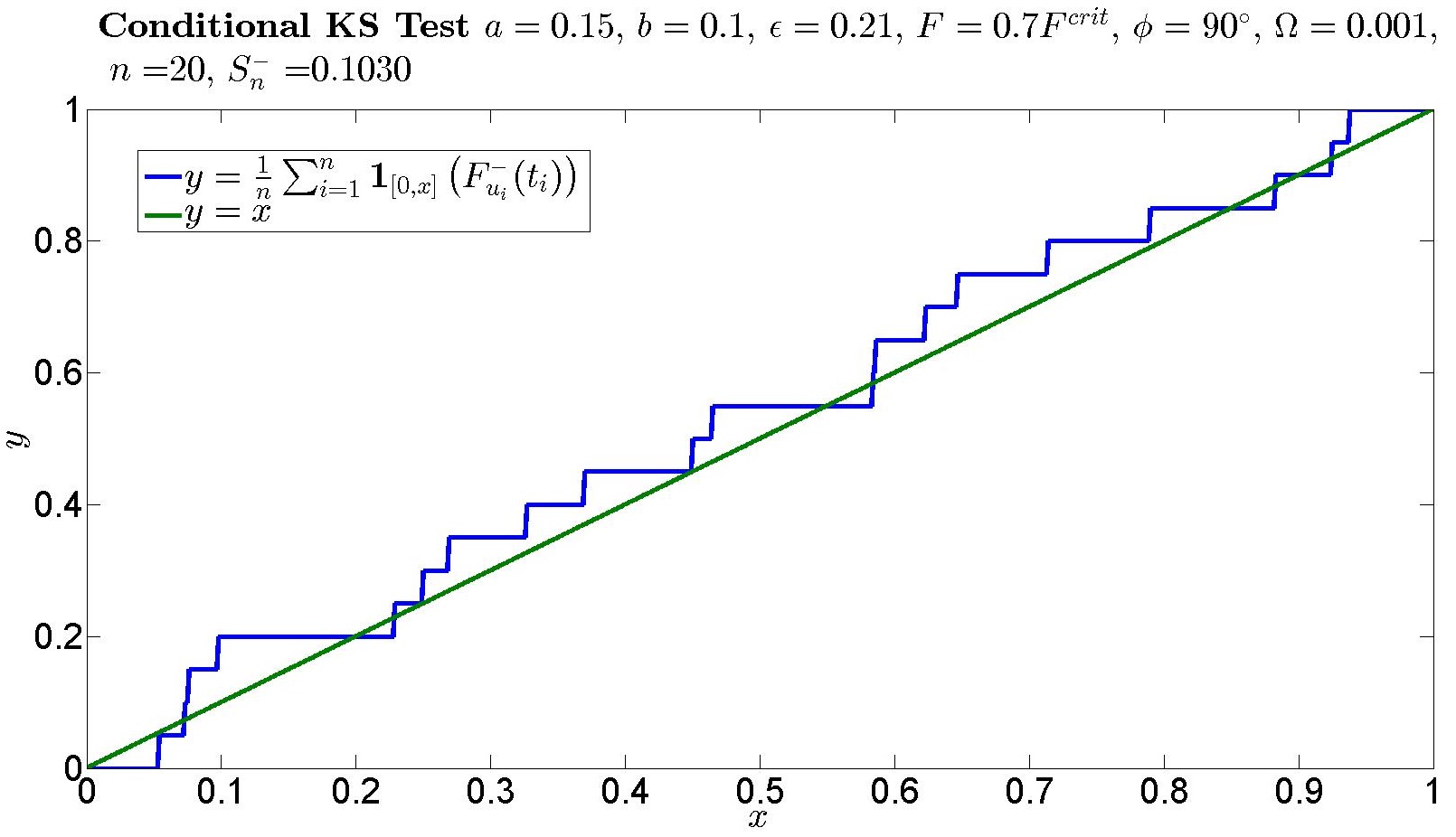}}
\caption{
This is Figure \ref{chap_8_g80_p90_e21_m1} redone with 20 transitions. 
Note that $\epsilon=0.21$, $\phi=90^\circ$, $n=20$, $S^-_n=0.1030$.
}\label{chap_8_g81_p90_e21_m1}
\end{figure}

\section{Empirical CDF}
Consider Figure \ref{chap_8_g81_p81_e17_m2}. Notice that the empirical CDF is on top on the $y=x$ line. There were enough data to give 10 more realisations of the random variable $S_n^+$. Note that all ten of these $S_n^+$ with $n=20$ were calculated from 200 transitions divided into ten sets for the ten $S_n^+$. This meant 10 more versions of the Figure \ref{chap_8_g81_p81_e17_m2} were plotted. Out of these 10 plots, one had the empirical CDF to the bottom of the $y=x$ line and one had roughly half the empirical CDF above and below the $y=x$ line. The noise level was very low at $\epsilon=0.17$, which meant the escape times were very long with a very large spread, which gave rise to data looking unreasonable. Only 200 transitions were detected which is significantly less than other parameters, which meant only 10 realisations of the $S_n^+$ random variable was possible. No further conclusions are drawn here.


\newpage
\addcontentsline{toc}{chapter}{References}
\bibliographystyle{ieeetr}
\bibliography{reportreferences}


\end{document}